\documentclass[a4paper,10pt,final ]{amsart}
\usepackage[utf8x]{inputenc}
\usepackage{fouriernc, amsmath, amsthm, amsfonts,stmaryrd, amssymb,epsfig,enumerate, wrapfig, scalerel, tikz,  color, tensor, accents, fancybox, rotating, comment, afterpage, changepage, colortbl,  pdfpages, datetime, fancyhdr, float, xlop}
\usepackage{hyperref}
\hypersetup{
    pdftoolbar=true,        
    pdfmenubar=true,        
    pdffitwindow=false,     
    pdfstartview={FitH},    
    pdfauthor={Louis-Hadrien Robert and Emmanuel Wagner},     
    pdfsubject={Symmetric Khovanov--Rozansky link homologies},
    pdftitle={Symmetric Khovanov--Rozansky link homologies},   
    pdfcreator={Louis-Hadrien Robert and Emmanuel Wagner},   
    pdfproducer={Louis-Hadrien Robert and Emmanuel Wagner}, 
    pdfkeywords={}, 
    pdfnewwindow=true,      
    colorlinks=true,       
    linkcolor=red,          
    citecolor=teal,        
    filecolor=magenta,      
    urlcolor=violet,          
    linkbordercolor=red,
    citebordercolor=teal,
    urlbordercolor=violet,  
    linktocpage=true
}
\usepackage{setspace}
\usepackage{showkeys}
\usetikzlibrary{arrows}
\usetikzlibrary{decorations.markings}
\usetikzlibrary{decorations}
\usetikzlibrary{patterns}
\usetikzlibrary{matrix}
\usepackage{tikz-3dplot}
\newcounter {res}[section]
\numberwithin{res}{section}
\newtheorem{thm}[res]{Theorem}
\newtheorem*{theo}{Theorem}
\newtheorem*{ques}{Question}

\newtheorem{lem}[res]{Lemma}
\newtheorem{prop}[res]{Proposition}
\newtheorem{cor}[res]{Corollary}

\theoremstyle{definition}
\newtheorem{notation}[res]{Notation}
\newtheorem{dfn}[res]{Definition}
\newtheorem{rmk}[res]{Remark}
\newtheorem{exa}[res]{Example}

\newtheorem{cjc}[res]{Conjecture}
\newcommand{\FA}{\mathcal{T}}
\newcommand{\NN}{\ensuremath{\mathbb{N}}} 
\newcommand{\ZZ}{\ensuremath{\mathbb{Z}}} 
\newcommand{\CC}{\ensuremath{\mathbb{C}}} 
\newcommand{\QQ}{\ensuremath{\mathbb{Q}}}
\newcommand{\RR}{\ensuremath{\mathbb{R}}} 

\newcommand{\PP}{\ensuremath{\mathbb{P}}}

\renewcommand{\SS}{\ensuremath{\mathbb{S}}} 

\newcommand {\Id}{\operatorname{Id}}
\newcommand{\SP}[1]{\ensuremath{R_{#1}}}

\newcommand{\Ker}{\mathop{\mathrm{Ker}}\nolimits}
\newcommand{\Hom}{\mathop{\mathrm{Hom}}}
\newcommand{\rY}{\ensuremath{\rotatebox[origin=c]{180}{\NB{$Y$}}}}

\newcommand{\id}{\mathrm{Id}}
\newcommand{\rot}{\mathrm{rot}}
\newcommand{\rk}{\mathrm{rk}}
\newcommand{\ie}{i.~e.~}
\newcommand{\F}{\mathcal{F}}
\newcommand{\listk}[1]{\ensuremath{\underline{#1}}}

\newcommand{\degext}{\ensuremath{\mathrm{deg}^{\Lambda}}}
\newcommand{\degD}{\ensuremath{\mathrm{deg}^{D}}}
\newcommand{\degR}{\ensuremath{\mathrm{deg}^{r}}}
\newcommand{\degT}{\ensuremath{\mathrm{deg}^{T}}}

\newcommand{\du}{\ensuremath{\boldsymbol{\bigsqcup}}} 

\newcommand{\syf}{\ensuremath{\mathcal{S}}} 
\let\pipi\pi
\renewcommand{\pi}{s}

\renewcommand\marginpar[1]{}

\newcommand\eqdef{\ensuremath{\stackrel{\textrm{def}}{=}}}
\newcommand\kup[1]{\left\langle #1 \right\rangle}
\input{llanglerrangle}
\newcommand\kups[1]{\left\llangle #1 \right\rrangle_N}
\newcommand\kupss[1]{\left\llangle #1 \right\rrangle}

\newcommand\kupsc[1]{\left\llangle #1 \right\rrangle^{\mathrm{sc}}}
\newcommand{\TL}{\ensuremath{\mathsf{TLF}}}
\newcommand{\DLF}{\ensuremath{\mathsf{DLF}}}
\newcommand{\BimS}{\ensuremath{\mathsf{BimS}}}

\newcommand{\Foam}{\ensuremath{\mathsf{Foam}}}
\newcommand{\BS}{\ensuremath{\mathcal{B}}}
\newcommand{\IE}{\ensuremath{\mathcal{F}_\infty}}
\newcommand{\IED}{\ensuremath{\mathcal{F}^D_\infty}}
\newcommand{\IET}{\ensuremath{\mathcal{F}^T_\infty}}

\newcommand{\qbina}[2]{\ensuremath
\begin{bmatrix}
  #1 \\
 #2
\end{bmatrix}
}
\newcommand{\qbinil}[2]{\ensuremath
\left[\begin{smallmatrix}
  #1 \\
 #2
\end{smallmatrix} \right]
}
\newcommand{\qbin}[2]{\ensuremath
\begin{bmatrix}
  #1 + #2 \\
#1 \quad #2
\end{bmatrix}
}

\newcommand{\imagesfolder}{.}
\newcommand{\gll}{\ensuremath{\mathfrak{gl}}}
\newcommand{\sll}{\ensuremath{\mathfrak{sl}}}

\newcommand{\Xing}{\NB{\scalebox{1.3}{\ensuremath{\times}}}}
\newcommand{\Col}{\ensuremath{\mathbb{P}}}
\newcommand{\NB}[1]{\ensuremath{\vcenter{\hbox{#1}}}}

\newcommand{\HH}{\ensuremath{H\!\!H}}
\newcommand{\HN}{\ensuremath{H\!\!N}}
\newcommand{\HT}{\ensuremath{H\!T}}
\newcommand{\HNT}{\ensuremath{H\!\!N\!T}}
\newcommand{\KH}{\ensuremath{K\!\!H}}

\newcommand{\MS}{\ensuremath{\mathcal{MS}}}

\newcommand{\Sym}{\ensuremath{\mathrm{Sym}}}
\newcommand{\ann}{\ensuremath{\mathcal{A}}}
\newcommand{\Vin}{\ensuremath{\mathcal{V}}}

\def\co{\colon\thinspace}
\title{Symmetric Khovanov--Rozansky link homologies}
\newcommand{\digona}{\ensuremath{\vcenter{\hbox{\tikz[scale=0.4]{
\coordinate (B) at (0,0);
\coordinate (V1) at (0,0.5);
\coordinate (V2) at (0,2.5);
\coordinate (T) at (0,3);
\draw[white] (0, -0.5) -- (0, 3.5);
\draw[->] (B) -- (V1) node[midway, left] {\tiny{$m+n$}};
\draw[->] (V2) -- (T) node[midway, left] {\tiny{$m+n$}};
\draw[->] (V1) .. controls +(+0.5, +0.5) and +(+0.5, -0.5).. (V2) node[midway, right] {\tiny{$n$}};
\draw[->] (V1) .. controls +(-0.5, +0.5) and +(-0.5, -0.5).. (V2) node[midway, left] {\tiny{$m$}};
}}}}}

\newcommand{\verta}{\ensuremath{\vcenter{\hbox{\tikz[scale=0.4]{
\coordinate (B) at (0,0);
\coordinate (T) at (0,3);
\draw[white] (0, -0.5) -- (0, 3.5);
\draw[->] (B) -- (T) node[midway, right] {\tiny{$m+n$}};
}}}}}

\newcommand{\digonb}{\ensuremath{\vcenter{\hbox{\tikz[scale=0.4]{
\coordinate (B) at (0,0);
\coordinate (V1) at (0,0.5);
\coordinate (V2) at (0,2.5);
\coordinate (T) at (0,3);
\draw[white] (0, -0.5) -- (0, 3.5);
\draw[->] (B) -- (V1) node[midway, left] {\tiny{$m$}};
\draw[->] (V2) -- (T) node[midway, left] {\tiny{$m$}};
\draw[<-] (V1) .. controls +(+0.5, +0.5) and +(+0.5, -0.5).. (V2) node[midway, right] {\tiny{$n$}};
\draw[->] (V1) .. controls +(-0.5, +0.5) and +(-0.5, -0.5).. (V2) node[midway, left] {\tiny{$m+n$}};
}}}}}

\newcommand{\vertb}{\ensuremath{\vcenter{\hbox{\tikz[scale=0.4]{
\coordinate (B) at (0,0);
\coordinate (T) at (0,3);
\draw[white] (0, -0.5) -- (0, 3.5);
\draw[->] (B) -- (T) node[midway, right] {\tiny{$m$}};
}}}}}

\newcommand{\stgamma}{\ensuremath{\vcenter{\hbox{\tikz[scale=0.3]{
\coordinate (B) at (0,0);
\coordinate (V1) at (0,1);
\coordinate (V2) at (1,2);
\coordinate (T1) at (-2,3);
\coordinate (T2) at (0,3);
\coordinate (T3) at (2,3);
\draw[>-] (B) -- (V1) node [at start, below] {\tiny{$i+j+k$}};
\draw[->] (V1) -- (T1) node [at end, above] {\tiny{$i$}};
\draw[->] (V1)  -- (V2) node[midway, right] {\tiny{$j+k$}};
\draw[->] (V2) -- (T2) node[at end, above] {\tiny{$j$}};
\draw[->] (V2) -- (T3) node[at end, above] {\tiny{$k$}};
}}}}}

\newcommand{\stgammaprime}{\ensuremath{\vcenter{\hbox{\tikz[scale=0.3]{
\coordinate (B) at (0,0);
\coordinate (V1) at (0,1);
\coordinate (V2) at (-1,2);
\coordinate (T1) at (-2,3);
\coordinate (T2) at (0,3);
\coordinate (T3) at (2,3);
\draw[>-] (B) -- (V1) node [at start, below] {\tiny{$i+j+k$}};
\draw[->] (V1) -- (T3) node [at end, above] {\tiny{$k$}};
\draw[->] (V1)  -- (V2) node[midway, left] {\tiny{$i+j$}};
\draw[->] (V2) -- (T1) node[at end, above] {\tiny{$i$}};
\draw[->] (V2) -- (T2) node[at end, above] {\tiny{$j$}};
}}}}}

\newcommand{\stgammar}{\ensuremath{\vcenter{\hbox{\tikz[scale=0.3]{
\coordinate (B) at (0,0);
\coordinate (V1) at (0,1);
\coordinate (V2) at (1,2);
\coordinate (T1) at (-2,3);
\coordinate (T2) at (0,3);
\coordinate (T3) at (2,3);
\draw[<-] (B) -- (V1) node [at start, below] {\tiny{$i+j+k$}};
\draw[-<] (V1) -- (T1) node [at end, above] {\tiny{$i$}};
\draw[-<] (V1)  -- (V2) node[midway, right] {\tiny{$j+k$}};
\draw[-<] (V2) -- (T2) node[at end, above] {\tiny{$j$}};
\draw[-<] (V2) -- (T3) node[at end, above] {\tiny{$k$}};
}}}}}

\newcommand{\stgammaprimer}{\ensuremath{\vcenter{\hbox{\tikz[scale=0.3]{
\coordinate (B) at (0,0);
\coordinate (V1) at (0,1);
\coordinate (V2) at (-1,2);
\coordinate (T1) at (-2,3);
\coordinate (T2) at (0,3);
\coordinate (T3) at (2,3);
\draw[<-] (B) -- (V1) node [at start, below] {\tiny{$i+j+k$}};
\draw[-<] (V1) -- (T3) node [at end, above] {\tiny{$k$}};
\draw[-<] (V1)  -- (V2) node[midway, left] {\tiny{$i+j$}};
\draw[-<] (V2) -- (T1) node[at end, above] {\tiny{$i$}};
\draw[-<] (V2) -- (T2) node[at end, above] {\tiny{$j$}};
}}}}}
\newcommand{\squarea}{\ensuremath{\vcenter{\hbox{\tikz[scale=0.4]{
\coordinate (B1) at (-1,0);
\coordinate (B2) at (1,0);
\coordinate (C1) at (-1,1);
\coordinate (D1) at (-1,2);
\coordinate (C2) at (1,1);
\coordinate (D2) at (1,2);
\coordinate (T1) at (-1,3);
\coordinate (T2) at (1,3);
\draw[->] (B1) -- (C1) node[at start, below] {\tiny{$1$}};
\draw[->] (D1) -- (C1) node[midway, left   ] {\tiny{$m$}};
\draw[->] (D1) -- (T1) node[at end , above ] {\tiny{$1$}};
\draw[->] (C2) -- (B2) node[at end, below] {\tiny{$m$}};
\draw[->] (C2) -- (D2) node[midway, right] {\tiny{$1$}};
\draw[->] (T2) -- (D2) node[at start, above] {\tiny{$m$}};
\draw[->] (D2) -- (D1) node[midway, above] {\tiny{$m+1$}};
\draw[->] (C1) -- (C2) node[midway, below] {\tiny{$m+1$}};
}}}}}

\newcommand{\twoverta}{\ensuremath{\vcenter{\hbox{\tikz[scale=0.4]{
\coordinate (B1) at (-1,0);
\coordinate (T1) at (-1,3);
\coordinate (B2) at (1,0);
\coordinate (T2) at (1,3);
\draw[->] (B1) -- (T1) node[midway, left] {\tiny{$1$}};
\draw[->] (T2) -- (B2) node[midway, right] {\tiny{$m$}};
}}}}}

\newcommand{\doubleYa}{\ensuremath{\vcenter{\hbox{\tikz[scale=0.4]{
\coordinate (B1) at (-1,0);
\coordinate (T1) at (-1,3);
\coordinate (C) at (0,1);
\coordinate (D) at (0,2);
\coordinate (B2) at (1,0);
\coordinate (T2) at (1,3);
\draw[->] (B1) -- (C) node[at start, below] {\tiny{$1$}};
\draw[->] (C) -- (B2) node[at end, below] {\tiny{$m$}};
\draw[->] (D) -- (C) node[midway, left] {\tiny{$m-1$}};
\draw[->] (T2) -- (D) node[at start, above] {\tiny{$m$}};
\draw[->] (D) -- (T1) node[at end, above] {\tiny{$1$}};
}}}}}

\newcommand{\squareb}{\ensuremath{\vcenter{\hbox{\tikz[scale=0.55]{
\coordinate (B1) at (-1,0);
\coordinate (B2) at (1,0);
\coordinate (C1) at (-1,1);
\coordinate (D1) at (-1,2);
\coordinate (C2) at (1,1);
\coordinate (D2) at (1,2);
\coordinate (T1) at (-1,3);
\coordinate (T2) at (1,3);
\draw[->] (B1) -- (C1) node[at start, below] {\tiny{$1$}};
\draw[->] (C1) -- (D1) node[midway, left   ] {\tiny{$l+n$}};
\draw[->] (D1) -- (T1) node[at end , above ] {\tiny{$l$}};
\draw[->] (B2) -- (C2) node[at start, below] {\tiny{$m+l-1$}};
\draw[->] (C2) -- (D2) node[midway, right] {\tiny{$m-n$}};
\draw[->] (D2) -- (T2) node[at end, above] {\tiny{$m$}};
\draw[->] (D1) -- (D2) node[midway, above] {\tiny{$n$}};
\draw[->] (C2) -- (C1) node[midway, below] {\tiny{$l+n-1$}};
}}}}}

\newcommand{\squarec}{\ensuremath{\vcenter{\hbox{\tikz[xscale=0.65, yscale=0.55]{
\coordinate (B1) at (-1,0);
\coordinate (B2) at (1,0);
\coordinate (C1) at (-1,1);
\coordinate (D1) at (-1,2);
\coordinate (C2) at (1,1);
\coordinate (D2) at (1,2);
\coordinate (T1) at (-1,3);
\coordinate (T2) at (1,3);
\draw[->] (B1) -- (C1) node[at start, below] {\tiny{$n$}};
\draw[->] (C1) -- (D1) node[midway, left   ] {\tiny{$n+k $}};
\draw[->] (D1) -- (T1) node[at end , above ] {\tiny{$m$}};
\draw[->] (B2) -- (C2) node[at start, below] {\tiny{$m+l$}};
\draw[->] (C2) -- (D2) node[midway, right] {\tiny{$m+l-k$}};
\draw[->] (D2) -- (T2) node[at end, above] {\tiny{$n+l$}};
\draw[->] (D1) -- (D2) node[midway, above] {\tiny{$n+k-m$}};
\draw[->] (C2) -- (C1) node[midway, below] {\tiny{$k$}};
}}}}}

\newcommand{\squared}{\ensuremath{\vcenter{\hbox{\tikz[yscale=0.55, xscale=0.65]{
\coordinate (B1) at (-1,0);
\coordinate (B2) at (1,0);
\coordinate (C1) at (-1,1);
\coordinate (D1) at (-1,2);
\coordinate (C2) at (1,1);
\coordinate (D2) at (1,2);
\coordinate (T1) at (-1,3);
\coordinate (T2) at (1,3);
\draw[->] (B1) -- (C1) node[at start, below] {\tiny{$n$}};
\draw[->] (C1) -- (D1) node[midway, left   ] {\tiny{$m-j$}};
\draw[->] (D1) -- (T1) node[at end , above ] {\tiny{$m$}};
\draw[->] (B2) -- (C2) node[at start, below] {\tiny{$m+l$}};
\draw[->] (C2) -- (D2) node[midway, right] {\tiny{$n+l+j$}};
\draw[->] (D2) -- (T2) node[at end, above] {\tiny{$n+l$}};
\draw[->] (D2) -- (D1) node[midway, above] {\tiny{$j$}};
\draw[->] (C1) -- (C2) node[midway, below] {\tiny{$n+j-m$}};
}}}}}

\newcommand{\squaree}{\ensuremath{\vcenter{\hbox{\tikz[yscale=0.55, xscale=0.65]{
\coordinate (B1) at (-1,0);
\coordinate (B2) at (1,0);
\coordinate (C1) at (-1,1);
\coordinate (D1) at (-1,2);
\coordinate (C2) at (1,1);
\coordinate (D2) at (1,2);
\coordinate (T1) at (-1,3);
\coordinate (T2) at (1,3);
\draw[->] (B1) -- (C1) node[at start, below] {\tiny{$k+s$}};
\draw[->] (C1) -- (D1) node[midway, left   ] {\tiny{$k$}};
\draw[->] (D1) -- (T1) node[at end , above ] {\tiny{$k-r$}};
\draw[->] (B2) -- (C2) node[at start, below] {\tiny{$l-s$}};
\draw[->] (C2) -- (D2) node[midway, right] {\tiny{$l$}};
\draw[->] (D2) -- (T2) node[at end, above] {\tiny{$l+r$}};
\draw[<-] (D2) -- (D1) node[midway, above] {\tiny{$r$}};
\draw[->] (C1) -- (C2) node[midway, below] {\tiny{$s$}};
}}}}}

\newcommand{\webHe}{\ensuremath{\vcenter{\hbox{\tikz[yscale=0.55, xscale=0.65]{
\coordinate (B1) at (-1,0);
\coordinate (B2) at (1,0);
\coordinate (C1) at (-1,1.5);
\coordinate (D1) at (-1,1.5);
\coordinate (C2) at (1,1.5);
\coordinate (D2) at (1,1.5);
\coordinate (T1) at (-1,3);
\coordinate (T2) at (1,3);
\draw[->] (B1) -- (C1) node[at start, below] {\tiny{$k+s$}};
\draw[->] (D1) -- (T1) node[at end , above ] {\tiny{$k-r$}};
\draw[->] (B2) -- (C2) node[at start, below] {\tiny{$l-s$}};
\draw[->] (D2) -- (T2) node[at end, above] {\tiny{$l+r$}};
\draw[->] (C1) -- (C2) node[midway, below] {\tiny{$r+s$}};
}}}}}

\newcommand{\doubleYb}{\ensuremath{\vcenter{\hbox{\tikz[scale=0.4]{
\coordinate (B1) at (-1,0);
\coordinate (T1) at (-1,3);
\coordinate (C) at (0,1);
\coordinate (D) at (0,2);
\coordinate (B2) at (1,0);
\coordinate (T2) at (1,3);
\draw[->] (B1) -- (C) node[at start, below] {\tiny{$1$}};
\draw[<-] (C) -- (B2) node[at end, below] {\tiny{$m+l-1$}};
\draw[<-] (D) -- (C) node[midway, left] {\tiny{$l+m$}};
\draw[<-] (T2) -- (D) node[at start, above] {\tiny{$m$}};
\draw[->] (D) -- (T1) node[at end, above] {\tiny{$l$}};
}}}}}

\newcommand{\bigHb}{\ensuremath{\vcenter{\hbox{\tikz[scale=0.4]{
\coordinate (B1) at (-1,0);
\coordinate (T1) at (-1,3);
\coordinate (M1) at (-1,1.5);
\coordinate (M2) at (1,1.5);
\coordinate (B2) at (1,0);
\coordinate (T2) at (1,3);
\draw[->] (B1) -- (M1) node[at start, below] {\tiny{$1$}};
\draw[->] (B2) -- (M2) node[at start, below] {\tiny{$m+l-1$}};
\draw[->] (M2) -- (M1) node[midway, above] {\tiny{$l-1$}};
\draw[->] (M2) -- (T2) node[at end, above] {\tiny{$m$}};
\draw[->] (M1) -- (T1) node[at end, above] {\tiny{$l$}};
}}}}}

\allowdisplaybreaks
 \author{Louis-Hadrien Robert}
 \author{Emmanuel Wagner}
\tikzset{>=latex}
\tikzset{->-/.style={decoration={
  markings,
  mark=at position .5 with {\arrow{>}}},postaction={decorate}}}
\tikzset{-<-/.style={decoration={
  markings,
  mark=at position .5 with {\arrow{<}}},postaction={decorate}}}
\newcommand{\todo}[1]{}
\begin{document}

\begin{abstract}
We provide a finite-dimensional categorification of the symmetric evaluation of $\mathfrak{sl}_N$-webs using foam technology. As an output we obtain a symmetric link homology theory categorifying the link invariant associated to symmetric powers of the standard representation of $\mathfrak{sl}_N$. The construction is  made in an equivariant setting. We prove also that there is a spectral sequence from the Khovanov--Rozansky triply graded link homology to the symmetric one and provide along the way a foam interpretation of Soergel bimodules. 
\end{abstract}

\maketitle

\tableofcontents

\section{Introduction}
\label{sec:introduction}

In \cite{RW1}, we provided a combinatorial evaluation of the foams underlying the (exterior) colored Khovanov--Rozansky link homologies \cite{zbMATH05278251, zbMATH05343986, zbMATH05656519, MR2710319, MR2491657, pre06302580, yonezawa2011}. See \cite{MR3787560} for an overview. This formula was the keystone to provide a down-to-earth treatment of these homologies, completely similar to the one in Khovanov's original paper \cite{MR2124557} or in his $\mathfrak{sl}_3$ paper \cite{MR2100691}. Immediate consequences of this formula were used by Ehrig--Tubbenhauer--Wedrich \cite{FunctorialitySLN} to prove functoriality of these homologies.\\

The present paper grew up as an attempt to provide a similar formula for foams underlying link homologies categorifying the Reshetikhin--Turaev invariants of links corresponding to symmetric powers of the standard representation of quantum $\mathfrak{sl}_N$. Providing manageable definitions of these link homologies is one of the keys of the program aiming at categorifying quantum invariants of $3$-manifolds. 

The first such link homologies were provided by Khovanov for the colored Jones polynomial \cite{MR2124557} (see as well \cite{MR2462446}). There are nowadays many definitions of link homologies, e.~g.\ using categorified projectors \cite{SSWenzl, FSS, CKWenzl, RWenzl, CautisClasp, CH} or using spectral sequences \cite{MR3709661} (see below for more details concerning this last one). They also fit in the higher representations techniques developed by Webster \cite{2013arXiv1309.3796W}.\\

In addition, it has been conjectured in \cite{GGS} that there exist symmetries between the categorification of Reshetikhin--Turaev invariants arising from exterior powers and symmetric powers of the standard representation of quantum $\mathfrak{sl}_N$. It has been proved by Tubbenhauer--Vaz--Wedrich at a decategorified level \cite{tubbenhauer2015super}. Moreover the work of Rose--Tubbenhauer \cite{RoseTub},  Queffelec--Rose \cite{QueffelecRoseAnnular2017} and Queffelec--Rose--Sartori \cite{InPrep} made clear that the planar graphical calculus underlying the description of the symmetric powers of the standard representation of quantum $\mathfrak{sl}_N$ is for a large part similar to the one for to the exterior powers. In the exterior case it was developed by Murakami--Ohtsuki--Yamada~\cite{MR1659228} (called in this paper the \emph{exterior} MOY calculus). Queffelec,  Rose and Sartori proved that the invariants only differ by the initializations on colored circles \cite{QueffelecRoseAnnular2017, InPrep}.  They work in an annular setting. We call it the \emph{symmetric} MOY calculus.\\

An attentive reader may have noticed that we spoke about an attempt. Let us explain why one cannot provide such a formula in the symmetric case at the level of generality we had in \cite{RW1} and therefore need to restrict the setup.\\

Foams are 2-dimensional CW-complexes which are naturally cobordisms between trivalent graphs (see below for an example). 
The closed formula of \cite{RW1} provided a singular TQFT using the universal construction \cite{MR1362791}. In this construction a facet of the foam is decorated with elements of a Frobenius algebra which is attached to the circle colored with the same color. Since the work of Bar-Natan and Khovanov the existence of the non-degenerate pairing on the algebra is rephrased in a topological type relation, known as the neck-cutting relation. Moreover, the TQFT feature also forces the evaluation of a planar graph times a circle to be the dimension of the vector space, the universal construction associates to the planar graph. We emphasized in \cite{RW1} that if such a construction works, not only circles are associated Frobenius algebras but all planar graphs which have a symmetry axis. In addition, the co-unit, for degree reasons, should be non-zero only on the maximal degree elements.  All the previous properties would be forced if one could obtain a closed formula providing a categorification of the symmetric MOY calculus. Elementary computations show a functor categorifying the symmetric MOY calculus cannot satisfies such properties essentially for degree reasons.\\




This is why we work in an annular setting. The drawback is that we cannot deal with general link diagrams. The benefit is that we can use part of the technology developed by Queffelec--Rose~\cite{QueffelecRoseAnnular2017}. We obtain an evaluation in this restricted case and apply a restricted universal construction to obtain the following result:

\begin{theo}
There exists a finite-dimensional categorification of the symmetric MOY calculus yielding a categorification of the Reshetikhin--Turaev invariants of links corresponding to symmetric powers of the standard representation of quantum $\mathfrak{sl}_N$.
\end{theo}

We call these link homologies \emph{symmetric} (colored) Khovanov--Rozansky link homologies. The construction applies in particular to the case where the representation is the standard representation of quantum $\mathfrak{sl}_N$. As shown on an example (Section~\ref{sec:an-example-comp}), it provides in this case a different categorification of the (uncolored)-$\mathfrak{sl}_N$-link invariants than those of Khovanov and Khovanov--Rozansky.  This observation is due to Queffelec--Rose--Sartori; 
actually they discuss in \cite{InPrep} how the annular link homologies constructed by Queffelec and Rose in \cite{QueffelecRoseAnnular2017} can be specialized to be invariant under Reidemeister I and give link invariants in  the 3-sphere.\\

Whereas the definition of the link homologies can be made only using the language of foams and symmetric polynomial, the proof of invariance at the moment requires a more algebraic treatment. This algebraic treatment uses Soergel bimodules and makes explicit the comparison with the work of Cautis \cite{MR3709661} \footnote{This strategy of proof was discussed with H. Queffelec and D. Rose.}. Cautis constructs a differential $d_N$ on the Hochschild homology of Soergel bimodules compatible with the differential of the Rickard complex such that the total homology provides a categorification of Reshetikhin--Turaev invariants of links corresponding to symmetric powers of the standard representation of quantum $\mathfrak{sl}_N$. We provide here  an explicit version of the additional differential in an equivariant setting. The non-equivariant case is studied by Cautis~\cite{MR3709661} and investigate by Queffelec--Rose--Sartori~\cite{InPrep}. Hence, one consequence of our proof of invariance is the following.

\begin{theo}
There exists a spectral sequence whose first page is isomorphic to the (unreduced) colored triply graded link homology converging to the symmetric Khovanov--Rozansky link homologies.
\end{theo}

We would like to stress that one can also see on the same picture the spectral sequences converging to the (exterior) colored Khovanov--Rozansky link homology \cite{MR3447099, 2016arXiv160202769W}. Hence in some sense the only differential missing from the perspective of the work of Dunfield--Gukov--Rasmussen is $d_0$ which seems to be tackled by Dowlin \cite{2017arXiv170301401D}.

The definition of the link homologies in this paper, starts with the links presented as closures of braids, hence regarding functoriality questions it only makes sense to consider braid-like movie moves.
It is an immediate consequence of our definitions and the work of Ehrig--Tubbenhauer--Wedrich concerning functoriality of the (exterior) colored Khovanov--Rozansky link homologies that the following holds.

\begin{theo}
The symmetric Khovanov--Rozansky link homologies  are functorial with respect to braid-like movie moves.
\end{theo}

One very important feature it that our construction works in an equivariant setting and will allow with a little more work to define Rasmussen type invariants for braids. In another direction, the fact that we are restricted to braid closures naturally suggest that one will obtain Morton--Franks--Williams type inequalities in this case (see Wu \cite{WuMFW}). The interactions between the two previous directions seem to us worth pursuing.\\

To conclude, the construction is done over rationals and we think the following question deserves attention.

\begin{ques}
Can one make the content of this paper work over integers?
\end{ques}

The main obstruction so far is that in our proof of invariance we need to invert $2$, just like for the stabilization in the triply graded homology~\cite{1203.5065, MR3687104}. The strategy adopted by Krasner~\cite{MR2726290} might be a starting point.\\

{\bf Outline of the paper.} The paper is divided as follows. In the first section we develop the symmetric MOY calculus. 
In the second section, we provide the needed definitions of the restricted class of foams we will be working with: disk-like and vinyl foams. We also give an overview on foams and the closed formula of \cite{RW1}. In the third section we explain how to think of Soergel bimodules as spaces of disk-like foams and part of their homologies as vinyl foams. In the fourth and central section we define an evaluation of vinyl foams, providing a categorification of the symmetric MOY calculus. The rest of the section is devoted to rephrasing it in terms of an additional differential on the Hochschild homology of Soergel bimodules. The algebraic description is used in the fifth section to prove the invariance of the link homologies. The link homologies are constructed using the well-known Rickard complexes.

There are as well three appendices. The first one deals with the representation theory of quantum $\gll_N$. The second is a reminder on Koszul resolutions and  
contains some technical homological lemmas. The third one present some inspiring algebraic geometry. 
 \\
{\bf Acknowledgments.} The authors thank Christian Blanchet, François Costantino, Hoel Queffelec, David Rose and Paul Wedrich for interesting discussions and Mikhail Khovanov and Daniel Tubbenhauer for their comments on a previous version of this work. They thank also the Newton Institute for Mathematical Sciences, the University of Burgundy and Marie Fillastre for their logistical support. L.-H.~R. was supported by SwissMAP.\\



\section{MOY graphs}
\label{sec:moy-graphs}

\begin{dfn}[\cite{MR1659228}]
  \label{dfn:abstractMOYgraph}
  An \emph{abstract MOY graph} is a finite oriented graph $\Gamma = (V_\Gamma,E_\Gamma)$ with a labeling of its edges : $\ell : E_\Gamma \to \NN_{>0}$ such that:
  \begin{itemize}
  \item the vertices are either univalent (we call this subset of vertices the \emph{boundary} of $\Gamma$ and denote it by $\partial\Gamma$) or trivalent (these are the \emph{internal vertices}),
  \item the flow given by labels and orientations is preserved along 
 the trivalent vertices, meaning that every trivalent vertex follows one of the two models (\emph{merge} and \emph{split} vertices) drawn here.
\[
\tikz{\begin{scope}
\draw[->] (0,0) -- (0,0.5) node [at end, above] {$a+b$};  
\draw[>-] (-0.5, -0.5) -- (0,0) node [at start, below] {$a$};  
\draw[>-] (+0.5, -0.5) -- (0,0) node [at start, below] {$b$};  
\begin{scope}[xshift = 5cm]
  \draw[-<] (0,0) -- (0,0.5) node [at end, above] {$a+b$};  
\draw[<-] (-0.5, -0.5) -- (0,0) node [at start, below] {$a$};  
\draw[<-] (+0.5, -0.5) -- (0,0) node [at start, below] {$b$};  
\end{scope}
\end{scope}}
\]
\end{itemize}
The univalent vertices are either sinks or sources. We call the first \emph{positive boundary points} and the later \emph{negative boundary points}. 
\end{dfn}
\begin{rmk}
  \label{rmk:0edgebigsmall}
  \begin{enumerate}
  \item Sometimes it will be convenient to allow edges labeled by $0$. However, this edges should be thought as ``irrelevant''. We simply delete them to recover the original definition.
  \item Each internal vertex has three adjacent edges. The label of one of these edges is strictly greater than the other two. This edge is called the \emph{big} edge relative to this vertex. The two other edges are called the \emph{small} edges relative to this vertex.
  \end{enumerate}
\end{rmk}



\begin{dfn}
  \label{dfn:MOY}
  A \emph{MOY graph} is the image of an abstract MOY graph $\Gamma$  by a smooth embedding
in the $[0,1] \times [0,1]$ such that:
  \begin{itemize}
  \item All the oriented tangent lines at vertices are equal\footnote{In pictures which follows we may forget about this technical condition, since it is clear that we can always deform the embedding  locally so that this condition is fulfilled}.
\[
\tikz{\begin{scope}
\draw[->] (0,0) -- (0,0.5) node [at end, above] {$a+b$};  
\draw[>-] (-0.5, -0.5) .. controls + (0,0) and +(0, -0.3) ..  (0,0) node [at start, below] {$a$};  
\draw[>-] (+0.5, -0.5) .. controls + (0,0) and +(0, -0.3) .. (0,0) node [at start, below] {$b$};  
\begin{scope}[xshift = 5cm]
  \draw[-<] (0,0) -- (0,0.5) node [at end, above] {$a+b$};  
\draw[<-] (-0.5, -0.5) .. controls + (0,0) and +(0, -0.3) .. (0,0) node [at start, below] {$a$};  
\draw[<-] (+0.5, -0.5) .. controls + (0,0) and +(0, -0.3) .. (0,0) node [at start, below] {$b$};  
\end{scope}
\end{scope}}
\]
  \item The boundary of $\Gamma$ is contained in $[0,1]\times \{0,1\}$.
  \item $]0,1[\times \{0,1\} \cap \Gamma = \partial \Gamma$.
  \item the tangent lines at the boundary points of $\Gamma$ are vertical, (that is, collinear with $
    \left(\begin{smallmatrix}
      0 \\1
    \end{smallmatrix} \right)
$).  
  \end{itemize}
In what follows it will be convenient to speak about \emph{the tangent vector} at a point $p$ of the graph $\Gamma$  (or more precisely of its image in $[0,1]\times [0,1]$). By this we mean the only vector which is tangent to $\Gamma$, has norm 1 and whose orientation agrees with the one of $\Gamma$. Note that the condition on the embedding on vertices ensure that it is well-defined everywhere. 

\begin{notation}
  \label{not:listk}
  In what follows, $\listk{k}$ always denotes a finite sequence of integers (the empty sequence is allowed). If $\listk{k}= (k_1, \dots, k_l)$,  $l$ is the \emph{length} of $\listk{k}$ and $\sum_{i=1}^lk_i$ is the \emph{level} of $\listk{k}$. If $\listk{k}$ is a sequence of length 1 and level $k$, we abuse notation and write $k$ instead of $\listk{k}$. 
\end{notation}

 If we want to specify the boundary of a MOY graph $\Gamma$, we will speak about $\listk{k}_1$-MOY graph-$\listk{k}_0$ (see the example in Figure~\ref{fig:exMOY} to understand the notations). If a MOY graph has an empty boundary we say that it is \emph{closed}.

If $\Gamma$ is an $\listk{k}_1$-MOY graph-$\listk{k}_0$, we denote by $-\Gamma$ the $(-\listk{k}_1)$-MOY graph-$(-\listk{k}_0)$, which is obtained from $\Gamma$ by reversing all orientations.
\end{dfn}

\begin{rmk}
  \label{rmk:MOYgraphisopy}
  MOY graphs are regarded up to ambient isotopy fixing the boundary. This fits into a category where objects are finite sequences of signed and labeled points in $]0,1[$, and morphisms are MOY graphs. The composition is then given by concatenation and rescaling (see Figure~\ref{fig:exMOY}). 
\end{rmk}
\begin{figure}
  \centering
  \begin{tikzpicture}[xscale =0.7, yscale=0.8]
   {\tiny \begin{scope}
\coordinate (B1) at (1, 0);
\coordinate (B2) at (2, 0);
\coordinate (B3) at (3, 0);
\coordinate (B4) at (4, 0);
\coordinate (T1) at (1.5, 4);
\coordinate (T2) at (2.5, 4);
\coordinate (T3) at (3.5, 4);
\coordinate (A1) at (1.5,0.5);
\coordinate (A2) at (2,1.5);
\coordinate (A3) at (2.5,2.5);
\coordinate (A4) at (3,3);
\coordinate (A5) at (3,3.5);
\draw[->-]  (B1) .. controls +(0,0.2) and + (0, -0.2) .. (A1) node[midway,above] {$2$};
\draw[->-]  (B2) .. controls +(0,0.2) and + (0, -0.2) .. (A1) node[midway,above] {$3$};
\draw[->-]  (B3) .. controls +(0,0.5) and + (0, -0.2) .. (A2) node[midway,left] {$1$};
\draw[->-]  (B4) .. controls +(0,0.5) and + (0, -0.2) .. (A4) node[midway,left] {$2$};
\draw[->-]  (A1) .. controls +(0,0.2) and + (0, -0.2) .. (A2) node[midway,left] {$5$};
\draw[->-]  (A2) .. controls +(0,0.2) and + (0, -0.2) .. (A3) node[midway,left] {$6$};
\draw[->-]  (A3) .. controls +(0,0.2) and + (0, -0.5) .. (T1) node[midway,left] {$4$};
\draw[->-]  (A3) .. controls +(0,0.2) and + (0, -0.2) .. (A4) node[midway,below] {$2$};
\draw[->-]  (A4) .. controls +(0,0.2) and + (0, -0.2) .. (A5) node[midway,left] {$4$};
\draw[->-]  (A5) .. controls +(0,0.2) and + (0, -0.2) .. (T2) node[midway,left] {$3$};
\draw[->-]  (A5) .. controls +(0,0.2) and + (0, -0.2) .. (T3) node[midway,right] {$1$};
\end{scope}

\begin{scope}[xshift= 6cm]
\coordinate (B1) at (1.5, 0);
\coordinate (B2) at (2.5, 0);
\coordinate (B3) at (3.5, 0);
\coordinate (T1) at (1, 4);
\coordinate (T2) at (2, 4);
\coordinate (T3) at (3, 4);
\coordinate (T4) at (4, 4);
\coordinate (A1) at (3,0.5);
\coordinate (A2) at (3,1.3);
\coordinate (A3) at (1.5,1.5);
\coordinate (A4) at (2.5,2.5);
\coordinate (A5) at (2.5,3);
\draw[->-]  (B1) .. controls +(0,0.2) and + (-0.2, 0) .. (A3) node[midway,left] {$4$};
\draw[->-]  (B2) .. controls +(0,0.2) and + (0, -0.2) .. (A1) node[midway,above] {$3$};
\draw[->-]  (B3) .. controls +(0,0.2) and + (0, -0.2) .. (A1) node[midway,above] {$1$};
\draw[->-]  (A1) .. controls +(0,0.2) and + (0, -0.2) .. (A2) node[midway,left] {$4$};
\draw[->-]  (A3) .. controls +(0.2,0) and + (0, -0.2) .. (A4) node[midway,left] {$6$};
\draw[->-]  (A2) .. controls +(0,0.2) and + (0, -0.2) .. (A4) node[midway,left] {$1$};
\draw[->-]  (T1) .. controls +(0,0.2) and + (-0.2, 0) .. (A3) node[midway,left] {$2$};
\draw[->-]  (A2) .. controls +(0,0.2) and + (0, -0.2) .. (T4) node[midway,left] {$3$};
\draw[->-]  (A4) .. controls +(0,0.2) and + (0, -0.2) .. (A5) node[midway,left] {$7$};
\draw[->-]  (A5) .. controls +(0,0.2) and + (0, -0.2) .. (T2) node[midway,left] {$2$};
\draw[->-]  (A5) .. controls +(0,0.2) and + (0, -0.2) .. (T3) node[midway,right] {$5$};
\end{scope}

\begin{scope}[xshift=12cm, yscale =0.5]
\coordinate (B1) at (1, 0);
\coordinate (B2) at (2, 0);
\coordinate (B3) at (3, 0);
\coordinate (B4) at (4, 0);
\coordinate (T1) at (1.5, 4);
\coordinate (T2) at (2.5, 4);
\coordinate (T3) at (3.5, 4);
\coordinate (A1) at (1.5,0.5);
\coordinate (A2) at (2,1.5);
\coordinate (A3) at (2.5,2.5);
\coordinate (A4) at (3,3);
\coordinate (A5) at (3,3.5);
\draw[->-]  (B1) .. controls +(0,0.2) and + (0, -0.2) .. (A1) node[midway,above] {$2$};
\draw[->-]  (B2) .. controls +(0,0.2) and + (0, -0.2) .. (A1) node[midway,above] {$3$};
\draw[->-]  (B3) .. controls +(0,0.5) and + (0, -0.2) .. (A2) node[midway,left] {$1$};
\draw[->-]  (B4) .. controls +(0,0.5) and + (0, -0.2) .. (A4) node[midway,left] {$2$};
\draw[->-]  (A1) .. controls +(0,0.2) and + (0, -0.2) .. (A2) node[midway,left] {$5$};
\draw[->-]  (A2) .. controls +(0,0.2) and + (0, -0.2) .. (A3) node[midway,left] {$6$};
\draw[->-]   (A3) .. controls +(0,0.2) and + (0, -0.5) .. (T1) node[left] {$4$};
\draw[->-]  (A3) .. controls +(0,0.2) and + (0, -0.2) .. (A4) node[midway,below] {$2$};
\draw[->-]  (A4) .. controls +(0,0.2) and + (0, -0.2) .. (A5) node[midway,left] {$4$};
\draw[->-]   (A5) .. controls +(0,0.2) and + (0, -0.2) .. (T2) node[left] {$3$};
\draw[->-]   (A5) .. controls +(0,0.2) and + (0, -0.2) .. (T3) node[right] {$1$};
\end{scope}

\begin{scope}[xshift= 12cm, yshift=2cm, yscale =0.5]
\coordinate (B1) at (1.5, 0);
\coordinate (B2) at (2.5, 0);
\coordinate (B3) at (3.5, 0);
\coordinate (T1) at (1, 4);
\coordinate (T2) at (2, 4);
\coordinate (T3) at (3, 4);
\coordinate (T4) at (4, 4);
\coordinate (A1) at (3,0.5);
\coordinate (A2) at (3,1.3);
\coordinate (A3) at (1.5,1.5);
\coordinate (A4) at (2.5,2.5);
\coordinate (A5) at (2.5,3);
\draw       (B1) .. controls +(0,0.2) and + (-0.2, 0) .. (A3);
\draw       (B2) .. controls +(0,0.2) and + (0, -0.2) .. (A1);
\draw       (B3) .. controls +(0,0.2) and + (0, -0.2) .. (A1);
\draw[->-]  (A1) .. controls +(0,0.2) and +(0, -0.2) .. (A2) node[midway,left] {$4$};
\draw[->-]  (A3) .. controls +(0.2,0) and +(0, -0.2) .. (A4) node[midway,left] {$6$};
\draw[->-]  (A2) .. controls +(0,0.2) and +(0, -0.2) .. (A4) node[midway,left] {$1$};
\draw[->-]  (T1) .. controls +(0,0.2) and +(-0.2, 0) .. (A3) node[midway,left] {$2$};
\draw[->-]  (A2) .. controls +(0,0.3) and +(0, -0.3) .. (T4) node[midway,left] {$3$};
\draw[->-]  (A4) .. controls +(0,0.2) and +(0, -0.2) .. (A5) node[midway,left] {$7$};
\draw[->-]  (A5) .. controls +(0,0.2) and +(0, -0.2) .. (T2) node[midway,left] {$2$};
\draw[->-]  (A5) .. controls +(0,0.2) and +(0, -0.2) .. (T3) node[midway,right] {$5$};
\end{scope} }
  \end{tikzpicture}
  \caption{Examples of a MOY graphs: a $(4,3,1)$-MOY graph-$(2,3,1,2)$, a $(-2,2,5,3)$-MOY graph-$(4,3,1)$ and their concatenation.}
  \label{fig:exMOY}
\end{figure}
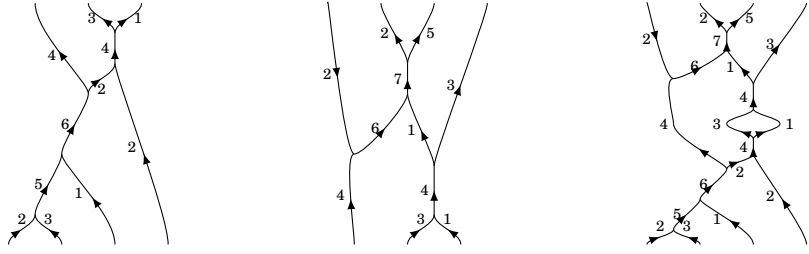

\begin{dfn}
  \label{dfn:rotational}
  Let $\Gamma$ be a closed MOY graph. The \emph{rotational} of $\Gamma$ is the sum of the rotational of the oriented circles appearing in the cabling of $\Gamma$. The \emph{rotational} of a circle is $+1$ if it winds counterclockwisely and $-1$ if it winds clockwisely.  It is denoted by $\mathrm{rot}(\Gamma)$. This definition is illustrated in Figure~\ref{fig:rotMOY}.
\end{dfn}
\begin{figure}
  \centering
  \begin{tikzpicture}[yscale = 0.4]
{\tiny    \begin{scope}
  \coordinate (A) at (0,0);
  \coordinate (B) at (0,2);
  \coordinate (C) at (0,4);
  \coordinate (D) at (0,6);
  \node at (0, -1.5) {$\Gamma$};
  \draw[->] (A) -- (B) node[midway, left] {$4$};
  \draw[->] (C) -- (D) node[midway, left] {$4$};
  \draw[->] (B) .. controls +(0.5,0.5) and +(0.5,-0.5) .. (C) node[midway, left] {$3$};
  \draw[->] (B) .. controls +(-0.5,0.5) and +(-0.5,-0.5) .. (C) node[midway, left] {$1$};
  \draw[->] (D) .. controls +(1.5,1.5) and +(1.5,-1.5) .. (A) node[midway, left] {$2$};
  \draw[->] (D) .. controls +(-1.5,1.5) and +(-1.5,-1.5) .. (A) node[midway, left] {$2$};
\end{scope}

\begin{scope}[xshift = 5cm]
  \coordinate (A1) at (-0.15,0);
  \coordinate (A2) at (-0.05,0);
  \coordinate (A4) at (+0.15,0);
  \coordinate (A3) at (+0.05,0);
  \coordinate (B1) at (-0.15,2);
  \coordinate (B2) at (-0.05,2);
  \coordinate (B4) at (+0.15,2);
  \coordinate (B3) at (+0.05,2);
  \coordinate (C1) at (-0.15,4);
  \coordinate (C2) at (-0.05,4);
  \coordinate (C4) at (+0.15,4);
  \coordinate (C3) at (+0.05,4);
  \coordinate (D1) at (-0.15,6);
  \coordinate (D2) at (-0.05,6);
  \coordinate (D4) at (+0.15,6);
  \coordinate (D3) at (+0.05,6);
  \draw[->] (A1) -- (B1);
  \draw[->] (A2) -- (B2);
  \draw[->] (A3) -- (B3);
  \draw[->] (A4) -- (B4);
  \draw[->] (C1) -- (D1);
  \draw[->] (C2) -- (D2);
  \draw[->] (C3) -- (D3);
  \draw[->] (C4) -- (D4);
  \draw (B1) .. controls +(-0.5,0.5) and +(-0.5,-0.5) .. (C1);
  \draw (B2) .. controls +(0.5,0.5) and +(0.5,-0.5) .. (C2);
  \draw (B3) .. controls +(0.5,0.5) and +(0.5,-0.5) .. (C3);
  \draw (B4) .. controls +(0.5,0.5) and +(0.5,-0.5) .. (C4); 
   \draw[->] (D3) .. controls +(1.65,1.9) and +(1.65,-1.9) .. (A3);
   \draw[->] (D4) .. controls +(1.4,1.3) and +(1.4,-1.3) .. (A4);
   \draw[->] (D1) .. controls +(-1.4,1.3) and +(-1.4,-1.3) .. (A1);
   \draw[->] (D2) .. controls +(-1.65,1.9) and +(-1.65,-1.9) .. (A2);
\end{scope} }
  \end{tikzpicture}
  \caption{The rotational of the closed MOY graph depicted on the left is equal to $2-2=0$. }
  \label{fig:rotMOY}
\end{figure}
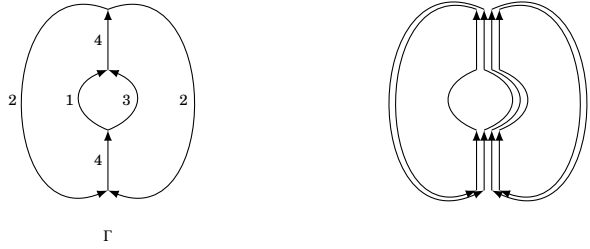

\subsection{MOY calculi}
\label{sec:moy-calculi}

In their seminal paper, Murakami, Ohtsuki and Yamada \cite{MR1659228} gave a combinatorial definition of the colored $U_q(\sll_N)$ framed link invariant. For clarity we refer to this construction as \emph{exterior} MOY calculus since it calculates the Reshetikhin--Turaev invariant of a framed link labeled with exterior powers of $V$ the standard representation of $U_q(\sll_N)$. 
We denote by $\kup{\bullet}_N$ this invariant (or simply by $\kup{\bullet}$ when there is no ambiguity about $N$).  This goes in two steps. We consider a framed link diagram labeled with integers between $0$ and $N$ (one should think about an integer $a$ as representing $\Lambda^a_qV$) and we replace every crossing by a formal $\ZZ[q,q^{-1}]$-linear combination of planar graphs following the formulas:
\begin{align}
\label{eq:extcrossplus}\kup{\scriptstyle{\NB{\tikz[scale=0.6]{\begin{scope}
  \draw[->] (1, -1) -- (-1, 1) node[near end, left] {\tiny{${m}$}};
  \fill[white] (0,0) circle (2mm);
  \draw[->] (-1, -1) -- (1, 1) node[near end, right] {\tiny{${n}$}};
\end{scope}}}}} &= \sum_{k= \max(0, m-n)}^m (-1)^{m-k}q^{k-m}\kup{\!\!\NB{\tikz[scale=0.9]{\begin{scope}
\coordinate (A) at (-1,-1);
\coordinate (B) at (1,-1);
\coordinate (C) at (1,1);
\coordinate (D) at (-1,1);
\coordinate (a) at (-.5,-.5);
\coordinate (b) at (.5,-.5);
\coordinate (c) at (.5,.5);
\coordinate (d) at (-.5,.5);
\draw[->] (A) -- (a) node[at start, below] {\tiny{$n$}};
\draw[->] (c) -- (C) node[at end, above ] {\tiny{$n$}};
\draw[->] (B) -- (b) node[at start , below ] {\tiny{$m$}};
\draw[->] (d) -- (D) node[at end, above] {\tiny{$m$}};
\draw[<-] (c) -- (d) node[midway, above] {\tiny{$n+k-m$}};
\draw[<-] (a) -- (b) node[midway, below] {\tiny{$k$}};
\draw[->] (a) -- (d) node[midway, left] {\tiny{$n+k$}};
\draw[->] (b) -- (c) node[midway, right] {\tiny{$m-k$}};
\end{scope}}}\!\!}\\
\label{eq:extcrossminus}
\kup{\scriptstyle{\NB{\tikz[scale=0.6]{\begin{scope}
  \draw[->] (-1, -1) -- (1, 1) node[near end, right] {\tiny{${n}$}};
  \fill[white] (0,0) circle (2mm);
  \draw[->] (1, -1) -- (-1, 1) node[near end, left] {\tiny{${m}$}};
\end{scope}}}}} &= \sum_{k= \max(0, m-n)}^m (-1)^{m-k}q^{m-k} \kup{\!\!\NB{\tikz[scale=0.9]{\begin{scope}
\coordinate (A) at (-1,-1);
\coordinate (B) at (1,-1);
\coordinate (C) at (1,1);
\coordinate (D) at (-1,1);
\coordinate (a) at (-.5,-.5);
\coordinate (b) at (.5,-.5);
\coordinate (c) at (.5,.5);
\coordinate (d) at (-.5,.5);
\draw[->] (A) -- (a) node[at start, below] {\tiny{$n$}};
\draw[->] (c) -- (C) node[at end, above ] {\tiny{$n$}};
\draw[->] (B) -- (b) node[at start , below ] {\tiny{$m$}};
\draw[->] (d) -- (D) node[at end, above] {\tiny{$m$}};
\draw[<-] (c) -- (d) node[midway, above] {\tiny{$n+k-m$}};
\draw[<-] (a) -- (b) node[midway, below] {\tiny{$k$}};
\draw[->] (a) -- (d) node[midway, left] {\tiny{$n+k$}};
\draw[->] (b) -- (c) node[midway, right] {\tiny{$m-k$}};
\end{scope}}}\!\!}.
\end{align}
The first crossing in the formula is said to have \emph{type} $(m,n, +)$, the second to have \emph{type} $(m,n, -)$.

Finally we can evaluate these planar graphs (which are closed MOY graphs) using the following identities and their mirror images:
\begin{align}\label{eq:extrelcircle}
  \kup{\vcenter{\hbox{\tikz[scale= 0.5]{\draw[->] (0,0) arc(0:360:1cm) node[right] {\small{$\!k\!$}};}}}}=
\begin{bmatrix}
  N \\ k
\end{bmatrix}
\end{align}
\begin{align} \label{eq:extrelass}
   \kup{\stgamma} = \kup{\stgammaprime}
 \end{align}
\begin{align} \label{eq:extrelbin1} 
\kup{\digona} = \arraycolsep=2.5pt
  \begin{bmatrix}
    m+n \\ m
  \end{bmatrix}
\kup{\verta}
\end{align}
\begin{align} \label{eq:extrelbin2}
\arraycolsep=2.5pt
\kup{\digonb} = 
  \begin{bmatrix}
    N-m \\ n
  \end{bmatrix}
\kup{\vertb} 
\end{align}

\begin{align}
 \kup{\squarea} = \kup{\twoverta} + [N-m-1]\kup{\doubleYa} \label{eq:extrelsquare1}
\end{align}

\begin{align}
\kup{\squareb}=\!
  \begin{bmatrix}
    m-1 \\ n
  \end{bmatrix}
\kup{\bigHb}  +
\!\begin{bmatrix}
  m-1 \\n-1
\end{bmatrix}
\kup{\doubleYb} \label{eq:extrelsquare2}
\end{align}
\begin{align}
  \kup{\squarec}= \sum_{j=\max{(0, m-n)}}^m\begin{bmatrix}l \\ k-j \end{bmatrix}
 \kup{\squared}\label{eq:extrelsquare3}
\end{align}

In the previous formulas, we used quantum integers and quantum binomials. These are symmetric Laurent polynomials in $q$ defined by $[k]:= \frac{q^{k}- q^{-k}}{q- q^{-1}}$ and 
\[
  \begin{bmatrix}
    l \\ k
  \end{bmatrix}=
\frac{\prod_{i=0}^{k-1} [l-i]}{[k]!}
\quad \textrm{where } [i]! = \prod_{j=1}^i [j].
\]

The first formal proof that these relations are enough to compute has been written by Wu \cite{pre06302580} and is based on a result of Kauffman and Vogel \cite[Appendix 4]{zbMATH06000869}. In particular this shows that there is a unique evaluation of MOY graphs which satisfies these relations. The coherence of these relations follows from the representation theoretic point of view. For more details we refer to \cite{MR1659228,zbMATH05656519} and to Appendix~\ref{sec:quant-link-invar}.

As pointed out by \cite{tubbenhauer2015super}, a similar story applies when one think about the integers  labeling the strands  of a link (which are now only required to be positive) as representing $q$-symmetric powers of $V$ (i. e. $\mathrm{Sym}_q^\bullet V$). We denote by $\kups{\bullet}$ this invariant (or simply by $\kups{\bullet}$ when there is no ambiguity about $N$). This yields what we call the \emph{symmetric} MOY calculus. 
\begin{align}
\label{eq:symcrossplus}
\kups{\scriptstyle{\NB{\tikz[scale=0.6]{}}}} &= \sum_{k= \max(0, m-n)}^m (-1)^{m-k}q^{k-m}\kups{\!\!\NB{\tikz[scale=0.9]{\begin{scope}
\coordinate (A) at (-1,-1);
\coordinate (B) at (1,-1);
\coordinate (C) at (1,1);
\coordinate (D) at (-1,1);
\coordinate (a) at (-.5,-.5);
\coordinate (b) at (.5,-.5);
\coordinate (c) at (.5,.5);
\coordinate (d) at (-.5,.5);
\draw[->] (A) -- (a) node[at start, below] {\tiny{$n$}};
\draw[->] (c) -- (C) node[at end, above ] {\tiny{$n$}};
\draw[->] (B) -- (b) node[at start , below ] {\tiny{$m$}};
\draw[->] (d) -- (D) node[at end, above] {\tiny{$m$}};
\draw[<-] (c) -- (d) node[midway, above] {\tiny{$n+k-m$}};
\draw[<-] (a) -- (b) node[midway, below] {\tiny{$k$}};
\draw[->] (a) -- (d) node[midway, left] {\tiny{$n+k$}};
\draw[->] (b) -- (c) node[midway, right] {\tiny{$m-k$}};
\end{scope}}}\!\!}\\
\label{eq:symcrossminus}
\kups{\scriptstyle{\NB{\tikz[scale=0.6]{}}}} &= \sum_{k= \max(0, m-n)}^m (-1)^{m-k}q^{m-k} \kups{\!\!\NB{\tikz[scale=0.9]{\begin{scope}
\coordinate (A) at (-1,-1);
\coordinate (B) at (1,-1);
\coordinate (C) at (1,1);
\coordinate (D) at (-1,1);
\coordinate (a) at (-.5,-.5);
\coordinate (b) at (.5,-.5);
\coordinate (c) at (.5,.5);
\coordinate (d) at (-.5,.5);
\draw[->] (A) -- (a) node[at start, below] {\tiny{$n$}};
\draw[->] (c) -- (C) node[at end, above ] {\tiny{$n$}};
\draw[->] (B) -- (b) node[at start , below ] {\tiny{$m$}};
\draw[->] (d) -- (D) node[at end, above] {\tiny{$m$}};
\draw[<-] (c) -- (d) node[midway, above] {\tiny{$n+k-m$}};
\draw[<-] (a) -- (b) node[midway, below] {\tiny{$k$}};
\draw[->] (a) -- (d) node[midway, left] {\tiny{$n+k$}};
\draw[->] (b) -- (c) node[midway, right] {\tiny{$m-k$}};
\end{scope}}}\!\!}.
\end{align}
The formulas for evaluating MOY graphs become:
\begin{align}\label{eq:symrelcircle}
  \kups{\vcenter{\hbox{\tikz[scale= 0.5]{\draw[->] (0,0) arc(0:360:1cm) node[right] {\small{$\!k\!$}};}}}}=
\begin{bmatrix}
  N +k -1 \\ k
\end{bmatrix}
\end{align}
\begin{align} \label{eq:symrelass}
   \kups{\stgamma} = \kups{\stgammaprime}
 \end{align}
\begin{align} \label{eq:symrelbin1} 
\kups{\digona} = \arraycolsep=2.5pt
  \begin{bmatrix}
    m+n \\ m
  \end{bmatrix}
\kups{\verta}
\end{align}
\begin{align} \label{eq:symrelbin2}
\arraycolsep=2.5pt
\kups{\digonb} = 
  \begin{bmatrix}
    N+m+n-1 \\ n
  \end{bmatrix}
\kups{\vertb} 
\end{align}

\begin{align}
 \kups{\squarea} = \kups{\twoverta} + [N+m+1]\kups{\doubleYa} \label{eq:relsquare1}
\end{align}

\begin{align}
\kups{\squareb}=\!
  \begin{bmatrix}
    m-1 \\ n
  \end{bmatrix}
\kups{\bigHb}  +
\!\begin{bmatrix}
  m-1 \\n-1
\end{bmatrix}
\kups{\doubleYb} \label{eq:symrelsquare2}
\end{align}
\begin{align}
  \kups{\squarec}= \sum_{j=\max{(0, m-n)}}^m\begin{bmatrix}l \\ k-j \end{bmatrix}
 \kups{\squared}\label{eq:symrelsquare3}
\end{align}

\begin{rmk}\label{rmk:normalisation-sym-MOY}  \label{rmk:minusN}
  \begin{enumerate}
  \item The proof of computability and uniqueness of
    Wu~\cite{pre06302580} still works in the symmetric case. As before
    consistency follows from the representation theoretic point of
    view. We describe explicitly in Appendix~\ref{sec:quant-link-invar} the $U_q(\gll_N)$-intertwiners between priducts of symmetric powers of the standard $U_q(\gll_N)$-module. One can check by brute force computation that these morphisms satisfy the identities defining the symmetric MOY calculus. 
  \item 
We chose a normalization making the polynomial associated by
    the symmetric MOY calculus with any planar graph a Laurent
    polynomial with positive coefficients. 
    The skein formula we use  for the crossings is not compatible with the choices made for the
    exterior MOY calculus. The braiding in the exterior MOY calculus
    is given by an $R$ matrix of $U_q(\sll_N)$, while for the
    symmetric MOY calculus it is given by its inverse. Hence the two
    calculi we present here cannot be merge into one bicolor calculus
    as it is done in \cite{tubbenhauer2015super}.
  \item In \cite{tubbenhauer2015super}, there is an overall sign which we choose to remove here. For recovering this sign one should multiply our symmetric evaluation of a MOY graph $\Gamma$ by $(-1)^{\mathrm{rot}(\Gamma)}$ (see Definition~\ref{dfn:rotational}). 
  \item Up to this sign the formulas are actually the same as the ones of the exterior MOY calculus applied to $-N$.
Since they do not involve $N$, the identities~(\ref{eq:extrelass}), (\ref{eq:extrelbin1}), (\ref{eq:extrelsquare2}) and (\ref{eq:extrelsquare3}) of the exterior MOY calculus are the same as the identities~(\ref{eq:symrelass}), (\ref{eq:symrelbin1}), (\ref{eq:symrelsquare2}) and (\ref{eq:symrelsquare3})  of the symmetric MOY calculus. 
  \item In order to turn the framed invariants $\kup{\bullet}_N$ and $\kups{\bullet}$ into invariants of unframed links, one needs to renormalize them. For any link diagram $D$, we define:
    \begin{align*}
      RT^{\Lambda}_N(D) &= (-1)^{e(D)} q^{k^{\Lambda}(D)} \kup{D}_N \quad \textrm{and}\\
      RT^{S}_N(D)      &=  q^{k^{S}(D)} \kups{D},
    \end{align*}
    where $e(D)$ (resp.\ $k^{\Lambda}(D)$, resp.\ $k^{S}(D)$) is the sum over all crossing $x$ of $D$ of $e_x$ (resp.\ $k^{\Lambda}_x$, resp.\ $k^{S}_x$)  defined by:
\[ (k^{\Lambda}_x, k^{S}_x , e_x) =
  \begin{cases}
    (m(N+1-m), -m(m+N-1) , -m)            & \textrm{if $x$ is of type $(m,m,+)$, } \\ 
    (-m(N+1-m), m(m+N-1), +m)            & \textrm{if $x$ is of type $(m,m,-)$,} \\ 
    (0,0,0)               & \textrm{else.}
  \end{cases}
\]

  \end{enumerate}
\end{rmk}

While the case $N=1$ is trivial in the exterior MOY calculus, it is not in the symmetric MOY calculus. However, the symmetric evaluation of a MOY graph $\Gamma$ for $N=1$ is especially simple.

\begin{lem}
  \label{lem:N1symmetriceval}
  Let $\Gamma$ be a MOY graph. For every vertex $v$ of $\Gamma$, let us denote by $W(v)$ the element of $\NN[q, q^{-1}]$ given by the formulas:
\[
W\left(\NB{\tikz[scale =0.8]{ 
\draw[->] (0,0) -- (0,0.5) node [at end, above] {$a+b$};  
\draw[>-] (-0.5, -0.5) -- (0,0) node [at start, below] {$a$};  
\draw[>-] (+0.5, -0.5) -- (0,0) node [at start, below] {$b$};}}\right) := \qbina{a+b}{a} =:
W\left(\NB{\tikz[scale =0.8]{ 
\draw[-<] (0,0) -- (0,0.5) node [at end, above] {$a+b$};  
\draw[<-] (-0.5, -0.5) -- (0,0) node [at start, below] {$a$};  
\draw[<-] (+0.5, -0.5) -- (0,0) node [at start, below] {$b$};}}\right).
\]

The following identities hold in $\ZZ[q, q^{-1}]$:
\begin{align*}
\kupss{\Gamma}_{N=1} &=
\prod_{v \textrm{ split}} W(v) =  \prod_{v \textrm{ merge}}W(v) = \left(\prod_{v \textrm{ vertex of $\Gamma$}} W(v) \right)^{1/2}.
\end{align*}
\end{lem}

\begin{proof}[Sketch of the proof.]
    We only treat the ``merge'' part of the statement.
    Since there is a unique polynomial satisfying the symmetric MOY calculus for $N=1$. It is enough to check that $W(\Gamma):= \prod_{v \textrm{ merge}}W(v)$ satisfies the symmetric MOY calculus of $N=1$.
 For instance, to prove that $W$ satisfies identity~(\ref{eq:symrelsquare3}), one shows the following $q$-binomial identity:
\[
\qbina{n+k}{m}\qbina{m+l}{k} = \sum_{j={\max(0, m-n)}}^n \qbina{l}{k-j}\qbina{n}{m-j}\qbina{n+l+j}{j}
\]
which can be done by induction on $n+m$. 
\end{proof}

Another identity holds both in the exterior and the symmetric MOY calculi. 

\begin{lem}
  \label{lem:otherrelation}The following local identities and their mirror images hold:
  \begin{align} \label{eq:otherrelext}
    \kup{\squaree} = 
    \qbina{r+s}{s}
    \kup{\webHe}
  \end{align}
  \begin{align} \label{eq:otherrelsym}
    \kups{\squaree} = 
    \qbina{r+s}{s}
    \kups{\webHe}
  \end{align}
\end{lem}

\begin{proof}[Sketch of proof.]
  For the exterior calculus: it is a consequence of identity (\ref{eq:extrelass}), its mirror image and identity (\ref{eq:extrelbin1}).
  For the symmetric calculus: it is a consequence of identity (\ref{eq:symrelass}), its mirror image and identity (\ref{eq:symrelbin1}).  
\end{proof}



\subsection{Braid-like MOY graphs}
\label{sec:braid-like-moy}
In this section we introduce a special class of MOY graphs which contains in particular the ones appearing in the expansion of braids when using identities (\ref{eq:extcrossplus}) and (\ref{eq:extcrossminus}) to get rid of crossings. We call these graphs braid-like.

\begin{dfn}
  \label{dfn:braidlike}
  A MOY graph $\Gamma$ is \emph{braid-like} if the scalar products of all its tangent vectors with $\left(\begin{smallmatrix} 0 \\ 1 \end{smallmatrix}\right)$ are strictly positive. In Figure~\ref{fig:exMOY}, the leftmost MOY graph is braid-like, while the two others are not.
\end{dfn}

\begin{rmk}\label{rmk:braidlikeisotopy}\begin{enumerate}
  \item Braid-like MOY graphs are regarded up to ambient isotopy fixing the boundary  and preserving the braid-like property. They fit into a category which is a non-full subcategory of the one described in Remark~\ref{rmk:MOYgraphisopy}.
\item The braid-likeness of a MOY graph $\Gamma$ implies that boundary points  of $\Gamma$ on $]0,1[\times \{0\}$ are negative, while the one on $]0,1[\times \{1\}$ are positive.   
\item Every braid-like MOY graph can be obtained as vertical concatenation of MOY graphs of type: 
\[ \NB{\tikz[scale= 0.7]{
\begin{scope}[yscale = -1]
\draw[-<] (0,0) -- (0,0.5) node [at end, below] {$a+b$};  
\draw[<-] (-0.5, -0.5) -- (0,0) node [at start, above] {$a$};  
\draw[<-] (+0.5, -0.5) -- (0,0) node [at start, above] {$b$};  

\draw[<-] (+1.5, -0.5) -- ++(0,1);
\draw[<-] (+3,   -0.5) -- ++(0,1);
\draw[<-] (-1.5, -0.5) -- ++(0,1);
\draw[<-] (-3,   -0.5) -- ++(0,1);
\node at (2.35, 0) {$\dots$};
\node at (-2.35, 0) {$\dots$};
\end{scope}}} \qquad \textrm{and} \qquad \NB{\tikz[scale= 0.7]{
\begin{scope}[yscale = 1]
\draw[->] (0,0) -- (0,0.5) node [at end, above] {$a+b$};  
\draw[>-] (-0.5, -0.5) -- (0,0) node [at start, below] {$a$};  
\draw[>-] (+0.5, -0.5) -- (0,0) node [at start, below] {$b$};  
\draw[->] (+1.5, -0.5) -- ++(0,1);
\draw[->] (+3,   -0.5) -- ++(0,1);
\draw[->] (-1.5, -0.5) -- ++(0,1);
\draw[->] (-3,   -0.5) -- ++(0,1);
\node at (2.35, 0) {$\dots$};
\node at (-2.35, 0) {$\dots$};
\end{scope}}}
.\]
\item If $\Gamma$ is a $\listk{k}_1$-MOY graph-$\listk{k}_0$, then the sum of the element of $\listk{k}_1$ is equal to the sum of the element of $\listk{k}_0$, we call the number the \emph{level} of $\Gamma$.
  \end{enumerate}
\end{rmk}

The following lemma, although quite elementary and completely combinatorial, is one of the keystones of this paper.

\begin{lem}
  \label{lem:trees}
  Let $k$ be a positive integer and $\listk{k}$ 
be a finite collection of positive integers of level $k$. We consider $M$ the $\ZZ[q,q^{-1}]$ 
module generated by braid-like $\listk{k}$-MOY graphs-$(k)$ and modded out by ambient isotopy and relations~(\ref{eq:extrelass}) and (\ref{eq:extrelbin1}) (or (\ref{eq:symrelass}) and (\ref{eq:symrelbin1})).
The module $M$ is generated by a braid-like tree. Moreover
all braid-like trees are equal in $M$.
\end{lem}

\begin{proof}
  Let $T$ be a braid-like $\listk{k}$-tree-$k$, that is a braid-like $\listk{k}$-MOY graph-$(k)$ which is a tree. 

Thanks to the relation~(\ref{eq:extrelass}), it is clear\footnote{This is similar to saying that the associativity of a product allows to remove parentheses in arbitrary long product. Indeed the first relation can be seen as an associativity property and a tree as a (big) product }  that all  braid-like $\listk{k}$-tree-$k$ are equal in $M$. 

It is enough to show that any braid-like $\listk{k}$-MOY graph-$k$ $\Gamma$ is equal in $M$ to $P$ (a Laurent polynomial in $q$) times a braid-like  $\listk{k}$-tree-$k$. We show this simultaneously on all finite sequences of integers of level $k$ by induction on the number of merge vertices. If there is no merge then $\Gamma$ is a tree and there is nothing to show. If $\Gamma$ contains a merge, we cut $\Gamma$ horizontally into two parts, just below its highest merge. We obtain $\Gamma_{\textrm{top}}$ and $\Gamma_{\textrm{bot}}$. The latter is a $\listk{k}'$-MOY graph-$k$ and has one merge vertex less than $\Gamma$. Hence we can use the induction hypothesis to write $\Gamma_{\textrm{bot}}= P(q)T$ and choose the tree $T$ to have a split vertex on its top part which is symmetric to the merge vertex below which we cut $\Gamma$. We now stack $\Gamma_{\textrm{top}}$ onto $T$, and reduce the digon thanks to the relation~(\ref{eq:extrelbin1}), we obtain that $\Gamma$ is equal in $M$ to a Laurent polynomial times a braid-like $\listk{k}$-tree-$k$. 
\end{proof}

\begin{rmk}
  \label{rmk:treeRT}
  \begin{enumerate}
  \item Note that with the representation theoretic interpretation
    of MOY-graph, this results should be interpreted as: the
    multiplicity of $\Lambda_q^k V$ (resp. $\Sym_q^k
    V$) 
    in $\bigotimes_{i=1}^l \Lambda_q^{k_i} V$
    (resp. $\bigotimes_{i=1}^l \Sym_q^{k_i} V$) is one (see Appendix~\ref{sec:quant-link-invar}).
  \item This lemma says, that for any braid-like $(k_1, \dots,
    k_l)$-MOY graph-$(k)$ $\Gamma$, there exists a Laurent polynomial
    $r(\Gamma)$, such that $\Gamma= r(\Gamma) T$ in the skein module, where $T$ is a braid-like tree. From the proof, one deduces that 
\[
r(\Gamma)= \prod_{\substack{v \in V(\Gamma) \\ \textrm{$v$ merge of type $(a,b,a+b)$}}}
\begin{bmatrix}
  a+b \\ a
\end{bmatrix}. 
\]
  \end{enumerate}
\end{rmk}

For a latter use it will be convenient to have a preferred tree.

\begin{dfn}
  \label{dfn:canonical-tree}
  Let $\listk{k}$ be a finite sequence of positive integers which add up to $k$. We denote by $T_{\listk{k}}$ the braid like $\listk{k}$-tree-$k$ which is obtained by this inductive definition:
  \begin{itemize}
  \item $T_{k}$ is a single vertical strand,
  \item $T_{(k_1, \dots ,k_l)}$ is obtained from $T_{(k_1, \dots,k_{l-2},k_{l-1}+ k_l)}$ by splitting its rightmost strand into two strands labeled by $k_{l-1}$ and $k_l$. 
  \end{itemize}
This is probably better understood with the following figure:
\[
\NB{\tikz[yscale=0.4, xscale=0.6]{
\draw[>->] (0,0) --(0,1) -- (-2,5);
\draw[->] (0,1) -- (2,5);
\draw[->] (1.5,4) -- (1,5);
\draw[->] (1,3) -- (0,5);
\draw[->] (0.5,2) -- (-1,5);
}}.
\]
\end{dfn}

\subsection{Vinyl graphs}
\label{sec:vinyl-moy-graphs}

We denote by $\ann$ the annulus $\{ x\in \RR^2 | 1 < \lVert x\rVert < 2\}$ and for all $x= \left(\begin{smallmatrix}   x_1 \\x _2 \end{smallmatrix}\right) $ in $\ann$, we denote by $t_x$ the vector $\left(\begin{smallmatrix} -x_2 \\x_1 \end{smallmatrix}\right)$. A \emph{ray} in $\RR^2$ is a half-line which starts at $O$, the origin of $\RR^2$.

\begin{dfn}\label{dfn:vynil-graph}
  A \emph{vinyl graph} is the image of an abstract closed MOY graph $\Gamma$ in $\ann$ by a smooth\footnote{The smoothness condition is the same as the one of Definition~\ref{dfn:MOY}.} embedding such that for every point $x$ in the image of $\Gamma$, the tangent vector at this point has a positive scalar product with $t_x$. 
The set of vinyl graphs is denoted by $\Vin$. We define the \emph{level} of a vinyl graph to be the rotational of the underlying MOY graph. If $k$ is a non-negative integer, we denote by $\Vin_k$ the set of vinyl graph with rotational equal to $k$. Vinyl graphs are regarded up to ambient isotopy preserving $\ann$.
\end{dfn}
\begin{figure}[ht]
  \centering
\NB{  \includegraphics[height=5.2cm]{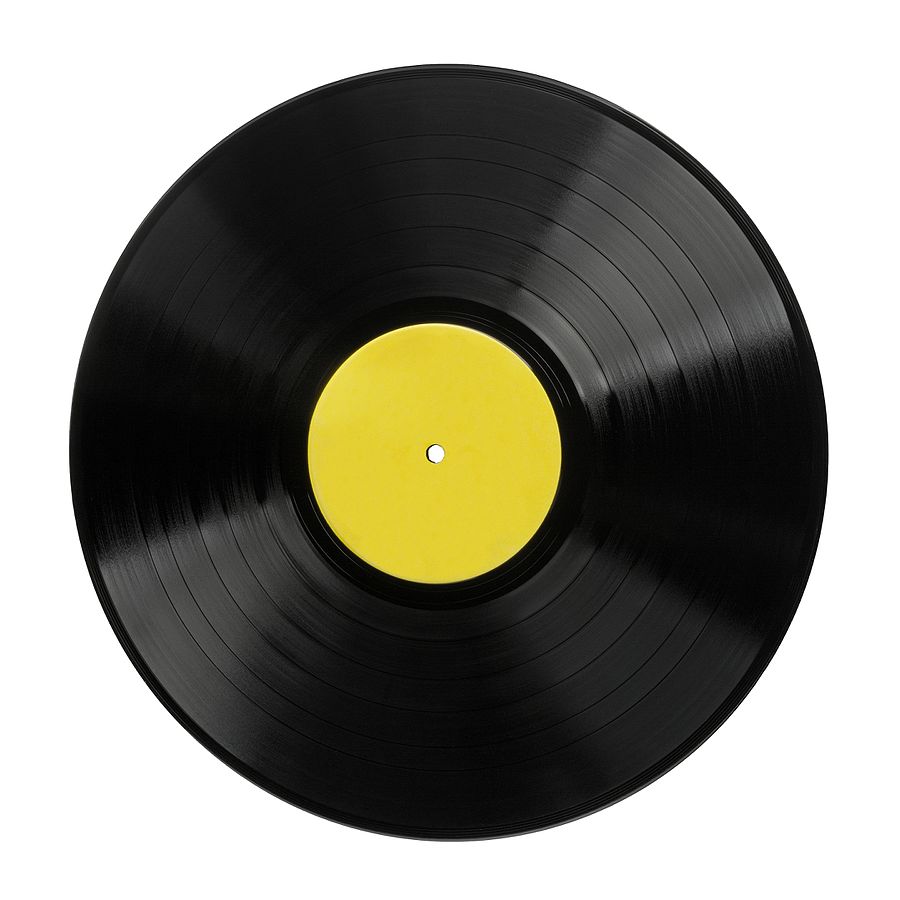}} \qquad \NB{\tiny\tikz[scale=0.6]{\begin{scope}
  \coordinate (O) at (0, 0);
  \draw[dotted] (O) circle (2cm);
  \draw[dotted] (O) circle (4cm);
  \coordinate (F1) at (30:3.8);
  \coordinate (F2) at (90:3.8);
  \coordinate (F3) at (-45:3.8);
  \coordinate (M1) at (135:3);
  \coordinate (M2) at (-90:3);
  \coordinate (M3) at (-45:3);
  \coordinate (I1) at (90:2.2);
  \coordinate (I2) at (180:2.2);
  \coordinate (I3) at (225:2.2);
  \coordinate (I4) at (-90:2.2);
  \draw[->-] (I1) arc (90:180:2.2)  node[midway, right] {$1$};
  \draw[->-] (I2) arc (180:225:2.2)  node[midway, right] {$6$};
  \draw[->-] (I3) arc (-135:-90:2.2)  node[midway, above] {$4$};;
  \draw[->-] (I4) arc (-90:90:2.2) node[midway, left] {$2$};
  \draw[->-] (M3) .. controls +(45:0.5) and +(-60:0.5) .. (F1) node[midway, right] {$3$};
  \draw[->-] (F2) .. controls +(180:0.5) and +(45:0.5) .. (M1)  node[midway, right] {$4$};
  \draw[->-] (I1) .. controls +(180:0.5) and +(45:0.5) .. (M1) node[midway, above] {$1$};
  \draw[->-] (M1) .. controls +(225:0.5) and +(90:0.5) .. (I2)  node[midway, right] {$5$};
  \draw[->-] (I3) .. controls +(-45:0.5) and +(180:0.5) .. (M2)  node[midway, below] {$2$};
  \draw[->-] (I4) .. controls +(0:0.5) and +(225:0.5) .. (M3)  node[near end, above] {$2$};
  \draw[->-] (M2) .. controls +(0:0.5) and +(225:0.5) .. (M3) node[near start, above] {$1$};
  \draw[->-] (M2) .. controls +(0:0.5) and +(225:0.5) .. (F3)  node[near end, above] {$1$};
  \draw[->-] (F1) .. controls +(120:1.5) and +(0:1.5) .. (F2)  node[midway, right] {$5$};
  \draw[->-] (F2) .. controls +(180:1.5) and  +(70:1) .. (160:3.5) .. controls +(250:1.5) and +(135:1.5) .. (225:3.5) node[left] {$1$} .. controls +(-45:1.5) and +(225:1.5) .. (F3);
  \draw[->-] (F3) .. controls +(45:1.5) and +(-60:1.5) .. (F1)  node[midway, right] {$2$};
\end{scope}}}
  \caption{A vinyl record and a vinyl graph with rotational equal to 7.}
  \label{fig:justforfun}
\end{figure}
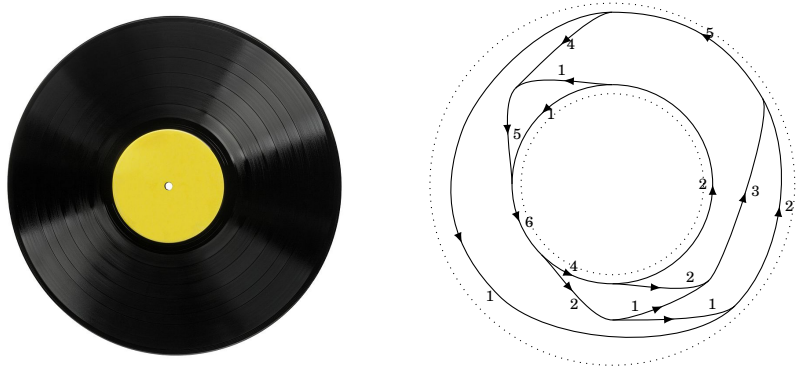
\begin{rmk}
  \label{rmk:rotational}
  Let $\Gamma$ be a vinyl graph with rotational $k$, and $D$ be a ray which does no contain any vertices of $\Gamma$. Then the condition on the tangent vectors of $\Gamma$, implies that: 
  \begin{itemize}
  \item the intersection points of the ray $D$ with $\Gamma$ are all transverse and positive,
  \item the sum of the labels of edges which intersects $\Gamma$ is equal to $k$.
  \end{itemize}
Informally the level counts the numbers of tracks of a vinyl graph.
\end{rmk}
Of course, a natural way to obtain vinyl graphs is by closing  braid-like MOY graphs. 

\begin{notation}
  \label{not:braidlikeclosure}
  Let $\listk{k}$ be a finite sequence of integers and $\Gamma$ be a braid-like $\listk{k}$-MOY graph-$\listk{k}$. Then we denote by $\widehat{\Gamma}$ the vinyl graph obtained by closing up $\Gamma$. The level of $\Gamma$ equals the level of $\listk{k}$. 
\end{notation}

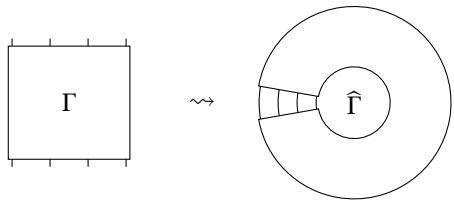
\begin{figure}[ht]
  \centering 
  \begin{tikzpicture}[scale=1]
    \begin{scope}[scale=0.5]
\begin{scope}
  \draw (-0.1,0) -- (3.1,0) -- (3.1,3) -- (-0.1,3) --cycle;
  \draw (0,0) -- +(0,-0.2);
  \draw (1,0) -- +(0,-0.2);
  \draw (2,0) -- +(0,-0.2);
  \draw (3,0) -- +(0,-0.2);
  \draw (0,3) -- +(0, 0.2);
  \draw (1,3) -- +(0, 0.2);
  \draw (2,3) -- +(0, 0.2);
  \draw (3,3) -- +(0, 0.2);
  \node at (1.5,1.5) {$\Gamma$};
\end{scope}
\node at (5,1.5) {$\rightsquigarrow$};
\begin{scope}[xshift= 9cm, yshift=1.5cm]
\draw (-170:0.95) arc (-170:170:0.95) -- ++(170:1.6) arc (170: -170:2.55) -- cycle;
\draw (-170:1) arc (-170:-190:1);
\draw (-170:1.5) arc (-170:-190:1.5);
\draw (-170:2) arc (-170:-190:2);
\draw (-170:2.5) arc (-170:-190:2.5);
\node at (0,0) {$\widehat{\Gamma}$};
\end{scope}
\end{scope}
  \end{tikzpicture}
  \caption{The vinyl graph $\widehat{\Gamma}$ is obtained by closing up the braid-like MOY graph $\Gamma$.}
  \label{fig:closebraidlike}
\end{figure}

The following theorem from Queffelec and Rose shows that the MOY relations~(\ref{eq:extrelcircle}), (\ref{eq:extrelass}), (\ref{eq:extrelbin1}), (\ref{eq:extrelsquare3}) and (\ref{eq:otherrelext}) (resp. (\ref{eq:symrelcircle}), (\ref{eq:symrelass}), (\ref{eq:symrelbin1}), (\ref{eq:symrelsquare3}) and (\ref{eq:otherrelsym})) defines uniquely the exterior (resp. symmetric) MOY calculus for vinyl graphs. 

\begin{thm}[{\cite[Lemma 5.2]{QueffelecRoseAnnular2017}}]
  \label{thm:QR}
  Let $S_k$ be the $\ZZ[q,q^{-1}]$-module generated by vinyl graphs of level $k$ modded out by the relations~(\ref{eq:extrelass}), (\ref{eq:extrelbin1}), (\ref{eq:extrelsquare3}) and (\ref{eq:otherrelext}). The module $S_k$ is generated by vinyl graphs which are collections of circles of level $k$. 
\end{thm}

The proof of this result is constructive. Queffelec and Rose give an algorithm which reduces a vinyl graph using  (\ref{eq:extrelass}), (\ref{eq:extrelbin1}), (\ref{eq:extrelsquare3}) and (\ref{eq:otherrelext}). Their algorithm produces a linear combination of collection of circles. Hence adding relation~(\ref{eq:extrelcircle}) (resp. relation~(\ref{eq:symrelcircle})) we obtain the exterior  (resp. symmetric) evaluation of any vinyl graph. 

In what follows we categorify the symmetric MOY calculus of vinyl graphs. We construct a category $\TL_\NN$ whose objects are vinyl graphs.  Suppose that $\mathcal{F}$ is a functor from the category $\TL_\NN$ to a category $\mathcal{C}$ with a grading, such that the relations~(\ref{eq:extrelass}), (\ref{eq:extrelbin1}), (\ref{eq:extrelsquare3}) and (\ref{eq:otherrelext}) are compatible with $\mathcal{F}$ (i.~e.~linear combinations translates into directs sums of objects with degree shifts). From the algorithm of Queffelec and Rose, we can deduce the following proposition: 

\begin{prop}[{\cite[Proof of Proposition 5.1]{QueffelecRoseAnnular2017}}]
  \label{prop:QR2}
For any vinyl graph $\Gamma$, there exist two $\NN[q,q^{-1}]$-linear combinations of collections of circles $\sum_i a_i C_i$ and  $\sum_j b_j C'_j$ such that:
\[
\mathcal{F}(\Gamma) \oplus \bigoplus_i \mathcal{F}(C_i)\{a_i\} \simeq \bigoplus_j \mathcal{F}(C'_j)\{b_j\}.  
\]
In particular, if $\mathcal{C}$ is a category of modules over an algebra $A$ and collections of circles are mapped to finitely generated projective modules, then $\mathcal{F}(\Gamma)$ is finitely generated and projective. Moreover, if $A$ is a polynomial algebra, then $\mathcal{F}(\Gamma)$ is free and relations~(\ref{eq:extrelass}), (\ref{eq:extrelbin1}), (\ref{eq:extrelsquare3}) and (\ref{eq:otherrelext}) are satisfied by the graded rank of the modules.
\end{prop}

\section{Foams}
\label{sec:foams}
Foams have been introduced in the realm of link homologies by Khovanov~\cite{MR2100691}. They have been used by Blanchet \cite{1195.57024} to fix functoriality of link homologies. They are now widely used \cite{queffelec2014mathfrak,1212.6076,MR3611714}. 
\subsection{Definitions}
\label{sec:definitions}
In the first two subsections we summarize some of the results of \cite{RW1}. However we think that familiarity with \cite{RW1} is essential to fully understand the constructions done in Sections~\ref{sec:disk-like-foams}, \ref{sec:quasi-annular-foams} and \ref{sec:evaluation-vinyl}. We fix a positive integer $N$. 

\begin{dfn}\label{dfn:foam}
  An \emph{abstract foam} $F$ is a finite collection
  of \emph{facets} $\mathcal{F}(F)=(\Sigma_i)_{i\in I}$, that is a finite set of
  oriented connected surfaces with boundary, together with the
  following data:
  \begin{itemize}
  \item A labeling $l\co (\Sigma_i)_{i\in I} \to \{0, \dots, N\}$,
  \item A ``gluing recipe'' of the facets along their boundaries such
    that when glued together using the recipe a neighborhood of a point of the foam has three possible local models:
\[
      \begin{tikzpicture}
        \begin{scope}
\tdplotsetmaincoords{80}{140}
  \begin{scope}[tdplot_main_coords]
    \filldraw [very thin, fill=red, opacity = 0.2] (0,-1, -1) -- (0,1,-1) -- (0,1,1) -- (0,-1,1) -- (0,-1,-1);
    \node[sloped, red] at (0,0,0) {$a$};
  \end{scope}
\begin{scope}[xshift = 3cm, tdplot_main_coords]
  \begin{scope}
    \filldraw [very thin, fill =red, opacity = 0.2] (0,-1, -1) -- (0,1,-1) -- (0,1,0) -- (0,-1,0) -- (0,-1,-1);
        \node[sloped, red] at (0,0,-0.5) {$a+b$};
    \end{scope}
  \begin{scope}[rotate around y = 150]
    \filldraw [very thin, fill =blue, opacity = 0.2] (0,-1, -1) -- (0,1,-1) -- (0,1,0) -- (0,-1,0) -- (0,-1,-1);
        \node[sloped, blue] at (0,0,-0.5) {$a$};
  \end{scope}
  \begin{scope}[rotate around y =-130]
    \filldraw [very thin, fill =green, opacity = 0.2] (0,-1, -1) -- (0,1,-1) -- (0,1,0) -- (0,-1,0) -- (0,-1,-1);
        \node[sloped, green!50!black] at (0,0,-0.5) {$b$};
        \draw[very thick, ->] (0,1,0) -- (0,-1,0);
  \end{scope}
  \end{scope}
\begin{scope}[scale = 1.6, xshift = 4.5cm, tdplot_main_coords]
  \begin{scope}
    \filldraw [very thin, fill =red, opacity = 0.2] (0,-1, -1) -- (0,1,-1) -- (0,1,0) -- (0,0,0) -- (0,-1,-1);
    \coordinate (a) at (0, -1, -1);
        \node[sloped, red] at (0,0.2,-0.5) {$a+b+c$};
    \end{scope}
  \begin{scope}[rotate around y = 150]
    \filldraw [very thin, fill= blue, opacity = 0.2] (0,-1, -1) -- (0,1,-1) -- (0,1,0) -- (0,0,0) -- (0,-1,-1);
    \coordinate (b) at (0, -1, -1);
        \node[sloped, blue] at (0,0.5,-0.5) {$a+b$};
  \end{scope}
  \begin{scope}[rotate around y =-130]
    \filldraw [very thin, fill= green, opacity = 0.2] (0,-1, -1) -- (0,1,-1) -- (0,1,0) -- (0,0,0) -- (0,-1,-1);
    \coordinate (c) at (0, -1, -1);
        \node[sloped, green!50!black] at (0,0.5,-0.5) {$c$};
        \node[sloped, orange!50!black] at (0,-1.3,0.8) {$a$};
        \node[sloped, purple!50!black] at (0,+0.5,-1.5) {$b$};
        \node[sloped, gray] at (0,-1.3,0.1) {$b+c$};
  \end{scope}
  \filldraw[very thin, fill= orange, opacity = 0.2] (a) -- (b) -- (0,0,0) -- (a);
  \filldraw[very thin, fill= gray, opacity = 0.2] (a) -- (c) -- (0,0,0) -- (a);
  \filldraw[very thin, fill= purple, opacity = 0.2] (b) -- (c) -- (0,0,0) -- (b);
  \draw[very thick, ->] (a) -- (0,0,0);
  \draw[very thick, <-] (b) -- (0,0,0);
  \draw[very thick, ->] (c) -- (0,0,0);
  \draw[very thick, <-] (0,1, 0) -- (0,0,0);
  \fill[green!50!black, opacity = 0.6] (0,0,0) circle (0.5mm);
\end{scope}
\end{scope}

      \end{tikzpicture}
      \label{fig:FBSP} \]
    The letter appearing on a facet indicates the label of this facet.
    That is we have facets, \emph{bindings} (which are compact oriented
    $1$-manifolds) and \emph{singular points}. Each binding carries:
    \begin{itemize}
    \item an orientation which agrees with the orientations of the facets with
      labels $a$ and $b$ and disagrees with the orientation of the facet with label
      $a+b$.
    \item a cyclic ordering of the three facets around it. When a foam
      is embedded in $\RR^3$, we require this cyclic ordering to agree with 
      the left-hand rule\footnote{This agrees with Khovanov's convention \cite{MR2100691}.} with respect to its
      orientation (the dotted circle in the middle indicates that the
      orientation of the binding points to the reader, a crossed circle indicates the other orientation, see Figure~\ref{fig:signsofcircles}):
      \[
        \begin{tikzpicture}[xscale=1]
          \begin{scope}
 \draw (0,0) -- +(0:1);
 \draw (0,0) -- +(120:1);
 \draw (0,0) -- +(240:1);
 \filldraw[fill= white, draw=black, very thin] (0,0) circle (0.15cm);
 \filldraw[fill = black] (0,0) circle (0.02cm);
 \draw[very thin,->] (-10:0.8cm) arc (-10:-110 :0.8); 
 \draw[very thin,->] (230:0.8cm) arc (230: 130:0.8); 
 \draw[very thin,->] (110:0.8cm) arc (110:10 :0.8); 
\end{scope}
        \end{tikzpicture}
      \]
    \end{itemize}
    The cyclic orderings of the different bindings adjacent to a
    singular point should be compatible. This means that a
    neighborhood of the singular point is embeddable in $\RR^3$ in a
    way that respects the left-hand rule for the four bindings adjacent to this singular point.
  \end{itemize}
\end{dfn}

\begin{rmk}
  Les us explain shortly what is meant by ``gluing recipe''. The boundaries of the facets forms a collection of circles. We denote it by  $\mathcal{S}$. The gluing recipe consists of:
  \begin{itemize}
  \item For a subset $\mathcal{S}'$ of $\mathcal{S}$, a subdivision of each circle of $\mathcal{S'}$ into a finite number of closed intervals. This gives us a collection $\mathcal{I}$ of closed intervals.  
  \item Partitions of $\mathcal{I} \cup (\mathcal{S} \setminus \mathcal{S'})$ into subsets of three elements. For every subset $(X_1, X_2, X_3)$ of this partition, three diffeomorphisms $\phi_1 : X_2 \to X_3$, $\phi_2 : X_3 \to X_1$, $\phi_3 : X_1 \to X_2$  such that $\phi_3 \circ \phi_2 \circ \phi_1 = \mathrm{id}_{X_2}$.
  \end{itemize}
A foam is obtained by gluing the facets along the diffeomorphisms, provided that the conditions given in the previous definition are fulfilled.  
\end{rmk}

\begin{dfn}\label{dfn:decoratedfoam}
  A \emph{decoration} of a foam $F$ is a map $f\mapsto P_f$ which associates with any facet $f$ of $F$ an homogeneous symmetric polynomial $P_f$ in $l(f)$ variables. A \emph{decorated foam} is a foam together with a decoration.
\end{dfn}

From now on all foams are decorated.

\begin{dfn}
A \emph{closed foam} is a smoothly embedded abstract foam in $\RR^3$. Smoothness means that the facets are smoothly embedded and the different oriented tangent planes agree on bindings and singular points as depicted here
  \[
\NB{      \begin{tikzpicture}
        \begin{scope}
\tdplotsetmaincoords{80}{140}
\begin{scope}[xshift = 3cm, tdplot_main_coords]
  \begin{scope}
    \filldraw [ thin, fill =red, opacity = 0.2] (0,-1, -1) -- (0,1,-1) -- (0,1,0) -- (0,-1,0) -- (0,-1,-1);
        \node[sloped, red] at (0,0,-0.5) {$a+b$};
    \end{scope}
  \begin{scope}[rotate around y = 150]
    \filldraw [thin, fill =blue, opacity = 0.2] (0,-1, -1) -- (0,1,-1) .. controls +(0,0,0) and +(135:2) .. (0,1,0) -- (0,-1,0) .. controls +(135:2) and +(0,0,0)  .. (0,-1,-1);
        \node[sloped, blue] at (0,0,-0.5) {$a$};
  \end{scope}
  \begin{scope}[rotate around y =-130]
    \filldraw [ thin, fill =green, opacity = 0.2] (0,-1, -1) -- (0,1,-1) .. controls +(0,0,0) and +(-40:0.5) .. (0,1,0) -- (0,-1,0) .. controls +(-40:0.5) and +(0,0,0) .. (0,-1,-1);
        \node[sloped, green!50!black] at (0,0,-0.5) {$b$};
        \draw[thick, ->] (0,1,0) -- (0,-1,0);
  \end{scope}
  \end{scope}
\tdplotsetmaincoords{70}{20}
\begin{scope}[scale = 1.6, xshift = 5cm, tdplot_main_coords, scale =0.9]
\coordinate (O) at (0, 0, 0);
\coordinate (A) at (-1, -1, 0);
\coordinate (B) at ( 1, -1, 0);
\coordinate (C) at ( 1,  1, 0);
\coordinate (D) at (-1,  1, 0);
\coordinate (a) at (0, -1, 0);
\coordinate (b) at ( 1, 0, 0);
\coordinate (c) at (0,  1, 0);
\coordinate (d) at (-1, 0, 0);
\coordinate (Ta) at ( -.2,-1, 1);
\coordinate (Bb) at ( 1, -.2,-1);
\coordinate (Tc) at ( -.2, 1, 1);
\coordinate (Bd) at (-1, -.2,-1);

\begin{scope}
  \filldraw [thin, fill =blue, opacity = 0.2] (Bd) -- (Bb) .. controls +(0,0) and  ($0.3*(B) + 0.7*(b)$).. (b) -- (d) .. controls  ($0.3*(A) + 0.7*(d)$) and +(0,0) ..  (Bd); 
  \node[blue] at ($0.5*(Bb) + 0.5*(d)$) {$a$};
 \draw[thick, ->] (d) -- (O);
  \filldraw [thin, fill =red, opacity = 0.2] (A) -- (a) -- (O) -- (d) -- (A); 
  \node[red] at ($0.5*(A) + 0.5*(O)$) {$b$};
  \filldraw [thin, fill =orange, opacity = 0.2] (B) -- (b) -- (O) -- (a) -- (B); 
  \node[orange] at ($0.5*(B) + 0.5*(O)$) {$b+c$};
  \filldraw [thin, fill =yellow, opacity = 0.2] (C) -- (c) -- (O) -- (b) -- (C); 
  \node[yellow!50!black] at ($0.5*(C) + 0.5*(O)$) {$a+b+c$};
  \filldraw [thin, fill =gray, opacity = 0.2] (D) -- (d) -- (O) -- (c) -- (D); 
  \node[gray!50!black] at ($0.5*(D) + 0.5*(O)$) {$a+b$};
  \filldraw [thin, fill =green, opacity = 0.2] (Ta) -- (Tc) .. controls +(0,0) and ($0.3*(D) + 0.7*(c)$).. (c) -- (a) .. controls ($0.3*(A) + 0.7*(a)$) and +(0,0) ..  (Ta); 
  \node[green!50!black] at ($0.5*(Tc) + 0.5*(a)$) {$c$};
  \draw[thick, <-] (a) -- (O);
   \draw[ thick, <-] (b) -- (O);
   \draw[ thick, ->] (c) -- (O);
 
\end{scope}

\end{scope}
\end{scope}

      \end{tikzpicture}}.
\]
Just like for MOY graphs (see Definition~\ref{dfn:MOY}), we will usually not care too much about the smoothness on bindings and singular points when drawing foams. 
\end{dfn}

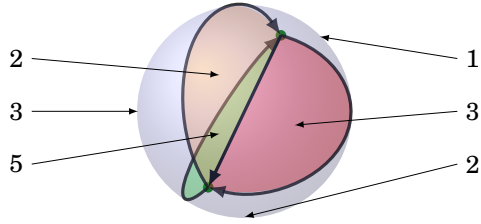
\begin{figure}[h]
  \centering
  \begin{tikzpicture}[scale = 0.7]
    \begin{scope}
\tdplotsetmaincoords{40}{20}
\begin{scope}[tdplot_main_coords]
  \fill[green!50!black] (0,2) circle (1mm);
  \fill[green!50!black] (0,-2) circle (1mm);
  \begin{scope}[rotate around y=0]
    \coordinate (b) at (1,0);
    \fill [fill =red,fill opacity=0.4 ] (0,2) arc (90:-90:2) ;
    \draw [draw = black, very thick , ->, opacity=0.8 ] (0,2) arc (90:-90:2) ;
  \end{scope}
  \begin{scope}[rotate around y=120]
        \coordinate (a) at (1,0);
    \fill [green ,fill opacity=0.3 ] (0,2) arc (90:-90:2) ;
    \draw [draw = black, very thick , <-, opacity=0.8 ] (0,2) arc (90:-90:2) ;
  \end{scope}
  \begin{scope}[rotate around y=240]
    \coordinate (c) at (1,0);
    \fill [orange, fill opacity=0.3 ] (0,2) arc (90:-90:2) ;
    \draw [draw = black, very thick , <-,opacity=0.8 ] (0,2) arc (90:-90:2) ;
  \end{scope}
  \draw[very thick, -> ] (0,2) -- (0,-2);
  \shade [ball color=blue!40!white,opacity=0.30] (0,0,0) circle (2.02cm);
\end{scope}
  \coordinate (d) at (45:2);
  \coordinate (e) at (-90:2);
  \coordinate (f) at (180:2);
  \draw[very thin, -> ] (-4,0)  node [left] {$3$} -- (f);
  \draw[very thin, -> ] (-4,1)    node [left] {$2$} -- (c);
  \draw[very thin, -> ] (-4,-1)  node [left] {$5$} -- (a);
  \draw[very thin, -> ] (4,0)   node [right] {$3$} -- (b);
  \draw[very thin, -> ] (4,-1)   node [right] {$2$} -- (e);
  \draw[very thin, -> ] (4,1)    node [right] {$1$}-- (d);
\end{scope}
  \end{tikzpicture}
  \caption{Example of a foam.
The cyclic ordering on the central binding is $(5,2,3)$.}
  \label{fig:exslnfoam}
\end{figure}


The notion of foam extends naturally to the notion of foam with boundary. The boundary of a foam has a structure of a MOY graph. We require that the facets and bindings are locally orthogonal to the boundary to be able to glue them together. Probably the most local framework is given by the concept of \emph{canopolis} of foams. We refer to \cite{MR2174270, FunctorialitySLN} for more details about this approach. In what follows we will consider:
\begin{itemize}
\item Foams in $\RR^2 \times [0,1]$ where the boundary is contained in $\RR^2 \times \{0,1\}$ (see Definition~\ref{dfn:catfoam});
\item Foams in the cube $[0,1]^3$, where the boundary is contained in $[0,1]^2 \times \{0,1\} \cup \{0,1\}\times [0,1]^2$ (see Section~\ref{sec:disk-like-foams});
\item Foams in the thickened annulus $\ann \times [0,1]$ where the boundary is contained in $\ann\times \{0,1\}$ (see Section~\ref{sec:quasi-annular-foams}).
\end{itemize}

\begin{dfn}
  \label{dfn:catfoam}
  The category $\Foam$ consists of the following data:
  \begin{itemize}
\item Objects are closed MOY graphs,
\item Morphisms from $\Gamma_0$ to $\Gamma_1$ are (ambient isotopy classes relatively to the boundary of) foams in $\RR^2\times [0,1]$ whose boundary is contained in  $\RR^2\times\{0,1\}$. The part of the boundary in $\RR^2\times\{0\}$ (resp. $\RR^2\times\{1\}$) is required to be equal to $-\Gamma_0$ (resp. $\Gamma_1$).     
  \end{itemize}
Composition of morphisms is given by stacking foams and rescaling in the vertical direction.
\end{dfn}

\begin{dfn}
\label{dfn:subsurfaces}
 If $F$ is a foam (possibly with boundary), a \emph{sub-surface} is a collection of  $\mathcal{U}$ of oriented facets $F$ such that
 their union is a smooth oriented surface $\Sigma$ whose boundary is contained in the boundary of $F$.
\end{dfn}


\subsection{Reminder on the exterior evaluation of foams}
\label{sec:rema-exter-foams}

The combinatorial evaluation can be thought of as a higher dimensional state sum formula of the state sum of Murakami--Ohtsuki--Yamada \cite{MR1659228} for evaluating MOY graphs. 

\begin{dfn}\label{dfn:color}
  A \emph{pigment} is an element of $\Col= \{1,\dots,N\}$. 
The set $\Col$ is endowed with the natural order.

  A \emph{coloring} of a foam $F$ is a map $c\co \mathcal{F}(F) \to \mathcal{P}(\Col)$, such that
  \begin{itemize}
  \item For each facet $f$, the number of elements $\# c(f)$ of $c(f)$ is equal to $l(f)$.
  \item For each binding joining a facet $f_1$ with label $a$, a facet $f_2$ with label $b$, and a facet $f_3$ with label $a+b$, we have $c(f_1) \cup c(f_2) = c(f_3)$. This condition is called the \emph{flow condition}.
  \end{itemize}
  A \emph{colored foam} is a foam together with a coloring.
  For a given foam $F$, the set of all its colorings is denoted $\mathrm{col}_N(F)$.
\end{dfn}

A careful inspection of the local behavior of colorings in the neighborhood of bindings and singular points gives the following lemma:

\begin{lem}\label{lem:monobicercle}
  \begin{enumerate}
  \item If $(F,c)$ is a colored foam and $i$ is an element of $\Col$, the union (with the identification coming from the gluing procedure) of all the facets which contain $i$ in their colors is a surface. It is called the \emph{monochrome surface of $(F,c)$ associated with $i$} and is denoted by $F_i(c)$. The restriction we imposed on the orientations of facets ensure that $F_i(c)$ is oriented. 
  \item  If $(F,c)$ is a colored foam and $i$ and $j$ are two distinct elements of $\Col$, the union (with the identification coming from the gluing procedure) of all the facets which contain $i$  or $j$ but not both in their color set is a surface. It is called the \emph{bichrome surface of $(F,c)$ associated with $i,j$}. This the symmetric difference of $F_i(c)$ and $F_j(c)$ and is denoted by $F_{ij}(c)$. The restriction imposed on the orientations of facets ensures that $F_{ij}(c)$ can be oriented via taking the orientation of facets containing $i$ and the reverse orientations on facets containing $j$.
  \item Let $i<j$ and consider a binding joining the facets $f_1$, $f_2$ and $f_3$. Suppose that $i\in c(f_1)$, $j \in c(f_2)$ and $\{i,j\} \subseteq c(f_3)$. We say that the binding is \emph{positive with respect to $(i,j)$} if the cyclic order on the binding is $(f_1, f_2, f_3)$ and \emph{negative with respect to $(i,j)$} otherwise. The set $F_{i}(c) \cap F_{j}(c) \cap F_{ij}(c)$ is a collection of disjoint circles. Each of these circles is a union of bindings; for every circle the bindings are either all positive or all negative with respect to $(i,j)$. 
\end{enumerate}
\end{lem}

Please note that the previous lemma contains the definition of \emph{monochrome} and \emph{bichrome surfaces}.
\begin{exa}[{\cite[Example 2.7]{RW1}}]
  \label{exa:torus}
  Suppose $N=4$ (and therefore $\Col = \{{\color{blue}1}, {\color{red}2}, {\color{orange!50!black}3}, {\color{green!50!black}4}\}$) and consider the colored foam $(F,c)$ given by the figure below 
\[
\NB{\tikz[scale =0.8]{
\begin{scope}
      \draw (4,0) arc (0:360: 4cm and 2cm) coordinate[pos = 0.8] (O1) coordinate[pos = 0.3] (O2);
      \draw (-2,0.3)  .. controls +(1.2,-0.6) and +(-1.2, -0.6) .. (2,0.3) coordinate[pos = 0.2] (L) coordinate[pos = 0.8] (R) coordinate[pos = 0.7] (M1) ;
      \draw (L) .. controls +(0.8,0.4) and +(-0.8, 0.4) .. (R) coordinate[pos = 0.2] (M2);
      \node at (125:1.9cm) {$3$}; 
      \node at (340:1.1cm) {$3$}; 
      \draw[thick, ->-] (O1)  .. controls +(0.9,0.2) and +(0.9,0.2 ) .. (M1);
      \draw[densely dotted ] (O1)  .. controls +(-0.9,-0.2) and +(-0.9,-0.1 ) .. (M1);
      \fill[opacity = 0.2, gray] (O1)  .. controls +(0.9,0.2) and +(0.9,0.2 ) .. (M1) ..  controls +(-0.9,-0.1 ) and +(-0.9,-0.2) .. (O1);
      \draw[thick, -<-] (O2)  .. controls +(0.9,0.2) and +(0.9,0.2 ) .. (M2);
      \draw[densely dotted] (O2)  .. controls +(-0.9,-0.2) and +(-0.9,-0.1 ) .. (M2);
      \fill[opacity = 0.2, gray] (O2)  .. controls +(0.9,0.2) and +(0.9,0.2 ) .. (M2) ..  controls +(-0.9,-0.1 ) and +(-0.9,-0.2) .. (O2);
      \node at (50:2cm) {$1$}; 
      \node at (170:2.9cm) {$2$}; 
      \begin{scope}[font= \tiny]
        \node at (-1, 1.05) {$\{{\color{blue}1}, {\color{red}2}, {\color{green!50!black}4}\}$};
        \node at (2.8, -0.3) {$\{ {\color{blue}1}\}$};
        \node at (1,-0.95) {$\{{\color{blue}1}, {\color{red}2}, {\color{green!50!black}4}\}$};
        \node at (-1.4,-0.9) {$\{ {\color{red}2}, {\color{green!50!black}4}\}$};
      \end{scope} 
\end{scope}}
}
\]
where the big digits represent labels. Note that the orientation of
every facet can be deduced from the orientations of the
bindings. Tables~\ref{tab:exa1} and \ref{tab:exa2} describe the monochrome and bichrome
surfaces as well as the values of $\theta^+_{ij}(c)$ for this colored foam.
\begin{table}[ht]
  \centering
  \begin{tabular}{|>{\centering\arraybackslash}m{1.5cm}|| >{\centering\arraybackslash}m{4.5cm} || >{\centering\arraybackslash}m{4cm}|}
  \hline  
  $i \in \mathbb{P}$ &Monochrome Surface $F_i(c)$&  In words \\ \hline \hline
  {\color{blue}{{${1}$}}}& 
 \begin{tikzpicture}[scale=0.3]
      \draw (4,0) arc (0:360: 4cm and 2cm) coordinate[pos = 0.8] (O1) coordinate[pos = 0.3] (O2);
      \draw (-2,0.3)  .. controls +(1.2,-0.6) and +(-1.2, -0.6) .. (2,0.3) coordinate[pos = 0.2] (L) coordinate[pos = 0.8] (R) coordinate[pos = 0.7] (M1) ;
      \draw (L) .. controls +(0.8,0.4) and +(-0.8, 0.4) .. (R) coordinate[pos = 0.2] (M2);
      \draw (O1)  .. controls +(0.9,0.2) and +(0.9,0.2 ) .. (M1);
      \draw[densely dotted ] (O1)  .. controls +(-0.9,-0.2) and +(-0.9,-0.1 ) .. (M1);
      \draw (O2)  .. controls +(0.9,0.2) and +(0.9,0.2 ) .. (M2);
      \draw[densely dotted] (O2)  .. controls +(-0.9,-0.2) and +(-0.9,-0.1 ) .. (M2);
      \fill[blue, opacity = 0.5] (M1) ..  controls +(-0.9,-0.1 ) and +(-0.9,-0.2) .. (O1) arc (-72: 108: 4cm and 2cm) .. controls +(-0.9,-0.2) and +(-0.9,-0.2 ) .. (M2);
      \fill[white]  (L) .. controls +(0.8,0.4) and +(-0.8, 0.4) .. (R)  .. controls +(-0.8,-0.27) and +(0.8, -0.27) ..(L);
 \draw[opacity = 0.5] (4,0) arc (0:360: 4cm and 2cm) coordinate[pos = 0.8] (O1) coordinate[pos = 0.3] (O2);
      \draw (-2,0.3)  .. controls +(1.2,-0.6) and +(-1.2, -0.6) .. (2,0.3) coordinate[pos = 0.2] (L) coordinate[pos = 0.8] (R) coordinate[pos = 0.7] (M1) ; 
      \draw (L) .. controls +(0.8,0.4) and +(-0.8, 0.4) .. (R); 
   \end{tikzpicture}
&Sphere (on the right)
\\ \hline
  \color{red}{$2$}& 
 \begin{tikzpicture}[scale=0.3]
      \draw (4,0) arc (0:360: 4cm and 2cm) coordinate[pos = 0.8] (O1) coordinate[pos = 0.3] (O2);
      \draw (-2,0.3)  .. controls +(1.2,-0.6) and +(-1.2, -0.6) .. (2,0.3) coordinate[pos = 0.2] (L) coordinate[pos = 0.8] (R) coordinate[pos = 0.7] (M1) ;
      \draw (L) .. controls +(0.8,0.4) and +(-0.8, 0.4) .. (R) coordinate[pos = 0.2] (M2);
      \draw (O1)  .. controls +(0.9,0.2) and +(0.9,0.2 ) .. (M1);
      \draw[densely dotted ] (O1)  .. controls +(-0.9,-0.2) and +(-0.9,-0.1 ) .. (M1);
      \draw (O2)  .. controls +(0.9,0.2) and +(0.9,0.2 ) .. (M2);
      \draw[densely dotted] (O2)  .. controls +(-0.9,-0.2) and +(-0.9,-0.1 ) .. (M2);
      \fill[red, opacity = 0.5] (M1) ..  controls +(0.9,0.1 ) and +(+0.9,+0.2) .. (O1) arc (288: 108: 4cm and 2cm) .. controls +(0.9,0.2) and +(0.9,0.2 ) .. (M2);
      \fill[white]  (L) .. controls +(0.8,0.4) and +(-0.8, 0.4) .. (R)  .. controls +(-0.8,-0.27) and +(0.8, -0.27) ..(L);
 \draw[opacity = 0.5] (4,0) arc (0:360: 4cm and 2cm) coordinate[pos = 0.8] (O1) coordinate[pos = 0.3] (O2);
      \draw (-2,0.3)  .. controls +(1.2,-0.6) and +(-1.2, -0.6) .. (2,0.3) coordinate[pos = 0.2] (L) coordinate[pos = 0.8] (R) coordinate[pos = 0.7] (M1) ; 
      \draw (L) .. controls +(0.8,0.4) and +(-0.8, 0.4) .. (R);
   \end{tikzpicture}
&Sphere (on the left)
\\ \hline
  \color{orange!50!black}{$3$}&  
  \begin{tikzpicture}[scale=0.3]
      \draw (4,0) arc (0:360: 4cm and 2cm) coordinate[pos = 0.8] (O1) coordinate[pos = 0.3] (O2);
      \draw (-2,0.3)  .. controls +(1.2,-0.6) and +(-1.2, -0.6) .. (2,0.3) coordinate[pos = 0.2] (L) coordinate[pos = 0.8] (R) coordinate[pos = 0.7] (M1) ;
      \draw (L) .. controls +(0.8,0.4) and +(-0.8, 0.4) .. (R) coordinate[pos = 0.2] (M2);
      \draw (O1)  .. controls +(0.9,0.2) and +(0.9,0.2 ) .. (M1);
      \draw[densely dotted ] (O1)  .. controls +(-0.9,-0.2) and +(-0.9,-0.1 ) .. (M1);
      \draw (O2)  .. controls +(0.9,0.2) and +(0.9,0.2 ) .. (M2);
      \draw[densely dotted] (O2)  .. controls +(-0.9,-0.2) and +(-0.9,-0.1 ) .. (M2);
    \end{tikzpicture}
& Empty set
\\ \hline
  \color{black!50!green}{$4$}&  
\begin{tikzpicture}[scale=0.3]
      \draw (4,0) arc (0:360: 4cm and 2cm) coordinate[pos = 0.8] (O1) coordinate[pos = 0.3] (O2);
      \draw (-2,0.3)  .. controls +(1.2,-0.6) and +(-1.2, -0.6) .. (2,0.3) coordinate[pos = 0.2] (L) coordinate[pos = 0.8] (R) coordinate[pos = 0.7] (M1) ;
      \draw (L) .. controls +(0.8,0.4) and +(-0.8, 0.4) .. (R) coordinate[pos = 0.2] (M2);
      \draw (O1)  .. controls +(0.9,0.2) and +(0.9,0.2 ) .. (M1);
      \draw[densely dotted ] (O1)  .. controls +(-0.9,-0.2) and +(-0.9,-0.1 ) .. (M1);
      \draw (O2)  .. controls +(0.9,0.2) and +(0.9,0.2 ) .. (M2);
      \draw[densely dotted] (O2)  .. controls +(-0.9,-0.2) and +(-0.9,-0.1 ) .. (M2);
      \fill[green!50!black, opacity = 0.5] (M1) ..  controls +(0.9,0.1 ) and +(+0.9,+0.2) .. (O1) arc (288: 108: 4cm and 2cm) .. controls +(0.9,0.2) and +(0.9,0.2 ) .. (M2);
      \fill[white]  (L) .. controls +(0.8,0.4) and +(-0.8, 0.4) .. (R)  .. controls +(-0.8,-0.27) and +(0.8, -0.27) ..(L);
 \draw[opacity = 0.5] (4,0) arc (0:360: 4cm and 2cm) coordinate[pos = 0.8] (O1) coordinate[pos = 0.3] (O2);
      \draw (-2,0.3)  .. controls +(1.2,-0.6) and +(-1.2, -0.6) .. (2,0.3) coordinate[pos = 0.2] (L) coordinate[pos = 0.8] (R) coordinate[pos = 0.7] (M1) ; 
      \draw (L) .. controls +(0.8,0.4) and +(-0.8, 0.4) .. (R);
   \end{tikzpicture}
&Sphere (on the left)
\\ \hline
\end{tabular}
  \caption{The monochrome surfaces of Example~\ref{exa:torus}.}
  \label{tab:exa1}
\end{table}

\begin{table}[ht]
  \centering
\begin{tabular}{| >{\centering\arraybackslash}m{1.3cm} || >{\centering\arraybackslash}m{0.9cm} | >{\centering\arraybackslash}m{1.7cm} | >{\centering\arraybackslash}m{3cm} | >{\centering\arraybackslash}m{3cm} | }
\hline
$(i,j)\in\Col$ & {\color{blue}{{${1}$}}}&  {\color{red}{{${2}$}}}&  {\color{orange!50!black}{{${3}$}}}&  {\color{green!50!black}{{${4}$}}}  \\ \hline \hline
{\color{blue}{{${1}$}}}& & Torus & Sphere\newline (on the right) & Torus \\ \hline
{\color{red}{{${2}$}}}& 2 & &Sphere \newline(on the left)& Empty set  \\ \hline
{\color{orange!50!black}{{${3}$}}}&0  &0 && Sphere \newline(on the left) \\ \hline
{\color{green!50!black}{{${4}$}}}& 2 & 0 & 0 & \\ \hline
\end{tabular} 
 \caption{The bichrome surfaces (top right) and the $\theta_{ij}^+$ (bottom left) of Example~\ref{exa:torus}.}
  \label{tab:exa2}
\end{table}
\end{exa}

\begin{rmk}
  \label{rmk:colorsubsurface}
  Monochrome surfaces of $(F,c)$ are sub-surfaces of $F$ while, in general, bichrome surfaces are not in the sense of Definition~\ref{dfn:subsurfaces}.
\end{rmk}

\begin{dfn}
Let $(F,c)$ be a colored foam and $i<j$ be two pigments. A circle in $F_{i}(c) \cap F_{j}(c) \cap F_{ij}(c)$ is 
\emph{positive} (resp. \emph{negative}) \emph{with respect to $(i,j)$} if it consists of positive (resp. negative)
bindings. 
We denote by $\theta^+_{ij}(c)_F$ (resp. $\theta^-_{ij}(c)_F$) or simply $\theta^+_{ij}(c)$ (resp. $\theta^-_{ij}(c)$) the number of positive (resp. negative) circles with respect to $(i,j)$. We set as well $\theta_{ij}(c)= \theta^+_{ij}(c) +\theta^{-}_{ij}(c)$. See Figure~\ref{fig:signsofcircles} for a pictorial definition.
\end{dfn}




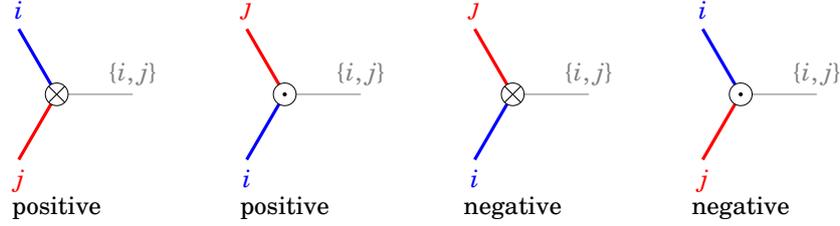
\begin{figure}[h]
  \centering
  \begin{tikzpicture}
    \begin{scope}[xshift = 0cm]
 \draw[gray] (0,0) -- +(0:1)  node[above] {$\{i,j\}$};
 \draw[blue, very thick ] (0,0) -- +(120:1) node[above] {${i}$};;
 \draw[red, very thick] (0,0) -- +(240:1) node[below] {${j}$};;
 \filldraw[fill= white, draw=black, very thin] (0,0) circle (0.15cm);
 \filldraw[fill = black] (0,0) circle (0.02cm);
 \node at (0, -1.5) {negative};
\end{scope}

 \begin{scope}[xshift = -3cm]
  \draw[gray] (0,0) -- +(0:1) node[above] {$\{i,j\}$};
  \draw[red, very thick ] (0,0) -- +(120:1) node[above] {${j}$};;
  \draw[blue, very thick] (0,0) -- +(240:1) node[below] {${i}$};;
  \filldraw[fill= white, draw=black, very thin] (0,0) circle (0.15cm);
  \draw (45:0.15) --  (-135:0.15);
  \draw (-45:0.15) -- (135:0.15);
  \node at (0, -1.5) {negative};
 \end{scope}

\begin{scope}[xshift = -9cm]
 \draw[gray] (0,0) -- +(0:1) node[above] {$\{i,j\}$};
 \draw[blue, very thick ] (0,0) -- +(120:1) node[above] {${i}$};;
 \draw[red, very thick] (0,0) -- +(240:1) node[below] {${j}$};;
 \filldraw[fill= white, draw=black, very thin] (0,0) circle (0.15cm);
  \draw (45:0.15) --  (-135:0.15);
  \draw (-45:0.15) -- (135:0.15);
 \node at (0, -1.5) {positive};
\end{scope}

\begin{scope}[xshift = -6cm]
 \draw[gray] (0,0) -- +(0:1) node[above] {$\{i,j\}$};
 \draw[red, very thick ] (0,0) -- +(120:1) node[above] {${j}$};;
 \draw[blue, very thick] (0,0) -- +(240:1) node[below] {${i}$};;
 \filldraw[fill= white, draw=black, very thin] (0,0) circle (0.15cm);
 \filldraw[fill = black] (0,0) circle (0.02cm);
 \node at (0, -1.5) {positive};
\end{scope}
  \end{tikzpicture}
  \caption{A pictorial definition of the signs of the circle, we assume $i<j$. Recall that a dotted circle in the middle indicates that the orientation of the binding points to the reader and a crossed circle indicates the other orientation.}
  \label{fig:signsofcircles}
\end{figure}



\begin{dfn}\label{dfn:degreefoam}
    The \emph{degree} $\degext_N$ of a foam $F$ is the sum of the following contributions:
    \begin{itemize}
    \item For each facet $f$ with label $a$, set $\deg(f)=a(N-a)\chi(f)$, where $\chi$ stands for the Euler characteristic;
    \item For each interval binding $e$ (i. e. not circle-like binding) surrounded by three facets with labels $a$, $b$ and $a+b$, set $\deg(e)= ab + (a+b)(N-a -b)$; 
    \item For each singular point $p$ surrounded with facets with labels $a$, $b$, $c$, $a+b$, $b+c$, $a+b+c$, set $\deg(p) = ab + bc+ cd + da + ac+ bd $ with $d =N -a -b -c $;
    \item Thus, \[
\degext_N(F) = -\sum_f\deg(f) + \sum_e\deg(e) -\sum_p \deg(p) + \sum_{f} \deg(P_f),
\]
where the variables of the polynomials $P_\bullet$ have degree 2.
    \end{itemize}
\end{dfn}

\begin{rmk}
  \begin{enumerate}
  \item The degree is additive with respect to the composition of
    foam. This is the same degree as in \cite{queffelec2014mathfrak},
    but since we are not in a $2$-categorical setting, the
    contributions of $\Gamma_0$ and $\Gamma_1$ to the degree are equal
    to $0$.
   \item The degree can be thought of as an analogue of the Euler
    characteristic. The degree of the foam of
    Figure~\ref{fig:exslnfoam} is equal to $-16$ when $N=6$.
\end{enumerate}
\end{rmk}

  \begin{dfn} \label{dfn:exteval} If $(F,c)$ is a colored foam, define:
    \begin{align*}
      s(F,c) &= \sum_{i=1}^N {i\chi(F_i(c))/2}  + \sum_{1\leq i < j \leq N} \theta^+_{ij}(F,c), \\
      P(F,c) &= \prod_{f \textrm{ facet of $F$}} P_f(c(f)),\\
      Q(F,c) &= \prod_{1\leq i < j \leq N} (x_i-x_j)^{\frac{\chi(F_{ij}(c))}{2}} \\
      \kup{F,c} &= (-1)^{s(F,c)} \frac{P(F,c)}{Q(F,c)}.
    \end{align*}
In the definition of $P(F,c)$, $P_f(c(f))$  means the polynomial $P$ evaluated on the variables $\{x_i\}_{i\in c(f)}$. Since the polynomial $P_f$ is symmetric, the order of the variables does not matter. A facet $f$ is called \emph{trivially decorated}, if $P_f=1$. 
   Define the \emph{evaluation of the foam $F$ by}:
    \[ \kup{F} := \sum_{c \textrm{ coloring of $F$}} \kup{F,c}. \]
  \end{dfn}


  \begin{rmk}\label{rmk:remove0faces}
    Let $F$ be a foam and denote by $\lambda$ the product of all decorations of facets of label $0$. Consider $F'$ the foam obtained  from $F$ by removing the facets with label $0$. There is a one-one correspondence between the colorings of $F$ and the colorings of $F'$. For every coloring $c$ of $F$ and its corresponding coloring $c'$ of $F'$, we have $\kup{F,c} = \lambda \kup{F',c'}$, and consequently $\kup{F} = \lambda \kup{F'}$.  
\end{rmk}


\begin{prop}[\cite{RW1}]\label{prop:sympol}
  Let $F$ be a foam, then $\kup{F}$ is an homogenous element of $\QQ[x_1, \dots, x_N]^{\mathfrak{S}_N}$ of degree $\degext_N(F)$.
\end{prop}


\begin{prop}[\cite{RW1}]
  \label{prop:relations-cat-ext}
  The following local identities and their mirror images hold:
\begin{align}
\kup{ \scriptstyle{\NB{\tikz[scale = 0.7]{\tdplotsetmaincoords{60}{110}
\begin{scope}[tdplot_main_coords]
  \coordinate (bT) at (-1, 0, 3);
  \coordinate (dT) at (1, 0, 3);
  \coordinate (AT) at (-2, -2, 3);
  \coordinate (BT) at (-2, 2, 3);
  \coordinate (CT) at (2, 2, 3);
  \coordinate (DT) at (2, -2, 3);
  \coordinate (MB) at (0,0, -1.5);
  \coordinate (bB) at (-1, 0, -3);
  \coordinate (dB) at (1, 0, -3);
  \coordinate (AB) at (-2, -2, -3);
  \coordinate (BB) at (-2, 2, -3);
  \coordinate (CB) at (2, 2, -3);
  \coordinate (DB) at (2, -2, -3);
  \draw[->, thick] (bT) -- (AT);
  \draw[->, thick] (bT) -- (BT);
  \draw[->, thick] (dT) -- (CT);
  \draw[->, thick] (DT) -- (dT);
  \draw[->, thick] (dT) -- (bT);
  \draw[->, thick] (bB) -- (AB);
  \draw[->, thick] (bB) -- (BB);
  \draw[->, thick] (dB) -- (CB);
  \draw[->, thick] (DB) -- (dB);
  \draw[->, thick] (dB) -- (bB);
  \draw[->, thick] (bB) -- (bT);
  \draw[->, thick] (dB) -- (dT);
  \filldraw[draw = black, rounded corners=1pt, thick, fill opacity = 0.3, fill = red]     (BT) -- (BB) -- (bB)  -- (bT)  -- (BT) node[midway, below, sloped, opacity=1] {$\scriptstyle{b}$};
  \filldraw[draw = black, rounded corners=1pt, thick, fill opacity = 0.3, fill = red]     (CT) -- (CB)  -- (dB) node[midway, above, sloped, opacity=1] { $\scriptstyle{c}$} -- (dT)-- cycle;
  \filldraw[draw = black, rounded corners=1pt, thick, fill opacity = 0.3, fill = red]     (AT) -- (AB) -- (bB) -- (bT) -- (AT) node[midway, below, sloped, opacity=1] {$\scriptstyle{a}$};
  \filldraw[draw = black, rounded corners=1pt, thick, fill opacity = 0.3, fill = red]     (DT) -- (DB) -- (dB) node[midway, above, sloped, opacity=1] { $\scriptstyle{a+b+c}$} -- (dT) -- cycle;
  \filldraw[draw = black, rounded corners=1pt, thick, fill opacity = 0.3, fill = blue]    (bT)  -- (dT)  -- (dB) node[pos= 0.5, above, sloped, opacity=1] {$\scriptstyle{a+b}$} -- (bB)-- cycle;
\end{scope}}}}}
= \kup{ \scriptstyle{\NB{\tikz[scale = 0.7]{\tdplotsetmaincoords{60}{115}
\begin{scope}[tdplot_main_coords]
  \coordinate (bT) at (-1, 0, 3);
  \coordinate (dT) at (1, 0, 3);
  \coordinate (AT) at (-2, -2, 3);
  \coordinate (BT) at (-2, 2, 3);
  \coordinate (CT) at (2, 2, 3);
  \coordinate (DT) at (2, -2, 3);
  \coordinate (MT) at (0,0, 1.5);
  \coordinate (aM) at (0, -1, 0);
  \coordinate (cM) at (0, 1, 0);
  \coordinate (AM) at (-2, -2, 0);
  \coordinate (BM) at (-2, 2, 0);
  \coordinate (CM) at (2, 2, 0);
  \coordinate (DM) at (2, -2, 0);
  \coordinate (MB) at (0,0, -1.5);
  \coordinate (bB) at (-1, 0, -3);
  \coordinate (dB) at (1, 0, -3);
  \coordinate (AB) at (-2, -2, -3);
  \coordinate (BB) at (-2, 2, -3);
  \coordinate (CB) at (2, 2, -3);
  \coordinate (DB) at (2, -2, -3);
  \draw[->, thick] (bT) -- (AT);
  \draw[->, thick] (bT) -- (BT);
  \draw[->, thick] (dT) -- (CT);
  \draw[->, thick] (DT) -- (dT);
  \draw[->, thick] (dT) -- (bT);
  \draw[->, thick] (bB) -- (AB);
  \draw[->, thick] (bB) -- (BB);
  \draw[->, thick] (dB) -- (CB);
  \draw[->, thick] (DB) -- (dB);
  \draw[->, thick] (dB) -- (bB);
  \draw[->, thick] (MT) -- (bT);
  \draw[->, thick] (MT) -- (dT);
  \draw[->, thick] (bB) -- (MB);
  \draw[->, thick] (dB) -- (MB);
  \draw[->, thick] (aM) -- (MT);
  \draw[->, thick] (cM) -- (MT);
  \filldraw[draw = black, rounded corners=1pt, thick, fill opacity = 0.3, fill = red]     (BT) -- (BB) -- (bB) -- (MB) -- (cM) -- (MT) -- (bT) -- (BT) node[midway, below, sloped, opacity=1] {$\scriptstyle{b}$};
  \filldraw[draw = black, rounded corners=1pt, thick, fill opacity = 0.3, fill = red]   (AT) -- (AB) -- (bB) -- (MB) -- (aM) -- (MT) -- (bT)-- (AT)  node[midway, below, sloped, opacity=1] { $\scriptstyle{a}$};;
  \filldraw[draw = black, rounded corners=1pt, thick, fill opacity = 0.3, fill = red]    (CT) -- (CB) -- (dB) node[midway, above, sloped, opacity=1] { $\scriptstyle{c}$}-- (MB) -- (cM) -- (MT) -- (dT)-- (CT);
  \filldraw[draw = black, rounded corners=1pt, thick, fill opacity = 0.3, fill = red]  (DT) -- (DB) -- (dB) node[midway, above, sloped, opacity=1] { $\scriptstyle{a+b+c}$} -- (MB) -- (aM) -- (MT) -- (dT)-- cycle;
  \filldraw[draw = black, rounded corners=1pt, thick, fill opacity = 0.3, fill = blue]  (bT) -- (dT) node[pos= 0.7, below=-0.1cm, sloped, opacity=1] {$\scriptstyle{a+b}$} -- (MT) -- cycle;
  \filldraw[draw = black, rounded corners=1pt, thick, fill opacity = 0.3, fill = blue]     (bB)-- (dB) node[pos = 0.3, above= -0.1cm, sloped, opacity=1] {$\scriptstyle{a+b}$}  -- (MB)-- cycle;
  \filldraw[draw = black, rounded corners=1pt, thick, fill opacity = 0.3, fill = green]     (MT)-- (aM) -- (MB)-- (cM)-- cycle;
  \draw[dotted] (AM) -- (aM) -- (cM) node[midway, above, sloped, opacity=1] {$\scriptstyle{b+c}$}  -- (CM);
  \draw[dotted] (BM) -- (cM);
  \draw[dotted] (DM) -- (aM);
\end{scope}}}} }, \label{eq:MPcat}
\end{align}
\begin{align}
%
\kup{\scriptstyle{\NB{\tikz[scale=0.5]{\tdplotsetmaincoords{90}{0}
\begin{scope}[tdplot_main_coords]
  \coordinate (aT) at (1, 0, 3);
  \coordinate (bT) at (-1, 0, 3);
  \coordinate (aB) at (1, 0, -1);
  \coordinate (bB) at (-1, 0, -1);
  \filldraw[draw = black, rounded corners=1pt, thick, fill opacity = 0.3, dotted, fill = red]  (aT) arc (0:-180:1cm and 0.5cm) -- (bB) arc (180:0:1cm and 0.5cm) -- (aT);
  \filldraw[draw = black, rounded corners=1pt, thick, fill opacity = 0.3, fill = red]    (aT) arc (0:180:1cm and 0.5cm) -- (bB) arc (-180:0:1cm and 0.5cm) -- (aT);
  \draw[thick, black, ->-]  (aT) arc (0:-180:1cm and 0.5cm) node[midway, below] {$\scriptstyle{a}$};
  \draw[thick, black, ->-]  (aB) arc (0:-180:1cm and 0.5cm);
\end{scope}
}}}} = \sum_{\alpha \in T(a,N-a)}(-1)^{|\widehat{\alpha}| + \frac{N(N+1)}2}
\kup{\scriptstyle{\NB{\tikz[scale=0.5]{\tdplotsetmaincoords{90}{0}
\begin{scope}[tdplot_main_coords]
  \coordinate (aT) at (1, 0, 3);
  \coordinate (bT) at (-1, 0, 3);
  \coordinate (aB) at (1, 0, -1);
  \coordinate (bB) at (-1, 0, -1);
  \coordinate (a2) at (1, 0, 0.4);
  \coordinate (b2) at (-1, 0, 0.4);
  \coordinate (a1) at (1, 0, -0.2);
  \coordinate (am) at (1, 0, 0.3);
  \coordinate (b1) at (-1, 0, -0.2);
  \coordinate (M2) at (0, 0, 0.4);
  \coordinate (M1) at (0, 0, -0.2);
  \fill[rounded corners=1pt, thick, fill opacity = 0.3, fill = red]   (aT) arc (0:-180:1cm and 0.5cm) arc (-180:0:1cm);
  \draw[thick]  (aT) arc (0:-180:1cm and 0.5cm) node[midway, below=-0.1cm] {$\scriptstyle{a}$};
  \fill[draw = black, rounded corners=1pt, thick, fill opacity = 0.3, fill = red] (aT) arc (0:180:1cm and 0.5cm) arc (-180:0:1cm);
  \filldraw[draw= black, dotted,  rounded corners=1pt, thick, fill opacity = 0.3, fill = red] (aB) arc (0:180: 1cm and 0.5cm) -- (b1) arc (180:0: 1cm and 0.5cm) -- (aB);
  \filldraw[draw= black, rounded corners=1pt, thick, fill opacity = 0.3, fill = yellow] (a1) arc (0:180: 1cm and 0.5cm) -- (b2) arc (180:0: 1cm) -- (a1);
  \fill[blue, opacity = 0.3]  (M1) ellipse (1cm and 0.5cm);
  \filldraw[pattern=north west lines, thick] (M1) ellipse (1cm and 0.5cm);
  \filldraw[draw= black, rounded corners=1pt, thick, fill opacity = 0.3, fill = red] (aB) arc (0:-180: 1cm and 0.5cm) node[midway, above, opacity=1] {$\scriptstyle{a}$} -- (b1) arc (-180:0: 1cm and 0.5cm) -- (aB);
  \filldraw[draw= black, rounded corners=1pt, thick, fill opacity = 0.3, fill = yellow] (a1) arc (0:-180: 1cm and 0.5cm) -- (b2) arc (180:0: 1cm)node[opacity = 1, black,midway, below] {$\scriptstyle{N-a}$} -- (a1); 
  \draw[thick, ->-] (a1) arc (0: -180: 1cm and 0.5cm);
  \draw[thick, ->-] (aT) arc (0: -180: 1cm and 0.5cm);
  \draw[thick, ->-] (aB) arc (0: -180: 1cm and 0.5cm);
\end{scope}
\draw[thick, red, <-] (am) -- +(0.5,0) node[right] {$\pi_{\widehat{\alpha}}$};
\draw[thick, red, <-] (aT) -- +(0.5,0) node[right] {$\pi_{{\alpha}}$};
\node at (M1) {$\scriptstyle{N}$};}}}},\label{eq:neckcuttingcat}
\end{align}
\begin{align}
%
\kup{\scriptstyle{\NB{\tikz[scale=0.8]{\begin{scope}
\tdplotsetmaincoords{60}{60}
\begin{scope}[tdplot_main_coords]
  \coordinate (OT) at (0, 0, 2);
  \coordinate (AT) at (0.5, 1, 2);
  \coordinate (BT) at (0.5, -1, 2);
  \coordinate (CT) at (-1.5, 0, 2);
  \coordinate (OB) at (0, 0, 0);
  \coordinate (AB) at (0.5, 1, 0);
  \coordinate (BB) at (0.5, -1, 0);
  \coordinate (CB) at (-1.5, 0, 0);
  \coordinate (a) at (0.3, 0.6, 1);
  \coordinate (b) at (0.3, -0.6, 1);
  \coordinate (c) at (1, 0, 1);

  \draw[thick, -<-] (OB) -- (OT);
  \draw[thick, ->-] (AB) -- (AT);
  \draw[thick, -<-] (BB) -- (BT);
  \draw[thick, -<-] (CB) -- (CT);

  \draw[thick, -<-] (OB) -- (AB);
  \draw[thick, ->-] (OB) -- (BB);
  \draw[thick, ->-] (OB) -- (CB);
  
  \draw[thick, -<-] (OT) -- (AT);
  \draw[thick, ->-] (OT) -- (BT);
  \draw[thick, ->-] (OT) -- (CT);
  \filldraw[draw = black, rounded corners=1pt, thick, fill opacity = 0.3, fill = red]    (OT) -- (AT) node[midway, opacity =1, below, sloped] {$\scriptstyle{a+b}$} -- (AB) -- (OB) -- (OT);
  \filldraw[draw = black, rounded corners=1pt, thick, fill opacity = 0.3, fill = green]    (OT) -- (CT) node[midway, opacity =1, below] {$\scriptstyle{b}$} -- (CB) -- (OB) -- (OT);
  \filldraw[draw = black, rounded corners=1pt, thick, fill opacity = 0.3, fill = blue]    (OT) -- (BT) -- (BB) -- (OB)  node[midway, opacity =1, above, sloped] {$\scriptstyle{a}$} -- (OT);

\end{scope}
\draw[thick, red, <-] (a) -- +(0.8, 0) node [right, red] {$\pi_{\lambda}$};  
\end{scope}}}}} = \sum_{\alpha, \beta } c_{\alpha \beta} ^\lambda
\kup{\scriptstyle{\NB{\tikz[scale=0.8]{\begin{scope}
\tdplotsetmaincoords{60}{60}
\begin{scope}[tdplot_main_coords]
  \coordinate (OT) at (0, 0, 2);
  \coordinate (AT) at (0.5, 1, 2);
  \coordinate (BT) at (0.5, -1, 2);
  \coordinate (CT) at (-1.5, 0, 2);
  \coordinate (OB) at (0, 0, 0);
  \coordinate (AB) at (0.5, 1, 0);
  \coordinate (BB) at (0.5, -1, 0);
  \coordinate (CB) at (-1.5, 0, 0);
  \coordinate (a) at (0.3, 0.6, 1);
  \coordinate (b) at (0.4, -0.8, 0.5);
  \coordinate (c) at (-1.2, 0, 1.5);

  \draw[thick, -<-] (OB) -- (OT);
  \draw[thick, ->-] (AB) -- (AT);
  \draw[thick, -<-] (BB) -- (BT);
  \draw[thick, -<-] (CB) -- (CT);

  \draw[thick, -<-] (OB) -- (AB);
  \draw[thick, ->-] (OB) -- (BB);
  \draw[thick, ->-] (OB) -- (CB);
  
  \draw[thick, -<-] (OT) -- (AT);
  \draw[thick, ->-] (OT) -- (BT);
  \draw[thick, ->-] (OT) -- (CT);

  \filldraw[draw = black, rounded corners=1pt, thick, fill opacity = 0.3, fill = red]    (OT) -- (AT) node[midway, opacity =1, below, sloped] {$\scriptstyle{a+b}$} -- (AB) -- (OB) -- (OT);
  \filldraw[draw = black, rounded corners=1pt, thick, fill opacity = 0.3, fill = green]    (OT) -- (CT) node[midway, opacity =1, below] {$\scriptstyle{b}$} -- (CB) -- (OB) -- (OT);
\draw[thick, red, <-] (b) -- +(-0.8, 0, -0.4) node [left, red] {$\pi_{\alpha}$};  
\draw[thick, red, <-] (c) -- +(-0.8, 0, -0.4) node [left, red] {$\pi_{\beta}$};  
  \filldraw[draw = black, rounded corners=1pt, thick, fill opacity = 0.3, fill = blue]    (OT) -- (BT) -- (BB) -- (OB)  node[midway, opacity =1, above, sloped] {$\scriptstyle{a}$} -- (OT);
\end{scope}

\end{scope}}}}}, \label{eq:dotmig}
\end{align}
\begin{align}
%
\kup{\scriptstyle{\NB{\tikz[scale=0.5]{\tdplotsetmaincoords{60}{140}
\begin{scope}[tdplot_main_coords]
  \coordinate (aT) at (-1, 0, 2);
  \coordinate (bT) at (+1, 0, 2);
  \coordinate (AT) at (-2.5, 0, 2);
  \coordinate (BT) at (+2.5, 0, 2);
  \coordinate (cT) at (0, +1, 2);
  \coordinate (dT) at (0, -1, 2);
  \coordinate (aB) at (-1, 0, -2);
  \coordinate (bB) at (+1, 0, -2);
  \coordinate (AB) at (-2.5, 0, -2);
  \coordinate (BB) at (+2.5, 0, -2);
  \coordinate (cB) at (0, +1, -2);
  \coordinate (dB) at (0, -1, -2);
  \coordinate (aM) at (-1, 0, 0);
  \coordinate (bM) at (+1, 0, 0);
\draw[thick, -<-] (AT) -- (aT);
\draw[thick, -<-] (bT) -- (BT);
\draw[thick, -<-] (aT) .. controls (dT) .. (bT);
\draw[thick, -<-] (aT) .. controls (cT) .. (bT);
\draw[thick, -<-] (AB) -- (aB);
\draw[thick, -<-] (bB) -- (BB);
\draw[thick, -<-] (aB) .. controls (dB) .. (bB);
\draw[thick, -<-] (aB) .. controls (cB) .. (bB);
\draw[thick, -<-] (aB) -- (aT);
\draw[thick, -<-] (bT) -- (bB);
  \filldraw[draw = black, rounded corners=1pt, thick, fill opacity = 0.3, fill = red]    (AT) -- (aT) node[pos=0.65, opacity =1, below, sloped] {$\scriptstyle{a+b}$} -- (aB) -- (AB) -- (AT);
  \filldraw[draw = black, rounded corners=1pt, thick, fill opacity = 0.3, fill = red]    (bT) -- (BT) node[pos=0.65, opacity =1, below, sloped] {$\scriptstyle{a+b}$} -- (BB) -- (bB) -- (bT);
  \filldraw[draw = black, rounded corners=1pt, thick, fill opacity = 0.3, fill = blue]   (bT) .. controls (dT) .. (aT) node[near end, below, opacity=1] {$\scriptstyle{a}$} -- (aB) .. controls (dB) .. (bB)-- (bT);
  \filldraw[draw = black, rounded corners=1pt, thick, fill opacity = 0.3, fill = green]   (bT) .. controls (cT) .. (aT) -- (aB)  .. controls (cB) .. (bB) node[near end, above= -0.1cm, opacity=1] {$\scriptstyle{b}$} -- (bT);

\end{scope}}}}} = \sum_{\alpha \in T(a,b)}(-1)^{|\widehat{\alpha}|}
\kup{\scriptstyle{\NB{\tikz[scale=0.5]{\tdplotsetmaincoords{60}{140}
\begin{scope}[tdplot_main_coords]
  \coordinate (aT) at (-1, 0, 2);
  \coordinate (bT) at (+1, 0, 2);
  \coordinate (AT) at (-2.5, 0, 2);
  \coordinate (BT) at (+2.5, 0, 2);
  \coordinate (cT) at (0, +1, 2);
  \coordinate (dT) at (0, -1, 2);
  \coordinate (aB) at (-1, 0, -2);
  \coordinate (bB) at (+1, 0, -2);
  \coordinate (AB) at (-2.5, 0, -2);
  \coordinate (BB) at (+2.5, 0, -2);
  \coordinate (cB) at (0, +1, -2);
  \coordinate (dB) at (0, -1, -2);
  \coordinate (aM) at (-1, 0, 0);
  \coordinate (bM) at (+1, 0, 0);
\draw[thick, -<-] (AT) -- (aT);
\draw[thick, -<-] (bT) -- (BT);
\draw[thick, -<-] (aT) .. controls (dT) .. (bT);
\draw[thick, -<-] (aT) .. controls (cT) .. (bT);
\draw[thick, -<-] (AB) -- (aB);
\draw[thick, -<-] (bB) -- (BB);
\draw[thick, -<-] (aB) .. controls (dB) .. (bB);
\draw[thick, -<-] (aB) .. controls (cB) .. (bB);
\draw[thick, -<-] (bT) .. controls (bM) and (aM) .. (aT);
\draw[thick, -<-] (aB) .. controls (aM) and (bM) .. (bB);
  \filldraw[draw = black, rounded corners=1pt, thick, fill opacity = 0.3, fill = red]    (AT) -- (aT)  .. controls (aM) and (bM) .. (bT) -- (BT) -- (BB) node[midway, opacity =1, right, sloped, rotate= 112 ] {$\scriptstyle{a+b}$} -- (bB) .. controls (bM) and (aM) .. (aB) -- (AB) -- (AT);
  \filldraw[draw = black, rounded corners=1pt, thick, fill opacity = 0.3, fill = blue]    (aT) .. controls (aM) and (bM) .. (bT) .. controls (dT) .. (aT) node [opacity=1, near end, below=-0.0cm] {$\scriptstyle{a}$} coordinate[pos=0.3] (alpha); 
  \filldraw[draw = black, rounded corners=1pt, thick, fill opacity = 0.3, fill = green]   (aT) .. controls (aM) and (bM) .. (bT) .. controls (cT) .. (aT) node [opacity=1, near start, below=-0.1cm] {$\scriptstyle{b}$} ;
  \filldraw[draw = black, rounded corners=1pt, thick, fill opacity = 0.3, fill = blue]    (aB) .. controls (aM) and (bM) .. (bB) .. controls (dB) .. (aB) node [opacity=1, near end, above=-0.0cm] {$\scriptstyle{a}$};
  \filldraw[draw = black, rounded corners=1pt, thick, fill opacity = 0.3, fill = green]   (aB) .. controls (aM) and (bM) .. (bB) .. controls (cB) .. (aB) node [opacity=1, near start, above=-0.1cm]  {$\scriptstyle{b}$} coordinate[pos=0.7] (alphah);
\end{scope}
  \draw[thick, red, <-] (alphah) -- +(0.5,0) node[right, red] {$\pi_{\widehat{\alpha}}$};
  \draw[thick, red, <-] (alpha) -- +(-0.5,0) node[left, red] {$\pi_\alpha$};}}}}, \label{eq:digoncat}
\end{align}
\begin{align}
 %
 \kup{\scriptstyle{\NB{\tikz[scale=0.7]{\tdplotsetmaincoords{60}{140}
\begin{scope}[tdplot_main_coords]
  \coordinate (aT) at (-1, 0, 2);
  \coordinate (bT) at (+1, 0, 2);
  \coordinate (AT) at (-2.5, 0, 2);
  \coordinate (BT) at (+2.5, 0, 2);
  \coordinate (cT) at (0, +1, 2);
  \coordinate (dT) at (0, -1, 2);
  \coordinate (aB) at (-1, 0, -2);
  \coordinate (bB) at (+1, 0, -2);
  \coordinate (AB) at (-2.5, 0, -2);
  \coordinate (BB) at (+2.5, 0, -2);
  \coordinate (cB) at (0, +1, -2);
  \coordinate (dB) at (0, -1, -2);
  \coordinate (aM) at (-1, 0, 0);
  \coordinate (bM) at (+1, 0, 0);
\draw[thick, -<-] (AT) -- (aT);
\draw[thick, -<-] (bT) -- (BT);
\draw[thick, ->-] (aT) .. controls (dT) .. (bT);
\draw[thick, -<-] (aT) .. controls (cT) .. (bT);
\draw[thick, -<-] (AB) -- (aB);
\draw[thick, -<-] (bB) -- (BB);
\draw[thick, ->-] (aB) .. controls (dB) .. (bB);
\draw[thick, -<-] (aB) .. controls (cB) .. (bB);
\draw[thick, -<-] (aT) -- (aB);
\draw[thick, -<-] (bB) -- (bT);

  \filldraw[draw = black, rounded corners=1pt, thick, fill opacity = 0.3, fill = red]    (AT) -- (aT) node[midway, opacity =1, below, sloped] {$\scriptstyle{a}$} -- (aB) -- (AB) -- (AT);
  \filldraw[draw = black, rounded corners=1pt, thick, fill opacity = 0.3, fill = red]    (bT) -- (BT) node[midway, opacity =1, below, sloped] {$\scriptstyle{a}$} -- (BB) -- (bB) -- (bT);
  \filldraw[draw = black, rounded corners=1pt, thick, fill opacity = 0.3, fill = blue]   (bT) .. controls (dT) .. (aT) node[near end, below, opacity=1] {$\scriptstyle{b}$} -- (aB) .. controls (dB) .. (bB)-- (bT);
  \filldraw[draw = black, rounded corners=1pt, thick, fill opacity = 0.3, fill = green]   (bT) .. controls (cT) .. (aT) -- (aB)  .. controls (cB) .. (bB) node[near end, above, opacity=1] {$\scriptstyle{a+b}$} -- (bT);
\end{scope}}}}} = \!\!\!\!\!\sum_{\alpha \in T(b,N-a -b)} \!\!\!\!\!(-1)^{|\alpha|}
 \kup{\scriptstyle{\NB{\tikz[scale=0.7]{\tdplotsetmaincoords{60}{140}
\begin{scope}[tdplot_main_coords]
  \coordinate (aT) at (-1, 0, 2);
  \coordinate (bT) at (+1, 0, 2);
  \coordinate (AT) at (-2.5, 0, 2);
  \coordinate (BT) at (+2.5, 0, 2);
  \coordinate (cT) at (0, +1, 2);
  \coordinate (dT) at (0, -1, 2);
  \coordinate (aB) at (-1, 0, -2);
  \coordinate (bB) at (+1, 0, -2);
  \coordinate (AB) at (-2.5, 0, -2);
  \coordinate (BB) at (+2.5, 0, -2);
  \coordinate (cB) at (0, +1, -2);
  \coordinate (dB) at (0, -1, -2);
  \coordinate (aM) at (-1, 0, 0);
  \coordinate (bM) at (+1, 0, 0);
\draw[thick, -<-] (AT) -- (aT);
\draw[thick, -<-] (bT) -- (BT);
\draw[thick, ->-] (aT) .. controls (dT) .. (bT);
\draw[thick, -<-] (aT) .. controls (cT) .. (bT);
\draw[thick, -<-] (AB) -- (aB);
\draw[thick, -<-] (bB) -- (BB);
\draw[thick, ->-] (aB) .. controls (dB) .. (bB);
\draw[thick, -<-] (aB) .. controls (cB) .. (bB);
\draw[thick, -<-] (bB) .. controls (bM) and (aM) .. (aB);
\draw[thick, -<-] (aT) .. controls (aM) and (bM) .. (bT);
  \filldraw[draw = black, rounded corners=1pt, thick, fill opacity = 0.3, fill = red]    (AT) -- (aT)  .. controls (aM) and (bM) .. (bT) -- (BT) -- (BB) node[midway, opacity =1, right, sloped, rotate= 112 ] {$\scriptstyle{a}$} -- (bB) .. controls (bM) and (aM) .. (aB) -- (AB) -- (AT);
  \filldraw[draw = black, rounded corners=1pt, thick, fill opacity = 0.3, fill = blue]    (aT) .. controls (aM) and (bM) .. (bT) .. controls (dT) .. (aT) node [opacity=1, near end, below] {$\scriptstyle{b}$} coordinate[pos=0.3] (alpha); 
  \filldraw[draw = black, rounded corners=1pt, thick, fill opacity = 0.3, fill = green]   (aT) .. controls (aM) and (bM) .. (bT) .. controls (cT) .. (aT) node [opacity=1, near start, below] {$\scriptstyle{a+b}$} ;
  \filldraw[draw = black, rounded corners=1pt, thick, fill opacity = 0.3, fill = blue]    (aB) .. controls (aM) and (bM) .. (bB) .. controls (dB) .. (aB) node [opacity=1, near end, above] {$\scriptstyle{b}$};
  \filldraw[draw = black, rounded corners=1pt, thick, fill opacity = 0.3, fill = green]   (aB) .. controls (aM) and (bM) .. (bB) .. controls (cB) .. (aB) node [opacity=1, near start, above]  {$\scriptstyle{a+b}$} coordinate[pos=0.7] (alphah);
\end{scope}
\filldraw[thick, black, pattern = north west lines ]  ($(alphah) + (-0.2,0.5)$) ellipse (0.1 and 0.3);
\draw[thick, black, <- ]  ($(alphah) + (-0.3,0.5)$)  arc (180:-180: 0.1 and 0.3);
\filldraw[thick, draw= black, fill = yellow, fill opacity =0.3] ($(alphah) + (-0.2, 0.8)$) -- +(1.5,0) node[midway, below, opacity = 1, black] {$\scriptstyle{N-a-b}$}  arc (90: -90:0.3) -- +(-1.5, 0) coordinate[midway] (aaa) {} arc (-90:90:0.1 and 0.3); 
  \filldraw[thick, draw= black, fill = yellow, fill opacity =0.3] ($(alphah) + (-0.2, 0.8)$) -- +(1.5,0) node[midway, below, opacity = 1, black] {$\scriptstyle{N-a-b}$}  arc (90: -90:0.3) -- +(-1.5, 0) coordinate[midway] (aaa) {} arc (270:90:0.1 and 0.3); 

  \draw[thick, red, <-] (aaa) -- +(0,-0.5) node[below, red] {$\pi_{\widehat{\alpha}}$};

  \draw[thick, red, <-] (alpha) -- +(-0.5,0) node[left, red] {$\pi_{\alpha}$};}}}}, \label{eq:digonDURcat}
\end{align}
\begin{align}
%
\kup{\scriptstyle{\NB{\tikz[scale=0.7]{\tdplotsetmaincoords{75}{60}
\begin{scope}[tdplot_main_coords]
  \coordinate (aT) at (-1, -1, 3);
  \coordinate (bT) at (-1, 1, 3);
  \coordinate (cT) at (1, 1, 3);
  \coordinate (dT) at (1, -1, 3);
  \coordinate (AT) at (-1, -2, 3);
  \coordinate (BT) at (-1, 2, 3);
  \coordinate (CT) at (1, 2, 3);
  \coordinate (DT) at (1, -2, 3);
  \coordinate (abMt) at (-1, 0, 1);
  \coordinate (abMb) at (-1, 0, -1);
  \coordinate (cdMt) at (1, 0, 1);
  \coordinate (cdMb) at (1, 0, -1);
  \coordinate (O) at (0,0,0);
  \coordinate (aB) at (-1, -1, -3);
  \coordinate (bB) at (-1, 1, -3);
  \coordinate (cB) at (1, 1, -3);
  \coordinate (dB) at (1, -1, -3);
  \coordinate (AB) at (-1, -2, -3);
  \coordinate (BB) at (-1, 2, -3);
  \coordinate (CB) at (1, 2, -3);
  \coordinate (DB) at (1, -2, -3);
\draw[thick, ->-] (AT) -- (aT);
\draw[thick, ->-] (DT) -- (dT);
\draw[thick, ->-] (cT) -- (CT);
\draw[thick, ->-] (bT) -- (BT);
\draw[thick, ->-] (aT) -- (bT);
\draw[thick, ->-] (dT) -- (cT);
\draw[thick, ->-] (aT) -- (dT);
\draw[thick, ->-] (bT) -- (cT);
\draw[thick, ->-] (AB) -- (aB);
\draw[thick, ->-] (DB) -- (dB);
\draw[thick, ->-] (cB) -- (CB);
\draw[thick, ->-] (bB) -- (BB);
\draw[thick, ->-] (aB) -- (bB);
\draw[thick, ->-] (dB) -- (cB);
\draw[thick, ->-] (aB) -- (dB);
\draw[thick, ->-] (bB) -- (cB);

\draw[thick, ->-] (aB) -- (aT);
\draw[thick, ->-] (cT) -- (cB);
\draw[thick, ->-] (dT) -- (dB);
\draw[thick, ->-] (bB) -- (bT);

  \filldraw[draw = black, rounded corners=1pt, thick, fill opacity = 0.3, fill = red]     (AT) -- (aT) node[pos = 0.5, below=-0.0cm, opacity =1, sloped] {$\scriptstyle{k+s}$} -- (aB) -- (AB) -- (AT);
  \filldraw[draw = black, rounded corners=1pt, thick, fill opacity = 0.3, fill = orange]  (BT) -- (bT) -- (bB) -- (BB) -- (BT)node[pos = 0.5, above=-0.0cm, opacity =1, sloped] {$\scriptstyle{k-r}$} ;
  \filldraw[draw = black, rounded corners=1pt, thick, fill opacity = 0.3, fill = yellow]  (aT) -- (bT) node[pos=0.5, below=-0.1cm, opacity =1] {$\scriptstyle{k}$} -- (bB) -- (aB) -- (aT);
  \filldraw[draw = black, rounded corners=1pt, thick, fill opacity = 0.3, fill = green]  (cT) -- (bT) node[pos=0.6, below=-0.05cm, opacity =1] {$\scriptstyle{r}$} -- (bB) -- (cB) -- (cT);
  \filldraw[draw = black, rounded corners=1pt, thick, fill opacity = 0.3, fill = blue]   (dT) -- (aT) node[pos=0.6, below=-0.05cm, opacity =1] {$\scriptstyle{s}$} -- (aB) -- (dB) -- (dT);
  \filldraw[draw = black, rounded corners=1pt, thick, fill opacity = 0.3, fill = yellow]  (cT) -- (dT) node[pos=0.7, below=-0.1cm, opacity =1] {$\scriptstyle{l}$} -- (dB) -- (cB) -- (cT);
  \filldraw[draw = black, rounded corners=1pt, thick, fill opacity = 0.3, fill = orange]  (DT) -- (dT) -- (dB) -- (DB) -- (DT) node[pos = 0.95, right, opacity =1, rotate =7] {$\scriptstyle{l-s}$} ;
  \filldraw[draw = black, rounded corners=1pt, thick, fill opacity = 0.3, fill = red]     (CT) -- (cT) node[pos = 0.5, below=-0.1cm, opacity =1, sloped] {$\scriptstyle{l+r}$}  -- (cB) -- (CB) -- (CT);
\end{scope}}}}} = \sum_{\alpha \in T(r,s)}(-1)^{|\widehat{\alpha}|}
\kup{\scriptstyle{\NB{\tikz[scale=0.7]{\tdplotsetmaincoords{75}{60}
\begin{scope}[tdplot_main_coords]
  \coordinate (aT) at (-1, -1, 3);
  \coordinate (bT) at (-1, 1, 3);
  \coordinate (cT) at (1, 1, 3);
  \coordinate (dT) at (1, -1, 3);
  \coordinate (AT) at (-1, -2, 3);
  \coordinate (BT) at (-1, 2, 3);
  \coordinate (CT) at (1, 2, 3);
  \coordinate (DT) at (1, -2, 3);
  \coordinate (abMt) at (-1, 0, 1);
  \coordinate (abMb) at (-1, 0, -1);
  \coordinate (cdMt) at (1, 0, 1);
  \coordinate (cdMb) at (1, 0, -1);
  \coordinate (O) at (0,0,0);
  \coordinate (aB) at (-1, -1, -3);
  \coordinate (bB) at (-1, 1, -3);
  \coordinate (cB) at (1, 1, -3);
  \coordinate (dB) at (1, -1, -3);
  \coordinate (AB) at (-1, -2, -3);
  \coordinate (BB) at (-1, 2, -3);
  \coordinate (CB) at (1, 2, -3);
  \coordinate (DB) at (1, -2, -3);
\draw[thick, ->-] (AT) -- (aT);
\draw[thick, ->-] (DT) -- (dT);
\draw[thick, ->-] (cT) -- (CT);
\draw[thick, ->-] (bT) -- (BT);
\draw[thick, ->-] (aT) -- (bT);
\draw[thick, ->-] (dT) -- (cT);
\draw[thick, ->-] (aT) -- (dT);
\draw[thick, ->-] (bT) -- (cT);
\draw[thick, ->-] (AB) -- (aB);
\draw[thick, ->-] (DB) -- (dB);
\draw[thick, ->-] (cB) -- (CB);
\draw[thick, ->-] (bB) -- (BB);
\draw[thick, ->-] (aB) -- (bB);
\draw[thick, ->-] (dB) -- (cB);
\draw[thick, ->-] (aB) -- (dB);
\draw[thick, ->-] (bB) -- (cB);

\draw[thick, ->-] (abMb) -- (cdMb);
\draw[thick, ->-] (cdMt) -- (abMt);

\draw[thick, ->-] (abMb) -- (abMt);
\draw[thick, ->-] (cdMt) -- (cdMb);

\draw[thick, ->-] (aB) -- (abMb);
\draw[thick, ->-] (cT) -- (cdMt);
\draw[thick, ->-] (dT) -- (cdMt);
\draw[thick, ->-] (bB) -- (abMb);

\draw[thick, ->-] (abMt) -- (aT);
\draw[thick, ->-] (cdMb) -- (cB);
\draw[thick, ->-] (cdMb) -- (dB);
\draw[thick, ->-] (abMt) -- (bT);

  \filldraw[draw = black, rounded corners=1pt, thick, fill opacity = 0.3, fill = red]     (AT) -- (aT) -- (abMt) -- (abMb) -- (aB) -- (AB) -- (AT) node[pos = 0.5, right, opacity =1, rotate =7] {$\scriptstyle{k+s}$} ;
  \filldraw[draw = black, rounded corners=1pt, thick, fill opacity = 0.3, fill = orange]  (BT) -- (bT) -- (abMt) -- (abMb) -- (bB) -- (BB) -- (BT) node[midway, above, opacity =1, sloped] {$\scriptstyle{k-r}$} ;
  \filldraw[draw = black, rounded corners=1pt, thick, fill opacity = 0.3, fill = yellow]  (aB) -- (bB) node[midway, above=0.3cm, opacity =1] {$\scriptstyle{k}$} -- (abMb) -- (aB);
  \filldraw[draw = black, rounded corners=1pt, thick, fill opacity = 0.3, fill = yellow]  (aT) -- (bT) node[midway, below, opacity =1] {$\scriptstyle{k}$} -- (abMt) -- (aT); 
  \draw [red, thick, <-] ($(bT)!0.40!(cdMt)$) -- +(0,0,1.5) node[above] {$\pi_{\alpha}$};
  \draw [red, thick, <-] ($(dB)!0.40!(abMb)$) -- +(0,0,-2) node[below] {$\pi_{\widehat{\alpha}}$};
  \filldraw[draw = black, rounded corners=1pt, thick, fill opacity = 0.3, fill = green]  (bT) -- (abMt) -- (cdMt) -- (cT) -- (bT) node[pos=0.6, below=-0.05cm,  opacity =1] {$\scriptstyle{r}$} ;
  \filldraw[draw = black, rounded corners=1pt, thick, fill opacity = 0.3, fill = green]  (bB) -- (abMb) -- (cdMb) -- (cB) -- (bB) node[pos=0.4, above=-0.05cm, opacity =1] {$\scriptstyle{r}$} ;
  \filldraw[draw = black, rounded corners=1pt, thick, fill opacity = 0.3, fill = blue]   (aT) -- (abMt) -- (cdMt) -- (dT) -- (aT) node[pos=0.6, , below=-0.05cm, opacity =1] {$\scriptstyle{s}$} ;
  \filldraw[draw = black, rounded corners=1pt, thick, fill opacity = 0.3, fill = blue]   (aB) -- (abMb) -- (cdMb) -- (dB) -- (aB) node[pos=0.4, above=-0.05cm, opacity =1] {$\scriptstyle{s}$} ;
  \filldraw[draw = black, rounded corners=1pt, thick, fill opacity = 0.3, fill = gray]   (abMb) -- (cdMb) -- (cdMt) -- (abMt)-- (abMb);
  \node[rotate = -15] at (O) {$\scriptstyle{r+s}$};
  \filldraw[draw = black, rounded corners=1pt, thick, fill opacity = 0.3, fill = yellow]  (cB) -- (dB) node[midway, above, opacity =1] {$\scriptstyle{l}$}-- (cdMb) -- (cB); 
  \filldraw[draw = black, rounded corners=1pt, thick, fill opacity = 0.3, fill = yellow]  (cT) -- (dT) node[midway, below=0.3cm, opacity =1] {$\scriptstyle{l}$}-- (cdMt) -- (cT); 
  \filldraw[draw = black, rounded corners=1pt, thick, fill opacity = 0.3, fill = orange]  (DT) -- (dT) -- (cdMt) -- (cdMb) -- (dB) -- (DB) -- (DT) node[pos = 0.5, below, opacity =1, sloped] {$\scriptstyle{l-s}$} ;
  \filldraw[draw = black, rounded corners=1pt, thick, fill opacity = 0.3, fill = red]     (CT) -- (cT) -- (cdMt) -- (cdMb) -- (cB) -- (CB) -- (CT) node[pos = 0.5, left, opacity =1, rotate =7] {$\scriptstyle{l+r}$} ; 
\end{scope}}}}}, \label{eq:cautiscat}
\end{align}
\begin{align}
%
\kup{\scriptstyle{\NB{\tikz[scale = 0.9]{\tdplotsetmaincoords{70}{100}
\begin{scope}[tdplot_main_coords]
  \coordinate (aT) at (-1, -1, 3);
  \coordinate (bT) at (-1, 1, 3);
  \coordinate (cT) at (1, 1, 3);
  \coordinate (dT) at (1, -1, 3);
  \coordinate (AT) at (-2, -2, 3);
  \coordinate (BT) at (-2, 2, 3);
  \coordinate (CT) at (2, 2, 3);
  \coordinate (DT) at (2, -2, 3);
  \coordinate (DM) at (2, -2, 0);
  \coordinate (aB) at (-1, -1, -3);
  \coordinate (bB) at (-1, 1, -3);
  \coordinate (cB) at (1, 1, -3);
  \coordinate (dB) at (1, -1, -3);
  \coordinate (AB) at (-2, -2, -3);
  \coordinate (BB) at (-2, 2, -3);
  \coordinate (CB) at (2, 2, -3);
  \coordinate (DB) at (2, -2, -3);
  \draw[thin, <-] ($(aT)!0.5!(dB)$) -- +(0,-1.5,0) node [left, sloped] {$\scriptstyle{n+k}$}; 
  \draw[thin, <-] ($(cB)!0.5!(bT)$) -- +(0,+1.5,0) node [right, sloped] {$\scriptstyle{m+l-k}$}; 
  \draw[->, thick] (aT) -- (AT);
  \draw[->, thick] (DT) -- (dT);
  \draw[->, thick] (bT) -- (BT);
  \draw[->, thick] (CT) -- (cT);
  \draw[->, thick] (dT) -- (aT);
  \draw[->, thick] (aT) -- (bT);
  \draw[->, thick] (cT) -- (dT);
  \draw[->, thick] (cT) -- (bT);
  \draw[->, thick] (dB) -- (aB);
  \draw[->, thick] (aB) -- (bB);
  \draw[->, thick] (cB) -- (dB);
  \draw[->, thick] (cB) -- (bB);
  \draw[->, thick] (aB) -- (AB);
  \draw[->, thick] (DB) -- (dB);
  \draw[->, thick] (bB) -- (BB);
  \draw[->, thick] (CB) -- (cB);
  \draw[->, thick] (aB) -- (aT);
  \draw[->, thick] (bT) -- (bB);
  \draw[->, thick] (cB) -- (cT);
  \draw[->, thick] (dT) -- (dB);
  \filldraw[draw = black, rounded corners=1pt, thick, fill opacity = 0.3, fill = red]  (aT) -- (aB) -- (AB) -- (AT) -- (aT) node[sloped, midway, below, opacity = 1] {$\scriptstyle{m}$};
  \filldraw[draw = black, rounded corners=1pt, thick, fill opacity = 0.3, fill = red]  (bT) -- (bB) -- (BB)  -- (BT) -- (bT) node[sloped, midway, below, opacity = 1] {$\scriptstyle{n+l}$};
  \filldraw[draw = black, rounded corners=1pt, thick, fill opacity = 0.3, fill = green]   (aT) --  (aB) -- (bB) -- (bT) --  (aT) node[sloped, midway, below, opacity = 1] {$\scriptstyle{n+k-m}$};
  \filldraw[draw = black, rounded corners=1pt, thick, fill opacity = 0.3, fill = red]  (dT) -- (dB) -- (DB) node[sloped, midway, above, opacity = 1] {$\scriptstyle{n}$}-- (DT) -- cycle;
  \filldraw[draw = black, rounded corners=1pt, thick, fill opacity = 0.3, fill = green]   (aT) --  (aB) -- (dB) -- (dT) --  cycle;
  \filldraw[draw = black, rounded corners=1pt, thick, fill opacity = 0.3, fill = green]   (cT) -- (cB) --  (bB) --  (bT) -- cycle;
  \filldraw[draw = black, rounded corners=1pt, thick, fill opacity = 0.3, fill = green]   (cT) -- (cB) --  (dB) node[sloped, midway, above, opacity = 1] {$\scriptstyle{k}$} --  (dT) -- cycle;
   \filldraw[draw = black, rounded corners=1pt, thick, fill opacity = 0.3, fill = red]  (cT) -- (cB) -- (CB) node[sloped, pos=0.45, above=-0.07cm, opacity = 1] {$\scriptstyle{m+l}$} -- (CT) -- cycle; 
\end{scope}}}}} 
= \sum_{\substack{j = \max(0, m-n),\dots, m \\ \alpha \in T(k-j, l-k+j)}} (-1)^{|\alpha| + (l-k+j)(m-j)} \kup{F^j_\alpha}, 
\label{eq:squarecat} 
\end{align}
where 
\[
F_\alpha^j\eqdef\!\! \sum_{\substack{\beta_1, \beta_2\\ \gamma_1, \gamma_2}} \!\!  c^\alpha_{\beta_1\beta_2} c^{\widehat{\alpha}}_{\gamma_1 \gamma_2} 
\scriptstyle{\NB{\tikz[scale = 0.9]{\tdplotsetmaincoords{70}{100}
\begin{scope}[tdplot_main_coords]
  \coordinate (aT) at (-1, -1, 3);
  \coordinate (bT) at (-1, 1, 3);
  \coordinate (cT) at (1, 1, 3);
  \coordinate (dT) at (1, -1, 3);
  \coordinate (AT) at (-2, -2, 3);
  \coordinate (BT) at (-2, 2, 3);
  \coordinate (CT) at (2, 2, 3);
  \coordinate (DT) at (2, -2, 3);
  \coordinate (aMt) at (-1, -1, 1);
  \coordinate (aMb) at (-1, -1, -1);
  \coordinate (bMt) at (-1, 1, 1);
  \coordinate (bMb) at (-1, 1, -1);
  \coordinate (cMt) at (1, 1, 1);
  \coordinate (cMb) at (1, 1, -1);
  \coordinate (dMt) at (1, -1, 1);
  \coordinate (dMb) at (1, -1, -1);
  \coordinate (AMt) at (-2, -2, 1);
  \coordinate (AMb) at (-2, -2, -1);
  \coordinate (BMt) at (-2, 2, 1);
  \coordinate (BMb) at (-2, 2, -1);
  \coordinate (CMt) at (2, 2, 1);
  \coordinate (CMb) at (2, 2, -1);
  \coordinate (DM) at (2, -2, 0);
  \coordinate (aB) at (-1, -1, -3);
  \coordinate (bB) at (-1, 1, -3);
  \coordinate (cB) at (1, 1, -3);
  \coordinate (dB) at (1, -1, -3);
  \coordinate (AB) at (-2, -2, -3);
  \coordinate (BB) at (-2, 2, -3);
  \coordinate (CB) at (2, 2, -3);
  \coordinate (DB) at (2, -2, -3);
  \draw[thin, <-] ($(dMb)!0.7!(aMt)$) --    +(0,-1.5,0) node [left, sloped] {$\scriptstyle{m-j}$}; 
  \draw[thin, <-] ($(aT)!0.3!(dMt)$) --     +(0,-1.5,0) node [left, sloped] {$\scriptstyle{n+k}$}; 
  \draw[red, thick, <-]($(aMb)!0.3!(dMt)$) --  +(0,-1.5,0) node [left, sloped] {$\scriptstyle{\pi_{\gamma_2}}$}; 
  \draw[red,thick, <-] ($(bT)!0.7!(cMt)$)-- +(0,+1.5,0) node [right, sloped] {$\scriptstyle{\pi_{\beta_1}}$}; 
  \draw[thin, <-] ($(dB)!0.5!(aMb)$) --     +(0,-1.5,0) node [left, sloped] {$\scriptstyle{n+k}$}; 
  \draw[thin, <-] ($(cMb)!0.7!(bMt)$) --    +(0,+1.5,0) node [right, sloped] {$\scriptstyle{n+l+j}$}; 
  \draw[thin, <-] ($(bT)!0.3!(cMt)$) --     +(0,+1.5,0) node [right, sloped] {$\scriptstyle{m+l-k}$}; 
  \draw[thin, <-] ($(cB)!0.7!(bMb)$) --     +(0,+1.5,0) node [right, sloped] {$\scriptstyle{m+l-k}$}; 
  \draw[thin, <-] ($(aMb)!0.3!(bMt)$) .. controls +(-3,0,0) and  +(0,0,-1) .. +(-3,0,+3) node [above, sloped] {$\scriptstyle{j}$}; 
  \draw[red, thick, <-] ($(aB)!0.5!(bMb)$) .. controls +(-3,0,0) and  +(0,+1,0) .. +(-3,-3.1,0) node [left, sloped] {$\scriptstyle{\pi_{\gamma_1}}$}; 
  \filldraw[draw = black, rounded corners=1pt, thick, fill opacity = 0.3, fill = red]  (aT) -- (aB) -- (AB) -- (AT) -- (aT) node[sloped, midway, below, opacity = 1] {$\scriptstyle{m}$};
  \filldraw[draw = black, rounded corners=1pt, thick, fill opacity = 0.3, fill = red]  (bT) -- (bB) -- (BB)  -- (BT) -- (bT) node[sloped, midway, below, opacity = 1] {$\scriptstyle{n+l}$};
  \filldraw[draw = black, rounded corners=1pt, thick, fill opacity = 0.3, fill = green]   (aT) --  (aMt) -- (bMt) -- (bT) --  (aT) node[sloped, midway, below, opacity = 1] {$\scriptstyle{n+k-m}$};
  \filldraw[draw = black, rounded corners=1pt, thick, fill opacity = 0.3, fill = green]   (aMb) -- (aB) --  (bB) node[sloped, midway, above, opacity = 1] {$\scriptstyle{n+k-m}$} --  (bMb) -- cycle;
  \filldraw[draw = black, rounded corners=1pt, thick, fill opacity = 0.3, fill = red]  (dT) -- (dB) -- (DB) node[sloped, midway, above, opacity = 1] {$\scriptstyle{n}$}-- (DT) -- cycle;
  \filldraw[draw = black, rounded corners=1pt, thick, fill opacity = 0.3, fill = yellow]  (aMt) -- (aMb) -- (dMb) -- (dMt) -- cycle;
  \filldraw[draw = black, rounded corners=1pt, thick, fill opacity = 0.3, fill = green]   (aT) --  (aMt) -- (dMt) -- (dT) --  cycle;
  \filldraw[draw = black, rounded corners=1pt, thick, fill opacity = 0.3, fill = green]   (aMb) -- (aB) --  (dB) --  (dMb) -- cycle;
  \filldraw[draw = black, rounded corners=1pt, thick, fill opacity = 0.3, fill = yellow]  (cMt) -- (cMb) -- (dMb)  -- (dMt) -- cycle;
  \filldraw[draw = black, rounded corners=1pt, thick, fill opacity = 0.3, fill = yellow]  (aMt) -- (aMb) -- (bMb) -- (bMt) -- cycle;
  \filldraw[draw = black, rounded corners=1pt, thick, fill opacity = 0.3, fill = yellow]  (cMt) -- (cMb) -- (bMb) -- (bMt) -- cycle;
  \filldraw[draw = black, rounded corners=1pt, thick, fill opacity = 0.3, fill = green]   (cT) --  (cMt) -- (bMt) -- (bT) --  (cT);
  \filldraw[draw = black, rounded corners=1pt, thick, fill opacity = 0.3, fill = green]   (cMb) -- (cB) --  (bB) --  (bMb) -- cycle;
  \filldraw[draw = black, rounded corners=1pt, thick, fill opacity = 0.3, fill = green]   (cT) --  (cMt) -- (dMt) -- (dT) --  (cT) node[sloped, midway, below, opacity = 1] {$\scriptstyle{k}$};
  \filldraw[draw = black, rounded corners=1pt, thick, fill opacity = 0.3, fill = green]   (cMb) -- (cB) --  (dB) node[sloped, midway, above, opacity = 1] {$\scriptstyle{k}$} --  (dMb) -- cycle;
  \filldraw[draw = black, rounded corners=1pt, thick, fill opacity = 0.3, fill = blue]    (aMt) -- (bMt) -- (cMt) -- (dMt) node[sloped, midway, above, opacity = 1] {$\scriptstyle{n+k-m+j}$} -- cycle;
  \filldraw[draw = black, rounded corners=1pt, thick, fill opacity = 0.3, fill = blue]    (aMb) -- (bMb) -- (cMb) -- (dMb) node[sloped, midway, above, opacity = 1] {$\scriptstyle{n+k-m+j}$} -- cycle;
  \filldraw[draw = black, rounded corners=1pt, thick, fill opacity = 0.3, fill = red]  (cT) -- (cB) -- (CB) node[sloped, pos=0.45, above=-0.07, opacity = 1] {$\scriptstyle{m+l}$} -- (CT) -- cycle; 
  \draw[thin, <-] ($(dMb)!0.3!(cMt)$) .. controls +(+3,0,0) and  +(0,0,+1) .. (+3,0,-3) node [below, sloped] {$\scriptstyle{n+j-m}$}; 
  \draw[red, thick, <-]($(cMt)!0.3!(dMb)$) .. controls  +(+1,0,0) and +(0, -1, 0) .. +(1,3,0)  node [right, sloped] {$\scriptstyle{\pi_{\beta_2}}$}; 
\end{scope}}}}.
\]
Moreover, in identities~(\ref{eq:MPcat}), (\ref{eq:digoncat}),  (\ref{eq:digonDURcat}), (\ref{eq:cautiscat}) and (\ref{eq:squarecat}), the terms on the right-hand sides are mutually orthogonal idempotents. In the previous formulas $T(a,b)$ denotes the set of all Young diagram contained in the rectangle of size $a\times b$, $\pi_\lambda$ denotes the Schur polynomial associated with $\lambda$ and $c_{\bullet\bullet}^\bullet$ denote the Littlewood--Richardson constant. Further explanations of notations and conventions can be found in \cite[Appendix 1]{RW1}.

In identity~(\ref{eq:squarecat}), $\kup{\bullet}$ needs to be extended linearly to formal $\QQ$-linear combinations of foams.
\end{prop}

Using this evaluation and the universal construction idea (see~\cite{MR1362791}), we define a functor $\F_N$ from the category of foams to the category of $\QQ[x_1, \dots, x_N]^{\mathfrak{S}_N}$-module.


If $\Gamma$ is a MOY graph, consider the free graded $\QQ[x_1, \dots, x_N]^{\mathfrak{S}_N}$-module spanned by $\Hom_\Foam(\emptyset, \Gamma)$. We mod this space out by 
\[
\bigcap_{G\in \Hom_{\Foam}(\Gamma, \emptyset)}\Ker\left(
  \begin{array}{rcl}
    \Hom_\Foam(\emptyset, \Gamma)& \to &  \QQ[x_1, \dots, x_N]^{\mathfrak{S}_N} \\
    F&\mapsto & \kup{G\circ F}
  \end{array}
\right).
\]
We define $\F_N(\Gamma)$ to be this quotient. The definition of $\F_N$ on morphisms follows. 
From Proposition~\ref{prop:relations-cat-ext}, we deduce that the functor $\F_N$ categorifies the exterior MOY calculus.

\begin{cor}[\cite{RW1}]
  \label{cor:catofMOYcalulus}
  Let $\Gamma$ be a closed MOY graph, then $\F_N(\Gamma)$ is a free graded $\QQ[x_1, \dots, x_N]^{\mathfrak{S}_N}$-module of graded rank equal to $\kup{\Gamma}_N$. 
\end{cor}

In Sections~\ref{sec:disk-like-foams} and \ref{sec:quasi-annular-foams} we will work with elements of $\F_N(\Gamma)$. Such elements are represented by $\QQ[x_1, \dots, x_N]^{\mathfrak{S}_N}$-linear combination of foams bounding $\Gamma$. Since it is more convenient to work with representatives of classes than with the classes themselves, we introduce the following terminology.
\begin{dfn}
  \label{dfn:equivalence-relations}
  \begin{enumerate}
  \item Let $\sum_i \lambda_i F_i$ and $\sum_j \mu_j G_j$ be two elements of the free graded $\QQ[x_1, \dots, x_N]^{\mathfrak{S}_N}$-module spanned by $\Hom_\Foam(\emptyset, \Gamma)$. We say that they are \emph{$N$-equivalent} if they represent the same element in $\F_N(\Gamma)$. 
\item  Let $\sum_i \lambda_i F_i$ and $\sum_j \mu_j G_j$ two elements of the graded $\QQ$-vector space generated by $\Hom_\Foam(\emptyset, \Gamma)$. We say that they are \emph{$\infty$-equivalent} if they are $N$-equivalent for all $N$ in $\NN$.
\item Define $\IE(\Gamma)$ the graded $\QQ$-vector space generated by $\Hom_\Foam(\emptyset, \Gamma)$ modded out by $\infty$-equivalence.
\end{enumerate}
\end{dfn}

\begin{rmk}
  \label{rmk:relation2equivalence}
  The local identities~(\ref{eq:MPcat}), (\ref{eq:dotmig}), (\ref{eq:digoncat}), (\ref{eq:cautiscat}) and (\ref{eq:squarecat}) can be translated into $\infty$-equivalences, while the local identities~(\ref{eq:neckcuttingcat}) and (\ref{eq:digonDURcat}) can only be translated into $N$-equivalences.
\end{rmk}

\subsection{Disk-like foams (or HOMPLYPT foams)}
\label{sec:disk-like-foams}

For this section we fix a non-negative integer. We will work in $\RR^3$ and we denote by $P$ the plane spanned by 
$\left(
\begin{smallmatrix}
  1 \\ 0 \\0
\end{smallmatrix}\right)$ and $
\left(\begin{smallmatrix}
  0 \\ 0 \\1
\end{smallmatrix} \right).$

We consider the cube $C = [0,1]^3$ and will use the following parametrization of its boundary:

\[
\NB{\begin{tikzpicture}[scale= 1.5]
  \tdplotsetmaincoords{80}{70}
\begin{scope}[tdplot_main_coords]
  \coordinate (OOO) at (0, 0, 0);
  \coordinate (IOO) at (1, 0, 0);
  \coordinate (OIO) at (0, 1, 0);
  \coordinate (IIO) at (1, 1, 0);
  \coordinate (OOI) at (0, 0, 1);
  \coordinate (IOI) at (1, 0, 1);
  \coordinate (OII) at (0, 1, 1);
  \coordinate (III) at (1, 1, 1);
  \coordinate (sf) at (+3.5,  0.5,  0.5);
  \coordinate (sh) at (-2.5,  0.5,  0.5);
  \coordinate (sr) at ( 0.5, +2.25,  0.5);
  \coordinate (sl) at ( 0.5, -1.25,  0.5);
  \coordinate (st) at ( 0.5,  0.5, +2.25);
  \coordinate (sb) at ( 0.5,  0.5, -1.25);
  \draw (OOO) -- (OOI);
  \draw [densely dotted] (OOO) -- (OIO);
  \draw (OOO) -- (IOO);
  \draw (III) -- (OII);
  \draw (III) -- (IOI);
  \draw (III) -- (IIO);
  \draw (OOI) -- (OII);
  \draw (OOI) -- (IOI);
  \draw [densely dotted] (OIO) -- (OII);
  \draw [densely dotted] (OIO) -- (IIO);
  \draw (IOO) -- (IOI);
  \draw (IOO) -- (IIO);
  \draw[<-] ($(OOO)!0.5!(OII)$) -- +(-2,0,0) ;
  \draw[<-] ($(OOO)!0.5!(IOI)$) -- +(0,-1,0) ;
  \draw[<-] ($(OOO)!0.5!(IIO)$) -- +(0,0,-1) ;
  \draw[<-] ($(III)!0.5!(IOO)$) -- +(+2,0,0) ;
  \draw[<-] ($(III)!0.5!(OIO)$) -- +(0,+1,0) ;
  \draw[<-] ($(III)!0.5!(OOI)$) -- +(0,0,+1) ;
  \node at (sf) {$s_f$};
  \node at (sh) {$s_h$};
  \node at (sr) {$s_r$};
  \node at (sl) {$s_l$};
  \node at (st) {$s_t$};
  \node at (sb) {$s_b$};
  \begin{scope}[xshift =2cm, yshift =1.5cm]
    \draw[->] (0,0,0) -- (-1,0,0);
    \node at (-1.5,0,0) {$x_2$};
    \draw[->] (0,0,0) -- (0,0.5,0);
    \node at (0,0.75,0) {$x_1$};
    \draw[->] (0,0,0) -- (0,0,0.5);
    \node at (0,0,.75) {$x_3$};
  \end{scope}

\end{scope} 
\end{tikzpicture}}
\qquad
\begin{array}{l}
  s_b = [0,1]^2\times \{0\} \\
  s_t = [0,1]^2\times \{1\} \\
  s_f = [0,1]\times \{0\} \times [0,1]\\
  s_h = [0,1]\times \{1\} \times [0,1]\\
  s_l = \{0\} \times [0,1]^2\\
 s_r = \{1\} \times [0,1]^2
\end{array}
\] 

The symbols $s_\bullet$ denote the 6 squares of the boundary of $C$ and the letters $f, h, l, r, b, t$ stand for \textbf{f}ront, \textbf{h}idden, \textbf{l}eft, \textbf{r}ight, \textbf{b}ottom and \textbf{t}op. The plan $P$ is parallel  to the square $s_f$ and $s_h$.

\begin{dfn}
  \label{dfn:foamincube}
  Let $F$ be a foam with boundary embedded in $C$. Suppose that the boundary of $F$ is contained in $s_l\cup s_r\cup s_b\cup s_t$, and that the MOY-graphs $F\cap s_f$, $F\cap s_h$, $F\cap s_b$ and $F\cap s_t$ are all braid-like. We say that $F$ is \emph{disk-like} if for every point $x$ of $F$, the normal line of the foam $F$ at $x$ is \emph{not} parallel to $P$.   

We say that  $F$ is a \emph{rooted $\Gamma$-foam of level $k$} if additionally:
\begin{itemize}
\item the restriction of $F$ on $s_b$ is a single strand labeled $k$,
\item the restriction of $F$ on $s_t$ is a braid-like MOY graph $\Gamma$,
\item the restriction of $F$ on  $s_l$ and $s_r$ are standard trees (see Definition~\ref{dfn:canonical-tree}).
\end{itemize}
\end{dfn}
\begin{rmk}
  \label{rmk:levelofdisklikefoam}
  The notion of \emph{level} (see Remark~\ref{rmk:braidlikeisotopy}) extends to disk-like foams. 
\end{rmk}

The name disk-like comes from the following lemma.

\begin{lem}
  \label{lem:disklikesubsurface} Let $F$ be a disk-like foam. Every non-empty connected subsurface of $F$ is a disk whose boundary circle intersects each of the four squares $s_l, s_r, s_b$ and $s_t$ non-trivially.
\end{lem}

\begin{proof}
  We consider a non-empty subsurface $\Sigma$ of $F$.
  The condition on the normal vector of disk-like foams implies that the projection on the second (resp. the third) coordinate provides a Morse function with no critical points. This implies that $\Sigma$ is diffeomorphic to its  intersection with  $s_l$ (resp. $s_b$ ) times the interval. Since $\Sigma$ is non-empty and connected $s_l \cap \Sigma$ is an interval. Finally,  $\Sigma$ is a disk which intersects non-trivially the four squares $s_l, s_r, s_b$ and $s_t$.
\end{proof}



\begin{rmk}
  \label{rmk:cutdisklike2braidlike} Let $F$ be a disk-like foam.  The condition on the normal vector implies that for all $t$ in $[0,1]$ the intersection of $\{t\}\times[0,1]^2$ (resp. $[0,1]\times \{t\}$ ) with $F$ is transverse. Moreover, if $\{t\}\times[0,1]^2$ (resp. $[0,1]\times \{t\}$ ) does not contain any singular point of $F$, $\{t\}\times[0,1]^2 \cap F$ (resp. $[0,1]\times \{t\} \cap F$ ) is a braid-like MOY graph. A very similar result (Corollary~\ref{cor:tubelike2vynil}) is given a proper proof in the next subsection.
\end{rmk}

\begin{dfn}
  \label{dfn:disklikefoamcat}
  Let us fix a non-negative integer $k$. The 2-category $\DLF_k$ of disk-like foams of level $k$ consists of the following data:
  \begin{itemize}
  \item Objects are finite sequences of positive integers of level $k$. 
  \item A 1-morphism from $\listk{k}_0$ to $\listk{k}_1$ is a braid-like $\listk{k}_1$-MOY graph-$\listk{k}_0$ (it has level $k$). Composition is given by concatenation of braid-like MOY graphs.
  \item A 2-morphism from a braid-like $\listk{k}_1$-MOY graph-$\listk{k}_0$ $\Gamma_{\mathrm{bot}}$ to a  braid-like $\listk{k}_1$-MOY graph-$\listk{k}_0$ $\Gamma_{\mathrm{top}}$ is an ambient isotopy class (relative to the boundary) of a disk-like foam $F$ in the cube $C$ such that:
    \begin{itemize}
    \item The intersection of $F$ with $s_b$ is equal to $-\Gamma_{\mathrm{bot}}$,
    \item The intersection of $F$ with $s_t$ is equal to $\Gamma_{\mathrm{top}}$,
    \item The intersection of $F$ with $s_l$  is equal to $-\listk{k}_0 \times [0,1]$,
    \item The intersection of $F$ with $s_r$  is equal to $\listk{k}_1 \times [0,1]$.
    \end{itemize}
Compositions are given by stacking disk-like foams and re-scaling. This is illustrated in Figure~\ref{fig:com2mor}.
  \end{itemize}
The $2$-category $\widehat{\DLF_k}$ is constructed as follows:
\begin{itemize}
\item Start from the $2$-category $\DLF_k$.
\item Linearize the 2-hom-spaces over $\QQ$.
\item Mod out every $2$-homspace by $\infty$-equivalence (disk-like foams are considered as foams from $\emptyset$ to the boundary vinyl graph).
\end{itemize}
\end{dfn}

\begin{figure}[ht]
  \centering
  \tikz[scale =1.15]{\input{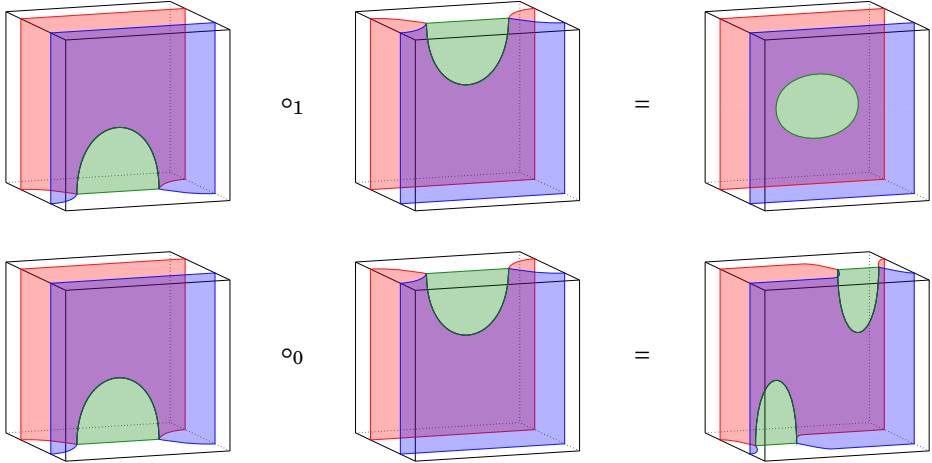}} 
\vspace{0.5cm}  

\tikz[scale =1.15]{\tdplotsetmaincoords{80}{70}
\begin{scope}[xshift =4cm,   tdplot_main_coords]
\draw[very thin] (-1,-1,-1) -- (-1, -1,  1);
\draw[densely dotted, very thin] (-1,-1,-1) -- (-1,  1, -1);
\draw[very thin] (-1,-1,-1) -- ( 1, -1, -1);
\draw[very thin] (-1, 1, 1) -- (-1, -1,  1);
\draw[densely dotted, very thin] ( 1, 1,-1) -- (-1,  1, -1);
\draw[very thin] ( 1,-1, 1) -- ( 1, -1, -1);
\draw[very thin] ( 1,-1, 1) -- (-1, -1,  1);
\draw[densely dotted,very thin] (-1, 1, 1) -- (-1,  1, -1);
\draw[very thin] ( 1, 1,-1) -- ( 1, -1, -1);
\draw[very thin] ( 1,-1, 1) -- ( 1,  1,  1);
\draw[very thin] (-1, 1, 1) -- ( 1,  1,  1);
\draw[very thin] ( 1, 1,-1) -- ( 1,  1,  1);
\filldraw[red, fill opacity =0.3, draw opacity =0.9] (-.5, -1, -1) -- (-.5,  1, -1) -- (-.5,  1,  1) 
.. controls +(0, -0.3 ,0) and +(0,0,0) .. (0, .5, 1)  .. controls +(0,0,-1) and +(0, 0, -1) .. (0, -.5, 1)  .. controls +(0,0,0) and +(0, +0.3, 0) .. ( -.5, -1, 1) --cycle;
\filldraw[blue, fill opacity =0.3, draw opacity =0.9] ( .5, -1, -1) -- ( .5,  1, -1) -- ( .5,  1,  1) 
.. controls +(0, -0.3 ,0) and +(0,0,0) .. (0, .5, 1)  .. controls +(0,0,-1) and +(0, 0, -1) .. (0, -.5, 1)  .. controls +(0,0,0) and +(0, +0.3, 0) .. (  .5, -1, 1) --cycle;
\filldraw[green!50!black, fill opacity = 0.3, draw opacity =0.9] (0, .5, 1)  .. controls +(0,0,-1) and +(0, 0, -1) .. (0, -.5, 1) -- cycle;
\end{scope}
\node at (2, 0) {$\circ_0$};
\node at (6, 0) {$=$};

\begin{scope}[ xshift =0cm, tdplot_main_coords]
\draw[very thin] (-1,-1,-1) -- (-1, -1,  1);
\draw[densely dotted, very thin] (-1,-1,-1) -- (-1,  1, -1);
\draw[very thin] (-1,-1,-1) -- ( 1, -1, -1);
\draw[very thin] (-1, 1, 1) -- (-1, -1,  1);
\draw[densely dotted, very thin] ( 1, 1,-1) -- (-1,  1, -1);
\draw[very thin] ( 1,-1, 1) -- ( 1, -1, -1);
\draw[very thin] ( 1,-1, 1) -- (-1, -1,  1);
\draw[densely dotted,very thin] (-1, 1, 1) -- (-1,  1, -1);
\draw[very thin] ( 1, 1,-1) -- ( 1, -1, -1);
\draw[very thin] ( 1,-1, 1) -- ( 1,  1,  1);
\draw[very thin] (-1, 1, 1) -- ( 1,  1,  1);
\draw[very thin] ( 1, 1,-1) -- ( 1,  1,  1);
\filldraw[red, fill opacity =0.3, draw opacity =0.9] (-.5, -1, 1) -- (-.5,  1,  1) -- (-.5,  1,  -1) 
.. controls +(0, -0.3 ,0) and +(0,0,0) .. (0, .5, -1)  .. controls +(0,0,1) and +(0, 0, 1) .. (0, -.5, -1)  .. controls +(0,0,0) and +(0, +0.3, 0) .. ( -.5, -1, -1) --cycle;
\filldraw[blue, fill opacity =0.3, draw opacity =0.9] ( .5, -1, 1) -- ( .5,  1, 1) -- ( .5,  1,  -1) 
.. controls +(0, -0.3 ,0) and +(0,0,0) .. (0, .5, -1)  .. controls +(0,0,1) and +(0, 0, 1) .. (0, -.5, -1)  .. controls +(0,0,0) and +(0, +0.3, 0) .. (  .5, -1, -1) --cycle;
\filldraw[green!50!black, fill opacity = 0.3, draw opacity =0.9] (0, .5, -1)  .. controls +(0,0,1) and +(0, 0, 1) .. (0, -.5, -1) -- cycle;
\end{scope}

\begin{scope}[ xshift=8cm, tdplot_main_coords]
\draw[very thin] (-1,-1,-1) -- (-1, -1,  1);
\draw[densely dotted,very thin] (-1,-1,-1) -- (-1,  1, -1);
\draw[very thin] (-1,-1,-1) -- ( 1, -1, -1);
\draw[very thin] (-1, 1, 1) -- (-1, -1,  1);
\draw[densely dotted,very thin] ( 1, 1,-1) -- (-1,  1, -1);
\draw[very thin] ( 1,-1, 1) -- ( 1, -1, -1);
\draw[very thin] ( 1,-1, 1) -- (-1, -1,  1);
\draw[densely dotted,very thin] (-1, 1, 1) -- (-1,  1, -1);
\draw[very thin] ( 1, 1,-1) -- ( 1, -1, -1);
\draw[very thin] ( 1,-1, 1) -- ( 1,  1,  1);
\draw[very thin] (-1, 1, 1) -- ( 1,  1,  1);
\draw[very thin] ( 1, 1,-1) -- ( 1,  1,  1);

\fill[red, fill opacity =0.3, draw opacity =0.9] (-.5, 1, -1) -- (-.5,  0, -1) -- (-.5,  0,  1) .. controls +(0, 0.15 ,0) and +(0,0,0) .. (0, .25, 1)  .. controls +(0,0,-1) and +(0, 0, -1) .. (0, .75, 1)  .. controls +(0,0,0) and +(0, -0.1, 0) .. ( -.5, 1, 1) --cycle;
\draw[red, fill opacity =0.3, draw opacity =0.9]  (-.5,  0,  1) .. controls +(0, 0.15 ,0) and +(0,0,0) .. (0, .25, 1)  .. controls +(0,0,-1) and +(0, 0, -1) .. (0, .75, 1)  .. controls +(0,0,0) and +(0, -0.1, 0) .. ( -.5, 1, 1) -- (-.5, 1, -1) -- (-.5,  0, -1);

\fill[red, fill opacity =0.3, draw opacity =0.9] (-.5, 0, 1) -- (-.5,  -1,  1) -- (-.5,  -1,  -1)  .. controls +(0, 0.15 ,0) and +(0,0,0) .. (0, -.75, -1)  .. controls +(0,0,1) and +(0, 0, 1) .. (0, -.25, -1)  .. controls +(0,0,0) and +(0, -0.1, 0) .. ( -.5, 0, -1) --cycle;
\draw[red, fill opacity =0.3, draw opacity =0.9] (-.5, 0, 1) -- (-.5,  -1,  1) -- (-.5,  -1,  -1)  .. controls +(0, 0.15 ,0) and +(0,0,0) .. (0, -.75, -1)  .. controls +(0,0,1) and +(0, 0, 1) .. (0, -.25, -1)  .. controls +(0,0,0) and +(0, -0.1, 0) .. ( -.5, 0, -1);

\fill[blue, fill opacity =0.3, draw opacity =0.9] ( .5, 1, -1) -- ( .5,  0, -1) -- (.5,  0,  1)  .. controls +(0, 0.15 ,0) and +(0,0,0) .. (0, .25, 1)  .. controls +(0,0,-1) and +(0, 0, -1) .. (0, .75, 1)  .. controls +(0,0,0) and +(0, -0.1, 0) .. ( .5, 1, 1) --cycle;
\draw[blue, fill opacity =0.3, draw opacity =0.9]  (.5,  0,  1)  .. controls +(0, 0.15 ,0) and +(0,0,0) .. (0, .25, 1)  .. controls +(0,0,-1) and +(0, 0, -1) .. (0, .75, 1)  .. controls +(0,0,0) and +(0, -0.1, 0) .. ( .5, 1, 1) -- ( .5, 1, -1) -- ( .5,  0, -1);

\fill[blue, fill opacity =0.3, draw opacity =0.9] ( .5, 0, 1) -- ( .5,  -1,  1) -- ( .5,  -1,  -1)  .. controls +(0, 0.15 ,0) and +(0,0,0) .. (0, -.75, -1)  .. controls +(0,0,1) and +(0, 0, 1) .. (0, -.25, -1)  .. controls +(0,0,0) and +(0, -0.1, 0) .. (  .5, 0, -1) --cycle;
\draw[blue, fill opacity =0.3, draw opacity =0.9] ( .5, 0, 1) -- ( .5,  -1,  1) -- ( .5,  -1,  -1)  .. controls +(0, 0.15 ,0) and +(0,0,0) .. (0, -.75, -1)  .. controls +(0,0,1) and +(0, 0, 1) .. (0, -.25, -1)  .. controls +(0,0,0) and +(0, -0.1, 0) .. (  .5, 0, -1);

\filldraw[green!50!black, fill opacity = 0.3, draw opacity =0.9] (0, .25, 1)  .. controls +(0,0,-1) and +(0, 0, -1) .. (0, .75, 1) -- cycle;
\filldraw[green!50!black, fill opacity = 0.3, draw opacity =0.9] (0, -.75, -1)  .. controls +(0,0,1) and +(0, 0, 1) .. (0, -.25, -1) -- cycle;

\end{scope}}
  \caption{ Vertical (top) and horizontal (bottom) compositions of 2-morphisms in $\DLF_k$.}
  \label{fig:com2mor}
\end{figure}

\begin{dfn}
  \label{dfn:degree-disk-like}
  Let $\listk{k}_0$ and $\listk{k}_1$ be two objects of the $\DLF_k$  and  $\Gamma_{\mathrm{bot}}$ and $\Gamma_{\mathrm{top}}$ two 1-morphism from $\listk{k}_0$ to $\listk{k}_1$. The \emph{degree} of a 2-morphism $F: \Gamma_{\textrm{bot}} \to \Gamma_{\textrm{top}}$ is given by formula
\[
\degD(F) = \degext_0(F) - \frac{||\listk{k}_0||^2 + ||\listk{k}_1||^2}{2}, 
\]
where $\degext_0(F)$ is the degree of $F$ as an exterior $0$-foam (see Definition~\ref{dfn:degreefoam}) and if $\listk{k}:= (k_1, \dots, k_l)$ is a finite sequence of non-negative integers, $||\listk{k}||^2:= \sum_{i=1}^l k_i^2$.

Let $\Gamma$ a $1$-morphism from $\listk{k}_0$ to $\listk{k}_1$. The \emph{degree} of a rooted $\Gamma$-foam $F$ is given by:
\[
\degR(F) = \degext_0(F) - \frac{||\listk{k}_0||^2 + ||\listk{k}_1||^2 + 2 k^2}{4}.
\]
\end{dfn}

\begin{rmk}
  \label{rmk:degreecompatible}
  One easily checks that with these definition the degree of $2$-morphisms is additive with respect to vertical and horizontal compositions. In particular, the degrees of identity $2$-morphisms are $0$. Since the relations defining the $\infty$-equivalence are homogeneous, this degree induces a grading on the $2$-homspaces of $\widehat{\DLF_k}$.  Moreover, the composition of a rooted $\Gamma_{\mathrm{bot}}$-foam with an element of $\hom(\Gamma_{\mathrm{bot}}, \Gamma_{\mathrm{top}})$ is a rooted $\Gamma_{\mathrm{top}}$-foam, and the degree is additive with respect to this composition.
\end{rmk}

\begin{dfn}
  \label{dfn:treelikefoams}
  Let $\Gamma$ be a braid-like MOY-graph, and $F$ be a rooted $\Gamma$-foam. We say that $F$ is \emph{tree-like}, if  $\{t\}\times [0,1]^2 \cap F$ is a tree for all $t$ in $[0,1]$. In particular if $\{t\}\times [0,1]^2$ does not contain any singular point of $F$,  then $\{t\}\times [0,1]^2 \cap F$ is a braid-like tree.  
\end{dfn}

From Lemma~\ref{lem:trees} we derive the following lemma which tells us that disk-like foams are combinatorially very simple:

\begin{lem}
  \label{lem:diskdisklike}
  Let $k$ be a non-negative integer, $\Gamma$ be the braid-like $k$-MOY {graph-$k$} consisting of one single strand labeled by $k$ and $F$ be a rooted $\Gamma$-foam (that is a disk-like foam which bound a circle with label $k$). Then $F$ is  $\infty$-equivalent to a disk with label $k$ decorated by $(-1)^{\frac{k(k+1)}{2}}\kup{\widehat{F}}_k$,  where $\widehat{F}$ is the foam obtained from $F$ by capping it with a disk labeled by $k$. 
\end{lem}

Note that this makes sense since $\kup{\widehat{F}}_k$ is a symmetric polynomial in $k$ variables.

\begin{proof}
  Let us denote by $D$ the disk with label $k$  decorated by $(-1)^{\frac{k(k+1)}{2}}\kup{\widehat{F}}_k$.
  It follows directly from the definition of the $\sll_k$-evaluation of foams, that $F$ is $k$-equivalent to $D$. The sign comes from the term  $\sum_{i=1}^ki\chi(F_i(c))/2$ in the definition of $s(F,c)$ (see Definition~\ref{dfn:exteval}). If $N<k$, then both $D$ and $F$ are $N$-equivalent to $0$.
  
  If $N>k$, we will see that the $N$-equivalence between $F$ and $D$ follows from their $k$-equivalence.  We need to prove that for any foam $G$ bounding a circle with label $k$, $\kup{F\circ G}_N= \kup{D\circ G}_N$.   First note that thanks to identity~(\ref{eq:neckcuttingcat}), we can suppose that $G$ is a decorated disk of label $k$. In this case, $D\circ G$ is a decorated sphere of label $k$.  Let $c$ be a coloring of $D\circ G$. The coloring $c$ is given by the color $I(c)=\{i_1(c),  \dots, i_k(c)\}\subseteq\Col$ of this sphere. We have:
  \begin{align} \label{eq:DGc}
    \kup{D\circ G,c}_N = \frac{\kup{D\circ G}_k(x_{i_1(c)}, \dots x_{i_k(c)})}{\prod_{\substack{i\in I(c)\\ j\in \Col\setminus I(c)}}(x_{j}-x_{i})}.
  \end{align}
  Note the foam $D\circ G$ being a sphere, it admits only one $\sll_k$-coloring. It can be obtained from $c$ by replacing $i_a(c)$ by $a$ in $I(c)$ for all $a\in \{1,\dots k\}$.
  
  Similarly, if $c$ is a coloring of $F\circ G$, it gives to the facet containing $G$ a color $I(c)=\{i_1(c), \dots, i_k(c)\}$. Let us denote by $f: I(c) \to \{1, \dots, k\}$ the one-to-one map given by $f(i_a)=a$ and $f(c)$ the $\sll_k$-coloring of $F\circ G$ induced from $c$ by $f$. We have:
  \begin{align}\label{eq:FGc}
    \kup{F\circ G,c}_N = \frac{\kup{F\circ G, f(c)}_k(x_{i_1(c)}, \dots, x_{i_k(c)})}{\prod_{\substack{i\in I\\ j\in \Col\setminus I}}(x_{j}-x_i)}.
  \end{align}
  Combining (\ref{eq:DGc}) and (\ref{eq:FGc}) and keeping the same notations, we get:
  \begin{align*}
    \kup{F\circ G}_N &= \sum_{c \in \mathrm{col}_N(F\circ G)} \kup{F \circ G,c}_N  \\
    &= \sum_{\substack{I \subseteq \Col \\ \#I = k }} \sum_{\substack{c \in \mathrm{col}_N(F\circ G)\\ I(c) =I}} \kup{F \circ G,c}_N \\
    &= \sum_{\substack{I \subseteq \Col \\ \#I = k }} \sum_{\substack{c \in \mathrm{col}_N(D\circ G)\\ I(c) =I}} \kup{D \circ G,c}_N \\
    &= \kup{D\circ G}_N.
  \end{align*}
\end{proof}

\begin{lem}
  \label{lem:treelike-are-enough}
  A rooted $\Gamma$-foam $F$ is $\infty$-equivalent to a $\ZZ$-linear combination of tree-like foams. 
\end{lem}

\begin{proof}
  First, assume that $\Gamma$ is a braid-like $k$-MOY graph-$k$.
  We will show that $F$ is $\infty$-equivalent to a $\ZZ$-linear combination of foams which are superposition of a tree-like $\Gamma$-foam on top of a disk-like foam which bounds a circle and conclude by Lemma~\ref{lem:diskdisklike}
  Assume further that $\Gamma$ has the form:
  \[\digona.\]
  Then, this is the content of identities (\ref{eq:digoncat}).

  If $\Gamma$ is a braid-like $k$-MOY graph-$k$, the result is obtained by induction using repeatedly identities~(\ref{eq:MPcat}) and (\ref{eq:digoncat}). See Figure~\ref{fig:alamain} for an illustration.

  \begin{figure}[ht]
    \[
      \NB{\tikz[scale=0.5]{\input{\imagesfolder/sym_gen-digon}}}
      \sim_\infty  \sum\,
       \NB{\tikz[scale=0.5]{\input{\imagesfolder/sym_gen-digon-2}}}
    \]
\caption{} \label{fig:alamain}
\end{figure}
  
  This induction can be thought of a categorified implementation of the algorithm described in Lemma~\ref{lem:trees}.


  The general case follows. The foam $F$ is embedded in the cube. Consider on $s_f$ the following curve
  \[
\NB{\tikz{\begin{scope}
  \draw (-1,-1) -- (1, -1) -- (1,1) -- (-1,1) -- cycle;
  \draw[rounded corners, red] (-1, -0.8) -- (-0.8, -0.8) -- (-0.8, 0.8) -- (0.8, 0.8) -- (0.8, -0.8) -- (1, -0.8);
\end{scope}}},
  \]
and $S$ the surface embedded in the cube obtained as a product of the previous curve with a unit interval. The intersection of a thickening of this surface with the foam $F$ along is diffeomorphic to a $k$-MOY graph-$k$ for which the first case applies.
\end{proof}

\begin{rmk}
  \label{rmk:dotonlyonleaves} Using the dots migration identity~(\ref{eq:dotmig}), we obtain that a tree-like rooted $\Gamma$-foam is $\infty$-equivalent to $\ZZ$-linear combination of tree-like rooted $\Gamma$-foams, where all non-trivial decoration are on facets which intersect $\Gamma$. These facets are the \emph{leaves} of the rooted $\Gamma$-foam. 
\end{rmk}

\begin{lem}
  \label{lem:tree-likeRequivalent}
  Let $F$ and $F'$ be two tree-like rooted $\Gamma$-foams with non-trivial decorations only on their leaves. Suppose furthermore that these decorations are the same\footnote{The fact that the non-trivial decoration are only on the leaves allows to see the decoration as a function associating a symmetric polynomial with every edge of $\Gamma$. We require that these functions to be the same for $F$ and $F'$.} for $F$ and $F'$. Then $F$ is $\infty$-equivalent to $F'$. 
\end{lem}
\begin{proof}
  This follows directly from the definition of the exterior evaluation of foams. The set of colorings of $F$ and of $F'$ are in one-to-one correspondence (because they are both in one-one correspondence with the set of colorings of their boundary). Let us denote $c$ and $c'$ two corresponding colorings of $F$ and $F'$. The monochrome and the bichrome surfaces of $(F,c)$ and $(F',c')$ are diffeomorphic. The oriented arcs in $(F,c)$ and $(F',c)$  are in one-one correspondence preserving their orientation. Finally the condition on the decorations of $F$ and $F'$ ensures that their contributions to the evaluation are equal.
\end{proof}

\subsection{Vinyl foams (or symmetric foams)}
\label{sec:quasi-annular-foams}

In this part, we work in the thickened annulus $\ann\times [0,1]$. If $
x:= \left(\begin{smallmatrix}
  x_1 \\ x_2 \\ x_3
\end{smallmatrix}\right)$ is an element of $\ann\times [0,1]$, we denote by $t_x$ the vector 
$ \left(\begin{smallmatrix}
  -x_2 \\ x_1 \\ 0
\end{smallmatrix}\right)$, by $v$ the vector
$ \left(\begin{smallmatrix}
  0 \\ 0 \\ 1
\end{smallmatrix}\right)$, and by $P_x$ the affine plane containing $x$ and spanned by $t_x$ and $v$.  If $\theta$ is an element of $[0,2\pipi[$, $P_\theta$ is the half-plane $\left\{
\left.\left(\begin{smallmatrix}
  \rho \cos \theta \\ \rho \sin \theta \\ t 
\end{smallmatrix}\right) \right | (\rho,t) \in \RR_+\times \RR\right\}$.

\[
\tikz[scale=0.8]{\begin{scope}[yscale =0.5]
\draw[dashed] (0, -4) -- (0,8);
\draw (0,4) circle (2cm and 2cm);
\draw (0,4) circle (0.5cm and 0.5cm);
\draw (-2,0) arc (-180:0:2cm and 2cm);
\draw[dotted] (-2,0) arc (180:0:2cm and 2cm);
\draw[dotted] (0,0) circle (0.5cm and 0.5cm);
\draw (-2, 0) -- +(0,4);
\draw ( 2, 0) -- +(0,4);
\draw[dotted] ( -0.5, 0) -- +(0,4);
\draw[dotted] ( +0.5, 0) -- +(0,4);
\filldraw[thick, draw= green!30!black, draw opacity =0.6, fill = green!30!black, fill opacity =0.2] (0, -1) -- (0, +5) -- (3, 8) -- (3,2) node[near start, left, opacity=100, green!30!black] {$P_\theta$} -- cycle;
\filldraw[thick, draw= orange, draw opacity =0.6, fill = orange, fill opacity =0.2] (0,0) -- (0.7,0) arc (0:45 :0.7cm) node[near end, right, orange!50!black, opacity= 100] {$\theta$} --cycle;
\draw[thin] (45:1.5) -- (0,0) -- (1.5,0);
\filldraw[thick, draw= red!70!black, draw opacity =0.6, fill = red!70!black, fill opacity =0.2] (-3, 5)  -- (0.3, 2) node[pos=0.1, below, opacity=100, red!30!black] {$P_x$} -- (0.3, -4) -- (-3,-1) --cycle;
\fill[red] (-1.35, 0.5) circle (0.03) node[right, black] {$x$};
\end{scope}}
\]

\begin{dfn}
  \label{dfn:tubelikefoam}
  Let $k$ be a non-negative integer and $\Gamma_0$ and $\Gamma_1$ two vinyl graphs of level $k$. Let $F$ be a foam with boundary embedded in $\ann \times [0,1]$. Suppose that $F\cap (\ann\times \{0\}) = -\Gamma_0$ and $F\cap (\ann\times \{1\}) = \Gamma_1$. We say $F$ is a \emph{vinyl $\Gamma_1$-foam-$\Gamma_0$ of level $k$} if for every point $x$ of $F$, the normal line of $F$ at $x$ is \emph{not} contained in $P_x$. See Figure~\ref{fig:tuble-like} for an example.
\end{dfn}

\begin{rmk}
  \label{rmk:cutopentubelike}
  Note that if we cut a vinyl foam $F$ along a half plane $P_\theta$, we obtain a disk-like foam.
\end{rmk}

\begin{figure}[ht]
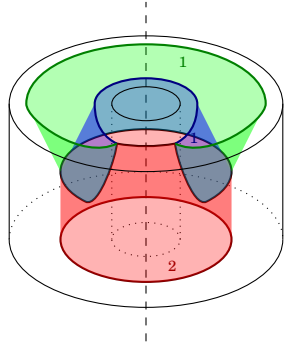

  \centering
\[
\tikz[scale=0.9]{\begin{scope}[yscale =0.5]
\begin{scope}[yshift=4cm]
\coordinate (A) at (-70:1.25 );
\coordinate (A2) at (-110:1.25 );
\end{scope}
\begin{scope}[yshift=2cm]
\coordinate (B) at (-135:1.25);
\coordinate (B2) at (-45:1.25);
\end{scope}
\draw[dashed] (0, -3) -- (0,7);
\draw (-2,0) arc (-180:0:2cm and 2cm);
\draw[dotted] (-2,0) arc (180:0:2cm and 2cm);
\draw[dotted] (0,0) circle (0.5cm and 0.5cm);
\draw (-2, 0) -- +(0,4);
\draw ( 2, 0) -- +(0,4);
\draw[dotted] ( -0.5, 0) -- +(0,4);
\draw[dotted] ( +0.5, 0) -- +(0,4);
\draw[red!50!black, thick] (1.25,0) arc (0:360:1.25cm) node[above, pos =0.8, font= \tiny] {\tiny{$2$}};
\draw[thick] (1.25,2) arc (0: -45:1.25cm)  .. controls +(-0.2,-0.2) and +(0, -0.3) .. (A) arc (-70:-110:1.25) .. controls +(0,-0.3) and +(0.2, -0.2) .. (B)  arc  (-135:-360:1.25cm); 
\fill[red, opacity =0.3] (-1.25,0) arc (180:0:1.25cm) -- +(0,2) arc (0:180:1.25cm) --cycle; 
\fill[red, opacity =0.3] (-1.25,0) arc (-180:0:1.25cm) -- +(0,2) arc (0: -45:1.25cm)  .. controls +(-0.2,-0.2) and +(0, -0.3) .. (A) arc (-70:-110:1.25) .. controls +(0,-0.3) and +(0.2, -0.2) .. (B)  arc  (-135:-180:1.25cm) -- cycle; 
\fill[green, opacity=0.3] (-1.25,2) arc (180:0:1.25cm) -- +(0.5,2) arc (0:180:1.75cm) --cycle; 
\fill[blue, opacity=0.3] (-1.25,2) arc (180:0:1.25cm) -- +(-0.5,2) arc (0:180:0.75cm) --cycle; 
\fill[green, opacity=0.3] (-1.25,2) arc (-180:-135:1.25cm)   .. controls +(0.2,-0.2) and +(0,-0.3)  .. (A2)  .. controls +(-0. 5,-0.5) and +(0,-0.5) .. (-1.75, 4) -- cycle;
\fill[green, opacity=0.3] (1.25,2) arc (0:-45:1.25cm)   .. controls +(-0.2,-0.2) and +(0,-0.3)  .. (A)  .. controls +(0. 5,-0.5) and +(0,-0.5) .. (1.75, 4) -- cycle;
\fill[blue, opacity=0.3] (-1.25,2) arc (-180:-135:1.25cm)   .. controls +(0.2,-0.2) and +(0,-0.3)  .. (A2)  .. controls +(-0.3, +0.3) and +(0,-0.3) .. (-0.75, 4) -- cycle;
\fill[blue, opacity=0.3] (1.25,2) arc      (0:-45:1.25cm)   .. controls +(-0.2,-0.2) and +(0,-0.3)  .. (A)  .. controls +(0. 3,0.3) and +(0,-0.3) .. (0.75, 4) -- cycle;
\draw[thick,blue!50!black] (A)  .. controls +(0. 3,0.3) and +(0,-0.3) .. (0.75, 4) node[midway, below, font = \tiny] {$1$}  arc (0:180:0.75) .. controls +(0,-0.3) and +(-0.3, 0.3) .. (A2); 
\draw[thick,green!50!black] (A)  .. controls +(0.5,-0.5) and +(0,-0.5) .. (1.75, 4) arc (0:180:1.75) node[pos=0.4, below, font= \tiny] {$1$} .. controls +(0,-0.5) and +(-0.5, -0.5) .. (A2); 
\draw (0,4) circle (2cm and 2cm);
\draw (0,4) circle (0.5cm and 0.5cm);
\end{scope}}
\] 
  \caption{An example 
of a vinyl foam}
  \label{fig:tuble-like}
\end{figure}

\begin{dfn}
  \label{dfn:cattubelike}
  The category $\TL_k$ of vinyl foams of level $k$ consists of the following data:
  \begin{itemize}
  \item The objects are elements of $\Vin_k$, \ie vinyl graphs of level $k$,
  \item Morphisms from $\Gamma_0$ to $\Gamma_1$ are (ambient isotopy classes of) vinyl $\Gamma_1$-foams-$\Gamma_0$.
  \end{itemize}
Composition is given by stacking vinyl foams together and rescaling. In the category $\TL_k$ we have one distinguished object which consists of a single essential circle with label $k$ denoted  by $\SS_k$. The \emph{degree} $\degT(F)$ of a vinyl foam $F$ is equal to $\degext_0(F)$. Note that the degree is additive with respect to the composition in $\TL_k$.
\end{dfn}

The name vinyl comes from the following lemma.
\begin{lem}
  \label{lem:tubelikesubsurface}
  Let $F$ be a vinyl $\Gamma_1$-foam-$\Gamma_0$. Then any non-empty connected subsurface $\Sigma$ of $F$ is an annulus. Moreover, for every $t$ in $[0,1]$, $\Sigma \cap \ann \times \{t\}$ is an essential circle in $\ann \times \{t\}$. Such annuli are called \emph{tubes}.
\end{lem}

\begin{proof}
  The condition on the tangent plane of vinyl foams implies that the projection on the last coordinate is a Morse function for $\Sigma$ and that it has no critical points. The result follows.
\end{proof}

\begin{cor}
  \label{cor:tubelike2vynil} Let $F$ be a vinyl $\Gamma_1$-foam-$\Gamma_0$.
  \begin{enumerate}
  \item Let $t$ be an element of $[0,1]$ such that the intersection of $\ann\times \{t\}$ and $F$ is generic (i. e.  $\ann\times \{t\}$ does not contain any singular points of $F$ and the intersection of $\ann\times \{t\}$ with the bindings of $F$ is transverse). Then $F\cap \ann\times \{t\}$ is a vinyl graph.
  \item Let $\theta$ be an element of $[0,2\pipi[$ and assume that the intersection of $F$ with $P_\theta$ is generic. Then the intersection of $F$ and $P_\theta$ is braid-like.
  \end{enumerate}
\end{cor}
Before proving the statements, let us emphasize that there are only finitely many $t$'s (resp. $\theta$'s) for which the intersection of $F$ and $\ann\times \{t\}$ (resp. $P_\theta$) is not generic.

\begin{proof} 
First note that every point of $F$ is contained in a connected subsurface which intersects $\ann\times\{0,1\}$ non-trivially. Indeed we can cable foams just like we can cable MOY graphs (see Figure~\ref{fig:rotMOY}). This gives us a collection of sub-surfaces of $F$ which covers it. 
Let us prove the first part.
Let $x=
\left(  \begin{smallmatrix}
    x_1 \\ x_2 \\ t_0
  \end{smallmatrix} \right)
$ be a point in $F$ and $\Sigma$ a connected subsurface of $F$ containing $x$. Since $F$ is vinyl, the scalar product of $t_x$ with the  tangent vector of $F\cap \ann\times \{t\}$ is non-zero. Since $\Sigma$ is connected, this quantity  is either always positive or always negative on $\Sigma$. The graphs $\Gamma_0$ and $\Gamma_1$ being vinyl, it is positive. This proves that  $F\cap \ann\times \{t\}$ is vinyl.

The second part is similar but we consider the scalar product of $P_\theta\cap F$ with 
$v= \left(\begin{smallmatrix}
  0 \\ 0 \\1
\end{smallmatrix} \right)$.
\end{proof}

The notion of tree-like foams developed in Section~\ref{sec:disk-like-foams} extends {\sl mutatis mutandis} to the concept of foams in the thickened annulus. The analogues of rooted $\Gamma$-foams are $\widehat{\Gamma}$-foams-$\SS_{k}$. We have analogues of Lemmas~\ref{lem:treelike-are-enough} and \ref{lem:tree-likeRequivalent}:

\begin{lem}
  \label{lem:vinyl2tree-like}
  Let $\Gamma$ be a vinyl graph of level $k$ and $F$ a  vinyl $\Gamma$-foam-$\SS_k$. Then $F$ is $\infty$-equivalent to a $\ZZ$-linear combination of tree-like foams.
\end{lem}

\begin{proof}
 First we choose a $\theta$ in $[0,2\pipi]$ such that the intersection of $P_\theta$ with the foam $F$ is generic. Thanks to Corollary~\ref{cor:tubelike2vynil}, we know that this intersection is a braid-like $\listk{k}$-MOY graph-$k$ $B$ for some finite sequence $\listk{k}$ of positive integers. We can suppose that $F$ is locally diffeomorphic to $B \times [\epsilon, \epsilon]$. The algorithm described in Lemma~\ref{lem:trees} and the local identities~(\ref{eq:MPcat}) and (\ref{eq:digoncat}) tells us that $F$ is $\infty$-equivalent to a $\ZZ$-linear combination of vinyl foams such that the intersection with $P_\theta$ is the canonical $\listk{k}$-tree-$k$. If we cut these foams along  $P_\theta$, we can apply Lemma~\ref{lem:treelike-are-enough} on each of these foams. Gluing back the result along the canonical $\listk{k}$-tree-$k$ gives us a $\ZZ$-linear combination of tree-like $\Gamma$-foams-$\SS_k$ which is $\infty$-equivalent to $F$.
\end{proof}

\begin{rmk}
  \label{rmk:leaves}
  Just like for the disk-like context, thanks to the dot migration identity~(\ref{eq:dotmig}), we can move all  non-trivial decorations of a tree-like foam on its leaves.
\end{rmk}


The proof of Lemma~\ref{lem:tree-likeRequivalent} can be easily adapted to the annular case. This gives the following lemma.

\begin{lem}
  \label{lem:tree-likeRequiv}
   Let $\Gamma$ be a vinyl graph and $F$ and $F'$ be two tree-like $\Gamma$-foams-$\SS_k$ with non-trivial decorations only on their leaves. Suppose furthermore that these decorations are the same for $F$ and $F'$. Then $F$ is $\infty$-equivalent to $F'$. 
\end{lem}

\section{Soergel Bimodules}
\label{sec:soergel-bimodules-1}
In this section we prove that for any braid-like MOY graph $\Gamma$  the space of rooted $\Gamma$-foams regarded up to $\infty$-equivalence is isomorphic to the Soergel bimodule associated with $\Gamma$.

\subsection{Some polynomial algebras}
\label{sec:some-polyn-algebr}

\begin{notation}
  \label{not:polynomial-algebras}
  \begin{enumerate}
  \item We denote the graded ring $\QQ[T_1, \dots,T_N]^{\mathfrak{S}_N}$ by $\SP{N}$, where the indeterminates $T_\bullet$ are homogeneous of degree $2$. This is the \emph{$q$-degree}.
   \item Denote  by $\mathsf{C}$ the category of $\ZZ$-graded finitely generated projective $\SP{N}$-modules. If $M$ is an object of $\mathsf{C}$, $Mq^{i}$ denotes the same object where the degree has been shifted by $i$. This means that $(Mq^i)_j = M_{j-i}$. If $P(q)= \sum_i a_i q^i$ is a Laurent polynomial in $q$ with positive integer coefficients, $MP(q)$ denotes the module
\[
\bigoplus_i (Mq^i)^{a_i}.
\]
\item Let $\listk{k}= (k_1, \dots, k_l)$ be a finite sequence of
  positive
  integers of level $k$ (if $k=0$ the empty sequence is allowed). 
  The group $\prod_{i=1}^l \mathfrak{S}_{k_i}$ is denoted by
  ${\mathfrak{S}_{\listk{k}}}$.  We define the algebra $A_{\listk{k}}$: 
\[ A_{\listk{k}}:=\SP{N}[x_1, \dots,   x_k]^{\mathfrak{S}_{\listk{k}}}. \]
 The indeterminates $x_\bullet$ are homogeneous of degree $2$. If $\listk{k}= (k)$ (that is if $\listk{k}$ has length $1$), we write $A_k$ instead of $A_{(k)}$.
   \item If $\Gamma$ is a vinyl graph, denote by $\IET(\Gamma)_\QQ$ the graded $\QQ$-vector space  generated by vinyl $\Gamma$-foams-$\SS_k$ modded out by $\infty$-equivalence (see Definition~\ref{dfn:equivalence-relations}). Define $\IET(\Gamma):= \IET(\Gamma)_\QQ\otimes_\QQ \SP{N}$. Since for all $k$, the exterior $\sll_k$-evaluation of foams is homogeneous, the $\SP{N}$-module $\IET(\Gamma)$ is naturally graded. 
  \end{enumerate}
\end{notation}

Before dealing with Soergel bimodules, we state the following lemma which relates the algebra $A_{\listk{k}}$ with vinyl foams.

\begin{lem}
  \label{lem:IE4circles}
  Let $\listk{k} := (k_1, \dots, k_l)$ be a finite sequence of positive integers and $\SS_{\listk{k}}$ be the vinyl graph which consists of $l$ oriented circles with labeling induced by $\listk{k}$. Then $\IET(\SS_{\listk{k}})$ is isomorphic to $A_{\listk{k}}$ as a graded $\SP{N}$-module. 
\end{lem}

\begin{proof}
  Let $k = \sum_{i}^l k_i$ and  $T$ be the canonical $\listk{k}$-tree-$(k)$. For $i$ in $\{1, \dots, l\}$ and $\lambda_i$ denotes a Young diagram with at most $k_i$ lines. Denote $\pi_{\lambda_i}^{(i)}$ the Schur polynomial associated with $\lambda_i$ in the variables $x_{1+r_i}, \dots, x_{k_i+ r_i}$ where  $r_i=\sum_{j=1}^{i-1} k_j$.
  
  A $\SP{N}$-base of $A_{\listk{k}}$ is given by $(\pi_{\lambda_1}^{(1)}, \dots, \pi_{\lambda_l}^{(l)})=: \pi_{\boldsymbol{\lambda}}$ where the $\lambda_i$'s take all possible shapes. Being given $\boldsymbol{\lambda}= (\lambda_1, \dots, \lambda_l)$ a sequence of of Young diagrams as described above, define $F_{\boldsymbol{\lambda}}$ to be the foam $T \times \SS^1$ where the $i$th leaf of $F$ is decorated by the Schur polynomial $\pi_{\lambda_i}^{(i)}$. 

  Let us prove that the $\SP{N}$-linear map sending  $\pi_{\boldsymbol{\lambda}} \in A_{\listk{k}}$ to $F_{\boldsymbol{\lambda}}\in \IET(\SS_{\listk{k}})$ is  bijective. It is surjective because of dots migration (\ref{eq:dotmig}). Let $\sum_{j=1}^h\mu_{\boldsymbol{\lambda}_j} \pi_{\boldsymbol{\lambda}_j}$ be a linear combination of monomials mapped to $0$. Let $\ell_{\max}$ be the maximal length of all lines appearing in the Young diagrams of all the size of the Young diagram of $(\boldsymbol{\lambda}_j)_{j=1, \dots, h}$. For $M = \ell_{max}(k+1)$, the foams $F_{\boldsymbol{\lambda}_h}$ precomposed by a cup labeled by $k$ are  linearly independent in $\F_{M}(\SS_{\listk{k}})$. Hence the coefficients $\mu_{\boldsymbol{\lambda}_j}$ must be all equal to $0$.
\end{proof}

\subsection{Singular Soergel bimodules}
\label{sec:sing-soerg-bimod}

We introduce singular Soergel bimodules. See for instance \cite{MR1173115, MR2097586, MR2339573, MR2844932} or   \cite{2016arXiv160202769W} for a pictorial description close to ours.

\begin{dfn}
  \label{dfn:soergel-bimodules}
  Let $\Gamma$ be a braid-like $\listk{k}_1$-MOY graph-$\listk{k}_0$. If $\Gamma$ has no trivalent vertices, we have $\listk{k}_0= \listk{k}_1$ and we define $\BS(\Gamma)$ to be equal to $A_{\listk{k}_0}$ as a $A_{\listk{k}_0}$-module-$A_{\listk{k}_0}$. If $\Gamma$ has only one trivalent vertex (which is supposed to be of type $(a,b,a+b)$), then:
  \begin{itemize}
  \item if the length of $\listk{k}_1$ is equal to the length of $\listk{k}_0$ plus 1, we define $\BS(\Gamma)$ to be $A_{\listk{k}_1}q^{-ab/2}$ as a $A_{\listk{k}_1}$-module-$A_{\listk{k}_0}$; 
  \item if the length of $\listk{k}_0$ is equal to the length of $\listk{k}_1$ plus 1, we define $\BS(\Gamma)$ to be $A_{\listk{k}_0}q^{-ab/2}$ as a $A_{\listk{k}_1}$-module-$A_{\listk{k}_0}$;
  \end{itemize}
If $\Gamma$ has more than one trivalent vertex, if necessary we perturb\footnote{In other words, we choose an ambient isotopy of the square such that the images of the trivalent vertices of $\Gamma$ have distinct $y$-coordinate} $\Gamma$ to see it as a composition:
\[
\Gamma = \Gamma_t \circ_{\listk{k^t}} \Gamma_{t-1} \circ_{\listk{k^{t-1}}} \cdots \circ_{\listk{k^2}}  \Gamma_1 \circ_{\listk{k^1}} \Gamma_0, 
\]
where $\Gamma_i$ is a braid-like $\listk{k^{i+1}}$-MOY graph-$\listk{k^{i}}$ with one trivalent vertex, for all $i$ in $\{0,\dots, t\}$. The symbols $\circ_{\listk{k^{i}}}$ means that  $\Gamma_i$ and $\Gamma_{i-1}$ are glued along $\listk{k^i}$. We have $\listk{k^{0}}= \listk{k}_{0}$ and $\listk{k^{t+1}} = \listk{k}_{1}$. We define
\[
\BS(\Gamma):= \BS(\Gamma_t) \otimes_{A_{\listk{k^t}}} \BS(\Gamma_{t-1}) \otimes_{A_{\listk{k^{t-1}}}} \cdots \otimes_{A_{\listk{k^2}}}  \BS(\Gamma_1) \otimes_{A_{\listk{k^1}}} \BS(\Gamma_0).
\]
The space $\BS(\Gamma)$ has a natural structure of $A_{\listk{k}_1}$-module-$A_{\listk{k}_0}$. It is called the \emph{Soergel bimodule associated with $\Gamma$}. Note that the grading of $\BS(\Gamma)$ takes values either in $\ZZ$ or in $\frac12 + \ZZ$.
\end{dfn}

\begin{exa} The singular Soergel bimodule associated with
  \label{exa:Soergel-bimodule-of-graph}
  \[\Gamma =\NB{\tikz[font=\tiny]{\begin{scope}
\coordinate (B1) at (1, 0);
\coordinate (B2) at (2, 0);
\coordinate (B3) at (3, 0);
\coordinate (B4) at (4, 0);
\coordinate (T1) at (1.5, 4);
\coordinate (T2) at (2.5, 4);
\coordinate (T3) at (3.5, 4);
\coordinate (A1) at (1.5,0.5);
\coordinate (A2) at (2,1.5);
\coordinate (A3) at (2.5,2.5);
\coordinate (A4) at (3,3);
\coordinate (A5) at (3,3.5);
\draw[->-]  (B1) .. controls +(0,0.2) and + (0, -0.2) .. (A1) node[midway,above] {$1$};
\draw[->-]  (B2) .. controls +(0,0.2) and + (0, -0.2) .. (A1) node[midway,above] {$2$};
\draw[->-]  (B3) .. controls +(0,0.5) and + (0, -0.2) .. (A2) node[midway,left] {$1$};
\draw[->-]  (B4) .. controls +(0,0.5) and + (0, -0.2) .. (A4) node[midway,left] {$1$};
\draw[->-]  (A1) .. controls +(0,0.2) and + (0, -0.2) .. (A2) node[midway,left] {$3$};
\draw[->-]  (A2) .. controls +(0,0.2) and + (0, -0.2) .. (A3) node[midway,left] {$4$};
\draw[->-]  (A3) .. controls +(0,0.2) and + (0, -0.5) .. (T1) node[midway,left] {$2$};
\draw[->-]  (A3) .. controls +(0,0.2) and + (0, -0.2) .. (A4) node[midway,below] {$2$};
\draw[->-]  (A4) .. controls +(0,0.2) and + (0, -0.2) .. (A5) node[midway,left] {$3$};
\draw[->-]  (A5) .. controls +(0,0.2) and + (0, -0.2) .. (T2) node[midway,left] {$1$};
\draw[->-]  (A5) .. controls +(0,0.2) and + (0, -0.2) .. (T3) node[midway,right] {$2$};
\end{scope}}}  
\]
is the following $A_{(2,1,2)}$-module-$A_{(1,2,1,1)}$:
\begin{align*}&A_{(2,1,2)} \otimes_{A_{(2,3)}} A_{(2,3)} \otimes_{A_{(2,3)}} A_{(2,2,1)} \otimes_{A_{(4,1)}} A_{(4,1)} \otimes_{A_{(4,1)}} A_{(3,1,1)}\otimes_{A_{(3,1,1)}} A_{(1,2,1,1)}q^{-13/2} \\
&\simeq A_{(2,1,2)} \otimes_{A_{(2,3)}} A_{(2,2,1)} \otimes_{A_{(4,1)}} A_{(1,2,1,1)} q^{-13/2}.
\end{align*}
\end{exa}


\begin{rmk}
  \label{rmk:soergel-well-defined}
  For Definition~\ref{dfn:soergel-bimodules} to be valid, the isomorphism type of the bimodule $\BS(\Gamma)$ should not depend on the decomposition of $\Gamma$. This is clear since two such decompositions are related by the ``commutation of faraway vertices'' for which the isomorphism is clear (see as well Remark~\ref{rmk:soergel-well-defined-2}). The purpose of the grading shift introduced in the previous definition is to ensure compatibility of gradings in Proposition~\ref{prop:Soergel-foam}. In order to keep track of this overall shift, define
\[
s(\Gamma) = \sum_{\substack{v \textrm{ vertex of $\Gamma$} \\ \textrm{of type $(a,b, a+b)$}}} \frac{ab}2.
\]
\end{rmk}

\subsection{A 2-functor}
\label{sec:2-functor}

The relationship between Soergel bimodules and foams has already been investigated, see for instance \cite{RW, 2016arXiv160202769W}. However, we develop in this section a foam interpretation of Soergel bimodules themselves and not only of morphisms between them.

\begin{dfn}
  \label{dfn:foam-modulo-infty}
  Let $\Gamma$ be a braid-like $\listk{k}_1$-MOY graph-$\listk{k}_0$. We set $\IED(\Gamma)$ to be the free $\SP{N}$-module generated by the set of rooted $\Gamma$-foams modded out by $\infty$-equivalence. If $F$ is in $2$-$\Hom_{\DLF_k}(\Gamma_0, \Gamma_1)$, we denote by $\IED(F)$ the map $\IED(\Gamma_0) \to \IED(\Gamma_1)$ induced by $F$. It is a map of graded $\SP{N}$-modules. Given two objects $\listk{k}_0$ and $\listk{k}_1$ of $\DLF_k$, this defines a functor $\IED$ from $1$-$\Hom_{\DLF_k}(\listk{k}_0, \listk{k}_1)$ to the category of graded $\SP{N}$-modules.
\end{dfn}

\begin{lem}
  \label{lem:id-dl-foam-2-algebra}
  Let $\listk{k}= (k_1, \dots, k_l)$ be a finite sequence of positive integers. The space $\IED(\listk{k} \times I)$ has a natural structure of $\SP{N}$-algebra. As an algebra it is isomorphic to $A_{\listk{k}}$.
\end{lem}

\begin{proof} 
  The algebra structure is induced by concatenation of disk-like $(\listk{k} \times I)$-foams along the standard tree. The unit is the standard tree times the interval, with each facets trivially decorated. 
Note, that there is an isomorphism of $\SP{N}$-algebras:
\[
A_{\listk{k}} \simeq \bigotimes_{i=1}^l \SP{N}[x_1, \dots,x_{k_i}]^{\mathfrak{S}_{k_i}}.
\]
It is convenient to use this description of $A_{\listk{k}}$ to define the isomorphism between $A_{\listk{k}}$ and $\IED(\listk{k}\times I)$. For $\mathbf{P}=P_1\otimes \dots \otimes P_l$ a pure tensor in $A_{\listk{k}}$, we define $\phi(\mathbf{P})$ to be the ($\infty$-equivalence class of the) standard tree times the interval with decorations $P_1, \dots, P_l$ on its leaves and trivial decorations on the other facets. 
This is clearly an algebra morphism and it is surjective thanks to Lemma~\ref{lem:treelike-are-enough}. We now focus on injectivity. Both $\IED(\listk{k} \times I)$ and $A_{\listk{k}}$ have natural structures of $A_{k}$-modules. For $A_{\listk{k}}$ this comes from the injection of $A_{k}$ in $A_{\listk{k}}$. For $\IED(\listk{k} \times I)$, this comes by decorating the ``root'' facet, that is the facet which bounds the edge  $k \times I$ (which it self is in $s_b$). The map $\phi$ respects these structures of $A_{k}$-modules because the dots migration identity~(\ref{eq:dotmig}) is part of the $\infty$-equivalence (see Remark~\ref{rmk:relation2equivalence}). Thanks to identity~(\ref{eq:digoncat}) (which is as well compatible with the $\infty$-equivalence) used $(k-1)$ times, we know that $\IED(\listk{k} \times I)$ is free of rank 
$\left[ \begin{smallmatrix}
  k \\
k_1\,\, k_2\,\dots\, k_l 
\end{smallmatrix} \right]$. The algebra $A_{\listk{k}}$ is as well a free $A_{k}$-module of rank $\left[\begin{smallmatrix}
  k \\
k_1\,\, k_2\,\dots\, k_l 
\end{smallmatrix} \right]$. This is enough to conclude that $\phi$ is indeed an isomorphism.
\end{proof}

\begin{cor}
  \label{cor:bimodule-strucure-on-IE-Gamma}
  Let $\Gamma$ be a braid-like $\listk{k}_1$-MOY graph-$\listk{k}_0$. The space $\IED(\Gamma)$ has a natural structure of $A_{\listk{k}_1}$-module-$A_{\listk{k}_0}$. Let us define $r_s$ and $r_m$ the two Laurent polynomials by the formulas:
\[
r_s(\Gamma):= \prod_{\substack{v \in V(\Gamma)\textrm{ split} \\ \textrm{$v$ of type $(a,b,a+b)$}}}
\begin{bmatrix}
  a+b \\ a
\end{bmatrix} 
\quad \textrm{and} \quad
r_m(\Gamma):= \prod_{\substack{v \in V(\Gamma) \textrm{ merge} \\ \textrm{$v$ of type $(a,b,a+b)$}}}
\begin{bmatrix}
  a+b \\ a
\end{bmatrix}. 
\]
The space $\IED(\Gamma)$ is a free $A_{\listk{k}_1}$-module of graded rank $r_m(\Gamma)q^{-s(\Gamma)}$ and a free module-$A_{\listk{k}_0}$ of graded rank $r_s(\Gamma)q^{-s(\Gamma)}$.

\end{cor}

\begin{proof}
  The algebras $A_{\listk{k}_0}$ and $A_{\listk{k}_1}$ are isomorphic to $\IED(\listk{k}_0\times I)$ and $\IED(\listk{k}_1\times I)$ and the action of $A_{\listk{k}_0}$ and $A_{\listk{k}_1}$ are given by concatenating disk-like $(\listk{k}_0\times I)$-foam and $(\listk{k}_1\times I)$ along $s_r$ and $s_l$. The statement about the freeness and the rank follows directly from Lemma~\ref{lem:trees}, Remark~\ref{rmk:treeRT} and the fact that the identities (\ref{eq:MPcat}) and (\ref{eq:digoncat}) holds in the $\infty$-equivalence setting (see Remark~\ref{rmk:relation2equivalence}.)
\end{proof}

\begin{rmk}
  \label{rmk:equality-dimension}
  Note that this corollary implies, that for any  braid-like $\listk{k}_1$-MOY graph-$\listk{k}_0$ $\Gamma$, we have:
\[
r_s(\Gamma)\dim^\QQ_q A_{\listk{k}_1}  = r_m(\Gamma)\dim^\QQ_q A_{\listk{k}_0}.
\]
where $\dim_q^\QQ W$ ($\in \ZZ[[q]]$) denotes the graded dimension of a $W$ as a graded $\QQ$-vector space provided each graded piece is finite-dimensional. 
\end{rmk}

\begin{cor}
  \label{cor:elementary-graph-2-soergel}
  Let us consider $\Gamma_m$ the braid-like $\listk{k}_1$-MOY graph-$\listk{k}_0$ and $\Gamma_s$
the braid-like $\listk{k}_0$-MOY graph-$\listk{k}_1$ given by
\begin{align*}
  &\Gamma_m:=\NB{\tikz[scale= 0.7]{
\begin{scope}[yscale = 1]
\draw[->] (0,0) -- (0,0.5) node [at end, above] {$a+b$};  
\draw[>-] (-0.5, -0.5) -- (0,0) node [at start, below] {$a$};  
\draw[>-] (+0.5, -0.5) -- (0,0) node [at start, below] {$b$};  
\draw[->] (+1.5, -0.5) -- ++(0,1);
\draw[->] (+3,   -0.5) -- ++(0,1);
\draw[->] (-1.5, -0.5) -- ++(0,1);
\draw[->] (-3,   -0.5) -- ++(0,1);
\node at (2.35, 0) {$\dots$};
\node at (-2.35, 0) {$\dots$};
\end{scope}}} \qquad \textrm{and} \\
  &\Gamma_s:=\NB{\tikz[scale= 0.7]{
\begin{scope}[yscale = -1]
\draw[-<] (0,0) -- (0,0.5) node [at end, below] {$a+b$};  
\draw[<-] (-0.5, -0.5) -- (0,0) node [at start, above] {$a$};  
\draw[<-] (+0.5, -0.5) -- (0,0) node [at start, above] {$b$};  

\draw[<-] (+1.5, -0.5) -- ++(0,1);
\draw[<-] (+3,   -0.5) -- ++(0,1);
\draw[<-] (-1.5, -0.5) -- ++(0,1);
\draw[<-] (-3,   -0.5) -- ++(0,1);
\node at (2.35, 0) {$\dots$};
\node at (-2.35, 0) {$\dots$};
\end{scope}}}.
\end{align*}
Then $\IED(\Gamma_s)$ isomorphic to $A_{\listk{k}_0}$ as a $A_{\listk{k}_1}$-module-$A_{\listk{k}_0}$ and
$\IED(\Gamma_m)$ isomorphic to $A_{\listk{k}_0}$ as a graded $A_{\listk{k}_0}$-module-$A_{\listk{k}_1}$. This makes sense, since $A_{\listk{k}_0}$ is a sub-algebra of $A_{\listk{k}_1}$.
\end{cor}

\begin{proof}
  We only prove the statement for $\Gamma_s$ (the proof for $\Gamma_m$ is similar). We can bend rooted $\Gamma_s$-foams on one side (see Figure~\ref{fig:bend}) to see them as rooted $\listk{k}_1\times I$-foams.
  \begin{figure}[ht]
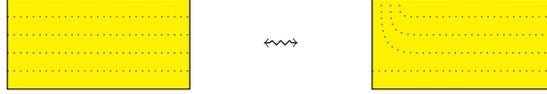

    \centering
    \[
 \tikz{\begin{scope}[scale=1.2]
  \begin{scope}
    \filldraw[draw= black, fill = yellow] (-1,-1) -- (1, -1) -- (1, 0)
    -- (-1, 0) -- cycle;
    \draw[dotted] (-1, -0.2) -- +(2,0);
    \draw[dotted] (-1, -0.4) -- +(2,0);
    \draw[dotted] (-1, -0.6) -- +(2,0);
    \draw[dotted] (-1, -0.8) -- +(2,0);
  \end{scope}
\node at (2, -0.5) {$\leftrightsquigarrow$};
  \begin{scope}[xshift= 4cm]
    \filldraw[draw= black, fill = yellow] (-1,-1) -- (1, -1) -- (1, 0)
    -- (-1, 0) -- cycle;
    \draw[dotted] (-0.7, -0)  .. controls +(0,-0.2) and +(-0.2, 0) .. (-0.5, -0.2) -- +(1.5,0);
    \draw[dotted] (-0.8, -0)  .. controls +(0,-0.3) and +(-0.3, 0) .. (-0.5, -0.4) -- +(1.5,0);
    \draw[dotted] (-0.9, -0)  .. controls +(0,-0.4) and +(-0.4, 0) .. (-0.5, -0.6) -- +(1.5,0); 
    \draw[dotted] (-1, -0.8) -- +(2,0);
  \end{scope}
\end{scope}}
 \]
 \caption{Bending a foam}
 \label{fig:bend}
\end{figure}
This shows that as a $\SP{N}$-module $\IED(\Gamma_s)$ is isomorphic to $A_{\listk{k}_1}$. The $A_{\listk{k}_1}$-module structures is straightforward since we did not bend on this part. The module-$A_{\listk{k}_0}$ structures follows from identity~(\ref{eq:dotmig}) (read from right to left). Note that the grading shift ensures compatibility with the grading of rooted foams (see Definition~\ref{dfn:degree-disk-like}).
\end{proof}

\begin{lem}
  \label{lem:comp-with-bim-str}
  Let $\Gamma_b$ and $\Gamma_t$ be two braid-like $\listk{k}_1$-MOY graphs-$\listk{k}_0$ and $F$ be an element of  $2-\hom_{\DLF}(\Gamma_b, \Gamma_t)$. The linear map $\IED(F)$ is a map of $A_{\listk{k}_1}$-modules-$A_{\listk{k}_0}$.
\end{lem}

\begin{proof}[Sketch of the proof.]
  This follows from the fact, that the operations \emph{concatenating foams horizontally} and \emph{concatenating foams vertically} commute up to isotopy and hence up to $\infty$-equivalence. The argument is the same as in \cite[Proposition 2.2.14]{LHRThese}.
\end{proof}

\begin{lem}
  \label{lem:concat-2-tensor}
  Let $\Gamma_0$ be a braid-like $\listk{k}_1$-MOY graph-$\listk{k}_0$ and $\Gamma_1$ be be a $\listk{k}_2$-MOY graph-$\listk{k}_1$. Then we have the following isomorphism of 
$A_{\listk{k}_2}$-module-$A_{\listk{k}_0}$:
\[\IED(\Gamma_1 \circ_{\listk{k}_1} \Gamma_2) \simeq \IED(\Gamma_1) \otimes_{A_{\listk{k}_1}} \IED( \Gamma_2).  \]
\end{lem}

\begin{proof}
  We have an  ${A_{\listk{k}_1}}$-bilinear morphism of $A_{\listk{k}_2}$-module-$A_{\listk{k}_0}$ from $\IED(\Gamma_1) \times \IED( \Gamma_2)$ (where $1^k= (1,\dots,1)$) to $\IED(\Gamma_1 \circ_{\listk{k}_1} \Gamma_2)$ given by concatenating foams along the standard tree for $\listk{k}_1$. This induce a map $\Psi:\IED(\Gamma_1) \otimes_{A_{\listk{k}_1}} \IED( \Gamma_2) \to \IED(\Gamma_1 \circ_{\listk{k}_1} \Gamma_2)$. We now prove that $\Psi$ is bijective.

The surjectivity is easy. Indeed, thanks to Lemmas~\ref{lem:treelike-are-enough} and \ref{lem:tree-likeRequivalent} every element of $\IED(\Gamma_1 \circ_{\listk{k}_1} \Gamma_2)$ is a linear combination of decorated tree-like foams $F_i$ and we can choose the shape of these tree-like foams. Hence we can suppose that at the locus where $\Gamma_1$ and $\Gamma_2$ are glued together, the intersection of the foams $F_i$ with a vertical plane are equal to the standard trees. Hence every $F_i$ is in the image of $\Psi$ and their sum as well.

To conclude, we argue with graded $\QQ$-dimensions since every graded piece is finite-dimensional. We have:
\begin{align*}
\dim_q^\QQ \left(\IED(\Gamma_1) \otimes_{A_{(k)}} \IED( \Gamma_2) \right) 
&= r_m(\Gamma_1)r_s(\Gamma_2)\dim_q^\QQ A_{\listk{k}_1} q^{-s(\Gamma_1)} q^{-s(\Gamma_2)}\\
&= r_m(\Gamma_1)r_m(\Gamma_2)\dim_q^\QQ A_{\listk{k}_2} q^{-s(\Gamma_1)-s(\Gamma_2)} \\
&= r_m(\Gamma_1 \circ_{\listk{k}_1} \Gamma_2)\dim_q^\QQ A_{\listk{k}_2}  q^{-s(\Gamma_1)-s(\Gamma_2)}\\
&= \dim_q^\QQ  \left(\IED(\Gamma_1 \circ_{\listk{k}_1} \Gamma_2) \right).
\end{align*}
\end{proof}

\begin{rmk}
  \label{rmk:soergel-well-defined-2}
  Note that from this lemma, we can re-obtain the fact that the bimodule $\BS(\Gamma)$ is well-defined. See Remark~\ref{rmk:soergel-well-defined}.
\end{rmk}

\begin{dfn}
  \label{dfn:2-functor-foam-soergel}
  For every $k$, we denote by $\BimS_k$ the 2-category of singular Soergel bimodules of level $k$. More precisely:
  \begin{enumerate}
  \item The objects of $\BS_k$ are finite sequences of positive integers which sum up to $k$.
  \item The category of 1-morphisms from $\listk{k}_0$ to $\listk{k}_1$ is the smallest abelian full sub-category of $A_{\listk{k}_1}$-module-$A_{\listk{k}_0}$ containing the $A_{\listk{k}_1}$-module-$A_{\listk{k}_0}$ $\BS(\Gamma)$ for any braid-like MOY graph $\Gamma$. Note that thanks to Corollary~\ref{cor:bimodule-strucure-on-IE-Gamma}, all objects of this category are projective (and therefore free) as $A_{\listk{k}_1}$-modules and as modules-$A_{\listk{k}_0}$, and are finitely generated for both of these structures.
  \end{enumerate}
\end{dfn}

From Lemmas~\ref{lem:id-dl-foam-2-algebra} and \ref{lem:concat-2-tensor} and Corollaries~\ref{cor:bimodule-strucure-on-IE-Gamma} and \ref{cor:elementary-graph-2-soergel}, we deduce the following proposition:

\begin{prop}
  \label{prop:Soergel-foam}
  We have a 2-functor 
\[
\begin{array}{crcl}
\IED: \thinspace & \DLF_k &\to &\BimS_k   \\
  & \listk{k} &\mapsto& \listk{k}  \\
  & \Gamma &\mapsto& \IED(\Gamma)  \\
  & F &\mapsto &\IED(F). 
\end{array}
\]
which factorizes through the 2-category $\widehat{\DLF_k}$.
\end{prop}

Actually, based on evidence given by Sto{\v{s}}i{\'c} \cite{2008arXiv0810.3578S}, we conjecture the following:

\begin{cjc}
  \label{cjc:equivalence-of-2-category}
  The 2-functor $\IED$ induces an equivalence of $2$-categories between  $\widehat{\DLF_k}$ and $\BimS_k$.
\end{cjc}

\subsection{Hochschild homology}
\label{sec:hochschild-homology}

If $A$ is an algebra and $M$ an $A$-module-$A$, the Hochschild homology of $A$ with coefficients in $M$ is denoted by $\HH_\bullet(A,M)$. 

\begin{lem}
  \label{lem:hochschild-trace}
  Let $\Gamma$ be a braid-like $\listk{k}_1$-MOY graph-$\listk{k}_0$ and $\Gamma'$ be a braid-like $\listk{k}_0$-MOY graph-$\listk{k}_1$, then  $\HH_\bullet(A_{\listk{k}_0},\BS(\Gamma'\circ_{\listk{k}_1}\Gamma))$ and 
$\HH_\bullet(A_{\listk{k}_1},\BS(\Gamma\circ_{\listk{k}_0}\Gamma'))$ are canonically isomorphic.
\end{lem}

\begin{proof}
  This follows from Corollary~\ref{cor:bimodule-strucure-on-IE-Gamma} and Lemma~\ref{lem:concat-2-tensor}. Let us write $M= \IED(\Gamma)$ and $M'=\IED(\Gamma')$. Let $C_\bullet(M)$ be a projective resolution of $M$ as $A_{\listk{k}_1}$-module-$A_{\listk{k}_0}$. Since $M'$ is projective as module-$A_{\listk{k}_1}$, $C_\bullet(M) \otimes_{A_{\listk{k}_0}} M'$ is a projective resolution of $M\otimes_{A_{\listk{k}_0}} M'$ as $A_{\listk{k}_1}$-module-$A_{\listk{k}_1}$. Similarly, $M' \otimes_{A_{\listk{k}_1}} C_\bullet(M)$ is a projective resolution of $M'\otimes_{A_{\listk{k}_1}} M$ as $A_{\listk{k}_0}$-module-$A_{\listk{k}_0}$.
We have a canonical isomorphism of chain complexes
\begin{align*}
  A_{\listk{k}_0} \otimes_{A_{\listk{k}_0}^{\mathrm{en}}} \left(M' \otimes_{A_{\listk{k}_1}} C_\bullet(M)\right)
\simeq
  A_{\listk{k}_1} \otimes_{A_{\listk{k}_1}^{\mathrm{en}}} \left(C_\bullet(M) \otimes_{A_{\listk{k}_0}} M'\right)
\end{align*}
The result follows because $\HH_\bullet(A_{\listk{k}_0},\BS(\Gamma'\circ_{\listk{k}_1}\Gamma))$ and 
$\HH_\bullet(A_{\listk{k}_1},\BS(\Gamma\circ_{\listk{k}_0}\Gamma'))$
are the homology groups of these two chain complexes.
\end{proof}

\begin{prop}
  \label{prop:HH0-2-vinyl}
  Let $\Gamma$ be a braid-like $\listk{k}$-MOY graph-$\listk{k}$ and denote by $\widehat{\Gamma}$ its closure and by $\IET(\widehat{\Gamma})$ the space of vinyl $\widehat{\Gamma}$-foams modulo $\infty$-equivalence. The space $\HH_0(A_{\listk{k}},\IED(\Gamma))$ is canonically isomorphic to $\IET(\widehat{\Gamma})$.
\end{prop}

\begin{proof}
  Closing up rooted $\Gamma$-foams into vinyl $\widehat{\Gamma}$-foams provides a well-defined map $\pipi: \IED(\Gamma) \to \IET(\widehat{\Gamma})$. The action of $A_{\listk{k}}$ can be seen as concatenating foams (see proof of Corollary~\ref{cor:elementary-graph-2-soergel}). Hence, $\pipi$  factories through $\IED(\Gamma)/ [A_{\listk{k}}, \IED(\Gamma)]$ which is equal to $\HH_0(A_{\listk{k}},\IED(\Gamma))$. We now denote by $\pipi$ the induced map from $\HH_0(A_{\listk{k}},\IED(\Gamma))$ to $\IET(\widehat{\Gamma})$. This map is surjective thanks to Lemma~\ref{lem:treelike-are-enough}. Instead of proving that $\pipi$ is injective, we prove that spaces $\HH_0(A_{\listk{k}},\IED(\Gamma))$ and $\IET(\widehat{\Gamma})$  have the same graded dimension (and each of their homogeneous parts is finite-dimensional). Thanks to Lemma~\ref{lem:hochschild-trace}, $\HH_0(A_{\listk{k}},\IED(\Gamma))$ only depends on $\widehat{\Gamma}$. 
Hochschild homology is compatible with direct sum of bimodules in the sense that: \[\HH_\bullet(A, M\oplus N) \simeq \HH_\bullet(A, M)\oplus \HH_\bullet(A,  N).\] Hence, thanks to Proposition~\ref{prop:QR2}, it is enough to prove the statement for $\widehat{\Gamma}$ a collection of circles. If $\widehat{\Gamma}$ is a collection of circles labeled by $\listk{k}:= (k_1, \dots, k_l)$, then we have $\IED(\Gamma) = A_{\listk{k}}$. Lemma~\ref{lem:kozsul_resolution} implies that the space $\HH_0(A_{\listk{k}},\IED(\Gamma))$ is isomorphic to $A_{\listk{k}}$. On the other hand, $\IET(\widehat{\Gamma})$ is isomorphic to $A_{\listk{k}}$ as well thanks to Lemma~\ref{lem:IE4circles}.
\end{proof}


\section{One quotient and two approaches}
\label{sec:one-quotient-two}

\subsection{A foamy approach}
\label{sec:foamy-approach}

\subsubsection{Evaluation of vinyl foams}
\label{sec:evaluation-vinyl}

\begin{notation}
  \label{not:YD}
  The set of Young diagrams with at most $a$ rows and at most $b$ columns is denoted by $T(a,b)$ and the set of Young diagrams with at most $a$ rows is denoted by $T(a,\infty)$. The rectangular Young diagram with $a$ rows and $b$ columns is denoted by $\rho(a,b)$.
\end{notation}

\begin{notation}
  \label{not:ringforsym}
  Recall that $\SP{N}$ denotes the ring of symmetric polynomials with coefficients in $\QQ$. 
  \begin{enumerate}
  \item Denote 
    the graded algebra
    $\SP{N}[x_1,\dots, x_k]^{\mathfrak{S}_k }$ by $A_{k}$, by
    $J_{N,k}$ the ideal of $R_N[x_1, \dots, x_k]$ generated by
    \[
\left\{\left. \prod_{i=1}^N(x_j-T_i)\right| j=
          1, \dots, k \right\}.
      \]
      Note that elements of this set are indeed symmetric in the $T_\bullet$. 
Denote by $M_{N,k}$ the $\SP{N}$-algebra 
    \[
    A_k\left/(J_{N,k}\cap
       A_k)\right.,
    \]seen as an $\SP{N}$-module. 
    The indeterminates $x_\bullet$ have degree $2$, just like the indeterminates $T_\bullet$ appearing in the definition of $\SP{N}$ (end of Section~\ref{sec:quasi-annular-foams}).
    \item If $\lambda=(\lambda_1, \dots, \lambda_k)$ is a Young diagram with at most $k$ rows, define $\mathbf{x}^{\lambda}:=\prod_{i=1}^k x_i^{\lambda_i}$. Denote by $m_{\lambda}(x_1, \dots, x_k)$ the symmetric polynomial $\sum_{\lambda'}\mathbf{x}^{\lambda'}$, where $\lambda'$ runs over all \emph{distinct} permutations of $\lambda$. Denote by $\tilde{m}_{\lambda}(x_1, \dots, x_k)$ the symmetric polynomial $\sum_{\lambda'}\mathbf{x}^{\lambda'}$, where $\lambda'$ runs over \emph{all} permutations of $\lambda$. The family $(m_\lambda)_{\lambda\in T(k, \infty)}$ is a $\ZZ$-basis of the ring of symmetric polynomials in $k$ variables with coefficients in $\ZZ$ (see~\cite{MR3443860}). The family $(\tilde{m}_\lambda)_{\lambda\in T(k, \infty)}$ is a $\QQ$-basis of the ring of symmetric polynomials in $k$ variables with coefficients in $\QQ$.   
  \end{enumerate}
\end{notation}


\begin{lem}
  \label{lem:MNkisfree} 
  The $\SP{N}$-module $M_{N,k}$ is free and has a basis given by images in $M_{N,k}$ of
  \[(m_{\lambda}(x_1, \dots, x_k))_{\lambda \in T(k, N-1)} \]
  seen as element of $A_k$.
\end{lem}

\begin{proof}
  The $\SP{N}$-module $M_{N,k}$ is isomorphic to
  \[
    \left(\SP{N}[x_1,\dots, x_k]/J_{N,k}\right)^{\mathfrak{S}_k}.
  \]
  Indeed, the $\SP{N}$-linear maps
  \[
    A_k \hookrightarrow \SP{N}[x_1,\dots, x_k]   \twoheadrightarrow \SP{N}[x_1,\dots, x_k]/J_{N,k}  \twoheadrightarrow  \left(\SP{N}[x_1,\dots, x_k]/J_{N,k}\right)^{\mathfrak{S}_k}
  \]
is surjective because if $P+J_{N,k} \in \left(\SP{N}[x_1,\dots, x_k]/J_{N,k}\right)^{\mathfrak{S}_k}$, one can assume that $P$ is itself $\mathfrak{S}_k$-invariant since the ideal $J_{N,k}$ is $\mathfrak{S}_k$-invariant. The kernel of this morphism is $J_{N,k}\cap A_k$, hence it induces an isomorphism of $\SP{N}$-modules between $M_{N,k}$ and   $\left(\SP{N}[x_1,\dots, x_k]/J_{N,k}\right)^{\mathfrak{S}_k}$.

The $\SP{N}$-module $\left(\SP{N}[x_1,\dots, x_k]/J_{N,k}\right)^{\mathfrak{S}_k}$ is isomorphic to
\[
\mathrm{Sym}^k\left( \SP{N}[x]\left/\left( \prod_{i=1}^N (x-T_i) \right)\right.\right).
\]

The $\SP{N}$-module $\SP{N}[x]\left/\left( \prod_{i=1}^N (x-T_i) \right.\right)$ is a free $\SP{N}$ module of rank $N$ and has a natural $\SP{N}$-basis given by $(x^i)_{0\leq i \leq N-1}$. Therefore, $(m_{\lambda}(x_1, \dots, x_k))_{\lambda \in T(k, N-1)}$ is a $\SP{N}$-basis of
$\left(\SP{N}[x_1,\dots, x_k]/J_{N,k}\right)^{\mathfrak{S}_k}$.
\end{proof}

Denote $\epsilon_{N,k}$ the 
following morphism of $\SP{N}$-modules:
\[
 \begin{array}{crcl} \epsilon_{N,k}\thinspace\colon
  & M_{N,k}  &\to     & \SP{N} \\ 
  & m_{\lambda}&\mapsto &
  \begin{cases}
    1 & \textrm{if $\lambda = \rho(k, N-1),$} \\
    0 & \textrm{if $\lambda \neq \rho(k, N-1),$} \\
  \end{cases}
\end{array}
\]

\begin{prop}
  \label{prop:Mnk-Frobenius}
  The $\SP{N}$-linear map $\epsilon_{N,k}$ endows the $\SP{N}$-algebra $M_{N,k}$ with a structure of symmetric algebra. In particular $M_{N,k}$ is a commutative Frobenius algebra.
\end{prop}
\begin{proof}
  It is enough to check that the composition of $\epsilon_{N,k}$ with the multiplication is a non-degenerate pairing. This follows from the following fact:
  \[
    \epsilon_{N,k}(\widetilde{m}_\lambda \widetilde{m}_\mu) =
    \begin{cases}
      (k!)^2 &\textrm{if $\lambda = \mu^c$,} \\
      0 &\textrm{if $|\lambda| + |\mu| \leq {k(N-1)}$ and $\lambda \neq \mu^c$.}
    \end{cases}
  \]
  Indeed, this implies that the pairing matrix in the bases $(\widetilde{m}_{\lambda})_{\lambda \in T(k, N-1)}$ and $(\widetilde{m}_{\mu^c})_{\mu \in T(k, N-1)}$ suitably ordered has the following form:
\[
  \tikz{
      \matrix[matrix of math nodes,
        left delimiter=(,
        right delimiter=),
        nodes in empty cells] (m)
        {
        (k!)^2&        &                      &        & \\
        &        (k!)^2&                      & 0 &        \\
        &        & {\ddots} &        &        \\
        & ?  &                       &    (k!)^2   &        \\
        &        &                       &        &    (k!)^2    \\
        };
        \draw[black, thin] (m-2-1.north west) |- (m-5-4.south east) --cycle;
        \draw[ black, thin] (m-1-2.north west) -| (m-4-5.south east) --cycle;
      }
    \]
    and is clearly invertible. Hence the pairing is non-degenerate.
\end{proof}



Denote $\Upsilon_{N,k}$ the  composition of the projection from $A_k$ to $M_{N,k}$ with $\epsilon_{N,k}$.

Let $F$ be a vinyl $\SS_k$-foam-$\SS_k$. In $\RR^3$, we can cap and cup $F$ with two disks labeled by $k$, to obtain a (non-vinyl) foam $\mathrm{cl}({F})$. 

\begin{dfn}
  \label{dfn:evaluationequiv}
  The \emph{equivariant symmetric evaluation} of a vinyl $\SS_k$-foams-$\SS_k$ $F$ is given by:
\[
\kups{F} := \Upsilon_{N,k}\left(\kup{\mathrm{cl}({F})}_k\right),
\]
where $\kup{\bullet}_k$ denotes the $\sll_k$-evaluation of closed foams (see Definition~\ref{dfn:exteval}).
\end{dfn}



\begin{rmk} 
  \label{rmk:explicitsymeval}
We can make the symmetric evaluation more explicit: let $F$ be a vinyl $\SS_k$-foam-$\SS_k$, then $\kups{F}$ is equal to the coefficient of $m_{\rho(k, N-1)}$ (in the basis $(m_{\lambda})$)  of $\kup{\mathrm{cl}({F})}_k$.  Alternatively it is equal to ${k!}$ times the $\tilde{m}_{\rho(k, N-1)}$-coefficient in the basis $(\tilde{m}_{\lambda})$ of $\kup{\mathrm{cl}({F})}_k$.
\end{rmk}

The basis $(\tilde{m}_{\lambda})$ is convenient because of the following lemma which is the key ingredient for the proof of monoidality of our construction.

\begin{lem}
  \label{lem:monomialwellbehave}
  Let $k_1$ and $k_2$ be two non-negative integers, $\lambda$ (resp. $\mu$) be a Young diagram with at most $k_1$ (resp $k_2$) rows and $A$ be a set of $k_1 +k_2$ variables. We have
\[
\sum_{\substack{A_1\sqcup A_2 = A \\ \#A_1=k_1 \\ \#A_2=k_2}} \tilde{m}_\lambda(A_1)\tilde{m}_{\mu}(A_2) =  \tilde{m}_{\lambda\cdot\mu}(A).
\]
where $\lambda\cdot\mu$ is Young diagram corresponding to the union of the partition of $\lambda$ and $\mu$. 
In particular for any integer $N$, $\lambda\cdot\mu=\rho(k_1+k_2, N)$ if and only if $\lambda=\rho(k_1,N)$ and $\mu=\rho(k_2, N)$. \hfill$\qed$
\end{lem}


The following lemma proves that this evaluation does not really depend on $k$.

\begin{lem}
  \label{lem:Nbiggerk}
  Let $F$ be a vinyl $\SS_k$-foam-$\SS_k$, then it is $\infty$-equivalent to the foam $\SS_k\times [0,1]$ decorated with $\kup{(-1)^{k(k+1)/2}\mathrm{cl}({F})}_k$. 
The decoration makes sense since $\kup{(-1)^{k(k+1)/2}\mathrm{cl}({F})}_k$ is a symmetric polynomial in $k$ variables.
\end{lem}

\begin{proof} Let us denote by $T$ the foam 
\[
\NB{\tikz[xscale =0.8, yscale=0.6]{\begin{scope}
  \draw[->-] (-2,0) arc (180:360: 2cm and 1cm) node[midway, above] {$k$};
  \draw[dashed] (-2,0) arc (180:0: 2cm and 1cm);
  \draw (-2,3) arc (180:360: 2cm and 1cm);
  \draw[-<-] (-2,3) arc (180:0: 2cm and 1cm)  node[midway, below] {$k$};
  \node[red, very thick] at (0, 1.5) {$(-1)^{k(k+1)/2}\kup{\mathrm{cl}({F})}_k$};
  \draw (2,0) -- +(0,3);
  \draw (-2,0) -- +(0,3);
\end{scope}}}.
\]
We need to show that $T$ and $F$ are $N$-equivalent for all $N$. If $N<k$, this is clear, since all $\SS_k$-foams-$\SS_k$ are $N$-equivalent to $0$. If $N=k$ this is clear as well since the identity~(\ref{eq:neckcuttingcat}) has a very simple form in this case: we can apply it on the top and on the bottom of $F$ and $T$, the result follows immediately.  Let us now pick an integer $N$ greater than $k$.
  Thanks to the definition of $\F_N$ (and its monoidality, see Corollary~\ref{cor:catofMOYcalulus}), the $N$-equivalence is equivalent to saying that for any $(G_1,G_2)$ in $\Hom_\Foam(\SS_k,\emptyset)\times \Hom_\Foam(\emptyset,\SS_k)$, we have:
\[
\kup{G_2\circ F\circ G_1}_N = \kup{G_2\circ T \circ G_1}_N
\]
If $N=k$, the result holds by hypothesis. Since $F$ is vinyl, for any $\sll_N$-coloring $c$ of $G_2\circ F \circ G_1$, the induced coloring on the two circles which form the boundary of $F$ are the same, hence it induces a coloring $c'$ on $G_2\circ T \circ G_1$. Since the difference of the Euler characteristics of monochrome and bichrome surfaces  as well as the parity of the difference of numbers of positive circles can be computed locally, the quantity
\[
\frac{\kup{G_2\circ F \circ G_1, c}_N}{\kup{G_2\circ T \circ G_1,c'}_N}
\]
only depends on the restrictions of $c$ and $c'$ to $F$ and $T$. Since $F$ and $T$ are vinyl, there are exactly $k$ pigments appearing in the restrictions of $c$ and $c'$ to $F$ and $T$. Hence if these pigments are $1, \dots, k$ and if we sum over all colorings $c$ of $G_2\circ F\circ G_1$,  which induce $c'$ on $G_2\circ T \circ G_1$ we obtain (thanks to the case $k=N$):
\[
\sum_{c \textrm{ induces } c'}\frac{\kup{G_2\circ F \circ G_1, c}_N}{\kup{G_2\circ T \circ G_1,c'}_N} =1
\]
Since permuting the pigments boils down to permuting the variables $x_1, \dots, x_N$ (see \cite[Lemma 2.16]{RW1}), we obtain:
\[
\sum_{\substack{c \textrm{ coloring of }\\ G_2\circ F\circ G_1}}{\kup{G_2\circ F \circ G_1, c}_N} = \sum_{\substack{c' \textrm{ coloring of } \\ G_2\circ T\circ G_1}}{\kup{G_2\circ T \circ G_1,c'}_N}.
\]
\end{proof}

\subsubsection{Universal construction}
\label{sec:univ-constr-1}

We will use the evaluation defined in Definition~\ref{dfn:evaluationequiv} and a universal construction à la \cite{MR1362791} in order to define a functor $\syf_{k,N}: \TL_k\to \mathsf{C}$, where $\mathsf{C}$ is the category of $\ZZ$-graded finitely generated projective $\SP{N}$-modules.

If $\Gamma$ is a vinyl graph of level $k$, we consider the free graded $\SP{N}$-module generated by $\Hom_{\TL_k}(\SS_k, \Gamma)q^{-k(N-1)}$. We mod this space out by 
\[
\bigcap_{G\in \Hom_{\TL_k}(\Gamma, \SS_k)}\Ker\left(
  \begin{array}{rcl}
    \Hom_{\TL_k}(\SS_k, \Gamma)& \to &  \SP{N} \\
    F&\mapsto & \kups{G\circ F}
  \end{array}
\right).
\]
We define $\syf_{k,N}(\Gamma)$ to be this quotient. The definition of $\syf_{k,N}$ on morphisms follows naturally. 



\subsubsection{Categorified identities}
\label{sec:categ-relat}

\begin{prop}
  \label{prop:catcircle}
  The $\SP{N}$-module $\syf_{k,N}(\SS_k)$ is isomorphic to $M_{N,k}q^{-k(N-1)}$ (see Notation~\ref{not:ringforsym}).
In particular, we have 
\[
\rk^{\SP{N}}_q(\syf_{k,N}(\SS_k)) =  
\begin{bmatrix}
  k+N-1 \\ k
\end{bmatrix}.
\]
\end{prop}

\begin{proof}
  
  Define  $\phi:A_k \to \syf_{k,N}(\SS_k)$ the $\SP{N}$-linear map which maps any symmetric polynomial in $k$ variables $P$ to the cylinder $\SS_k\times [0,1]$ decorated by $P$. 
  Thanks to Lemma~\ref{lem:Nbiggerk}, this map is surjective. By the very definition of the equivariant symmetric evaluation $J_{N,k}\cap A_k$ is in the kernel of this map. Hence it induces a $\SP{N}$-linear map $M_{N,k} \to \syf_{k,N}(\SS_k)$ denoted by $\phi'$.

  The map $\phi'$ is injective. Indeed, let $x$ be a non-zero element of $M_{N,k}$. Since $M_{N,k}$ is a Frobenius algebra there exists $y$ in $M_{N,k}$ such that $\epsilon_{N,k}(xy)\neq 0$. Let $X$ and $Y$ be two $\SP{N}$-linear combinations of vinyl $\SS_k$-foams-$\SS_k$ representing $\phi'(x)$ and $\phi'(y)$ in $\syf_{k,N}(\SS_k)$. By definition, $\kups{X \circ Y} = \epsilon(xy) \neq 0$. Hence $\phi'(xy) = \phi'(x) \circ \phi'(y)\neq 0$ and $\phi'(x)\neq 0$.
  It follows that $\phi'$ is an isomorphism.
  Note that it is homogeneous of degree $-(N-1)k$. 
\end{proof}


From the definition of the evaluation of vinyl foams, we immediately deduce that the identities of Section~\ref{sec:rema-exter-foams} which can be expressed by vinyl foams are still valid. This gives the following proposition.

\begin{prop}
  \label{prop:isofromext}
  Let $k$ be a non-negative integer and $N$ be a positive integer, then the functor $\syf_{k,N}$ satisfies the following local relations:
  \begin{align}
    \syf_{k, N}\left(\stgamma\right) &\simeq \syf_{k, N}\left(\stgammaprime\right) \label{eq:catsymass} \\
    \syf_{k, N}\left(\digona\right) &\simeq \syf_{k, N}\left(\verta\right) \begin{bmatrix}
    m+n \\ m
  \end{bmatrix}
 \label{eq:catsymdig}\\
   \syf_{k, N}\left(\squarec\right)&\simeq  \bigoplus_{j=\max{(0, m-n)}}^m
  \syf_{k, N}\left(\squared\right) \begin{bmatrix}l \\ k-j \end{bmatrix} \label{eq:catsymsquare} \\
        \syf_{k, N}\left(\squaree\right) &\simeq 
        \syf_{k, N}\left(\webHe\right) \qbina{r+s}{s}
  \end{align} \hfill $\qed$
\end{prop}

\subsubsection{Monoidality}
\label{sec:monoidality}


The category $\TL_k$ does not have a notion of disjoint union, hence
for a fixed $k$ we cannot have monoidality of the functor
$\syf_{k,N}$. However, if we consider  the disjoint union $\TL_{\NN}$
of the categories $\TL_k$ for $k$ in $\NN$, then we can speak about
disjoint union, and this obviously endows $\TL_\NN$ with a structure of a
monoidal category. The empty vinyl graph seen as an object of $\TL_0$ is the
monoidal unit. Note that in this category $\Gamma_1\sqcup \Gamma_2$ is
in general not isomorphic to $\Gamma_2\sqcup \Gamma_1$.

We consider the functor $\syf_{\NN,N}$: it is given by the functor
$\syf_{k,N}$ on $\TL_k$ for all $k\in \NN$. In order to fix notations,
set
\[
\du_{k_1,k_2} : \TL_{k_1} \times \TL_{k_2} \to \TL_{k_1+ k_2}
\]
which sends pairs of objects (resp. morphisms) onto their (rescaled)
disjoint union. This is illustrated on object below:

\[\NB{\tikz{\begin{scope}
  \filldraw[draw= black, fill = yellow] (0,0) circle (1cm);
  \node at (0.6,0) {$\Gamma_1$};
  \filldraw[draw= black, fill = white] (0,0) circle (0.2cm);
\end{scope}
\node at (1.5, 0) {$,$};
\begin{scope}[xshift=3cm]
  \filldraw[draw= black, fill = orange] (0,0) circle (1cm);
  \node at (0.6,0) {$\Gamma_2$};
  \filldraw[draw= black, fill = white] (0,0) circle (0.2cm);
\end{scope}
\node at (4.75, 0) {$\mapsto$};
\begin{scope}[xshift=6.5cm]
  \filldraw[draw= black, fill = orange] (0,0) circle (1cm);
  \filldraw[draw = black, dotted, fill = yellow] (0,0) circle (0.6cm);
  \node at (0.4,0) {\scalebox{0.5}[1]{$\Gamma_1$}};
  \node at (0.8,0) {\scalebox{0.5}[1]{$\Gamma_2$}};
  \filldraw[draw= black, fill = white] (0,0) circle (0.2cm);
\end{scope}}}. \]

\begin{prop}
  \label{prop:mono}
  Let $\Gamma_1$ (resp. $\Gamma_2$) be a vinyl graph of level $k_1$ (resp. $k_2$). Suppose that $\syf_{k_1, N}(\Gamma_1)$ and $\syf_{k_2, N}(\Gamma_2)$ are free $\SP{N}$-modules, then $\syf_{k_1+k_2, N}(\du_{k_1, k_2}(\Gamma_1, \Gamma_2))$ is isomorphic to $\syf_{k_1, N}(\Gamma_1) \otimes_{\SP{N}} \syf_{k_2, N}(\Gamma_2)$. In particular, $\syf_{k_1+k_2, N}(\du_{k_1, k_2}(\Gamma_1, \Gamma_2))$ is a free $\SP{N}$-module.
\end{prop}

\begin{proof}
  Let us fix two vinyl graphs $\Gamma_1$ and $\Gamma_2$ of level $k_1$
  and $k_2$. We denote $\du_{k_1, k_2}(\Gamma_1, \Gamma_2)$ by $\Gamma$ and
  $k_1+k_2$ by $k$.  We will define a map $\phi_{\Gamma_1,
  \Gamma_2}$ from $\syf_{k_1, N}(\Gamma_1) \otimes_{\SP{N}} \syf_{k_2, N}(\Gamma_2)$ to 
  $\syf_{k_1+k_2, N}(\du_{k_1, k_2}(\Gamma_1, \Gamma_2))$.
It is enough to define $\phi_{\Gamma_1,
  \Gamma_2}$ on pure tensors. Let $v_1$ (resp. $v_2$) be an element of
  $\syf_{k_1,N}(\Gamma_1)$ (resp. $\syf_{k_2,N}(\Gamma_2)$). We can suppose
  that $v_1$ is represented by a $\Gamma_1$-foam-$\SS_{k_1}$
  $F_1$ and $v_2$ by a $\Gamma_2$-foam-$\SS_{k_2}$ $F_2$. We define
  $\phi_{\Gamma_1, \Gamma_2}(v_1\otimes v_2)$ to be the element of
  $\syf_{k,N}\left(\Gamma \right)$ obtained by re-scaling $F_1$ and
  $F_2$, taking their disjoint union (this gives an element of
  $\Hom_{\TL_{k}}\left(\Gamma, \du_{k_1, k_2}\left(\SS_{k_1}, \SS_{k_2}\right)\right)$ and pre-composing it with
  the foam
\[
Y_{k}^{k_1, k_2} := \NB{\tikz{
\draw[-<] (0,0) -- + (0,-0.7) node [below] {$k_1+k_2$}; 
\draw[->] (0,0) .. controls  (0,0.2) and (-0.5,0.5)  .. (-0.5,0.7) node [above] {$k_1$};  
\draw[->] (0,0) .. controls (0,0.2) and (0.5,0.5) .. (+0.5,0.7) node [above] {$k_2$}; 
}} \times \SS^1.
\]

We extend this definition linearly.
We now need to show that:
\begin{enumerate}
  \item\label{it:wd} this is well-defined,
 \item\label{it:iso} this is an isomorphism,
\end{enumerate}

{(\ref{it:wd})} In order to prove that $\phi_{\Gamma_1, \Gamma_2}$ is
well-defined, we only need to show that if for all vinyl
$\Gamma_2$-foam-$\SS_{k_2}$ $F_2$ (resp. $\Gamma_1$-foam-$\SS_{k_1}$
$F_1$) and $\SP{N}$-linear combination of vinyl
$\Gamma_1$-foams-$\SS_{k_1}$ $\sum_ia_iF_1^i$
(resp. $\Gamma_2$-foams-$\SS_{k_2}$ $\sum_j b_jF_2^j$) representing
$0$ in $\syf_{k_1,N}(\Gamma_1)$ (resp. in
$\syf_{k_2,N}(\Gamma_2)$), we have $\phi_{\Gamma_1,
\Gamma_2}\left(\sum_ia_i[F_1^i]\otimes [F_2]\right) =0$
(resp. $\phi_{\Gamma_1,
\Gamma_2}\left(\sum_ia_i\left[F_1^i\right]\otimes [F_2]\right)=0$
). By symmetry we only prove \[\phi_{\Gamma_1,
\Gamma_2}\left(\sum_ia_i\left[F_1^i\right]\otimes [F_2]\right) =0.\]
The square brackets stand for ``the element of the appropriate graded
$\SP{N}$-module represented by this foam''.

Suppose that the element $\sum_ia_i[F_1^i]$ is equal to $0$. This means that for
any vinyl $\SS_{k_1}$-foam-$\Gamma_1$ $G_1$, $\sum_{i}
a_i\kups{G_1\circ F_1^i}=0$. In other words, the coefficient of
$\tilde{m}_{\rho(k_1, N-1)}$ (in the base $(\tilde{m}_\lambda)_{\lambda \in T(k_1, N-1)}$) of
$\sum_{i} a_i\kup{\mathrm{cl}(G_1\circ F_1^i)}_k$ is equal to $0$.

We want to prove that for any vinyl $\SS_{k}$-foam-$\Gamma$ $G$,
\[\sum_{i} a_i\kups{G\circ (F_1^i\sqcup F_2) \circ Y_{k}^{k_1, k_2}}
=0.\] Thanks to Lemmas~\ref{lem:vinyl2tree-like} and
\ref{lem:tree-likeRequiv}, we might suppose that $G$ is tree-like and
that it can be obtained by re-scaling the disjoint union of a
tree-like $\SS_{k_1}$-foam-$\Gamma_1$ $G_1$ and a
$\SS_{k_2}$-foam-$\Gamma_2$ $G_2$ composed with $\rY_{k_1, k_2}^k$
(which is the foam $Y^{k_1,k_2}_k$ turned upside down). Thanks to
Lemma~\ref{lem:Nbiggerk}, for all $i$, $G_1\circ F_1^i$ is
$\infty$-equivalent to $\SS_{k_1}\times [0,1]$ decorated with
$\kup{\mathrm{cl}(G_1\circ F_1^i)}_{k_1}$ and $G_2 \circ F_2$ is
$\infty$-equivalent to $\SS_{k_2}\times [0,1]$ decorated with
$\kup{\mathrm{cl}(G_2\circ F_2)}_{k_2}$. Hence $\sum_{i} a_i G\circ
(F_1^i\sqcup F_2) \circ Y_{k}^{k_1, k_2}$ is $\infty$-equivalent to
\[
\NB{\tikz{
\draw[>-] (0,0) node[below,  font=\tiny] {$k_1+k_2$} -- (0,0.3);
\draw[->] (0,0.3)  .. controls +(0,+0.2) and +(0, -0.2) .. (-0.3, 1) node[midway, left,  font=\tiny] {$k_1$};
\draw (-0.3,1)  .. controls +(0,0.2) and +(0, -0.2) .. (0, 1.7) coordinate[midway] (L);
\draw[->] (0,0.3)  .. controls +(0,+0.2) and +(0, -0.2) .. (0.3, 1) node[midway, right,  font=\tiny] {$k_2$};
\draw (0.3,1)  .. controls +(0,0.2) and +(0, -0.2) .. (0, 1.7) coordinate[midway] (R);
\draw[red, thick, <-] (L) -- +(-0.5,0.1) node[left, font=\tiny] {$\sum_i a_i (-1)^{k_1(k_1+1)/2}\kup{\mathrm{cl}(G_1\circ F_1^i)}_{k_1}$};
\draw[red, thick, <-] (R) -- +( 0.5,0.1) node[right, font=\tiny] {$(-1)^{k_2(k_2+1)/2}\kup{\mathrm{cl}(G_1\circ F_2)}_{k_2}$};
\draw[->] (0,1.7) -- (0,2) node[above,  font=\tiny] {$k_1+k_2$};}}
\times \SS^1
\]

Let us write 
\begin{align*}
  &\sum_i a_i(-1)^{k_1(k_1+1)/2}\kup{\mathrm{cl}(G_1\circ F_1^i)}_{k_1}
  = \sum_{\lambda \in T(k_1, \infty)} c_\lambda \tilde{m}_{\lambda} \quad \textrm{and} \\
  &(-1)^{k_2(k_2+1)/2}\kup{\mathrm{cl}(G_2\circ F_2^i)}_{k_2} =
  \sum_{\mu \in T(k_2, \infty)} d_\mu \tilde{m}_{\mu}
\end{align*}
with $c_\bullet$ and $d_\bullet$ in $\SP{N}$. Thanks to
Lemma~\ref{lem:monomialwellbehave}, we have:
\[
(-1)^{k_1(k_1+1)/2 + k_2(k_2+1)/2} \sum_{i} a_i\kup{\mathrm{cl}(G\circ (F_1^i\sqcup F_2) \circ Y_{k}^{k_1, k_2})}_k = 
\sum_{\substack{\lambda \in T(k_1, \infty) \\ \mu \in T(k_2, \infty)}} c_{\lambda}d_{ \mu} \tilde{m}_{\lambda\cdot \mu}
\]
By hypothesis we know that $c_{\rho(k_1, N-1)}=0$. Thanks to
Lemma~\ref{lem:monomialwellbehave}, if $\lambda$ is different from
$\rho(k_1, N-1)$, then for all $\mu$ in $T(k_2, \infty)$,
$\lambda\cdot \mu$ is different from $\rho(k, N-1)$. This proves that
$\sum_{i} a_i\kups{G\circ (F_1^i\sqcup F_2) \circ Y_{k}^{k_1,
k_2}}=0$.

(\ref{it:iso}) Let $F_1$ (resp. $F_2$) be an trivially decorated tree-like
$\Gamma_1$-foam-$\SS_{k_1}$ (resp. $\Gamma_2$-foam-$\SS_{k_2}$). The
foam $F$ obtained by pre-composing $F_1\sqcup F_2$ with the foam $Y_{k}^{k_1, k_2}$
described in the construction of $\phi_{\Gamma_1, \Gamma_2}$ is in the
image of $\phi_{\Gamma_1, \Gamma_2}$ as well as its decorated version
with non-trivial decoration on leaves.  Thanks to
Lemma~\ref{lem:vinyl2tree-like} and \ref{lem:tree-likeRequiv} the
$\SP{N}$-vector space $\syf_{k,N}\left( \Gamma \right)$ is spanned by
elements represented by foams of this type. This proves the surjectivity.

To prove injectivity, we first pick bases $(B^1_i)_{i\in I}$ and
$(B^2_j)_{j\in J}$ of $\syf_{k_1,N}(\Gamma_1)$ and
$\syf_{k_2,N}(\Gamma_2)$ and their dual bases $(B^{*1}_i)_{i\in I}$
and $(B^{*2}_j)_{j\in J}$. All these elements can be represented by
($\SP{N}$-linear combinations of) tree-like foams. 
Denote these ($\SP{N}$-linear combinations of) tree-like foams by $F_i^1$,
$F_j^2$, $F_i^{*1}$ and $F_j^{*2}$. Suppose that 
\[
\left[
\sum_{\substack{{i\in I} \\
{j\in J} }} a_{ij} (F_i^1\sqcup F_j^2)\circ Y_k^{k_1, k_2} \right] =0.
\]
Let us fix an $i_0$ in $I$ and a  $j_0$ in $J$. The hypothesis implies that
\[
\sum_{\substack{{i\in I} \\{j\in J} }}
a_{ij}\kups{ \rY_{k_1, k_2}^k \circ (B_{i_0}^{*1} \sqcup B_{j_0}^{*2})\circ (F_i^1\sqcup F_j^2)\circ Y_k^{k_1, k_2}
}=0
\]
In the previous expression, $\kups{\bullet}$ has been $\SP{N}$-linearly extended to makes sense on $\SP{N}$-linear combinations of foams.
Thanks to Lemma~\ref{lem:monomialwellbehave}, the $\tilde{m}_{\rho(k, N-1)}$-coefficient in the base $(\tilde{m}_{\lambda})_{\lambda \in T(k, \infty)}$ of
\[
\sum_{\substack{{i\in I} \\{j\in J} }}
a_{ij}\kup{ \mathrm{cl}(\rY_{k_1, k_2}^k \circ (B_{i_0}^{*1} \sqcup B_{j_0}^{*2})\circ (F_i^1\sqcup F_j^2)\circ Y_k^{k_1, k_2})
}_k
\]
is equal to $a_{i_0j_0}$. However, this coefficient should be equal to 0. This prove injectivity. \end{proof}
\begin{thm}
  \label{thm:monoidality}
  The functor $\syf_{\NN,N}$ is monoidal.
\end{thm}

\begin{proof}
  First of all, since the only vinyl foam of level $0$ is the
  empty set, it is clear that $\syf_{\NN,N}(\mathbf{1}_{\TL_\NN}) =
  \SP{N} = \mathbf{1}_{\mathsf{C}}$ ($\mathbf{1}$ denotes the unital
  object of monoidal categories).  We need to construct a natural
  isomorphism $\left(\phi_{\Gamma_1, \Gamma_2} \right)_{\Gamma_1,
  \Gamma_2 \in \mathrm{ob(\TL_\NN)}}$ from $\syf_{\NN,N}(\bullet)
  \otimes_{\SP{N}} \syf_{\NN,N}(\bullet )$ to $\syf_{\NN,N}\left(\du(\bullet,
    \bullet )\right)$.
  
  This isomorphism is provided by (the proof of) Proposition~\ref{prop:mono}. In order to use it, we only need to show that for any vinyl graph $\Gamma$, $\syf_{\NN,N}(\mathbf{1}_{\TL_\NN})$ is a free $\SP{N}$-module.

  If $\Gamma$ is a collection of circles we can argue by induction on the number of circles. If $\Gamma$ consists of only one circle, we can use Proposition~\ref{prop:catcircle}. If it consists of more than one circle, we can use the induction hypothesis and Proposition~\ref{prop:mono}. We deduce the general case from the case of collection of circles thanks to Propostion~\ref{prop:QR2} and Proposition~\ref{prop:isofromext}.
\end{proof}

\begin{rmk}
  \label{rmk:ass}
  Formally we should have checked the compatibility of associators. Since we did not write down them explicitly, we cannot be very precise here. However, this compatibility trivially holds because the foams
\[
\stgamma \times \SS^1 \quad\textrm{and} \quad \stgammaprime \times \SS^1
\]
are $\infty$-equivalent, (because they are both tree-like with only trivial decorations (see Lemma~\ref{lem:tree-likeRequiv})). 

\end{rmk}

We can now prove that the functor $\syf_{\NN,N}$ categorifies the symmetric MOY calculus. 

\begin{thm}
  \label{thm:catsym}
  The functor $\syf_{\NN,N}: \TL_\NN \to \mathcal{C}$ satisfies for every vinyl graph $\Gamma$.
\[
\rk^{\SP{N}}_q\left(\syf_{\NN,N}(\Gamma)\right) = \kups{\Gamma}.
\]
\end{thm}


\begin{proof}
We have already seen in the proof of Theorem~\ref{thm:monoidality}, that the $\SP{N}$-module $\syf_{\NN,N}(\Gamma)$ is free for every vinyl graph $\Gamma$. Thanks to Theorem~\ref{thm:monoidality} and Propositions~\ref{prop:catcircle} and \ref{prop:isofromext} we obtain that the function $\rk^{\SP{N}}_q\left(\syf_{\NN,N}(\bullet)\right)$ satisfies identities~(\ref{eq:symrelcircle}), (\ref{eq:symrelass}), (\ref{eq:symrelbin1}), (\ref{eq:symrelsquare3}) and (\ref{eq:otherrelext}). We conclude by Theorem~\ref{thm:QR}.
\end{proof}

\subsection{An algebraic approach}
\label{sec:alg-approach}

\subsubsection{Hochschild and Koszul homologies}
\label{sec:kozsul-homology}

Koszul homology has been formalized in \cite{berger_lambre_solotar_2017}. If $R$ is a unital commutative ring, $A$ an $R$-algebra and $M$ a $A$-module-$A$, it associates with the pair $(A,M)$ a sequence $\KH_\bullet(A,M)$ of $R$-modules in a functorial way. If $A$ is Koszul (and this will be our case), then $\KH_\bullet(A,M) = \HH(A,M)$. We do not aim to discuss Koszul homology in details, we refer to  \cite{berger_lambre_solotar_2017} and references therein for a nice presentation. We will use Koszul homology instead of Hochschild homology because it enables to have more structure: in fact an extra differential.


\begin{notation}
  \label{not:polynomial-algebra}
  In what follows, 
$\listk{k}= (k_1,\dots ,k_l)$ is a finite sequence of positive integers of level $k$, $A_{\listk{k}}$ is the polynomial algebra $\SP{N}[x_1, \dots x_k]^{\mathfrak{S}_{\listk{k}}}$ and $A_k$ denotes $\SP{N}[x_1, \dots x_k]^{\mathfrak{S}_{k}}$. 
Note that $A_{\listk{k}}$ is a polynomial algebra over $\SP{N}$. For $i$ in $\{1,\dots, l\}$ and $j$ in $\{1, \dots, k_i\}$, we set $e_{j}^{(i)}$ to be the $j$th elementary symmetric polynomial in variables $x_{r_i+1},\dots, x_{r_i+ k_i}$, where $r_i= \sum_{t=1}^{i-1} k_t$. It is standard that we have (see for example \cite[Chapter IV, \S 6]{MR1878556}): 
\[
A_{\listk{k}{}}= \SP{N}[e_1^{(1)},\dots, e_{k_1}^{(1)}, e_1^{(2)}\dots, e_{k_2}^{(2)}, \dots, e_1^{(l)}\dots, e_{k_l}^{(l)} ]  \subseteq A_{1^k},
\]
where $A_{1^k} := A_{(1,\dots,1)} = \SP{N}[x_1, \dots, x_k]$. 
\end{notation}

\begin{dfn}
  \label{dfn:Koszul-complex}
The \emph{Koszul resolution} of $A_{\listk{k}}$ is the complex
  \[
C_\bullet(A_{\listk{k}}) := \bigotimes_{i=1}^l\bigotimes_{j=1}^{k_i} \left( \SP{N}[e_j^{(i)}]\otimes \SP{N}[e_j^{(i)}]q^{2j}   \xrightarrow{e_j^{(i)}\otimes 1 - 1 \otimes e_j^{(i)}} 
{\SP{N}[e_j^{(i)}]\otimes \SP{N}[e_j^{(i)}] }
\right).
\]
The homological degree of $C_\bullet(A_{\listk{k}})$ is called the $H$-degree.
\end{dfn}
It is convenient to think of this complex in this way: let $V_{\listk{k}}$ be the  $\SP{N}$-module generated by $(e_j^{(i)})_{\substack{i=1,\dots, l \\ j= 1, \dots, k_i}}$ (with $(e_j^{(i)})$ having degree $2j$). 
Then $C_\bullet(A_{\listk{k}})=A_{\listk{k}}\otimes \Lambda V_{\listk{k}} \otimes A_{\listk{k}}$ with the differential:
\[
\begin{array}{crcl}
d^{\listk{k}} \colon\thinspace & C_\bullet(A_{\listk{k}}) & \to & C_\bullet(A_{\listk{k}}) \\
  & a\otimes v_1 \wedge \dots \wedge v_l\otimes b &\mapsto & \displaystyle{\sum_{i=1}^l (-1)^{i+1} \left( av_i \otimes  v_1 \wedge \dots \wedge \widehat{v_i}\wedge\dots \wedge v_l\otimes b \right.} \\ &&& \qquad \quad \displaystyle{\left.- a\otimes  v_1 \wedge \dots \wedge \widehat{v_i}\wedge\dots \wedge v_l\otimes v_ib\right).}
\end{array}
\]

It is standard that it is a projective resolution of $A_{\listk{k}}$ as $A_{\listk{k}}$-module-$A_{\listk{k}}$. Hence, for any  $A_{\listk{k}}$-module-$A_{\listk{k}}$ $M$, we have:
\[
\HH_\bullet(A_{\listk{k}},M) \simeq H(C_\bullet(A_{\listk{k}})\otimes_{A_{\listk{k}}^{\mathrm{en}}} M)=: \KH(A_{\listk{k}}, M).
\]

From now on, when speaking about Hochschild homology, we mean Hochschild homology computed in this way. Of course this precision is irrelevant when we only look at the homology groups. But we will shortly introduce an extra differential $d_N$ on $C_\bullet(A_{\listk{k}})$. It will equip the Hochschild homology with a structure of chain complex. As far as we understand, the differential $d_N$ can be thought as an equivariant version of the extra differential introduced by Cautis in \cite{MR3709661}. Some proofs are postponed to Appendix~\ref{sec:kosz-reosl-polyn}.

\subsubsection{An extra differential}
\label{sec:an-extra-diff}

Let $N$ be a positive integer.

\begin{notation} 
  \label{notation:DN-and-dN}
  We denote by $D_N$ the following derivation on $A_{1^k}$:
  \[
  \begin{array}{crcl}
D^N_{1^k}   \colon\thinspace & A_{1^k} & \to & A_{1^k} \\
    & P(x_1,\dots,x_k)  &\mapsto & \sum_{i=1}^k \prod_{j=1}^N (x_i - T_j) \partial_{x_i}P(x_1, \dots, x_k).
  \end{array}
  \] 
We denote by $d^N_{\listk{k}}$ the following map $A_{\listk{k}}$-linear-$A_{\listk{k}}$ map on $C_{\bullet}(A_{\listk{k}})$:
\[
\begin{array}{crcl}
d^N_{\listk{k}}  \colon\,& C_\bullet(A_{\listk{k}}) & \to & C_\bullet(A_{\listk{k}}) \\
  & 1\otimes v_1 \wedge \dots \wedge v_l\otimes 1 &\mapsto & \sum_{i=1}^l 
(-1)^{i+1} D^N_k(v_i) \otimes  v_1 \wedge \dots \wedge \widehat{v_i}\wedge\dots \wedge v_l\otimes 1.  
\end{array}
\]
note that this is $H$-homogeneous of degree $-1$ and $q$-homogeneous of degree $2(N-1)$. When the context is clear we will drop the subscript $\listk{k}$.
\end{notation}
 
\begin{lem}
  \label{lem:d_Ndiff-commute-with-d}
  For any finite sequence of positive integers $\listk{k}$, the map $d^N_{\listk{k}}$ anti-commutes with $d^{\listk{k}}$ and is a differential on $C_{\bullet}(A_{\listk{k}})$.
\end{lem}

Since we have this extra structure, we need a refined version of Lemma~\ref{lem:hochschild-trace}. The following lemma should be compared to \cite[Lemma 6.2]{MR3709661}.
\begin{lem}
  \label{lem:hochschild-trace-dN}
  Let $\Gamma$ be a braid-like $\listk{k}_1$-MOY graph-$\listk{k}_0$ and $\Gamma'$ be a braid-like $\listk{k}_0$-MOY graph-$\listk{k}_1$, then the complexes   \[(\HH_\bullet(A_{\listk{k}_0},\IED(\Gamma'\circ_{\listk{k}_1}\Gamma)), d^N_{\listk{k}_0}) \quad \textrm{and} \quad
(\HH_\bullet(A_{\listk{k}_1},\IED(\Gamma\circ_{\listk{k}_0}\Gamma')), d^N_{\listk{k}_1}) \] are isomorphic. 
\end{lem}

This is proved in Appendix~\ref{sec:kosz-reosl-polyn} where we restrict to the special case where $\Gamma$ contains only one vertex. Note that this is actually enough to conclude in general.
We define an explicit homotopy equivalence $\varphi$ between two complexes computing the Hochschild homology. Finally we prove that $\varphi \circ d^N_{\listk{k}_0} - d^N_{\listk{k}_1}\circ\varphi$ is null-homotopic.
Note that the previous lemma shows that the complex $(\HH_\bullet(A_{\listk{k}},\IED(\Gamma)), d^N)$ only depends on the vinyl graph $\widehat{\Gamma}$. 

In order to make the differential $d_N$ of $q$-degree $0$ we shift the $\HH_i(A_{\listk{k}}, M)$ by $2i(N-1)$ in $q$-degree. Note that this adjustment does not change $\HH_0(A_{\listk{k}}, M)$. We denote this normalization by 
$\HH^N_\bullet(A_{\listk{k}}, M)$

The following proposition should be compared to \cite[Lemma 3.23]{2016arXiv160202769W}.

\begin{prop}\label{prop:all-in-deg-0}
  Let $\Gamma$ be a vinyl graph, then the homology of $(\HH^N_\bullet(A_{\listk{k}},\IED(\Gamma)), d^N)$ (denoted by $\HN_\bullet$) is concentrated in $H$-degree $0$.
\end{prop}

\begin{proof}
  Thanks to Queffelec--Rose's algorithm (see Theorem~\ref{thm:QR}), it is enough to show the statement when $\Gamma$ is a collection of circles. Suppose that $\Gamma$ is a collection of circles labeled by $\listk{k}$. The result follows from the fact 
that the polynomials $\left(D^N(e_j^{(i)}) \right)_{\substack{1\leq i \leq l \\ 1 \leq j \leq k_i}}$ are pairwise co-prime in $A_{\listk{k}}$. First note that $D^N(e_j^{(i)})$ is a  polynomial in the variables $x_{r_i+1},\dots, x_{r_i+ k_i}$ and $T_1, \dots, T_N$. It is homogeneous and has degree $N-1 +j$ which is bigger than 1. Hence for $i_1\neq {i_2}$, $1\leq j_1\leq k_{i_1}$ and  $1\leq j_2\leq i_2$, if $D^N(e_{j_1}^{(i_1)})$ and $D^N(e_{j_2}^{(i_2)})$ would have a non-trivial common divisor, they would have an homogeneous divisor in the variables $T_\bullet$. However $D^N(e_{j_1}^{(i_1)})$ is not divisible by any non-trivial homogeneous polynomial in the variables $T_\bullet$, because evaluating all theses variables to $0$ in $D^N(e_{j_1}^{(i_1)})$ does not give the 0 polynomial.

It remains to show that for a fixed $i$ and $1\leq j_1< j_2 \leq k_i$,  $D^N(e_{j_1}^{(i)})$ and $D^N(e_{j_2}^{(i)})$ are coprime. By symmetry we may assume that $i=1$ and even that $\listk{k}= (k)$. Then the result follows from Lemma~\ref{lem:DNei-coprime} below.
\end{proof}

\begin{lem}
  \label{lem:DNei-coprime}
  For $1\leq i \leq k$, let $e_i$ be the $i$th  elementary symmetric polynomial in $x_1, \dots, x_k$. The polynomials $\left(D^N(e_i) \right)_{1\leq i \leq k}$ are pairwise co-prime in $A_{k}$ ($=A_{1^k}^{\mathfrak{S}_k}$).
\end{lem}

\begin{proof}
    We work by induction on $k$. For $1\leq i \leq k$, we write $P_{i,k}:= D^N(e_i)$ with $e_i$ symmetric in $k$ variables and $P_{i,k}^0:= D^N(e_i)_{|T_1 = \dots = T_N =0}$. For $k=1$, there is nothing to show. For $k=2$, we have:
\[ P^0_{1,2} = x_1^N + x_2^N \quad \textrm{and} \quad P^0_{2,2} = x_1x_2^N + x_2x_1^N = x_1x_2(x_1^{N-1} + x_2^{N-1})\]
which are co-prime since $x_1^{N-1} + x_2^{N-1}$ and $x_1^N + x_2^N$ are co-prime as polynomial in $x_1$ with coefficients in $\QQ[x_2]$. If $Q$ is a non-trivial homogeneous element of $A_{k}$ which divides $P_{1,2}$ and $P_{2,2}$, then $Q_{{|T_1 = \dots = T_N =0}}$ is of degree 0. Since $Q$ is not equal to $0$ it has degree $0$ which proves that $P_{1,2}$ and $P_{2,2}$ are co-prime.
Suppose now that $k\geq 3$.

If $P$ is a polynomial in $A_{k}$, $P(x_k=0)$ denotes the polynomial of $A_{k-1}$ obtained by specializing the variable $x_k$ to $0$ in $P$.
From the very definition from $D^N$, we have:
\[P^0_{i,k} = \sum_{j=1}^{k} x_j^N e_{i-1}(x_1, \dots, \widehat{x_j}, \dots x_k)\]
Hence if $1\leq i \leq k-1$, $P^0_{i,k}(x_k=0) = P^0_{i, k-1}$. Let $1\leq i_1 <i_2 \leq k-1$. Suppose that a polynomial $Q$ in $A_{1^k}$ divides $P_{i_1,k}$ and $P_{i_2,k}$. The polynomial $Q^0:= Q_{|T_1=\dots=T_N=0}$ is homogeneous and $Q(x_k=0)$ divides  $P^0_{i_1,k-1}$ and $P^0_{i_2,k-1}$. By induction, we know that $Q^0(x_k=0)$ has degree $0$. This implies that $Q^0$ and therefore $Q$ has degree $0$. Hence $P_{i_1,k}$ and $P_{i_2,k}$ are co-prime. 

It remains to show that for $1\leq i \leq k-1$, $P_{i,k}$ and $P_{k,k}$ are co-prime. Let $Q$ be a polynomial of $A_{k}$ which divides $P_{i,k}$ and $P_{k,k}$. 
\[
P^0_{k,k} = x_1x_2 \cdots x_k(x_1^{N-1} + \dots + x_k^{N-1}) = e_k(x_1, \dots, x_k)(x_1^{N-1} + \dots + x_k^{N-1})
\]
The polynomial $e_k$ is prime in $A_{k}$ and does not divide $P_{i,k}$, since $P_{i,k}(x_k=0)$ is not equal to $0$.
Hence $Q^0$ divides $x_1^{N-1} + \dots + x_k^{N-1}=:p_{k,N-1}$.  Since $Q^0$ divides $p_{k,N-1}$ and $P_{i,k}^0$, it divides (in $A_{k}$)
\[
P_{i,k}- p_{k,N-1}x_ke_{i-1}(x_1,\dots,x_{k-1})
\]
Hence its $x_k$-degree  is at most equal to $1$. Since $Q^0$ is symmetric in the $x_\bullet$, if it does not have degree equal to $0$, it is a multiple of $x_1 + \dots + x_k = e_1$. But for $e_1$ to divide $p_{k, N-1}$, one must have $k=2$ and $N-1$ odd. But we supposed $k\geq 3$. Hence $Q^0$ has degree $0$. This implies that $Q$ has degree $0$ and finally that $P_{i_1,k}$ and $P_{i_2,k}$ are co-prime. 
\end{proof}

\begin{notation}\label{not:spaceTN}
  If $\widehat{\Gamma}$ is a vinyl graph of level $k$, denote
  $\FA_N(\widehat{\Gamma})$, the space
  \[\HN_0(\HH^N_\bullet(A_{\listk{k}}, \IED(\Gamma)))q^{-k(N-1)}\] for a
  braid-like $\listk{k}$-MOY graph-$\listk{k}$ $\Gamma$ whose closure
  is equal to
  $\widehat{\Gamma}$.  
  This is legitimate thanks to Proposition~\ref{prop:all-in-deg-0}.
\end{notation}


\subsection{When algebra meets foams}
\label{sec:when-algebra-meets}

The aim of this section is to compare $\FA_N(\widehat{\Gamma})$ and $\syf_{\NN, N}(\widehat{\Gamma})$, namely to prove that these spaces are isomorphic.

Let us consider a braid-like $\listk{k}$-MOY graph-$\listk{k}$ ${\Gamma}$ and denote $\widehat{\Gamma}$ the closure of ${\Gamma}$. We know, thanks to Proposition~\ref{prop:HH0-2-vinyl}, that there is a canonical isomorphism $\phi$ from $\IET(\widehat{\Gamma})$ to $\HH_0(A_{\listk{k}},\BS({\Gamma}))$. The space $\syf_{\NN,N}(\widehat{\Gamma})$ is a quotient of $\IET(\widehat{\Gamma})q^{-k(N-1)}$ while the space $\FA_N(\widehat{\Gamma})$ is a quotient of $\HH_0(A_{\listk{k}} , \BS({\Gamma}))q^{-k(N-1)}$ (thanks to Proposition~\ref{prop:all-in-deg-0}). Using the isomorphism $\phi$, we can think of $\FA_N(\widehat{\Gamma})$ and $\syf_{\NN, N}(\widehat{\Gamma})$ as being both quotient of $\HH_0(A_{\listk{k}}, \BS(\widehat{\Gamma}))q^{-k(N-1)}$. The rest of the section is devoted to prove the following proposition:

\begin{prop}
  \label{prop:alg-eq-foam}
  The spaces  $\FA_N(\widehat{\Gamma})$ and $\syf_{\NN, N}(\widehat{\Gamma})$ are isomorphic.
\end{prop}

Thanks to Queffelec--Rose algorithm's and Proposition~\ref{prop:Soergel-foam}, it is enough to prove the statement when $\widehat{\Gamma}$ is a collection of circles labeled by $\listk{k}$. Since the spaces $\FA_N(\widehat{\Gamma})$ and $\syf_{\NN, N}(\widehat{\Gamma})$ are both quotients of $\HH_0(A_{\listk{k}},\BS({\Gamma}))q^{-k(N-1)}$ which is itself isomorphic to $A_{\listk{k}}q^{-k(N-1)}$, let us write $\FA_N(\widehat{\Gamma})=A_{\listk{k}}q^{-k(N-1)}/I_1$ and $\syf_{\NN, N}(\widehat{\Gamma})= A_{\listk{k}} q^{-k(N-1)}/I_2$. With these notations, we only need to show that the space $I_1$ and $I_2$ of $A_{\listk{k}}q^{-k(N-1)}$ are equal. 

\begin{lem}
  \label{lem:desctiption-of-I1}
  The space $I_1$ is generated by the polynomials $\left(D^N(e_j^{(i)}) \right)_{\substack{1\leq i \leq l \\ 1 \leq j \leq k_i}}$. Forgetting about the action of the variables $T_\bullet$, it is a graded vector space. Its graded dimension over $\QQ$ is equal to:
\[
\dim^\QQ_q I_1 = q^{-k(N-1)} \left(1 -  \prod_{b=1}^l\prod_{i=1}^{k_b}\left(1-q^{2(i+N-1)} \right) \right) \dim^\QQ_q A_{\listk{k}}.
\] 
\end{lem}

\begin{proof}
  The first statement is obvious. The second one follows from the fact that the polynomials $\left(D^N(e_j^{(i)}) \right)_{\substack{1\leq i \leq l \\ 1 \leq j \leq k_i}}$ are pairwise co-prime. Indeed this implies that 
\begin{align*}
&\left\langle D^N(e_{j_1}^{(i_1)})  \right\rangle_{A_{\listk{k}}} \cap 
\left\langle D^N(e_{j_2}^{(i_2)})  \right\rangle_{A_{\listk{k}}} \cap 
\dots \cap
\left\langle D^N(e_{j_a}^{(i_a)})   \right\rangle_{A_{\listk{k}}} \\ &\qquad \qquad =
\left\langle D^N(e_{j_1}^{(i_1)})\cdot D^N(e_{j_2}^{(i_2)}) \cdots D^N(e_{j_a}^{(i_a)})  \right\rangle_{A_{\listk{k}}}.
\end{align*}
Since the polynomial $D^N(e_j^{(i)})$ is homogeneous of degree $2(N+j-1)$, we have:
\begin{align*}
\dim_q \left\langle D^N(e_{j_1}^{(i_1)})\cdot D^N(e_{j_2}^{(i_2)}) \,\cdots\, D^N(e_{j_a}^{(i_a)})\right\rangle_{A_{\listk{k}}} &=q^{-k(N-1)} (\dim_q^{\QQ} A_{\listk{k}} )\prod_{l=1}^{a}q^{2(j_l+ N -1)}.
\end{align*}
The space $I_1$ is the sum of all spaces $\left\langle D^N(e_j^{(i)}) \right\rangle_{A_{\listk{k}}}$. This implies:
\[
\dim^\QQ_q I_1 = q^{-k(N-1)}\left(1 -  \prod_{b=1}^l\prod_{i=1}^{k_b}\left(1-q^{2(i+N-1)} \right) \right) \dim^\QQ_q A_{\listk{k}}.
\]
\end{proof}

The proof of Proposition~\ref{prop:alg-eq-foam}, follows from next lemma.

\begin{lem}
  \label{lem:ideal-I2}
  The spaces $I_1$ and $I_2$ are equal. 
\end{lem}

\begin{proof}
  It is clear that $I_1$ is in $I_2$ because the polynomials $\left(D^N(e_j^{(i)}) \right)_{\substack{1\leq i \leq l \\ 1 \leq j \leq k_i}}$ are all in $I_2$ (this follows from the definition of the evaluation of vinyl foams). 
We know thanks to Lemma~\ref{lem:MNkisfree} that $\syf_{\NN, N}$ associates with a circle labeled $k_b$ a free $\SP{N}$-module of graded rank
\[
\prod_{i=1}^{k_b} \frac{q^{-i-(N-1)} - q^{i+N-1}}{q^{-i} - q^i} = q^{-k_b(N-1)}
\prod_{i=1}^{k_b} \frac{1 - q^{2(i+N-1)}}{1 - q^{2i}}.
\]
Thanks to the monoidality of this functor, we obtain that the graded rank of $\syf_{\NN, N}(\Gamma)$ is equal to
\[
q^{-k(N-1)}\prod_{b=1}^l\prod_{i=1}^{k_b} \frac{1 - q^{2(i+N-1)}}{1 - q^{2i}}.
\]
But $A_{\listk{k}}q^{-k(N-1)}$ has a graded rank over $\SP{N}$ equal to:
\[
q^{-k(N-1)}\prod_{b=1}^l\prod_{i=1}^{k_b} \frac{1}{1 - q^{2i}}.
\]
Hence we have:
\begin{align*}
\dim_q^\QQ I_2 &= q^{-k(N-1)}\dim_q^\QQ A_{\listk{k}} - q^{-k(N-1)}\prod_{b=1}^l\prod_{i=1}^{k_b} \left(1 - q^{2(i+N-1)}\right) \dim_q^\QQ A_{\listk{k}}\\
&=  q^{-k(N-1)}\left(1 -  \prod_{b=1}^l\prod_{i=1}^{k_b}\left(1-q^{2(i+N-1)} \right) \right) \dim^\QQ_q A_{\listk{k}} \\
&=\dim_q^\QQ I_1.
\end{align*}
\end{proof}

\section{Link homologies}
\label{sec:link-homologies}
In this section, we define the symmetric Khovanov--Rozansky homology on diagrams of braid-closures and prove that they are indeed links invariants. The definition is of a purely foamy nature. However for proving the invariance we need to use the dictionary developed in Section~\ref{sec:soergel-bimodules-1}. We derive the invariance of the symmetric homologies from the invariance of the triply graded homology \cite{MR2421131}. We use the description of this homology as Hochschild homology of complexes of Soergel bimodules due to Khovanov and Rouquier. \cite{MR2339573, 1203.5065}. We show that the extra-differential introduced in Section~\ref{sec:an-extra-diff}, is compatible with their construction. Finally, we prove that when taking the homology with respect to this extra differential, one gets the same link homology as the one obtained by applying the foamy functor of the previous sections.
This link homology categorifies the Reshetikhin--Turaev link invariant associated with $q$-symmetric powers of the standard representation of $U_q(\sll_N)$. We call it the \emph{symmetric Khovanov--Rozansky homology}.


\subsection{The chain complexes}
\label{sec:chain-complexes}



The idea of the construction is somewhat classical. We follow the normalization used in \cite{MR3447099}. 
Let $D$ be a diagram of a braid-closure of level $k$, and
$\Xing(D)$ be its set of crossings.  For $x$ in $\Xing(D)$ we define a
finite set $I_x$ by the following rules:
\begin{align*}
  \begin{array}{ccl}
    \textrm{if } x= \scriptstyle{\NB{\tikz[scale=0.4]{}}}& \textrm{and $m\leq n$}& \textrm{then $I_x=\{-m, \dots, -1, 0\}$},  \\[0.3cm]
    \textrm{if } x= \scriptstyle{\NB{\tikz[scale=0.4]{}}}& \textrm{and $m>n$}& \textrm{then $I_x=\{-n, \dots, -1, 0\}$},  \\[0.3cm]
\textrm{if } x= \scriptstyle{\NB{\tikz[scale=0.4]{}}}& \textrm{and $m\leq n$}& \textrm{then $I_x=\{0, 1, \dots, m\}$} \\[0.3cm]
    \textrm{if } x= \scriptstyle{\NB{\tikz[scale=0.4]{}}}& \textrm{and $m>n$}& \textrm{then $I_x=\{0, 1, \dots, n\}$}. 
  \end{array}
\end{align*}
In the first two cases, we say that $x$ is of type $(m,n,+)$
and in the last two cases of type $(m,n,-)$.  
If $x$ is a crossing and $i$ is an element of $I_x$ we define
\begin{align*}
(\eta_{x,i}, \kappa_{x,i}) &=
  \begin{cases}
    (m+i,-i-m)             & \textrm{if $x$ is of type $(m,n,+)$, and $m\leq n$} \\ 
    (n+i,-i-n)             & \textrm{if $x$ is of type $(m,n,+)$, and $m>n$,} \\ 
    (i-m,m-i)             & \textrm{if $x$ is of type $(m,n,-)$, and $m\leq n$} \\
    (i-n,n-i)             & \textrm{if $x$ is of type $(m,n,-)$ and $m>n$.} 
  \end{cases}                                                           
\end{align*} 

We set $I(D)$ to
be $\prod_{x\in \Xing(D)}I_x$ and call the elements of $I(D)$,
\emph{the states of $D$}. With every state $s=(s_x)_{x\in \Xing(D)}$ of $D$
we associate a vinyl graph $D_s$ of level $k$ by replacing every
crossing $x$ of $\Xing(D)$ according to the following rules:

\begin{align*}
  \NB{\tikz{\begin{scope}
\coordinate (A) at (-1,-1);
\coordinate (B) at (1,-1);
\coordinate (C) at (1,1);
\coordinate (D) at (-1,1);
\coordinate (a) at (-.5,-.5);
\coordinate (b) at (.6,-.6);
\coordinate (c) at (.6,.6);
\coordinate (d) at (-.5,.5);
\draw[->] (A) -- (a) node[at start, below] {\tiny{$n$}};
\draw[->] (c) -- (C) node[at end, above ] {\tiny{$n$}};
\draw[->] (B) -- (b) node[at start , below ] {\tiny{$m$}};
\draw[->] (d) -- (D) node[at end, above] {\tiny{$m$}};
\draw[<-] (c) -- (d) node[midway, above,sloped] {\tiny{$n-s_x-m$}};
\draw[<-] (a) -- (b) node[midway, below, sloped] {\tiny{$-s_x$}};
\draw[->] (a) -- (d) node[midway, left] {\tiny{$n-s_x$}};
\draw[->] (b) -- (c) node[midway, right] {\tiny{$m+s_x$}};
\end{scope}}}  & \qquad \textrm{if $x$ is of type $(m,n,+)$ and $m \leq n $,} \\
  \NB{\tikz{\begin{scope}[xscale=1]
\coordinate (A) at (-1,-1);
\coordinate (B) at (1,-1);
\coordinate (C) at (1,1);
\coordinate (D) at (-1,1);
\coordinate (a) at (-.6,-.6);
\coordinate (b) at (.5,-.5);
\coordinate (c) at (.5,.5);
\coordinate (d) at (-.6,.6);
\draw[->] (A) -- (a) node[at start, below] {\tiny{$n$}};
\draw[->] (c) -- (C) node[at end, above ] {\tiny{$n$}};
\draw[->] (B) -- (b) node[at start , below ] {\tiny{$m$}};
\draw[->] (d) -- (D) node[at end, above] {\tiny{$m$}};
\draw[->] (c) -- (d) node[midway, above, sloped] {\tiny{$m-s_x-n$}};
\draw[->] (a) -- (b) node[midway, below, sloped] {\tiny{$-s_x$}};
\draw[->] (a) -- (d) node[midway, left] {\tiny{$n+s_x$}};
\draw[->] (b) -- (c) node[midway, right] {\tiny{$m-s_x$}};
\end{scope}}}   & \qquad \textrm{if $x$ is of type $(m,n,+)$ and $m >n $,} \\
  \NB{\tikz{\begin{scope}
\coordinate (A) at (-1,-1);
\coordinate (B) at (1,-1);
\coordinate (C) at (1,1);
\coordinate (D) at (-1,1);
\coordinate (a) at (-.5,-.5);
\coordinate (b) at (.6,-.6);
\coordinate (c) at (.6,.6);
\coordinate (d) at (-.5,.5);
\draw[->] (A) -- (a) node[at start, below] {\tiny{$n$}};
\draw[->] (c) -- (C) node[at end, above ] {\tiny{$n$}};
\draw[->] (B) -- (b) node[at start , below ] {\tiny{$m$}};
\draw[->] (d) -- (D) node[at end, above] {\tiny{$m$}};
\draw[<-] (c) -- (d) node[midway, above,sloped] {\tiny{$n+s_x-m$}};
\draw[<-] (a) -- (b) node[midway, below,sloped] {\tiny{$s_x$}};
\draw[->] (a) -- (d) node[midway, left] {\tiny{$n+s_x$}};
\draw[->] (b) -- (c) node[midway, right] {\tiny{$m-s_x$}};
\end{scope}}}    & \qquad \textrm{if $x$ is of type $(m,n,-)$ and $m\leq n$,}  \\
  \NB{\tikz{\begin{scope}[xscale=1]
\coordinate (A) at (-1,-1);
\coordinate (B) at (1,-1);
\coordinate (C) at (1,1);
\coordinate (D) at (-1,1);
\coordinate (a) at (-.6,-.6);
\coordinate (b) at (.5,-.5);
\coordinate (c) at (.5,.5);
\coordinate (d) at (-.6,.6);
\draw[->] (A) -- (a) node[at start, below] {\tiny{$n$}};
\draw[->] (c) -- (C) node[at end, above ] {\tiny{$n$}};
\draw[->] (B) -- (b) node[at start , below ] {\tiny{$m$}};
\draw[->] (d) -- (D) node[at end, above] {\tiny{$m$}};
\draw[->] (c) -- (d) node[midway, above,sloped] {\tiny{$m+s_x-n$}};
\draw[->] (a) -- (b) node[midway, below,sloped] {\tiny{$s_x$}};
\draw[->] (a) -- (d) node[midway, left] {\tiny{$n-s_x$}};
\draw[->] (b) -- (c) node[midway, right] {\tiny{$m+s_x$}};
\end{scope}}}    & \qquad \textrm{if $x$ is of type $(m,n,-)$ and $m> n$.}  
\end{align*}

If $s$ is a state, we define 
\[\eta_s = \sum_{x\in \Xing} \eta_{x,s_x}\quad \textrm{and} \quad \kappa_s= \sum_{x\in \Xing} \kappa_{x,s_x} \] and we set $D_s$ to seat in topological degree $\eta_s$ and to be shifted in $q$-degree by $\kappa_s$.

If $s=(s_x)_{x\in \Xing(D)}$ and $s'=(s'_x)_{x\in \Xing(D})$ are two states
which are equal on all but one of their coordinate $x$, for which
$s'_x = s_x +1$, we write  $(s \to s')$ (or $(s\to_x s')$ to be precise). In
this case, we define $F_{D,s\to s'}$ to be the vinyl 
$D_{s'}$-foam-$D_{s}$, which is the identity everywhere but in a
neighborhood of $x$, where it is given by:

\begin{align*}
  \begin{array}{ccc}
    \NB{\tikz[scale =0.7]{\tdplotsetmaincoords{60}{100}
\begin{scope}[tdplot_main_coords]
  \coordinate (aT) at (-1, -1, 3);
  \coordinate (bT) at (-1, 1, 3);
  \coordinate (cT) at (1, 1, 3);
  \coordinate (dT) at (1, -1, 3);
  \coordinate (AT) at (-2, -2, 3);
  \coordinate (BT) at (-2, 2, 3);
  \coordinate (CT) at (2, 2, 3);
  \coordinate (DT) at (2, -2, 3);
  \coordinate (DM) at (2, -2, 0);
  \coordinate (aB) at (-1, -1, -3);
  \coordinate (bB) at (-1, 1, -3);
  \coordinate (cB) at (1, 1, -3);
  \coordinate (dB) at (1, -1, -3);
  \coordinate (AB) at (-2, -2, -3);
  \coordinate (BB) at (-2, 2, -3);
  \coordinate (CB) at (2, 2, -3);
  \coordinate (DB) at (2, -2, -3);
  \coordinate (aM) at (-1, -1, 1);
  \coordinate (bM) at (-1, 1, -1);
  \coordinate (cM) at (1, 1, -1);
  \coordinate (dM) at (1, -1, 1);
  \draw[thin, <-] ($(aT)!0.5!(dM)$) -- +(0,-1.5,0) node [left, sloped] {$\scriptstyle{n-s'_x}$}; 
  \draw[thin, <-] ($(aB)!0.5!(dM)$) -- +(0,-1.5,0) node [left, sloped] {$\scriptstyle{n-s_x}$}; 
  \draw[thin, <-] ($(cT)!0.5!(bM)$) -- +(0,+1.5,0) node [right, sloped] {$\scriptstyle{m+s'_x}$}; 
  \draw[thin, <-] ($(cB)!0.5!(bM)$) -- +(0,+1.5,0) node [right, sloped] {$\scriptstyle{m+s_x}$}; 
  \draw[->, thick] (aT) -- (AT);
  \draw[->, thick] (DT) -- (dT);
  \draw[->, thick] (bT) -- (BT);
  \draw[->, thick] (CT) -- (cT);
  \draw[->, thick] (dT) -- (aT);
  \draw[->, thick] (aT) -- (bT);
  \draw[->, thick] (cT) -- (dT);
  \draw[->, thick] (cT) -- (bT);
  \draw[->, thick] (dB) -- (aB);
  \draw[->, thick] (aB) -- (bB);
  \draw[->, thick] (cB) -- (dB);
  \draw[->, thick] (cB) -- (bB);
  \draw[->, thick] (aB) -- (AB);
  \draw[->, thick] (DB) -- (dB);
  \draw[->, thick] (bB) -- (BB);
  \draw[->, thick] (CB) -- (cB);
  \draw[->, thick] (aB) -- (aT);
  \draw[->, thick] (bT) -- (bB);
  \draw[->, thick] (cB) -- (cT);
  \draw[->, thick] (dT) -- (dB);
  \filldraw[draw = black, rounded corners=1pt, thick, fill opacity = 0.3, fill = red]  (aT) -- (aB) -- (AB) -- (AT) -- (aT) node[sloped, midway, below, opacity = 1] {$\scriptstyle{m}$};
  \filldraw[draw = black, rounded corners=1pt, thick, fill opacity = 0.3, fill = red]  (bT) -- (bB) -- (BB)  -- (BT) -- (bT) node[sloped, midway, below, opacity = 1] {$\scriptstyle{n}$};
  \filldraw[draw = black, rounded corners=1pt, thick, fill opacity = 0.3, fill = green]   (aT) --  (aB) -- (bB) node[sloped, midway, above, opacity = 1] {$\scriptstyle{n-s_x-m}$} -- (bT) --  (aT) node[sloped, midway, below, opacity = 1] {$\scriptstyle{n-s'_x-m}$};
  \filldraw[draw = black, rounded corners=1pt, thick, fill opacity = 0.3, fill = blue]  (aM) -- (bM) node[ midway, sloped, below, opacity = 1] {$\scriptstyle{1}$} --  (cM)  -- (dM) -- cycle; 
  \filldraw[draw = black, rounded corners=1pt, thick, fill opacity = 0.3, fill = red]  (dT) -- (dB) -- (DB) node[sloped, midway, above, opacity = 1] {$\scriptstyle{n}$}-- (DT) -- cycle;
  \filldraw[draw = black, rounded corners=1pt, thick, fill opacity = 0.3, fill = green]   (aT) --  (aB) -- (dB) -- (dT) --  cycle;
  \filldraw[draw = black, rounded corners=1pt, thick, fill opacity = 0.3, fill = green]   (cT) -- (cB) --  (bB) --  (bT) -- cycle;
  \filldraw[draw = black, rounded corners=1pt, thick, fill opacity = 0.3, fill = green]   (cT) -- (cB) --  (dB) node[sloped, midway, above, opacity = 1] {$\scriptstyle{-s_x}$} --  (dT) -- (cT) node[sloped, midway, below, opacity = 1] {$\scriptstyle{-s'_x}$};
   \filldraw[draw = black, rounded corners=1pt, thick, fill opacity = 0.3, fill = red]  (cT) -- (cB) -- (CB) node[sloped, midway, above, opacity = 1] {$\scriptstyle{m}$} -- (CT) -- cycle;  
\end{scope}}} & & \NB{\tikz[scale =0.7]{\tdplotsetmaincoords{60}{100}
\begin{scope}[tdplot_main_coords]
  \coordinate (aT) at (-1, -1, 3);
  \coordinate (bT) at (-1, 1, 3);
  \coordinate (cT) at (1, 1, 3);
  \coordinate (dT) at (1, -1, 3);
  \coordinate (AT) at (-2, -2, 3);
  \coordinate (BT) at (-2, 2, 3);
  \coordinate (CT) at (2, 2, 3);
  \coordinate (DT) at (2, -2, 3);
  \coordinate (DM) at (2, -2, 0);
  \coordinate (aB) at (-1, -1, -3);
  \coordinate (bB) at (-1, 1, -3);
  \coordinate (cB) at (1, 1, -3);
  \coordinate (dB) at (1, -1, -3);
  \coordinate (AB) at (-2, -2, -3);
  \coordinate (BB) at (-2, 2, -3);
  \coordinate (CB) at (2, 2, -3);
  \coordinate (DB) at (2, -2, -3);
  \coordinate (aM) at (-1, -1, -1);
  \coordinate (bM) at (-1, 1, 1);
  \coordinate (cM) at (1, 1, 1);
  \coordinate (dM) at (1, -1, -1);
  \draw[thin, <-] ($(aT)!0.5!(dM)$) -- +(0,-1.5,0) node [left, sloped] {$\scriptstyle{n+s'_x}$}; 
  \draw[thin, <-] ($(aB)!0.5!(dM)$) -- +(0,-1.5,0) node [left, sloped] {$\scriptstyle{n+s_x}$}; 
  \draw[thin, <-] ($(cT)!0.5!(bM)$) -- +(0,+1.5,0) node [right, sloped] {$\scriptstyle{m-s'_x}$}; 
  \draw[thin, <-] ($(cB)!0.5!(bM)$) -- +(0,+1.5,0) node [right, sloped] {$\scriptstyle{m-s_x}$}; 
  \draw[->, thick] (aT) -- (AT);
  \draw[->, thick] (DT) -- (dT);
  \draw[->, thick] (bT) -- (BT);
  \draw[->, thick] (CT) -- (cT);
  \draw[->, thick] (dT) -- (aT);
  \draw[<-, thick] (aT) -- (bT);
  \draw[<-, thick] (cT) -- (dT);
  \draw[->, thick] (cT) -- (bT);
  \draw[->, thick] (dB) -- (aB);
  \draw[<-, thick] (aB) -- (bB);
  \draw[<-, thick] (cB) -- (dB);
  \draw[->, thick] (cB) -- (bB);
  \draw[->, thick] (aB) -- (AB);
  \draw[->, thick] (DB) -- (dB);
  \draw[->, thick] (bB) -- (BB);
  \draw[->, thick] (CB) -- (cB);
  \draw[->, thick] (aB) -- (aT);
  \draw[->, thick] (bT) -- (bB);
  \draw[->, thick] (cB) -- (cT);
  \draw[->, thick] (dT) -- (dB);
  \filldraw[draw = black, rounded corners=1pt, thick, fill opacity = 0.3, fill = red]  (aT) -- (aB) -- (AB) -- (AT) -- (aT) node[sloped, midway, below, opacity = 1] {$\scriptstyle{m}$};
  \filldraw[draw = black, rounded corners=1pt, thick, fill opacity = 0.3, fill = red]  (bT) -- (bB) -- (BB)  -- (BT) -- (bT) node[sloped, midway, below, opacity = 1] {$\scriptstyle{n}$};
  \filldraw[draw = black, rounded corners=1pt, thick, fill opacity = 0.3, fill = green]   (aT) --  (aB) -- (bB) node[sloped, midway, above, opacity = 1] {$\scriptstyle{m-s_x-n}$} -- (bT) --  (aT) node[sloped, midway, below, opacity = 1] {$\scriptstyle{m-s'_x-n}$};
  \filldraw[draw = black, rounded corners=1pt, thick, fill opacity = 0.3, fill = blue]  (aM) -- (bM) node[ midway, sloped, below, opacity = 1] {$\scriptstyle{1}$} --  (cM)  -- (dM) -- cycle; 
  \filldraw[draw = black, rounded corners=1pt, thick, fill opacity = 0.3, fill = red]  (dT) -- (dB) -- (DB) node[sloped, midway, above, opacity = 1] {$\scriptstyle{n}$}-- (DT) -- cycle;
  \filldraw[draw = black, rounded corners=1pt, thick, fill opacity = 0.3, fill = green]   (aT) --  (aB) -- (dB) -- (dT) --  cycle;
  \filldraw[draw = black, rounded corners=1pt, thick, fill opacity = 0.3, fill = green]   (cT) -- (cB) --  (bB) --  (bT) -- cycle;
  \filldraw[draw = black, rounded corners=1pt, thick, fill opacity = 0.3, fill = green]   (cT) -- (cB) --  (dB) node[sloped, midway, above, opacity = 1] {$\scriptstyle{-s_x}$} --  (dT) -- (cT) node[sloped, midway, below, opacity = 1] {$\scriptstyle{-s'_x}$};
   \filldraw[draw = black, rounded corners=1pt, thick, fill opacity = 0.3, fill = red]  (cT) -- (cB) -- (CB) node[sloped, midway, above, opacity = 1] {$\scriptstyle{m}$} -- (CT) -- cycle;  
\end{scope}}} \\
    \textrm{if $x$ is of type $(m,n,+)$ and $m \leq n $,}   & &   \textrm{if $x$ is of type $(m,n,-)$ and $n\leq m$,}  
\end{array} \\ \\
  \begin{array}{ccc}
    \NB{\tikz[scale =0.7]{\tdplotsetmaincoords{60}{100}
\begin{scope}[tdplot_main_coords]
  \coordinate (aT) at (-1, -1, 3);
  \coordinate (bT) at (-1, 1, 3);
  \coordinate (cT) at (1, 1, 3);
  \coordinate (dT) at (1, -1, 3);
  \coordinate (AT) at (-2, -2, 3);
  \coordinate (BT) at (-2, 2, 3);
  \coordinate (CT) at (2, 2, 3);
  \coordinate (DT) at (2, -2, 3);
  \coordinate (DM) at (2, -2, 0);
  \coordinate (aB) at (-1, -1, -3);
  \coordinate (bB) at (-1, 1, -3);
  \coordinate (cB) at (1, 1, -3);
  \coordinate (dB) at (1, -1, -3);
  \coordinate (AB) at (-2, -2, -3);
  \coordinate (BB) at (-2, 2, -3);
  \coordinate (CB) at (2, 2, -3);
  \coordinate (DB) at (2, -2, -3);
  \coordinate (aM) at (-1, -1, -1);
  \coordinate (bM) at (-1, 1, 1);
  \coordinate (cM) at (1, 1, 1);
  \coordinate (dM) at (1, -1, -1);
  \draw[thin, <-] ($(aT)!0.5!(dM)$) -- +(0,-1.5,0) node [left, sloped] {$\scriptstyle{n-s'_x}$}; 
  \draw[thin, <-] ($(aB)!0.5!(dM)$) -- +(0,-1.5,0) node [left, sloped] {$\scriptstyle{n-s_x}$}; 
  \draw[thin, <-] ($(cT)!0.5!(bM)$) -- +(0,+1.5,0) node [right, sloped] {$\scriptstyle{m+s'_x}$}; 
  \draw[thin, <-] ($(cB)!0.5!(bM)$) -- +(0,+1.5,0) node [right, sloped] {$\scriptstyle{m+s_x}$}; 
  \draw[->, thick] (aT) -- (AT);
  \draw[->, thick] (DT) -- (dT);
  \draw[->, thick] (bT) -- (BT);
  \draw[->, thick] (CT) -- (cT);
  \draw[->, thick] (dT) -- (aT);
  \draw[<-, thick] (aT) -- (bT);
  \draw[<-, thick] (cT) -- (dT);
  \draw[->, thick] (cT) -- (bT);
  \draw[->, thick] (dB) -- (aB);
  \draw[<-, thick] (aB) -- (bB);
  \draw[<-, thick] (cB) -- (dB);
  \draw[->, thick] (cB) -- (bB);
  \draw[->, thick] (aB) -- (AB);
  \draw[->, thick] (DB) -- (dB);
  \draw[->, thick] (bB) -- (BB);
  \draw[->, thick] (CB) -- (cB);
  \draw[->, thick] (aB) -- (aT);
  \draw[->, thick] (bT) -- (bB);
  \draw[->, thick] (cB) -- (cT);
  \draw[->, thick] (dT) -- (dB);
  \filldraw[draw = black, rounded corners=1pt, thick, fill opacity = 0.3, fill = red]  (aT) -- (aB) -- (AB) -- (AT) -- (aT) node[sloped, midway, below, opacity = 1] {$\scriptstyle{m}$};
  \filldraw[draw = black, rounded corners=1pt, thick, fill opacity = 0.3, fill = red]  (bT) -- (bB) -- (BB)  -- (BT) -- (bT) node[sloped, midway, below, opacity = 1] {$\scriptstyle{n}$};
  \filldraw[draw = black, rounded corners=1pt, thick, fill opacity = 0.3, fill = green]   (aT) --  (aB) -- (bB) node[sloped, midway, above, opacity = 1] {$\scriptstyle{m+s_x-n}$} -- (bT) --  (aT) node[sloped, midway, below, opacity = 1] {$\scriptstyle{m+s'_x-n}$};
  \filldraw[draw = black, rounded corners=1pt, thick, fill opacity = 0.3, fill = blue]  (aM) -- (bM) node[ midway, sloped, below, opacity = 1] {$\scriptstyle{1}$} --  (cM)  -- (dM) -- cycle; 
  \filldraw[draw = black, rounded corners=1pt, thick, fill opacity = 0.3, fill = red]  (dT) -- (dB) -- (DB) node[sloped, midway, above, opacity = 1] {$\scriptstyle{n}$}-- (DT) -- cycle;
  \filldraw[draw = black, rounded corners=1pt, thick, fill opacity = 0.3, fill = green]   (aT) --  (aB) -- (dB) -- (dT) --  cycle;
  \filldraw[draw = black, rounded corners=1pt, thick, fill opacity = 0.3, fill = green]   (cT) -- (cB) --  (bB) --  (bT) -- cycle;
  \filldraw[draw = black, rounded corners=1pt, thick, fill opacity = 0.3, fill = green]   (cT) -- (cB) --  (dB) node[sloped, midway, above, opacity = 1] {$\scriptstyle{-s_x}$} --  (dT) -- (cT) node[sloped, midway, below, opacity = 1] {$\scriptstyle{-s'_x}$};
   \filldraw[draw = black, rounded corners=1pt, thick, fill opacity = 0.3, fill = red]  (cT) -- (cB) -- (CB) node[sloped, midway, above, opacity = 1] {$\scriptstyle{m}$} -- (CT) -- cycle;  
\end{scope}}} & &  \NB{\tikz[scale =0.7]{\tdplotsetmaincoords{60}{100}
\begin{scope}[tdplot_main_coords]
  \coordinate (aT) at (-1, -1, 3);
  \coordinate (bT) at (-1, 1, 3);
  \coordinate (cT) at (1, 1, 3);
  \coordinate (dT) at (1, -1, 3);
  \coordinate (AT) at (-2, -2, 3);
  \coordinate (BT) at (-2, 2, 3);
  \coordinate (CT) at (2, 2, 3);
  \coordinate (DT) at (2, -2, 3);
  \coordinate (DM) at (2, -2, 0);
  \coordinate (aB) at (-1, -1, -3);
  \coordinate (bB) at (-1, 1, -3);
  \coordinate (cB) at (1, 1, -3);
  \coordinate (dB) at (1, -1, -3);
  \coordinate (AB) at (-2, -2, -3);
  \coordinate (BB) at (-2, 2, -3);
  \coordinate (CB) at (2, 2, -3);
  \coordinate (DB) at (2, -2, -3);
  \coordinate (aM) at (-1, -1, 1);
  \coordinate (bM) at (-1, 1, -1);
  \coordinate (cM) at (1, 1, -1);
  \coordinate (dM) at (1, -1, 1);
  \draw[thin, <-] ($(aT)!0.5!(dM)$) -- +(0,-1.5,0) node [left, sloped] {$\scriptstyle{n+s'_x}$}; 
  \draw[thin, <-] ($(aB)!0.5!(dM)$) -- +(0,-1.5,0) node [left, sloped] {$\scriptstyle{n+s_x}$}; 
  \draw[thin, <-] ($(cT)!0.5!(bM)$) -- +(0,+1.5,0) node [right, sloped] {$\scriptstyle{m-s'_x}$}; 
  \draw[thin, <-] ($(cB)!0.5!(bM)$) -- +(0,+1.5,0) node [right, sloped] {$\scriptstyle{m-s_x}$}; 
  \draw[->, thick] (aT) -- (AT);
  \draw[->, thick] (DT) -- (dT);
  \draw[->, thick] (bT) -- (BT);
  \draw[->, thick] (CT) -- (cT);
  \draw[->, thick] (dT) -- (aT);
  \draw[->, thick] (aT) -- (bT);
  \draw[->, thick] (cT) -- (dT);
  \draw[->, thick] (cT) -- (bT);
  \draw[->, thick] (dB) -- (aB);
  \draw[->, thick] (aB) -- (bB);
  \draw[->, thick] (cB) -- (dB);
  \draw[->, thick] (cB) -- (bB);
  \draw[->, thick] (aB) -- (AB);
  \draw[->, thick] (DB) -- (dB);
  \draw[->, thick] (bB) -- (BB);
  \draw[->, thick] (CB) -- (cB);
  \draw[->, thick] (aB) -- (aT);
  \draw[->, thick] (bT) -- (bB);
  \draw[->, thick] (cB) -- (cT);
  \draw[->, thick] (dT) -- (dB);
  \filldraw[draw = black, rounded corners=1pt, thick, fill opacity = 0.3, fill = red]  (aT) -- (aB) -- (AB) -- (AT) -- (aT) node[sloped, midway, below, opacity = 1] {$\scriptstyle{m}$};
  \filldraw[draw = black, rounded corners=1pt, thick, fill opacity = 0.3, fill = red]  (bT) -- (bB) -- (BB)  -- (BT) -- (bT) node[sloped, midway, below, opacity = 1] {$\scriptstyle{n}$};
  \filldraw[draw = black, rounded corners=1pt, thick, fill opacity = 0.3, fill = green]   (aT) --  (aB) -- (bB) node[sloped, midway, above, opacity = 1] {$\scriptstyle{n+s_x-m}$} -- (bT) --  (aT) node[sloped, midway, below, opacity = 1] {$\scriptstyle{n+s'_x-m}$};
  \filldraw[draw = black, rounded corners=1pt, thick, fill opacity = 0.3, fill = blue]  (aM) -- (bM) node[ midway, sloped, below, opacity = 1] {$\scriptstyle{1}$} --  (cM)  -- (dM) -- cycle; 
  \filldraw[draw = black, rounded corners=1pt, thick, fill opacity = 0.3, fill = red]  (dT) -- (dB) -- (DB) node[sloped, midway, above, opacity = 1] {$\scriptstyle{n}$}-- (DT) -- cycle;
  \filldraw[draw = black, rounded corners=1pt, thick, fill opacity = 0.3, fill = green]   (aT) --  (aB) -- (dB) -- (dT) --  cycle;
  \filldraw[draw = black, rounded corners=1pt, thick, fill opacity = 0.3, fill = green]   (cT) -- (cB) --  (bB) --  (bT) -- cycle;
  \filldraw[draw = black, rounded corners=1pt, thick, fill opacity = 0.3, fill = green]   (cT) -- (cB) --  (dB) node[sloped, midway, above, opacity = 1] {$\scriptstyle{s_x}$} --  (dT) -- (cT) node[sloped, midway, below, opacity = 1] {$\scriptstyle{s'_x}$};
   \filldraw[draw = black, rounded corners=1pt, thick, fill opacity = 0.3, fill = red]  (cT) -- (cB) -- (CB) node[sloped, midway, above, opacity = 1] {$\scriptstyle{m}$} -- (CT) -- cycle;  
\end{scope}}}    \\  
     \textrm{if $x$ is of type $(m,n,+)$ and $m >n $,}  && \textrm{if $x$ is of type $(m,n,-)$ and $n< m$.} 
  \end{array}
\end{align*}
It is worth noting that this has $q$-degree $1$. 

\begin{rmk}
  \label{rmk:diff1}
  If $x$ is of type $(1, 1, +)$ or $(1,1, -)$, the set $I_x$ has two elements and the foam $F_{D,s\to_x s'}$ is simpler because we can remove the facets labeled by $0$:
\[
\begin{array}{cc}
\NB{\tikz[scale=1]{\tdplotsetmaincoords{125}{115}
\begin{scope}[tdplot_main_coords]
  \coordinate (A1B) at (0, 0, 0);
  \coordinate (A2B) at (1, 0, 0);
  \coordinate (C1B) at (0.5, 0.5, 0);
  \coordinate (C2B) at (0.5, 1.5, 0);
  \coordinate (B1B) at (0, 2, 0);
  \coordinate (B2B) at (1, 2, 0);
  \coordinate (A1T) at (0, 0, 2);
  \coordinate (A2T) at (1, 0, 2);
  \coordinate (B1T) at (0, 2, 2);
  \coordinate (B2T) at (1, 2, 2);
  \coordinate (CM) at (0.5, 1,1);
  \filldraw [draw= black, fill =green, fill opacity =0.4] (A2B) --  (A2T) -- (B2T) -- (B2B) -- (C2B)  .. controls +(0,0,0) and +(0,0.5,0 ) .. (CM)  .. controls +(0, -0.5,0) and +(0,0, 0) .. (C1B) -- cycle;
 \filldraw [draw= black, fill =red, fill opacity =0.4] (A1B) --  (A1T) -- (B1T) -- (B1B) -- (C2B)  .. controls +(0,0,0) and +(0,0.5,0 ) .. (CM)  .. controls +(0, -0.5,0) and +(0,0, 0) .. (C1B) -- cycle;
 \filldraw [draw= black, fill =blue, fill opacity =0.4] (C1B) -- (C2B)  .. controls +(0,0,0) and +(0,0.5,0 ) .. (CM)  .. controls +(0, -0.5,0) and +(0,0, 0) .. (C1B) -- cycle;
\end{scope}}} & \NB{\tikz[scale=1]{\tdplotsetmaincoords{-125}{115}
\begin{scope}[tdplot_main_coords]
  \coordinate (A1B) at (0, 0, 0);
  \coordinate (A2B) at (1, 0, 0);
  \coordinate (C1B) at (0.5, 0.5, 0);
  \coordinate (C2B) at (0.5, 1.5, 0);
  \coordinate (B1B) at (0, 2, 0);
  \coordinate (B2B) at (1, 2, 0);
  \coordinate (A1T) at (0, 0, 2);
  \coordinate (A2T) at (1, 0, 2);
  \coordinate (B1T) at (0, 2, 2);
  \coordinate (B2T) at (1, 2, 2);
  \coordinate (CM) at (0.5, 1,1);
  \filldraw [draw= black, fill =green, fill opacity =0.4] (A2B) --  (A2T) -- (B2T) -- (B2B) -- (C2B)  .. controls +(0,0,0) and +(0,0.5,0 ) .. (CM)  .. controls +(0, -0.5,0) and +(0,0, 0) .. (C1B) -- cycle;
 \filldraw [draw= black, fill =red, fill opacity =0.4] (A1B) --  (A1T) -- (B1T) -- (B1B) -- (C2B)  .. controls +(0,0,0) and +(0,0.5,0 ) .. (CM)  .. controls +(0, -0.5,0) and +(0,0, 0) .. (C1B) -- cycle;
 \filldraw [draw= black, fill =blue, fill opacity =0.4] (C1B) -- (C2B)  .. controls +(0,0,0) and +(0,0.5,0 ) .. (CM)  .. controls +(0, -0.5,0) and +(0,0, 0) .. (C1B) -- cycle;
\end{scope}}} \\ 
F_{D,s\to_x s'}\textrm{ for $x$ of type $(1,1, +)$,} & F_{D,s\to_x s'}\textrm{ for $x$ of type $(1,1,-)$.} 
\end{array}
\]
\end{rmk}

We define an hyper-rectangle $R(D)$ of graded $\SP{N}$-modules. The vertices of this hyper-rectangle are labeled by states and
the edges by pair of states $(s,s')$ for which $s \to s'$. With every
state $s=(s_x)_{x\in \Xing}$, we associated the graded $\SP{N}$- module
$V_s:=\syf_{k,N}(D_s)q^{\kappa_s}$, and we declare that it has
homological degree $\eta_s$.

With every edge $(s\to s')$, we associate the map
$d_{s\to s'}:=\syf_{k,N}(F_{D,s\to s'}) : V_s \to V_{s'}$. One easily checks that all these maps are
$q$-homogeneous of degree $0$ (thanks to the degree shift $q^{\kappa_s}$) and
increase 
the homological degree by $1$. Hence
we call them \emph{pre-differentials}.

All squares in $R(D)$ commute because of the TQFT nature of the
functor $\syf_{k,N}$. Furthermore, if the composition of two pre-differentials $d_{s\to s'}$ and $d_{s'\to s''}$  does not fit into a square in $R(D)$, this means that we have  $(s\to_x s')$  and $(s'\to_x s'')$ for the same $x$ in $\Xing$. In this case $d_{s'\to s''} \circ d_{s\to s'} =0$ because the foams 
\[
\NB{\tikz[scale=0.7]{\tdplotsetmaincoords{60}{100}
\begin{scope}[tdplot_main_coords]
  \coordinate (aT) at (-1, -1, 3);
  \coordinate (bT) at (-1, 1, 3);
  \coordinate (cT) at (1, 1, 3);
  \coordinate (dT) at (1, -1, 3);
  \coordinate (AT) at (-2, -2, 3);
  \coordinate (BT) at (-2, 2, 3);
  \coordinate (CT) at (2, 2, 3);
  \coordinate (DT) at (2, -2, 3);
  \coordinate (DM) at (2, -2, 0);
  \coordinate (aB) at (-1, -1, -3);
  \coordinate (bB) at (-1, 1, -3);
  \coordinate (cB) at (1, 1, -3);
  \coordinate (dB) at (1, -1, -3);
  \coordinate (AB) at (-2, -2, -3);
  \coordinate (BB) at (-2, 2, -3);
  \coordinate (CB) at (2, 2, -3);
  \coordinate (DB) at (2, -2, -3);
  \coordinate (aM1) at (-1, -1, -1);
  \coordinate (bM1) at (-1, 1, -2);
  \coordinate (cM1) at (1, 1, -2);
  \coordinate (dM1) at (1, -1, -1);
  \coordinate (aM2) at (-1, -1, 2);
  \coordinate (bM2) at (-1, 1, 1);
  \coordinate (cM2) at (1, 1, 1);
  \coordinate (dM2) at (1, -1, 2);
  \draw[thin, <-] ($(aT)!0.5!(dM2)$) -- +(0,-1.5,0) node [left, sloped] {$\scriptstyle{n+k+2}$};
  \draw[thin, <-] ($(aM1)!0.5!(dM2)$) -- +(0,-1.5,0) node [left, sloped] {$\scriptstyle{n+k+1}$};  
  \draw[thin, <-] ($(aB)!0.5!(dM1)$) -- +(0,-1.5,0) node [left, sloped] {$\scriptstyle{n+k}$}; 
  \draw[thin, <-] ($(cT)!0.5!(bM2)$) -- +(0,+1.5,0) node [right, sloped] {$\scriptstyle{m-k-2}$}; 
  \draw[thin, <-] ($(cM1)!0.5!(bM2)$) -- +(0,+1.5,0) node [right, sloped] {$\scriptstyle{m-k-1}$}; 
  \draw[thin, <-] ($(cB)!0.5!(bM1)$) -- +(0,+1.5,0) node [right, sloped] {$\scriptstyle{m-k}$}; 
  \draw[->, thick] (aT) -- (AT);
  \draw[->, thick] (DT) -- (dT);
  \draw[->, thick] (bT) -- (BT);
  \draw[->, thick] (CT) -- (cT);
  \draw[->, thick] (dT) -- (aT);
  \draw[->, thick] (aT) -- (bT);
  \draw[->, thick] (cT) -- (dT);
  \draw[->, thick] (cT) -- (bT);
  \draw[->, thick] (dB) -- (aB);
  \draw[->, thick] (aB) -- (bB);
  \draw[->, thick] (cB) -- (dB);
  \draw[->, thick] (cB) -- (bB);
  \draw[->, thick] (aB) -- (AB);
  \draw[->, thick] (DB) -- (dB);
  \draw[->, thick] (bB) -- (BB);
  \draw[->, thick] (CB) -- (cB);
  \draw[->, thick] (aB) -- (aT);
  \draw[->, thick] (bT) -- (bB);
  \draw[->, thick] (cB) -- (cT);
  \draw[->, thick] (dT) -- (dB);
  \filldraw[draw = black, rounded corners=1pt, thick, fill opacity = 0.3, fill = red]  (aT) -- (aB) -- (AB) -- (AT) -- (aT) node[sloped, midway, below, opacity = 1] {$\scriptstyle{m}$};
  \filldraw[draw = black, rounded corners=1pt, thick, fill opacity = 0.3, fill = red]  (bT) -- (bB) -- (BB)  -- (BT) -- (bT) node[sloped, midway, below, opacity = 1] {$\scriptstyle{n}$};
  \filldraw[draw = black, rounded corners=1pt, thick, fill opacity = 0.3, fill = green]   (aT) --  (aB) -- (bB) node[sloped, midway, above, opacity = 1] {$\scriptstyle{n+k-m}$} -- (bT) --  (aT) node[sloped, midway, below, opacity = 1] {$\scriptstyle{n+k+2-m}$};
  \filldraw[draw = black, rounded corners=1pt, thick, fill opacity = 0.3, fill = blue]  (aM2) -- (bM2) node[ midway, sloped, below, opacity = 1] {$\scriptstyle{1}$} --  (cM2)  -- (dM2) -- cycle; 
  \filldraw[draw = black, rounded corners=1pt, thick, fill opacity = 0.3, fill = blue]  (aM1) -- (bM1) node[ midway, sloped, below, opacity = 1] {$\scriptstyle{1}$} --  (cM1)  -- (dM1) -- cycle; 
  \filldraw[draw = black, rounded corners=1pt, thick, fill opacity = 0.3, fill = red]  (dT) -- (dB) -- (DB) node[sloped, midway, above, opacity = 1] {$\scriptstyle{n}$}-- (DT) -- cycle;
  \filldraw[draw = black, rounded corners=1pt, thick, fill opacity = 0.3, fill = green]   (aT) --  (aB) -- (dB) -- (dT) --  cycle;
  \filldraw[draw = black, rounded corners=1pt, thick, fill opacity = 0.3, fill = green]   (cT) -- (cB) --  (bB) --  (bT) -- cycle;
  \filldraw[draw = black, rounded corners=1pt, thick, fill opacity = 0.3, fill = green]   (cT) -- (cB) --  (dB) node[sloped, midway, above, opacity = 1] {$\scriptstyle{k}$} --  (dT) -- (cT) node[sloped, midway, below, opacity = 1] {$\scriptstyle{k+2}$};
   \filldraw[draw = black, rounded corners=1pt, thick, fill opacity = 0.3, fill = red]  (cT) -- (cB) -- (CB) node[sloped, midway, above, opacity = 1] {$\scriptstyle{m}$} -- (CT) -- cycle;  
\end{scope}}}\,\,\textrm{and} \,\,
\NB{\tikz[scale=0.7]{\tdplotsetmaincoords{60}{100}
\begin{scope}[tdplot_main_coords]
  \coordinate (aT) at (-1, -1, 3);
  \coordinate (bT) at (-1, 1, 3);
  \coordinate (cT) at (1, 1, 3);
  \coordinate (dT) at (1, -1, 3);
  \coordinate (AT) at (-2, -2, 3);
  \coordinate (BT) at (-2, 2, 3);
  \coordinate (CT) at (2, 2, 3);
  \coordinate (DT) at (2, -2, 3);
  \coordinate (DM) at (2, -2, 0);
  \coordinate (aB) at (-1, -1, -3);
  \coordinate (bB) at (-1, 1, -3);
  \coordinate (cB) at (1, 1, -3);
  \coordinate (dB) at (1, -1, -3);
  \coordinate (AB) at (-2, -2, -3);
  \coordinate (BB) at (-2, 2, -3);
  \coordinate (CB) at (2, 2, -3);
  \coordinate (DB) at (2, -2, -3);
  \coordinate (aM1) at (-1, -1, -2);
  \coordinate (bM1) at (-1, 1, -1);
  \coordinate (cM1) at (1, 1, -1);
  \coordinate (dM1) at (1, -1, -2);
  \coordinate (aM2) at (-1, -1, 1);
  \coordinate (bM2) at (-1, 1, 2);
  \coordinate (cM2) at (1, 1, 2);
  \coordinate (dM2) at (1, -1, 1);
  \draw[thin, <-] ($(aT)!0.5!(dM2)$) -- +(0,-1.5,0) node [left, sloped] {$\scriptstyle{n-k}$}; 
  \draw[thin, <-] ($(aM1)!0.5!(dM2)$) -- +(0,-1.5,0) node [left, sloped] {$\scriptstyle{n-k+1}$}; 
  \draw[thin, <-] ($(aB)!0.5!(dM1)$) -- +(0,-1.5,0) node [left, sloped] {$\scriptstyle{n-k+2}$}; 
  \draw[thin, <-] ($(cT)!0.5!(bM2)$) -- +(0,+1.5,0) node [right, sloped] {$\scriptstyle{m+k}$}; 
  \draw[thin, <-] ($(cM1)!0.5!(bM2)$) -- +(0,+1.5,0) node [right, sloped] {$\scriptstyle{m+k-1}$}; 
  \draw[thin, <-] ($(cB)!0.5!(bM1)$) -- +(0,+1.5,0) node [right, sloped] {$\scriptstyle{m+k-2}$}; 
  \draw[->, thick] (aT) -- (AT);
  \draw[->, thick] (DT) -- (dT);
  \draw[->, thick] (bT) -- (BT);
  \draw[->, thick] (CT) -- (cT);
  \draw[->, thick] (dT) -- (aT);
  \draw[<-, thick] (aT) -- (bT);
  \draw[<-, thick] (cT) -- (dT);
  \draw[->, thick] (cT) -- (bT);
  \draw[->, thick] (dB) -- (aB);
  \draw[<-, thick] (aB) -- (bB);
  \draw[<-, thick] (cB) -- (dB);
  \draw[->, thick] (cB) -- (bB);
  \draw[->, thick] (aB) -- (AB);
  \draw[->, thick] (DB) -- (dB);
  \draw[->, thick] (bB) -- (BB);
  \draw[->, thick] (CB) -- (cB);
  \draw[->, thick] (aB) -- (aT);
  \draw[->, thick] (bT) -- (bB);
  \draw[->, thick] (cB) -- (cT);
  \draw[->, thick] (dT) -- (dB);
  \filldraw[draw = black, rounded corners=1pt, thick, fill opacity = 0.3, fill = red]  (aT) -- (aB) -- (AB) -- (AT) -- (aT) node[sloped, midway, below, opacity = 1] {$\scriptstyle{m}$};
  \filldraw[draw = black, rounded corners=1pt, thick, fill opacity = 0.3, fill = red]  (bT) -- (bB) -- (BB)  -- (BT) -- (bT) node[sloped, midway, below, opacity = 1] {$\scriptstyle{n}$};
  \filldraw[draw = black, rounded corners=1pt, thick, fill opacity = 0.3, fill = green]   (aT) --  (aB) -- (bB) node[sloped, midway, above, opacity = 1] {$\scriptstyle{m+k-2-n}$} -- (bT) --  (aT) node[sloped, midway, below, opacity = 1] {$\scriptstyle{m+k-n}$};
  \filldraw[draw = black, rounded corners=1pt, thick, fill opacity = 0.3, fill = blue]  (aM1) -- (bM1) node[ midway, sloped, below, opacity = 1] {$\scriptstyle{1}$} --  (cM1)  -- (dM1) -- cycle; 
  \filldraw[draw = black, rounded corners=1pt, thick, fill opacity = 0.3, fill = blue]  (aM2) -- (bM2) node[ midway, sloped, below, opacity = 1] {$\scriptstyle{1}$} --  (cM2)  -- (dM2) -- cycle; 
  \filldraw[draw = black, rounded corners=1pt, thick, fill opacity = 0.3, fill = red]  (dT) -- (dB) -- (DB) node[sloped, midway, above, opacity = 1] {$\scriptstyle{n}$}-- (DT) -- cycle;
  \filldraw[draw = black, rounded corners=1pt, thick, fill opacity = 0.3, fill = green]   (aT) --  (aB) -- (dB) -- (dT) --  cycle;
  \filldraw[draw = black, rounded corners=1pt, thick, fill opacity = 0.3, fill = green]   (cT) -- (cB) --  (bB) --  (bT) -- cycle;
  \filldraw[draw = black, rounded corners=1pt, thick, fill opacity = 0.3, fill = green]   (cT) -- (cB) --  (dB) node[sloped, midway, above, opacity = 1] {$\scriptstyle{k-2}$} --  (dT) -- (cT) node[sloped, midway, below, opacity = 1] {$\scriptstyle{k}$};
   \filldraw[draw = black, rounded corners=1pt, thick, fill opacity = 0.3, fill = red]  (cT) -- (cB) -- (CB) node[sloped, midway, above, opacity = 1] {$\scriptstyle{m}$} -- (CT) -- cycle;  
\end{scope}}}
\]
are $\infty$-equivalent to $0$. 

Hence if we add\footnote{There are many ways to do it, but all possibilities produce isomorphic chain complexes. In order to get functoriality of the construction announced in the introduction, one should be careful in this choice. This boils down to endowing the set $\Xing$ with a total order as detailed in \cite{FunctorialitySLN}. A more systematical construction is given in \cite{1195.57024} making use of the exterior algebra generated by $\Xing$. This last approach works only for uncolored links, but can be easily adapted to the colored case. } some signs to the pre-differential to turn the commutativity of squares into anti-commutativity we can flatten the hyper-rectangle $R(D)$ and obtain a complex of graded $\SP{N}$-modules $S^{\bullet, \bullet}(D)$.


Finally we set $\widehat{S}^{\bullet, \bullet}(D)$ to be equal to $S^{\bullet, \bullet}(D)q^{\kappa'(D)}$, where
\[\kappa'(D) = \sum_{x \in \Xing} \kappa'_x \quad \textrm{and} \quad 
\kappa'_x,  =
  \begin{cases}
    -m(m+N-1)             & \textrm{if $x$ is of type $(m,m,+)$, } \\ 
     m(m+N-1)            & \textrm{if $x$ is of type $(m,m,-)$,} \\ 
    0               & \textrm{else.}
  \end{cases}
\]


\begin{thm}
  \label{thm:symlinkhom}
  The homology of $\widehat{S}^{\bullet, \bullet}(D)$ is a link invariant which categorifies the symmetric Reshetikhin--Turaev invariant.
\end{thm}

\begin{notation} 
  In what follows, we deal with three different differentials: the Hochschild one ($d_H$), the topological one ($d_T$) and the extra one ($d^N$). We denote by $\HH(\bullet)$ (resp. $\HT(\bullet)$, $\HN(\bullet)$), the homology taken with respect to $d_H$ (resp. $d_T$, $d^N$) and by $\HNT(\bullet)$ the homology taken with respect to the total complex\footnote{Note that $d^N$ is a differential of \emph{chain} complex (it has $H$-degree $-1$) while $d_T$  is a differential of \emph{cochain} complex (it has $T$-degree $+1$), the total complex we consider is a complex of cochain complex: the total homological degree is equal to the $T$-degree minus the $H$-degree.} built out of the bicomplex with bi-differentials $(d^N,d_T)$.
The Hochschild homology is computed using the Koszul complex (as explain in Section~\ref{sec:hochschild-homology} and Appendix~\ref{sec:kosz-reosl-polyn}). Moreover, for the Hochschild homology, we will often drop the algebra in the notation: writing $\HH(M)$ instead of $\HH(A,M)$. As explain in Section~\ref{sec:an-extra-diff}, we denote by $\HH^N$ the Hochschild homology with an additional $q$-degree shift making the extra differential of $q$-degree $0$.

Note that homological degree for $\HH$ and $\HN$ coincide. We denote by $[\bullet]_H$ and $[\bullet]_T$ grading shifts\footnote{We use the \emph{homological} convention for grading shifts: If $V = \oplus_i V_i$, we have $(V[k])_i = V_{i-k}$.} with respect to the $H$ and the $T$-degree. When considering  the total complex   built out of the bicomplex with bi-differentials $(d^N,d_T)$, the homological degree shift is denoted by $[\bullet]_{T\!H}$. Let us recall that grading shifts with respect to the $q$-degree are denoted by $q^{\bullet}$.
\end{notation}

\begin{proof}[Proof of Theorem~\ref{thm:symlinkhom}]
  The fact the graded Euler characteristic of  $\widehat{S}^{\bullet, \bullet}(\bullet)$ is indeed the symmetric Reshetikhin--Turaev invariant follows from Theorem~\ref{thm:catsym} and from the construction of the hyper-rectangle which is clearly designed to  categorify the identities~(\ref{eq:symcrossplus}) and (\ref{eq:symcrossminus}).
  In order to prove invariance, we proceed in two steps:
  \begin{enumerate}
  \item \label{item:braids} First, we prove that if $D$ is a diagram of a knotted vinyl graph (that is a diagram of a knotted MOY graph which satisfies the same condition on tangent than vinyl graph), then the homotopy type of complex $S^{\bullet, \bullet}(D)$ (up to a $q$-grading shift) only depends on the isotopy type of knotted graph in the annulus.
  \item \label{item:stab} Then we prove that we have a stabilization property. Namely, we will show, that for any diagram of knotted braid-like $\listk{k}$-MOY graph-$\listk{k}$ $\Gamma$, the complexes associated with the following three diagrams 
\[
D_+ := \NB{\tikz[scale=0.38]{\begin{scope}[xscale=-1, yscale= 1]
\filldraw[fill=blue!20!white, draw = black] (-145:1.45) arc (-145:145:1.45) -- +(145:1.6) arc (145: -145:3.05) -- cycle;
\draw[>->] (-145:1.5) .. controls +(125:0.5) and +(-125:0.5) .. (145:1) node[near end, left] {\tiny{$1$}};
\fill[white] (-1.17, 0) circle (0.1cm);
\draw[<-<] ( 145:1.5) .. controls +(-125:0.5) and +(+125:0.5) .. (-145:1) node[near start, right] {\tiny{$1$}};
\draw[->] (-145:2) arc (-145:-215:2);
\node at (180:2.5) {\tiny{$_{\ \dots}$}};
\node at (0:2.25) {$\Gamma$};
\draw[->] (-145:3) arc (-145:-215:3);
\draw (-145:1) arc (-145:145:1) node[midway, right] {\tiny{$1$}};
\end{scope}
}},
\quad
D_- := \NB{\tikz[scale=0.38]{\begin{scope}[xscale=-1, yscale= 1]
\filldraw[fill=blue!20!white, draw = black] (-145:1.45) arc (-145:145:1.45) -- +(145:1.6) arc (145: -145:3.05) -- cycle;
\draw[<-<] ( 145:1.5) .. controls +(-125:0.5) and +(+125:0.5) .. (-145:1) node[near start, right] {\tiny{$1$}};
\fill[white] (-1.17, 0) circle (0.1cm);
\draw[>->] (-145:1.5) .. controls +(125:0.5) and +(-125:0.5) .. (145:1) node[near end, left] {\tiny{$1$}};
\draw[->] (-145:2) arc (-145:-215:2);
\node at (180:2.5) {\tiny{$_{\ \dots}$}};
\node at (0:2.25) {$\Gamma$};
\draw[->] (-145:3) arc (-145:-215:3);
\draw (-145:1) arc (-145:145:1) node[midway, right] {\tiny{$1$}};
\end{scope}
}}
\quad  \textrm{and} \quad
D_0 := \NB{\tikz[scale=0.38]{\begin{scope}[xscale=-1]
\filldraw[fill=blue!20!white, draw = black] (-145:1.45) arc (-145:145:1.45) -- +(145:1.6) arc (145: -145:3.05) -- cycle;
\draw[->] (-145:1.5) arc (-145:-215:1.5) node[left, midway] {\tiny{$1$}}; 
\draw[->] (-145:2) arc (-145:-215:2);
\node at (180:2.5) {\tiny{$_{\ \dots}$}};
\node at (0:2.25) {$\Gamma$};
\draw[->] (-145:3) arc (-145:-215:3);
\end{scope}
}}
\]
have the same homology. We write $\Gamma_+$  (resp. $\Gamma_-$) for the two knotted braid-like $\listk{k'}$-MOY graph-$\listk{k'}$ obtained from $\Gamma$ by adding one strand labeled by $1$ on the right and a positive (resp. negative) crossing on the top of it.
  \end{enumerate}
Note that in our stabilization, we only deal with a strand labeled by $1$. Thanks to a trick due to Mackaay--Sto\v si\'c--Vaz~\cite{MR2746676}  (see as well \cite{pre06302580} and \cite[Figure 1]{MR3709661}), this implies (together with the homotopy equivalences~(\ref{eq:twistfork})) that we actually get the stabilization property for any labels. Using this trick requires that we actually deal with knotted vinyl graphs in step~(\ref{item:braids}) (and not only with links). Finally thanks to Markov theorem, we can conclude that $S^{\bullet, \bullet}(D)$ is a link invariant.

The proof of step~(\ref{item:braids}) is quite standard. One first consider the case where all strands involved in the braid relations have label $1$. This case is treated in terms of foams in \cite[Figures 5.14 and 5.16]{VAZPHD}. The case with strands of arbitrary labels follows from the first case and the invariance under the so-called \emph{fork slide moves} (see Figure~\ref{fig:fkslmv}).
\begin{figure}[ht!]
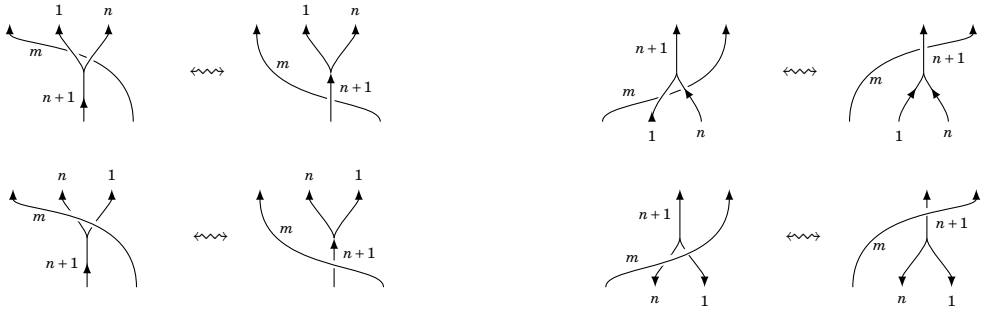

  \centering
  \tikz[scale= 0.65]{\begin{scope}
  \coordinate (A1) at (   0, -1);
  \coordinate (A2) at (   1, -1);
  \coordinate (B1) at (   0,  0);
  \coordinate (C1) at (-1.5,  1);
  \coordinate (C2) at (-0.5,  1);
  \coordinate (C3) at ( 0.5,  1);
  \draw[->] (A2) ..controls +(0,1.7) and + (0, -0.5) .. (C1) node[above, below, pos= 0.7, font=\tiny] {$m$};
  \draw[white, line width = 1mm] (A1) -- (B1);
  \draw[white, line width = 1mm] (B1) .. controls +(0,0.2) and +(0,-0.3) .. (C2);
  \draw[white, line width = 1mm] (B1) .. controls +(0,0.2) and +(0,-0.3) .. (C3);
  \draw[->-] (A1) -- (B1) node[midway, left, font=\tiny] {$n+1$};
  \draw[->] (B1) .. controls +(0,0.2) and +(0,-0.3) .. (C2) node[ above,  font=\tiny] {$1$};
  \draw[->] (B1) .. controls +(0,0.2) and +(0,-0.3) .. (C3) node[above, font=\tiny] {$n$};


\end{scope}
\node (R) at (2.5,0) {$\leftrightsquigarrow$};

\begin{scope}[xshift=5cm]
  \coordinate (A1) at (   0, -1);
  \coordinate (A2) at (   1, -1);
  \coordinate (B1) at (   0,  0);
  \coordinate (C1) at (-1.5,  1);
  \coordinate (C2) at (-0.5,  1);
  \coordinate (C3) at ( 0.5,  1);
  \draw[->] (A2) ..controls +(0,0.5) and + (0, -1.7) .. (C1) node[above, above, pos= 0.7, font=\tiny] {$m$};
  \draw[white, line width = 1mm] (A1) -- (B1);
  \draw[white, line width = 1mm] (B1) .. controls +(0,0.2) and +(0,-0.3) .. (C2);
  \draw[white, line width = 1mm] (B1) .. controls +(0,0.2) and +(0,-0.3) .. (C3);
  \draw[->] (A1) -- (B1) node[pos = 0.7, right, font=\tiny] {$n+1$};
  \draw[->] (B1) .. controls +(0,0.2) and +(0,-0.3) .. (C2) node[above,  font=\tiny] {$1$};
  \draw[->] (B1) .. controls +(0,0.2) and +(0,-0.3) .. (C3) node[above, font=\tiny] {$n$};

\end{scope}

\begin{scope}[xshift =12cm, yscale= -1]
  \begin{scope}
  \coordinate (A1) at (   0, -1);
  \coordinate (A2) at (   1, -1);
  \coordinate (B1) at (   0,  0);
  \coordinate (C1) at (-1.5,  1);
  \coordinate (C2) at (-0.5,  1);
  \coordinate (C3) at ( 0.5,  1);
  \draw[<-] (A2) ..controls +(0,1.7) and + (0, -0.5) .. (C1) node[above, above, pos= 0.7, font=\tiny] {$m$};
  \draw[white, line width = 1mm] (A1) -- (B1);
  \draw[white, line width = 1mm] (B1) .. controls +(0,0.2) and +(0,-0.3) .. (C2);
  \draw[white, line width = 1mm] (B1) .. controls +(0,0.2) and +(0,-0.3) .. (C3);
  \draw[<-] (A1) -- (B1) node[midway, left, font=\tiny] {$n+1$};
  \draw[-<] (B1) .. controls +(0,0.2) and +(0,-0.4) .. (C2) node[pos= 1, below, font=\tiny] {$1$};
  \draw[-<-] (B1) .. controls +(0,0.2) and +(0,-0.3) .. (C3) node[below, font=\tiny] {$n$};
\end{scope}
\node (R) at (2.5,0) {$\leftrightsquigarrow$};

\begin{scope}[xshift=5cm]
  \coordinate (A1) at (   0, -1);
  \coordinate (A2) at (   1, -1);
  \coordinate (B1) at (   0,  0);
  \coordinate (C1) at (-1.5,  1);
  \coordinate (C2) at (-0.5,  1);
  \coordinate (C3) at ( 0.5,  1);
  \draw[<-] (A2) ..controls +(0,0.5) and + (0, -1.7) .. (C1) node[above, below, pos= 0.7, font=\tiny] {$m$};
  \draw[white, line width = 1mm] (A1) -- (B1);
  \draw[white, line width = 1mm] (B1) .. controls +(0,0.2) and +(0,-0.3) .. (C2);
  \draw[white, line width = 1mm] (B1) .. controls +(0,0.2) and +(0,-0.3) .. (C3);
  \draw[<-] (A1) -- (B1) node[pos = 0.7, right, font=\tiny] {$n+1$};
  \draw[-<-] (B1) .. controls +(0,0.2) and +(0,-0.3) .. (C2) node[below,  font=\tiny] {$1$};
  \draw[-<-] (B1) .. controls +(0,0.2) and +(0,-0.3) .. (C3) node[below, font=\tiny] {$n$};
\end{scope}
\end{scope}} \\[4pt]
  \tikz[scale = 0.65]{\begin{scope}
  \coordinate (A1) at (   0, -1);
  \coordinate (A2) at (   1, -1);
  \coordinate (B1) at (   0,  0);
  \coordinate (C1) at (-1.5,  1);
  \coordinate (C2) at (-0.5,  1);
  \coordinate (C3) at ( 0.5,  1);
  \draw[->-] (A1) -- (B1) node[midway, left, font=\tiny] {$n+1$};
  \draw[->] (B1) .. controls +(0,0.2) and +(0,-0.3) .. (C2) node[above,  font=\tiny] {$n$};
  \draw[->] (B1) .. controls +(0,0.2) and +(0,-0.3) .. (C3) node[above, font=\tiny] {$1$};
  \draw[white, line width= 1mm,] (A2) ..controls +(0,1.7) and + (0, -0.5) .. (C1);
  \draw[->] (A2) ..controls +(0,1.7) and + (0, -0.5) .. (C1) node[above, below, pos= 0.7, font=\tiny] {$m$};
\end{scope}
\node (R) at (2.5,0) {$\leftrightsquigarrow$};

\begin{scope}[xshift=5cm]
  \coordinate (A1) at (   0, -1);
  \coordinate (A2) at (   1, -1);
  \coordinate (B1) at (   0,  0);
  \coordinate (C1) at (-1.5,  1);
  \coordinate (C2) at (-0.5,  1);
  \coordinate (C3) at ( 0.5,  1);
  \draw[->] (A1) -- (B1) node[pos = 0.7, right, font=\tiny] {$n+1$};
  \draw[->] (B1) .. controls +(0,0.2) and +(0,-0.3) .. (C2) node[above,  font=\tiny] {$n$};
  \draw[->] (B1) .. controls +(0,0.2) and +(0,-0.3) .. (C3) node[above, font=\tiny] {$1$};
  \draw[white, line width= 1mm,] (A2) ..controls +(0,0.5) and + (0, -1.7) .. (C1);
  \draw[->] (A2) ..controls +(0,0.5) and + (0, -1.7) .. (C1) node[above, above, pos= 0.7, font=\tiny] {$m$};
\end{scope}

\begin{scope}[xshift =12cm, yscale= -1]
  \begin{scope}
  \coordinate (A1) at (   0, -1);
  \coordinate (A2) at (   1, -1);
  \coordinate (B1) at (   0,  0);
  \coordinate (C1) at (-1.5,  1);
  \coordinate (C2) at (-0.5,  1);
  \coordinate (C3) at ( 0.5,  1);
  \draw[<-] (A1) -- (B1) node[midway, left, font=\tiny] {$n+1$};
  \draw[->] (B1) .. controls +(0,0.2) and +(0,-0.3) .. (C2) node[below,  font=\tiny] {$n$};
  \draw[->] (B1) .. controls +(0,0.2) and +(0,-0.3) .. (C3) node[below, font=\tiny] {$1$};
  \draw[white, line width= 1mm,] (A2) ..controls +(0,1.7) and + (0, -0.5) .. (C1);
  \draw[<-] (A2) ..controls +(0,1.7) and + (0, -0.5) .. (C1) node[above, above, pos= 0.7, font=\tiny] {$m$};
\end{scope}
\node (R) at (2.5,0) {$\leftrightsquigarrow$};

\begin{scope}[xshift=5cm]
  \coordinate (A1) at (   0, -1);
  \coordinate (A2) at (   1, -1);
  \coordinate (B1) at (   0,  0);
  \coordinate (C1) at (-1.5,  1);
  \coordinate (C2) at (-0.5,  1);
  \coordinate (C3) at ( 0.5,  1);
  \draw[<-] (A1) -- (B1) node[pos = 0.7, right, font=\tiny] {$n+1$};
  \draw[->] (B1) .. controls +(0,0.2) and +(0,-0.3) .. (C2) node[below,  font=\tiny] {$n$};
  \draw[->] (B1) .. controls +(0,0.2) and +(0,-0.3) .. (C3) node[below, font=\tiny] {$1$};
  \draw[white, line width= 1mm,] (A2) ..controls +(0,0.5) and + (0, -1.7) .. (C1);
  \draw[<-] (A2) ..controls +(0,0.5) and + (0, -1.7) .. (C1) node[above, below, pos= 0.7, font=\tiny] {$m$};
\end{scope}
\end{scope}}
  \caption{Fork slide moves}
  \label{fig:fkslmv}
\end{figure}
Indeed, one can blist (see Figure~\ref{fig:blist}) each strands and use the fork slide moves and the $1$-labeled  braid relation to deduce the arbitrary labeled braid relations.  See \cite{MR3687104, MR2491657,pre06302580} for similar arguments.
\begin{figure}[ht!]
  \centering
    \tikz[scale = 1]{\input{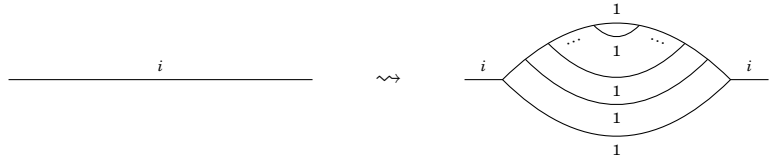}}
  \caption{Blisting an edge of label $i$.}
  \label{fig:blist}
\end{figure}
Proof of invariance for the fork slide move using foam is quite standard see for instance \cite[Proof of Proposition 4.10]{queffelec2014mathfrak} or \cite[Section 3.3]{FunctorialitySLN}. But let us briefly sketch how it works on the fork slide move on the top left move of Figure~\ref{fig:fkslmv}. We consider the diagram on the left-hand side of the move and its associated complex:
\[
  \tikz[xscale =3, yscale= 2.5]{
  \node (A) at (0, 1) {$ \NB{\tikz[yscale = 0.5, xscale = 0.9]{\begin{scope}[scale =1.5, font = \tiny]
  \coordinate (A1) at ( -1, 4);
  \coordinate (A2) at (  0, 4);
  \coordinate (A3) at (  1, 4);
  \coordinate (B)  at ( -1, 3.5);
  \coordinate (C)  at (  0, 3);
  \coordinate (D)  at (  0, 2.5 );
  \coordinate (E)  at ( -1, 2);
  \coordinate (F)  at (0.5, 1.5);
  \coordinate (G)  at (0.5, 1);
  \coordinate (H)  at (  0, 0.5);
  \coordinate (I1) at (  0, 0);
  \coordinate (I2) at (  1, 0);

  \draw[->-] (I1) -- (H) node[pos = 0, below] {$n+1$};
  \draw[->-] (I2)  .. controls +(0, 0.5) and + (0.1, -0.1) .. (G) node[pos = 0, below] {$m$};
  \draw[->-] (H)  -- (G) node[pos = 0.5, below] {$1$};
  \draw[->-] (G)  -- (F) node[pos = 0.5, right] {$m+1$};
  \draw[->-] (H)  -- (E) node[pos = 0.5, left] {$n$};
  \draw[->-] (E)  -- (D) node[pos = 0.5, above] {$j$};
  \draw[->-] (F)  .. controls +(0.1,0.1) and +(0, -1) .. (A3) node[pos = 1, above] {$1$};
  \draw[->-] (F)  -- (D) node [pos =0.5 , right] {$m$};
  \draw[->-] (D)  -- (C) node [pos =0.5 , right] {$m+j$};
  \draw[->-] (E)  -- (B) node [pos =0.5 , left] {$n-j$};
  \draw[->-] (C)  -- (B) node [pos =0.5 , sloped, above] {$m+j-n$};
  \draw[->-] (C)  -- (A2) node [pos =1 , above] {$n$};
  \draw[->-] (B)  -- (A1) node [pos =1 , above] {$i$};

\end{scope}} }$};
  \node (B) at (0,-1) {$ \NB{\tikz[yscale = 0.5, xscale = 0.9]{\begin{scope}[scale =1.5, font = \tiny]
  \coordinate (A1) at ( -1, 4);
  \coordinate (A2) at (  0, 4);
  \coordinate (A3) at (  1, 4);
  \coordinate (B)  at ( -1, 3.5);
  \coordinate (C)  at (  0, 3);
  \coordinate (D)  at (  0, 2.5 );
  \coordinate (E)  at ( -1, 2);
  \coordinate (F)  at (0, 1.5);
  \coordinate (G)  at (1, 1);
  \coordinate (H)  at (  0, 0.5);
  \coordinate (I1) at (  0, 0);
  \coordinate (I2) at (  1, 0);
  \draw[->-] (I1) -- (H) node[pos = 0, below] {$n+1$};
  \draw[->-] (I2) -- (G) node[pos = 0, below] {$m$};
  \draw[->-] (H)  -- (F) node[pos = 0.5, right] {$1$};
  \draw[->-] (G)  -- (F) node[pos = 0.5, below, sloped] {$m-1$};
  \draw[->-] (H)  -- (E) node[pos = 0.5, left] {$n$};
  \draw[->-] (E)  -- (D) node[pos = 0.5, above] {$j$};
  \draw[->-] (G) -- (A3) node[pos = 1, above] {$1$};
  \draw[->-] (F)  -- (D) node [pos =0.5 , right] {$m$};
  \draw[->-] (D)  -- (C) node [pos =0.5 , right] {$m+j$};
  \draw[->-] (E)  -- (B) node [pos =0.5 , left] {$n-j$};
  \draw[->-] (C)  -- (B) node [pos =0.5 , sloped, above] {$m+j-n$};
  \draw[->-] (C)  -- (A2) node [pos =1 , above] {$n$};
  \draw[->-] (B)  -- (A1) node [pos =1 , above] {$m$};
\end{scope}} }$};
  \node (a) at (-1, 1) {$\dots$};
  \node (aa) at (1, 1) {$\dots$};
  \node (b) at (-1,-1) {$\dots$};
  \node (bb) at (1,-1) {$\dots$};
  \draw[-to] (a) -- (A);
  \draw[-to] (A) -- (aa);
  \draw[-to] (b) -- (B);
  \draw[-to] (B) -- (bb);
  \draw[-to] (A) -- (B);
  }
\]
where for simplicity we dropped the symbols $\syf_N$.
The space on the first line can be decomposed using the isomorphisms (\ref{eq:catsymass}) and (\ref{eq:catsymsquare}). The space on the second line can be decomposed using isomorphisms
(\ref{eq:catsymass}) and (\ref{eq:digoncat}). This gives:
\[
  \tikz[xscale =3, yscale = 2.5]{
  \node (A) at (0, 1) {$\NB{\tikz[yscale = 0.5, xscale = 0.9]{\begin{scope}[scale =1.5, font = \tiny]
  \coordinate (A1) at ( -1, 4);
  \coordinate (A2) at (  0, 4);
  \coordinate (A3) at (  1, 4);
  \coordinate (B)  at ( -1, 3.5);
  \coordinate (C)  at (  0, 2.75);
  \coordinate (D)  at (  0.5, 2 );
  \coordinate (E)  at ( 0.5, 1.25);
  \coordinate (F)  at (0, 0.75);
  \coordinate (I1) at (  0, 0);
  \coordinate (I2) at (  1, 0);

  \draw[->-] (I1) -- (F) node[pos = 0, below] {$n+1$};
  \draw[->-] (I2)  .. controls +(0, 0.5) and +(0.1, -0.1) .. (E) node[pos = 0, below] {$m$};
  \draw[->-] (F)  -- (E) node[pos = 0.5, sloped, below] {$j+1$};
  \draw[->-] (F)  -- (B) node[pos = 0.5, left] {$n-j$};
  \draw[->-] (E)  -- (D) node[pos = 0.5, right] {$m+j+1$};
  \draw[->-] (D)  .. controls +(0.1,0.1) and +(0, -1) .. (A3) node[pos = 1, above] {$1$};
  \draw[->-] (D)  -- (C) node [pos =0.5 , left] {$m+j$};
  \draw[->-] (C)  -- (B) node [pos =0.5 , sloped, above] {$m+j-n$};
  \draw[->-] (C)  -- (A2) node [pos =1 , above] {$n$};
  \draw[->-] (B)  -- (A1) node [pos =1 , above] {$m$};
\end{scope}} }
  \oplus
  [j]\NB{\tikz[yscale = 0.5, xscale = 0.9]{\input{\imagesfolder/sym_cx-fk-sl-3}} }
  $};
  \node (B) at (0,-1) {$[j+1]\NB{\tikz[yscale = 0.5, xscale = 0.9]{\input{\imagesfolder/sym_cx-fk-sl-3}} }$};
  \node (a) at (-1.7, 1) {$\dots$};
  \node (aa) at (1.7, 1) {$\dots$};
  \node (b) at (-1.7,-1) {$\dots$};
  \node (bb) at (1.7,-1) {$\dots$};
  \draw[-to] (a) -- (A);
  \draw[-to] (A) -- (aa);
  \draw[-to] (b) -- (B);
  \draw[-to] (B) -- (bb);
  \draw[-to] (A) -- (B);
  }
\]
One can check that the vertical maps injective on the second term of the direct sum. We can use them to simplify the complex. The complex is homotopy equivalent to 
\[
  \tikz[xscale =3, yscale=2.5]{
  \node (A) at (0, 1) {$\NB{\tikz[yscale = 0.5, xscale = 0.9]{\begin{scope}[scale =1.5, font = \tiny]
  \coordinate (A1) at ( -1, 4);
  \coordinate (A2) at (  0, 4);
  \coordinate (A3) at (  1, 4);
  \coordinate (B)  at ( -1, 3.5);
  \coordinate (C)  at (  0, 2.75);
  \coordinate (D)  at (  0.5, 2 );
  \coordinate (E)  at ( 0.5, 1.25);
  \coordinate (F)  at (0, 0.75);
  \coordinate (I1) at (  0, 0);
  \coordinate (I2) at (  1, 0);

  \draw[->-] (I1) -- (F) node[pos = 0, below] {$n+1$};
  \draw[->-] (I2)  .. controls +(0, 0.5) and +(0.1, -0.1) .. (E) node[pos = 0, below] {$m$};
  \draw[->-] (F)  -- (E) node[pos = 0.5, sloped, below] {$j+1$};
  \draw[->-] (F)  -- (B) node[pos = 0.5, left] {$n-j$};
  \draw[->-] (E)  -- (D) node[pos = 0.5, right] {$m+j+1$};
  \draw[->-] (D)  .. controls +(0.1,0.1) and +(0, -1) .. (A3) node[pos = 1, above] {$1$};
  \draw[->-] (D)  -- (C) node [pos =0.5 , left] {$m+j$};
  \draw[->-] (C)  -- (B) node [pos =0.5 , sloped, above] {$m+j-n$};
  \draw[->-] (C)  -- (A2) node [pos =1 , above] {$n$};
  \draw[->-] (B)  -- (A1) node [pos =1 , above] {$m$};
\end{scope}} }
  $};
  \node (B) at (0,-1) {$q^{-j}\NB{\tikz[yscale = 0.5, xscale = 0.9]{\input{\imagesfolder/sym_cx-fk-sl-3}} }$};
  \node (a) at (-1.5, 1) {$\dots$};
  \node (aa) at (1.5, 1) {$\dots$};
  \node (b) at (-1.5,-1) {$\dots$};
  \node (bb) at (1.5,-1) {$\dots$};
  \draw[-to] (a) -- (A);
  \draw[-to] (A) -- (aa);
  \draw[-to] (b) -- (B);
  \draw[-to] (B) -- (bb);
  \draw[-to] (A) -- (B);
  }
\]
The second line turns out to be exact (thanks to isomorphism~(\ref{eq:catsymsquare})) except on the rightmost (corresponding to $j=0$) term if $n<m$. If the line is exact we can simplify and obtain that the complex is homotopy equivalent to that corresponding to the right-hand side diagram of the for slide move and we are interested in.

If $n<m$, we can cancel out the second line except the rightmost term which becomes:
\[
\NB{\tikz[yscale = 0.5, xscale = 0.9]{\begin{scope}[scale =1.5, font = \tiny]
  \coordinate (A1) at ( -1, 4);
  \coordinate (A2) at (  0, 4);
  \coordinate (A3) at (  1, 4);
  \coordinate (B)  at (  0.5, 3);
  \coordinate (C)  at (  -0.5, 2);
  \coordinate (D)  at (  0.5, 1 );
  \coordinate (I1) at (  0, 0);
  \coordinate (I2) at (  1, 0);
  \draw[->-] (I2) -- (D) node[pos = 0, below] {$m$};
  \draw[->-] (I1) ..controls +(0,0.5) and +(-0,-0.5) ..  (C) node[pos = 0, below] {$n+1$};
  \draw[->-] (D)  -- (C) node[pos = 0.5, above, sloped] {$n+1-m$};
  \draw[->-] (D)  -- (B) node[pos = 0.5, right] {$n+1$};
  \draw[->-] (B)  ..controls +(0.1,0.1) and +(-0,-0.5) .. (A3) node[pos = 1, above] {$1$};
  \draw[->-] (B)  ..controls +(-0.1,0.1) and +(-0,-0.5) .. (A2) node [pos =1 , above] {$n$};
  \draw[->-] (C)  ..controls +(-0.1,0.1) and +(-0,-0.5) .. (A1) node [pos =1 , above] {$m$};
\end{scope}}}
\]
Hence the whole complex is homotopy equivalent to that corresponding to the right-hand side diagram of the for slide move and we are interested in

The invariance under braid relations can as well be deduced from the algebraic setting using Soergel bimodules see \cite{MR3687104}.


\begin{align}\label{eq:twistfork}
q^{-ab }S\left(
\NB{\tikz[scale=0.5]{\begin{scope}[font= \tiny]
  \coordinate (T) at (0, 1);
  \coordinate (M) at (0, 0);
  \coordinate (B1) at (-1, -1);
  \coordinate (B2) at (1, -1);
  \draw[->] (M) -- (T) node[above] {$a+b$};
  \draw[->-] (B2)  node[below] {$b$} .. controls +(-0.1, 0.1) and +(-0.7 ,-0.3) .. (M);
  \fill[white] (0,-0.45) circle (0.1);
  \draw[->-] (B1)  node[below] {$a$} .. controls +( 0.1, 0.1) and +( 0.7 ,-0.3) .. (M);
\end{scope}}}
\right) 
\simeq 
S\left(
\NB{\tikz[scale=0.5]{\begin{scope}[font= \tiny]
  \coordinate (T) at (0, 1);
  \coordinate (M) at (0, 0);
  \coordinate (B1) at (-1, -1);
  \coordinate (B2) at (1, -1);
  \draw[->] (M) -- (T) node[above] {$a+b$};
  \draw[->-] (B1) node[below] {$a$} -- (M);
  \draw[->-] (B2) node[below] {$b$} -- (M);
\end{scope}}}
\right) 
\simeq 
q^{ab} S\left(
\NB{\tikz[scale=0.5]{\begin{scope}[font= \tiny]
  \coordinate (T) at (0, 1);
  \coordinate (M) at (0, 0);
  \coordinate (B1) at (-1, -1);
  \coordinate (B2) at (1, -1);
  \draw[->] (M) -- (T) node[above] {$a+b$};
  \draw[->-] (B1)  node[below] {$a$} .. controls +( 0.1, 0.1) and +( 0.7 ,-0.3) .. (M);
  \fill[white] (0,-0.45) circle (0.1);
  \draw[->-] (B2)  node[below] {$b$} .. controls +(-0.1, 0.1) and +(-0.7 ,-0.3) .. (M);

\end{scope}}}
\right) 
\end{align}

The proof of step~(\ref{item:stab}) is more involved. We need to use the  dictionary between Soergel bimodules and vinyl foams developed in Section~\ref{sec:one-quotient-two} and a stabilization result which holds for Soergel bimodules. 

First note that if $D$ has braid index equal to $k$, $S(D) = \HT(\HN(\HH^N_\bullet( \BS( R(D))))q^{-k(N-1)}$ thanks to Proposition~\ref{prop:alg-eq-foam}. We claim that we have $\HNT(\HH^N_\bullet(\BS( R(D))))= \HT(\HN(\HH^N_\bullet (\BS( R(D)))))$. We consider $\HH^N_\bullet (\BS( R(D)))$. It is a bi-complex which we temporarily denote by $C$. Proposition~\ref{prop:all-in-deg-0} tells us that $\HN(C)$ is concentrated in $H$-degree equal to $0$. This implies that the spectral sequence $E(C)$ induced by the bi-complex structure with $\HT(\HN(C))$ on the second page has only trivial differentials on this page. This spectral sequence converges to $\HNT(C)$ because the bi-complex is bounded. Hence, we have  $\HNT(\HH^N_\bullet (\BS( R(D))))=E^\infty(C) =E^2(C) = \HT(\HN(\HH^N_\bullet(\BS(R(D) ))))$.

The stabilization result for Soergel bimodules is given by Lemma~\ref{lem:stab} and its proof occupies Section~\ref{sec:stabilization}. It tells us that: 
\[\HT(\HH_\bullet(\BS( \Gamma_+)))[-1]_T[-1]_Hq^{-1}\simeq \HT(\HH_\bullet( \BS( \Gamma_0))) \simeq  \HT(\HH_\bullet( \BS( \Gamma_-)))q^{1}\]
Using $\HH^N$ instead of $\HH$ and collapsing the $T$-grading and the $H$-grading to their difference, we get:
\[\HT(\HH^N_\bullet(\BS( \Gamma_+)))q^{-2(N-1)-1}\simeq \HT(\HH^N_\bullet( \BS( \Gamma_0))) \simeq  \HT(\HH^N_\bullet( \BS( \Gamma_-)))q^{1}\]
moreover these isomorphisms preserve the extra-differential $d^N$. We shift by $q^{-k(N-1)}$. This proves that 
\begin{align*}
&\HN(\HT(\HH^N_\bullet(\BS( \Gamma_+)))) q^{-N}q^{-(k+1)(N-1)}\simeq \HN(\HT(\HH^N_\bullet(\BS( \Gamma_0)))) q^{-k(N-1)} \\ &\qquad \qquad\simeq  \HN(\HT(\HH^N_\bullet(\BS( \Gamma_-))))[1]_{T\!H}q^{N} q^{-(k+1)(N-1)}.
\end{align*}
Since homological degrees are all finite, there is a spectral sequence with $\HN(\HT(\HH_\bullet A_{\listk{k}}, \BS( \bullet)))$ on the second page converging to $\HNT(\HH_\bullet A_{\listk{k}}, \BS( \bullet))$. The previous isomorphisms descend to the spectral sequences and to their limits. Hence we have:
\begin{align*}&\HNT(\HH^N_\bullet A_{\listk{k}}, \BS( \Gamma_+))q^{N}q^{-(k+1)(N-1)} \simeq \HNT(\HH_\bullet^N (A_{\listk{k}}, \BS( \Gamma_0))q^{-k(N-1)}\\& \qquad \qquad \simeq  \HNT(\HH^N_\bullet( A_{\listk{k}}, \BS( \Gamma_-))q^{-N}q^{-(k+1)(N-1)}.\end{align*}
This implies that:
\[S(D_+)q^{N} \simeq S(D_0) \simeq  S(D_-)q^{-N},\]
And finally that 
\[\widehat{S}(D_+)  \simeq S(D_0) \simeq  \widehat{S}(D_-).\]
\end{proof}

\begin{rmk}
  \label{rmk:triplygraded}
  Note that the previous proof can be adapted to disk-like foams up to $\infty$-equivalence. Then it turns out to be a rewriting of the invariance of the triply graded homology very close to \cite{1203.5065} (see also \cite{MR2421131, MR3447099}).
\end{rmk}

\subsection{Stabilization}
\label{sec:stabilization}

The aim of this section is to prove the stabilization move for Soergel bimodules. This basically follows from Rouquier \cite{1203.5065}. However, since for further use we need to be careful with some additional structures going on, we repeat the proof. Note, however, that the framework in which the results of \cite{1203.5065} are stated is more general\footnote{Namely, he deals with arbitrary Coxeter groups, while we only consider the type $A$.} than ours.

Let $k$ be a positive integer and $\listk{k}^{1\!1}:= (k_1, \dots, k_{l-1},1,1)$ 
be a finite sequence of positive integers of level $k$. We define $\listk{k}^2= (k_1, \dots, k_{l-1},2)$ to be the same sequence where the last two $1$s has been merged $\listk{k}^1:= (k_1, \dots, k_{l-1},1)$ to be the same sequence where the last $1$ has been dropped, and finally $\listk{k}^0:= (k_1, \dots, k_{l-1})$ to be the same sequence where the last two $1$s has been dropped.  We consider the algebras $A^{1\!1}:= A_{\listk{k}^{1\!1}}$, $A^{2}:= A_{\listk{k}^{2}}$ and $A^{1}:= A_{\listk{k}^{1}}$.

We consider $M$ a complex of $A^1$-modules-$A^1$ which is projective as an $A^1$-module and as a module-$A^1$, and we define $MI := M\otimes \SP{N}[x_{k}]$. It has a natural structure complex of $A^{1\!1}$-modules-$A^{1\!1}$ and these modules are projective as $A^1$-modules and as a modules-$A^1$. We denote by $\theta$ the $A^{1\!1}$-module-$A^{1\!1}$ $A^{1\!1}\otimes_{A^2}A^{1\!1}q^{-1}$. Note that with the notations of Section~\ref{sec:soergel-bimodules-1}, $\theta$ is the Soergel bimodule associated with $\NB{\tikz[scale=0.3]{\begin{scope}[font=\tiny]
  \coordinate (T1) at (-1, 1);
  \coordinate (T2) at ( 1, 1);
  \coordinate (B1) at (-1,-1);
  \coordinate (B2) at ( 1,-1);
  \coordinate (M2) at ( 0, .5);
  \coordinate (M1) at ( 0,-.5);
  \draw[>-] (B1) node[below] {$1$} -- (M1);
  \draw[>-] (B2) node[below] {$1$} -- (M1);
  \draw[<-] (T1) node[above] {$1$} -- (M2);
  \draw[<-] (T2) node[above] {$1$} -- (M2);
  \draw[->-] (M1) -- (M2) node[midway, right] {$2$};
\end{scope}}}$ and the $A^{1\!1}$-module-$A^{1\!1}$ $A^{1\!1}$ is the Soergel bimodule associated with $\NB{\tikz[scale=0.3]{\begin{scope}[font=\tiny]
  \coordinate (T1) at (-1, 1);
  \coordinate (T2) at ( 1, 1);
  \coordinate (B1) at (-1,-1);
  \coordinate (B2) at ( 1,-1);
  \draw[->-] (B1)  .. controls +( 0.3,0.3) and +( 0.3, -0.3) .. (T1) node[midway, left] {$1$};
  \draw[->-] (B2)  .. controls +(-0.3,0.3) and +(-0.3, -0.3) .. (T2) node[midway, right] {$1$};
\end{scope}}}$.
Finally we consider two morphisms of $A^{1\!1}$-modules-$A^{1\!1}$.
\[
\begin{array}{crcl}
s  \colon\thinspace &\theta & \to & A^{1\!1} \\
  & P\otimes Q &\mapsto &PQ 
\end{array}\quad \textrm{and}\quad
\begin{array}{crcl}
m  \colon\thinspace & A^{1\!1} & \to & \theta \\
  &  P &\mapsto & Px_{k-1}\otimes 1 - P\otimes x_{{k}}
\end{array}
\]
The notation may seem confusing: \emph{m} is for \emph{m}erge and \emph{s} is for \emph{s}plit. We define
\begin{align*}
  F&:= 0 \longrightarrow \theta \stackrel{s} {\longrightarrow} A^{1\!1}q^{-1} \longrightarrow 0, \\
  F^{-1}&:= 0 \longrightarrow  A^{1\!1}q \stackrel{m}{\longrightarrow} \theta \longrightarrow 0.
\end{align*}
In $F^{-1}$ and $F$, $\theta$ is in $T$-degree $0$, $A^{1\!1}q^{-1}$ in $T$-degree $1$ and $A^{1\!1}q$ in $T$-degree $-1$.

Note that in the language of Section~\ref{sec:soergel-bimodules-1}, $m$ and $s$ are the maps induced by the foams given in Remark~\ref{rmk:diff1}.

The following lemma should be compared to \cite[Lemma 6.3]{MR3709661}.

\begin{lem}
  \label{lem:stab}
  The homology of the complexes $(\HH_*(A^{1}, M), d_T)$, $(\HH_*(A^{1\!1}, MI\otimes_{A^{1\!1}} F)[-1]_T[-1]_Hq^{-1}, d_T)$ and $(\HH_*(A^{1\!1}, MI\otimes_{A^{1\!1}} F^{-1})q^{1}, d_T)$ are isomorphic as triply graded $\SP{N}$-modules. Moreover we can choose the isomorphisms to commute with the extra-differentials $d^N$.  
\end{lem}

\begin{proof}
  We will deal with complexes carrying three different differentials: the Hochschild differential $d_H$, the topological differential $d_T$ and the additional differential $d^N$.

\[
C(A^{1\!1}) = C(A^1)\otimes_{\SP{N}} X
\]

where 
\[
\ensuremath{\vcenter{\hbox{
\begin{tikzpicture}[xscale =2]
  \node (X) at (-1,1) {$X:=$};
  \node (A) at (0,1) {$\SP{N}[x_k] \otimes \SP{N}[x_k]q^2$};
  \node (B) at (3,1) {$\SP{N}[x_k] \otimes \SP{N}[x_k]$};
  \draw[-to] (A)--(B) node [midway,above] {\tiny{$x_k\otimes 1 - 1 \otimes x_k$}};
  \draw[densely dotted, -to] (A)  .. controls +(0,-1) and +(0, -1) .. (B) node [midway,below] {$x_k^N\partial_{x_k}$};
\end{tikzpicture}
}}}
\]

Hence 
\[
C(A^{1\!1}) \otimes_{(A^{1\!1})^{\mathrm{en}}} (MI \otimes_{A^{1\!1}} F) \simeq C(A^{1}) \otimes_{(A^{1})^{\mathrm{en}}}( M \otimes_ {A^{1}}  (X\otimes_{\SP{N}[x_k]^{\mathrm{en}}} F)).
\]
where $B^{\mathrm{en}}$ denotes the algebra $B\otimes_RB^{\mathrm{opp}}$ for any $R$-algebra.
We focus on $X\otimes_{\SP{N}[x_k]^{\mathrm{en}}} F$. We have

\begin{align*}
  X\otimes_{\SP{N}[x_k]^{\mathrm{en}}} F \simeq
\NB{
\ensuremath{
\vcenter{\hbox{
\begin{tikzpicture}[yscale =1.3]
  \node (A)  at (0,1) {$\theta q^2$};
  \node (B) at (2,1)  {$A^{1\!1}q$};
  \node (C)  at (0,0) {$\theta $};
  \node (D) at (2,0) {$A^{1\!1}q^{-1}$};
  \draw[-to] (A)--(B) node [midway,above] {\tiny{$s$}};
  \draw[-to] (A)--(C) node [midway,left] {\tiny{$x_k\otimes 1 - 1\otimes x_k$}};
  \draw[-to] (B)--(D) node [midway,right] {\tiny{$0$}};
  \draw[-to] (C)--(D) node [midway,below] {\tiny{$s$}};
  \draw[densely dotted, -to] (A)  .. controls +(0.5,-0.50) and +(0.5,0.5) .. (C);
  \draw[densely dotted, -to] (B)  .. controls +(0.5,-0.5) and +(0.5,0.5) .. (D);
\end{tikzpicture}
}}}
}.
\end{align*}

In $\theta$, the elements $2(x_k\otimes 1 - 1\otimes x_k)$ and $(x_{k}- x_{k-1})\otimes 1 - 1\otimes (x_k - x_{k-1})$ are equal. Hence, we have the following exact sequence of bicomplexes (for $d_T$ and $d_H$) of graded $A^{1\!1}$-modules-$A^{1\!1}$:
\begin{align*}
  \begin{tikzpicture}[yscale=1.5, xscale=1.3]
  \node (O1)  at (5.5,1.5) {$0$}; 
  \node (Y1)  at (5.5,0.5) {$Y_1$}; 
  \node (FX)  at (5.5,-1.5) {$F\otimes_{\QQ[x_k]^{\mathrm{en}}}X$};
  \node (Y2)  at (5.5,-3.5) {$Y_2$};
  \node (O2)  at (5.5,-4.5) {$0$};  
 \begin{scope}
  \node (A1)  at (0,1) {$A^{1\!1}q^3 $}; 
  \node (B1) at (2,1)   {$A^{1\!1}q$}; 
  \node (C1)  at (0,0) {$0$};         
  \node (D1) at (2,0)   {$0$};        
  \draw[-to] (A1)--(B1) node [midway,above] {\tiny{$2(x_{k} - x_{k-1})$}};
  \draw[-to] (A1)--(C1) node [midway,left] {};
  \draw[-to] (B1)--(D1) node [midway,right] {};
  \draw[-to] (C1)--(D1) node [midway,below] {};
  \draw [dotted, rounded corners] (-0.5,-0.3) -- ++(3,0) -- ++(0,1.6) -- ++(-3,0) -- cycle;
  \draw [dotted] (2.5, 0.5) -- (Y1);
 \end{scope}
 \begin{scope}[yshift =0cm]
  \node (A2)  at (0,-1) {$\theta q^2$};   
  \node (B2) at (2,-1)  {$A^{1\!1}q$};  
  \node (C2)  at (0,-2) {$\theta $};   
  \node (D2) at (2,-2)  {$A^{1\!1}q^{-1}$};  
  \draw[-to] (A2)--(B2) node [midway,above] {\tiny{$s$}};
  \draw[-to] (A2)--(C2) node [midway,left] {\tiny{$x_k\otimes 1 - 1\otimes x_k$}};
  \draw[-to] (B2)--(D2) node [midway,right] {\tiny{$0$}};
  \draw[-to] (C2)--(D2) node [midway,below] {\tiny{$s$}};
  \draw[densely dotted, -to] (A2)  .. controls +(0.5,-0.50) and +(0.5,0.5) .. (C2);
  \draw[densely dotted, -to] (B2)  .. controls +(0.5,-0.5) and +(0.5,0.5) .. (D2);
  \draw [dotted, rounded corners] (-0.5,-2.3) -- ++(3,0) -- ++(0,1.6) -- ++(-3,0) -- cycle;
  \draw [dotted] (2.5, -1.5) -- (FX);
 \end{scope}
 \begin{scope}[yshift = 0cm]
  \node (A3)  at (0,-3) {$A'^{1\!1}q$};
  \node (B3) at (2,-3)  {$0$};
  \node (C3)  at (0,-4) {$\theta $};
  \node (D3) at (2,-4) {$A^{1\!1}q^{-1}$};
  \draw[-to] (A3)--(B3) node [midway,above] {};
  \draw[-to] (A3)--(C3) node [midway,left] {\tiny{$a\mapsto
  \begin{array}{l}
    a(x_{k}- x_{k-1})\otimes 1 \\- a\otimes (x_k - x_{k-1})
  \end{array}$}};
  \draw[-to] (B3)--(D3) node [midway,right] {\tiny{$0$}};
  \draw[-to] (C3)--(D3) node [midway,below] {\tiny{$s$}};
  \draw [dotted, rounded corners] (-0.5,-4.3) -- ++(3,0) -- ++(0,1.6) -- ++(-3,0) -- cycle;
  \draw [dotted] (2.5, -3.5) -- (Y2);
 \end{scope}
\draw[-to] (A1) .. controls +(-1,-1) and +(-1, 1) .. (A2) node[midway, left] 
{\tiny{$ \begin{array}{c} (x_{k}-x_{k-1})\otimes 1 \\ + 1\otimes (x_{k}-x_{k-1})  \end{array}$}};  
\draw[-to] (A2) .. controls +(-1,-1) and +(-1, 1) .. (A3) node[midway, left] {\tiny{$a\otimes b \mapsto a\sigma_{k-1}(b)$}};
\draw[-to] (C2) .. controls +(-1,-1) and +(-1, 1) .. (C3) node[midway, left] {\tiny{$2\id$}};
\draw[dashed, -to] (A3) .. controls +(1,1) and +(1, -1) .. (A2) node[near end, right] {\tiny{$\phi$}};
\draw[-to] (B1) .. controls +(1,-1) and +(1, 1) .. (B2) node[midway, right] {\tiny{$\id$}};
\draw[-to] (D2) .. controls +(1,-1) and +(1, 1) .. (D3) node[midway, right] {\tiny{$2\id$}};
\draw[->] (O1) -- (Y1);
\draw[->] (Y1) -- (FX);
\draw[->] (FX) -- (Y2);
\draw[->] (Y2) -- (O2);
\end{tikzpicture}
\end{align*}
where $A'^{1\!1}$ is equal to $A^{1\!1}$ as a $A^{1\!1}$-module and has a right $A^{1\!1}$-action twisted by the transposition $\sigma_{k-1}$ which exchanges $x_k$ and $x_{k-1}$. Note that this is \emph{not} a sequence of tri-complexes: it does not respect the differential $d^N$. The dotted arrows represents the part of $d^N$ appearing in $X$. In each topological degree, this sequence splits as a sequence of complexes (for the Hochschild differential) of 
\[
\phi(p(x_1, \dots, x_{k_2})x_{k-1}^i x_k^j) = p(x_1, \dots, x_{k_2}) x_{k-1}^i \otimes x_{k-1}^j.
\]
We now take the homology with respect with the Hochschild differentials\footnote{Here, we abuse a little bit the appellation Hochschild differential: in $Y_1$ and $Y_2$ the vertical arrows are part of the Hochschild differential, while the horizontal ones are part of the topological differential.}. It implies, that for each $i$ in $\NN$, that we have the sequence of graded complexes:
\begin{align*} 
0\longrightarrow \HH_i(C(A^{1}) \otimes_{(A^{1})^{\mathrm{en}}}( M \otimes_ {A^{1}} Y_1)) \longrightarrow &\HH_i(A^{1\!1}, MI\otimes F) \longrightarrow \\  \longrightarrow &\HH_i(C(A^{1}) \otimes_{(A^{1})^{\mathrm{en}}}( M \otimes_ {A^{1}} Y_2)) \longrightarrow 0.
\end{align*}
We have a short exact sequence of complexes of $A^{1\!1}$-modules-$A^{1\!1}$:
\[
0 \longrightarrow A'^{1\!1}q \xrightarrow{a\mapsto  a(x_k- x_{k-1})\otimes 1 -a \otimes (x_k- x_{k-1})}
\theta \stackrel{s}{\longrightarrow} A^{1\!1}q^{-1} \longrightarrow 0.
\]
This implies that  for all $i$, the complex $\HH_i(C(A^{1}) \otimes_{(A^{1})^{\mathrm{en}}}( M \otimes_ {A^{1}} Y_2))$  is homotopically trivial (with respect to $d_T$). 
It follows that: 
\begin{align*}
\HT(\HH_i(A^{1\!1}, MI\otimes F^{-1})) 
\simeq&  \HT(\HH_i(C(A^{1}) \otimes_{(A^{1})^{\mathrm{en}}}( M \otimes_ {A^{1}} Y_1))). 
\end{align*}
This implies that the part of the differential $d^N$ coming from $X$ is equal to $0$ in $\HT(\HH_i(A^{1\!1}, MI\otimes F^{-1}))$, and therefore that this isomorphism commute with $d^N$.
The complex $(\HH_i(C(A^{1}) \otimes_{(A^{1})^{\mathrm{en}}}( M \otimes_ {A^{1}} Y_1)), d_T)$ has the same homology as the complex $(\HH_i(C(A^{1}) \otimes_{(A^{1})^{\mathrm{en}}}( M \otimes_ {A^{1}} A^1q^1))[+1]_H, d_T)$ which itself is equal to $\HH_i(A^{1}, M) [+1]_Hq^1$.

The argument for $F^{-1}$ is similar, with the following short exact sequence of bi-complexes:
\begin{align*}
 \begin{tikzpicture}[yscale=1.5, xscale =1.3]
  \node (O1)  at (-3.5,1.5) {$0$}; 
  \node (Y1)  at (-3.5,0.5) {$Z_1$}; 
  \node (FX)  at (-3.5,-1.5) {$F^{-1}\otimes_{\QQ[x_k]^{\mathrm{en}}}X$};
  \node (Y2)  at (-3.5,-3.5) {$Z_2$};
  \node (O2)  at (-3.5,-4.5) {$0$};  
 \begin{scope}
  \node (A1)  at (0,1) {$A^{1\!1}q^3$}; 
  \node (B1) at (2,1)   {$A^{1\!1}q^3$}; 
  \node (C1)  at (0,0) {$0$};         
  \node (D1) at (2,0)   {$0$};        
  \draw[-to] (A1)--(B1) node [midway,above] {\tiny{$\Id$}};
  \draw[-to] (A1)--(C1) node [midway,left] {};
  \draw[-to] (B1)--(D1) node [midway,right] {};
  \draw[-to] (C1)--(D1) node [midway,below] {};
  \draw [dotted, rounded corners] (-0.5,-0.3) -- ++(3,0) -- ++(0,1.6) -- ++(-3,0) -- cycle;
  \draw [dotted] (-0.5, 0.5) -- (Y1);
 \end{scope}
 \begin{scope}[yshift =0cm]
   \node (A2)  at (0,-1)  {$A^{1\!1}q^3$}; 
   \node (B2) at (2,-1)   {$\theta q^2$};  
   \node (C2)  at (0,-2)  {$A^{1\!1}q$}; 
   \node (D2) at (2,-2)   {$\theta $};  
  \draw[-to] (A2)--(B2) node [midway,above] {\tiny{$m$}};
  \draw[-to] (A2)--(C2) node [midway,left] {\tiny{$0$}};
  \draw[-to] (B2)--(D2) node [midway,right] {\tiny{$x_k\otimes 1 - 1\otimes x_k$}};
  \draw[-to] (C2)--(D2) node [midway,below]  {\tiny{$m$}};
  \draw[densely dotted, -to] (A2)  .. controls +(0.5,-0.50) and +(0.5,0.5) .. (C2);
  \draw[densely dotted, -to] (B2)  .. controls +(-0.5,-0.5) and +(-0.5,0.5) .. (D2);
  \draw [dotted, rounded corners] (-0.5,-2.3) -- ++(3,0) -- ++(0,1.6) -- ++(-3,0) -- cycle;
  \draw [dotted] (-0.5, -1.5) -- (FX);
 \end{scope}
 \begin{scope}[yshift = 0cm]
  \node (A3)  at (0,-3) {$0$};
  \node (B3) at (2,-3)  {$A'^{1\!1}q$};
  \node (C3)  at (0,-4) {$A^{1\!1}q$};
  \node (D3) at (2,-4) {$\theta$};
  \draw[-to] (A3)--(B3) node [midway,above] {};
  \draw[-to] (A3)--(C3) node [midway,left] {}; 
  \draw[-to] (B3)--(D3) node [midway,right] 
{\tiny{$a\mapsto
  \begin{array}{l}
    (x_{k}- x_{k-1})a\otimes 1 \\- a\otimes (x_k - x_{k-1})
  \end{array}$}};
  \draw[-to] (C3)--(D3) node [midway,below] {\tiny{$m$}};
  \draw [dotted, rounded corners] (-0.5,-4.3) -- ++(3,0) -- ++(0,1.6) -- ++(-3,0) -- cycle;
  \draw [dotted] (-0.5, -3.5) -- (Y2);
 \end{scope}
\draw[-to] (A1) .. controls +(-1,-1) and +(-1, 1) .. (A2) node[midway, left] {\tiny{$2\id $}};
\draw[-to] (B2) .. controls +(1,-1) and +(1, 1) .. (B3) node[midway, right] {\tiny{$a\otimes b \mapsto a\sigma_{k-1}(b)$}};
\draw[-to] (C2) .. controls +(-1,-1) and +(-1, 1) .. (C3) node[midway, left] {\tiny{$\id$}};
\draw[dashed, -to] (B3) .. controls +(-1,1) and +(-1, -1) .. (B2) node[near end, left] {\tiny{$\phi$}};
\draw[-to] (B1) .. controls +(1,-1) and +(1, 1) .. (B2) node[midway, right] 
{\tiny{$a\mapsto
  \begin{array}{l}
    (x_{k}- x_{k-1})a\otimes 1 \\- a\otimes (x_k - x_{k-1})
  \end{array}$}};
\draw[-to] (D2) .. controls +(1,-1) and +(1, 1) .. (D3) node[midway, right] {\tiny{$\Id$}};
\draw[->] (O1) -- (Y1);
\draw[->] (Y1) -- (FX);
\draw[->] (FX) -- (Y2);
\draw[->] (Y2) -- (O2);
\end{tikzpicture}
\end{align*}

\end{proof}

\subsection{An example of computation }
\label{sec:an-example-comp}
In this section, we compute the homology of the positive uncolored Hopf link in the non-equivariant setting (i.~e.\ evaluating the variables $T_\bullet$ on $0$).

First remark (see Remark~\ref{rmk:diff1}) that the pre-differential given by the foam $F_{D,s\to_x s'}$ for  $x$ of type $(1,1, +)$ (resp. of type $(1,1,-)$)  is always surjective (resp. injective). This implies in particular that if an uncolored braid diagram $D$ has $n$ crossings, then the homological length (for the topological degree) of $H(\widehat{S}(D)$ is at most $n$. If $D$ contains both positive and negative crossings then $H(\widehat{S}(D)$ has length at most $n-1$.  

The complex of resolutions of the positive uncolored Hopf link is given by:
\[
\NB{\tikz[scale =0.65]{\begin{scope}
  \coordinate (A0) at (-1, 0.5);
  \coordinate (B0) at (-1, 1.5);
  \coordinate (C1) at (-0.3, 1);
  \coordinate (C2) at ( 0.3, 1);
  \coordinate (A2) at (1, 0.5);
  \coordinate (B2) at (1, 1.5);
  \coordinate (D2) at (-1, -1.5);
  \coordinate (E2) at (-1, -0.5);
  \coordinate (F2) at (-0.3, -1);
  \coordinate (F1) at (0.3, -1);
  \coordinate (D0) at (1, -1.5);
  \coordinate (E0) at (1, -0.5);
  \coordinate (OOT) at (2.5,1);
  \coordinate (OOB) at (2.5,-1);
  \draw (A0) arc (90:270:0.5);
  \draw (B0) arc (90:270:1.5) coordinate[pos =0.8] (left);
  \draw (A2) arc (90:-90:0.5);
  \draw (B2) arc (90:-90:1.5) coordinate[pos =0.8] (right);
  \draw[>-]  (A0)  .. controls +(0.3,0) and +(-0.3, 0) .. (C1);
  \draw[>-]  (B0)  .. controls +(0.3,0) and +(-0.3, 0) .. (C1);
  \draw[->-, thick] (C1) -- (C2) node[midway, above ] {$\scriptstyle{2}$};
  \draw[->]  (C2)  .. controls +(0.3,0) and +(-0.3, 0) .. (A2);
  \draw[->]  (C2)  .. controls +(0.3,0) and +(-0.3, 0) .. (B2);
  \draw[>-] (D0)  .. controls +(-0.3,0) and +( 0.3, 0) .. (F1);
  \draw[>-] (E0)  .. controls +(-0.3,0) and +( 0.3, 0) .. (F1);
  \draw[->-, thick] (F1) -- (F2) node[midway, above ] {$\scriptstyle{2}$} coordinate[pos=0.5] (F0);
  \draw[->]  (F2)  .. controls +(-0.3,0) and +( 0.3, 0) .. (D2);
  \draw[->]  (F2)  .. controls +(-0.3,0) and +( 0.3, 0) .. (E2);
  \draw[red, very thin, <-] (F0) -- + (0, -1) node[below] {$\scriptstyle{\{m_\lambda|\lambda \in T(2,N-1)\}}$}; 
    \draw[red, very thin, <-] (left) -- + (0, -1.5) node[below] {$\scriptstyle{\{x^i| i \in \{0,1\}\}}$}; 
        \draw[red, very thin, <-] (right) -- + (0, -1.5) node[below] {$\scriptstyle{\{x^i| i \in \{0,1\}\}}$}; 
\node[blue] at (0,0) {$\scriptstyle{[N][N+1][2]}$};
\end{scope}

\begin{scope}[xshift= 6.5cm, yshift = -3cm]
  \coordinate (A0) at (-1, 0.5);
  \coordinate (B0) at (-1, 1.5);
  \coordinate (C1) at (-0.3, 1);
  \coordinate (C2) at ( 0.3, 1);
  \coordinate (A2) at (1, 0.5);
  \coordinate (B2) at (1, 1.5);
  \coordinate (D2) at (-1, -1.5);
  \coordinate (E2) at (-1, -0.5);
  \coordinate (F2) at (-0.3, -1);
  \coordinate (F1) at (0.3, -1);
  \coordinate (D0) at (1, -1.5);
  \coordinate (E0) at (1, -0.5);
  \coordinate (OIL) at (-2.7,0);
  \coordinate (OIR) at (2.7,0);
  \draw (A0) arc (90:270:0.5);
  \draw (B0) arc (90:270:1.5);
  \draw (A2) arc (90:-90:0.5);
  \draw (B2) arc (90:-90:1.5) coordinate[pos =0.8] (right);
  \draw[->-]  (A0)  -- (A2);
  \draw[->-]  (B0)  -- (B2);
  \draw[>-] (D0)  .. controls +(-0.3,0) and +( 0.3, 0) .. (F1);
  \draw[>-] (E0)  .. controls +(-0.3,0) and +( 0.3, 0) .. (F1);
  \draw[->-, thick] (F1) -- (F2) node[midway, above ] {$\scriptstyle{2}$}
coordinate[pos=0.5] (F0);
  \draw[->]  (F2)  .. controls +(-0.3,0) and +( 0.3, 0) .. (D2);
  \draw[->]  (F2)  .. controls +(-0.3,0) and +( 0.3, 0) .. (E2);
  \draw[red, very thin, <-] (F0) -- + (0, -1) node[below] {$\scriptstyle{\{m_\lambda|\lambda \in T(2,N-1)\}}$}; 
        \draw[red, very thin, <-] (right) -- + (0, -1.5) node[below] {$\scriptstyle{\{x^i| i \in \{0,1\}\}}$}; 
\node[blue] at (0,0) {$\scriptstyle{[N][N+1]}$};
\end{scope}

\begin{scope}[xshift= 6.5cm, yshift = 3cm]
  \coordinate (A0) at (-1, 0.5);
  \coordinate (B0) at (-1, 1.5);
  \coordinate (C1) at (-0.3, 1);
  \coordinate (C2) at ( 0.3, 1);
  \coordinate (A2) at (1, 0.5);
  \coordinate (B2) at (1, 1.5);
  \coordinate (D2) at (-1, -1.5);
  \coordinate (E2) at (-1, -0.5);
  \coordinate (F2) at (-0.3, -1);
  \coordinate (F1) at (0.3, -1);
  \coordinate (D0) at (1, -1.5);
  \coordinate (E0) at (1, -0.5);
  \coordinate (IOL) at (-2.7,0);
  \coordinate (IOR) at (2.7,0);
  \draw (A0) arc (90:270:0.5);
  \draw (B0) arc (90:270:1.5);
  \draw (A2) arc (90:-90:0.5);
  \draw (B2) arc (90:-90:1.5) coordinate[pos=0.2] (right);
  \draw[->-]  (D0)  -- (D2);
  \draw[->-]  (E0)  -- (E2);  
  \draw[>-]  (A0)  .. controls +(0.3,0) and +(-0.3, 0) .. (C1);
  \draw[>-]  (B0)  .. controls +(0.3,0) and +(-0.3, 0) .. (C1);
  \draw[->-, thick] (C1) -- (C2) node[midway, below ] {$\scriptstyle{2}$}
  coordinate[midway] (C0);
  \draw[->]  (C2)  .. controls +(0.3,0) and +(-0.3, 0) .. (A2);
  \draw[->]  (C2)  .. controls +(0.3,0) and +(-0.3, 0) .. (B2);
  \draw[red, very thin, <-] (C0) -- + (0, 1) node[above] {$\scriptstyle{\{m_\lambda|\lambda \in T(2,N-1)\}}$}; 
        \draw[red, very thin, <-] (right) -- + (0, 1.5) node[above] {$\scriptstyle{\{x^i| i \in \{0,1\}\}}$}; 
\node[blue] at (0,0) {$\scriptstyle{[N][N+1]}$};
\end{scope}

\begin{scope}[xshift =13cm]
  \coordinate (A0) at (-1, 0.5);
  \coordinate (B0) at (-1, 1.5);
  \coordinate (C1) at (-0.3, 1);
  \coordinate (C2) at ( 0.3, 1);
  \coordinate (A2) at (1, 0.5);
  \coordinate (B2) at (1, 1.5);
  \coordinate (D2) at (-1, -1.5);
  \coordinate (E2) at (-1, -0.5);
  \coordinate (F2) at (-0.3, -1);
  \coordinate (F1) at (0.3, -1);
  \coordinate (D0) at (1, -1.5);
  \coordinate (E0) at (1, -0.5);
  \coordinate (IIT) at (-2.5,1);
  \coordinate (IIB) at (-2.5,-1);
  \draw (A0) arc (90:270:0.5);
  \draw (B0) arc (90:270:1.5);
  \draw (A2) arc (90:-90:0.5);
  \draw (B2) arc (90:-90:1.5);
  \draw[->-]  (A0)  -- (A2);
  \draw[->-]  (B0)  -- (B2);
  \draw[->-]  (D0)  -- (D2);
  \draw[->-]  (E0)  -- (E2);  
  \draw[red, very thin, <-] (B2) -- + (0, 1.5) node[above] {$\scriptstyle{\{m_\lambda|\lambda \in T(1,N-1)\}}$}; 
        \draw[red, very thin, <-] (A0) -- + (0, 1.5) node[above] {$\scriptstyle{\{m_\lambda| \lambda \in T(1,N-1)\}}$}; 
\node[blue] at (0,0) {$\scriptstyle{[N][N]}$};
\end{scope}

\draw[->] (OOB) -- (OIL);
\draw[->] (OOT) -- (IOL);
\draw[->] (OIR) -- (IIB);
\draw[->] (IOR) -- (IIT);}}
\]

In this picture, the red labels are meant to represent a basis. For a given diagram $\Gamma$, pick a $\Gamma$-tree-like foam and decorate it with the elements given by red sets. The blue quantum numbers are meant to represent the quantum dimension of the spaces.

Observe that the two spaces in the middle are isomorphic and that the pre-differential on the left are equal up to this isomorphism, and therefore have the same kernel. On the other hand, the pre-differential on the right are surjective. Taking care of the different grading shifts, this gives:
 \[
 H_{i}(\widehat{S}(D_n)) = H_{i}(S(D_n))q^{-Nn} =
\begin{cases}
  \QQ[N][N+1] q^{-2N+1} & \textrm{if $i =0$,}\\
  \QQ[N]q^{-(N+1)} & \textrm{if $i=1$,}\\
  0 & \textrm{else.}
\end{cases}
\]
\begin{rmk}
  \label{rmk:N12}
  \begin{enumerate}
  \item Note in particular that the symmetric uncolored homology of
    links is not trivial for $N=1$ (for which the corresponding
    polynomial invariant is always $1$). We expect that a simple
    combinatorial description of this homology is achievable since the
    polynomial invariant of MOY graph has an especially simple form for $N=1$
    (see Lemma~\ref{lem:N1symmetriceval}).
  \item For $N=2$, this gives a different homology than the Khovanov homology and odd Khovanov homology.
  \end{enumerate}
\end{rmk}




\appendix

\section{Quantum link invariants and representations of $U_q(\gll_N)$.}
\label{sec:quant-link-invar}
In this appendix we provide details about the relation between the graphical MOY calculi defined in Section~\ref{sec:moy-calculi} and the representations of $U_q(\gll_N)$. The aim is to give explicit definitions the Reshetikhin--Turaev functors which associate with a MOY graph an intertwiner of $U_q(\gll_N)$-representations. We will define two such functors: one sends an edge labeled $k$ onto the identity of $\Lambda^k_qV_q$ where $V_q$ is the standard representation of $U_q(\sll_N)$. This yields the \emph{exterior} MOY calculus\footnote{This is the ``classical'' MOY calculus. MOY stands for Murakami--Ohtsuki--Yamada who gave the identities described in Section~\ref{sec:moy-calculi} in \cite{MR1659228}. }. The other one associates with such an edge the identity of $\mathrm{Sym}_q^kV_q$. This yields what we call the \emph{symmetric} MOY calculus. 

\subsection{The quantum group $U_q(\gll_N)$}
\label{sec:quant-groups-u_qsll}

In this section we follow the presentation of $U_q(\gll_N)$ given by Tubbenhauer--Vaz--Wedrich in \cite{tubbenhauer2015super}.

\begin{dfn} 
Let $N$ be a positive integer. The \emph{quantum general linear algebra} $U_q(\gll_N)$ is
the associative, unital $\CC(q)$-algebra generated by $L_i$, $L_i^{-1}$, $F_j$ and $E_j$, with $1\leq i \leq N$ and $1\leq j \leq N-1$ 
subject to the relations
\begin{gather*}
L_iL_j = L_jL_i,  \quad L_iL_i^{-1}= L_i^{-1}L_i= 1, \\ 
L_iF_i = q^{-1}F_iL_i, \quad L_{i+1}F_i = qF_iL_{i+1}, \quad L_iE_i = qE_iL_i, \quad L_{i+1}E_i = q^{-1}E_iL_{i+1}, \\ 
L_jF_i = F_iL_j, \quad L_jE_i = E_iL_j \qquad \textrm{for $j\neq i, i+1$,} \\
E_iF_j - F_jE_i = \delta_{ij}\frac{L_iL_{i+1}^{-1} - L_i^{-1}L_{i+1}}{q-q^{-1}},  \\
[2]_qF_i F_j F_i = F_i^2F_j + F_j F_i^2   \qquad \textrm{if $|i − j| = 1$,}  \\
[2]_qE_i E_j E_i = E_i^2 E_j + E_j E_i^2   \qquad \textrm{if $|i − j| = 1$,}  \\
E_i E_j = E_j E_i, \quad F_iF_j = F_jF_i \qquad  \textrm{if $|i − j| > 1$.}
\end{gather*}
\end{dfn}


\begin{prop} Defining $\Delta:U_q(\gll_N)\to U_q(\gll_N)^{\otimes 2}$, $S:U_q(\gll_N)^{\mathrm{op}} \to  U_q(\gll_N)$ and $\epsilon: U_q(\gll_N)\to \CC(q)$ to be the $\CC(q)$ algebra maps defined by:
  \begin{align*}
    &\Delta(L_i^{\pm 1}) = L_i^{\pm1}\otimes L_i^{\pm1}        & \quad &S(L_i^{\pm1})= L_i^{\mp1}               & \quad & \epsilon(L_i^{\pm 1}) =1  &  \\
    &\Delta(F_i) = F_i\otimes 1 + L_i^{-1}L_{i+1}\otimes F_i  & \quad &S(F_i) = - L_iL_{i+1}^{-1}F_i           & \quad  & \epsilon(F_i)=0 & \\
    &\Delta(E_i)= E_i\otimes 1 + L_iL_{i+1}^{-1}\otimes E_i   &\quad & S(E_i) = - E_iL_i^{-1}L_{i+1}           &  \quad &\epsilon(E_i)=0&\\
  \end{align*}
  endow $U_q(\gll_N)$ 
with a structure of Hopf algebras with antipode. Furthermore the category of finite-dimensional $U_q(\gll_N)$-modules is braided.
\end{prop}

\begin{prop}
We define $V_q$ to be an $N\!$th dimensional $\CC(q)$-vector space with basis $(b_i)_{i=1,\dots, N}$. The formulas:
\begin{align*}
&L_i b_i = qb_i,& &L_i^{-1} b_{i} =q^{-1}b_{i},&  &L_i^{\pm 1} b_j = b_j \quad j \neq i, &\\
&E_{i-1} b_i = b_{i-1}& &E_i b_{j} =0,&  &\textrm{if $i\neq j-1$,} & \\
&F_i b_i = b_{i+1}& &F_i b_{j} =0,&  &\textrm{if $i\neq j$}& 
\end{align*}
endow $V_q$ with a structure of $U_q(\gll_N)$-modules. 
\end{prop}

Following \cite{1701.02932} we now consider the tensor algebra $T^\bullet V_q$. This algebra is naturally graded and endowed with an action of $U_q(\gll_N)$ which preserve the grading (i. e. for every integer $a$, $T^aV_q$ is a $U_q(\gll_N)$-submodule of $T^\bullet V_q$.). We consider two two-sided ideals $E^2V_q$ and $S^2V_q$ inside this algebra $TV_q$:
\begin{align*}
EV_q &:= \langle qb_i\otimes b_j - b_j\otimes b_i| \textrm{for $i<j$} \rangle \qquad\textrm{and}   \\
SV_q &:= \langle b_m\otimes b_m,  b_i\otimes b_j + qb_j\otimes b_i| \textrm{for all $m$ and for $i<j$} \rangle.
\end{align*}
Since these two ideals are homogeneous the quotient
\[
\Lambda^\bullet_qV_q := T^\bullet V_q / SV_q \quad \textrm{and} \quad \mathrm{Sym}^\bullet_qV_q := T^\bullet V_q / EV_q 
\]
inherits a grading from $T^\bullet V_q$. One easily checks that $E^2V_q$ and  $S^2V_q$ are stable under the action of $U_q(\gll_N)$ over $T^\bullet V_q$. This implies that for any non-negative integer $a$, $\Lambda^a_qV_q$ and $\mathrm{Sym}^a_qV_q$ inherit structure of $U_q(\gll_N)$-modules. One can show that for every  integer $a$, $\Lambda^a_qV_q$ and $\mathrm{Sym}^a_qV_q$ are simple modules. The image of a pure tensor $x_1\otimes \dots \otimes x_a$ is denoted by
\[
x_1\wedge \dots \wedge x_a\, \textrm{in $\Lambda^a_qV_q$}\quad \textrm{and by} \quad  x_1\otimes \dots \otimes x_a\, \textrm{in $\mathrm{Sym}^a_qV_q$}.
\]

The $\CC(q)$ vector space $\Lambda^a_qV_q$ has dimension $\left(\begin{smallmatrix} N \\a\end{smallmatrix}\right)$ and is spanned by the vectors
\[
(b_{i_1}\wedge b_{i_2}\wedge \dots \wedge b_{i_a})_{1\leq i_1< i_2< \dots < i_a\leq N}.
\]
If $1\leq i_1< i_2< \dots < i_a\leq N$ and $I= \{i_1,\dots i_a\}$, we write $b_I = b_{i_1}\wedge b_{i_2}\wedge \dots \wedge b_{i_a}$. 

\begin{dfn}
  Let $X$ be a set.
  A \emph{multi-subset  of $X$} is a map $Y: X\to \NN$. If $\sum_{x\in X}Y(x)< \infty$, the multi-subset $Y$ is said to be \emph{finite} and the sum is its \emph{cardinal} (denoted by $\#Y$). If $x$ is an element of $X$, the number $Y(x)$ is the \emph{multiplicity of $x$ in $Y$}. 
\end{dfn}

The $\CC(q)$ vector space $\mathrm{Sym}^a_qV_q$ has dimension $\left(\begin{smallmatrix} N+a-1 \\a\end{smallmatrix}\right)$ and is spanned by the vectors
\[
(b_{i_1}\otimes b_{i_2}\otimes \dots \otimes b_{i_a})_{1\leq i_1\leq i_2\leq \dots \leq i_a\leq N}.
\]
If $1\leq i_1\leq i_2\leq \dots \leq i_a\leq N$ and $I$ is the multi-subset of $I_N:=\{1,\dots,N\}$  $\{i_1,\dots i_a\}$, we write $b'_I = b_{i_1}\otimes b_{i_2}\otimes \dots \otimes b_{i_a}$. 

\subsection{Exterior MOY calculus}
\label{sec:ext-moy-calc-1}
In this subsection we work in the full subcategory $U_q(\gll_N)$-$\mathsf{mod}_{\Lambda}$ of finite-dimensional $U_q(\gll_N)$-modules generated (as a monoidal category) by the modules $\Lambda^a_qV_q$ for $0\leq a\leq N$ and their duals. We define a few morphisms:
\[
\begin{array}{crcl}
\Lambda_{a,b}  :& \Lambda_q^a V_q\otimes \Lambda_q^b V_q  &\to     & \Lambda^{a+b}_q V_q \\
  & b_I \otimes b_J & \mapsto &
  \begin{cases}
q^{-|J<I|} b_{I\sqcup J} &   \textrm{if $I\cap J =\emptyset$,}    \\
0 & \textrm{else.}
  \end{cases}
\end{array}
\]
\[
\begin{array}{crcl}
Y_{a,b}  :& \Lambda^{a+b}_q V_q   &\to     & \Lambda_q^a V_q\otimes \Lambda_q^b V_q  \\
  & b_K  &\mapsto &  \displaystyle{\sum_{I\sqcup J = K}  q^{|I<J|} b_I\otimes  b_J}
\end{array}
\]
\[
\begin{array}{crcl}
\stackrel{\leftarrow}{\cup}_{a}  :& \CC(q)   &\to     & \Lambda_q^a V_q\otimes (\Lambda_q^a V_q)^*  \\
  & 1  &\mapsto &  \displaystyle{\sum_{\#I=a}  b_I\otimes  b_I^*}
\end{array}
\]
\[
\begin{array}{crcl}
\stackrel{\leftarrow}{\cap}_{a}  :& (\Lambda_q^a V_q)^*\otimes \Lambda_q^a V_q    &\to     & \CC(q) \\
  &  f\otimes x &\mapsto &  f(x)
\end{array}
\]
\[
\begin{array}{crcl}
\stackrel{\rightarrow}{\cup}_{a}  :& \CC(q)   &\to     & (\Lambda_q^a V_q)^*\otimes \Lambda_q^a V_q  \\
  & 1  &\mapsto &  \displaystyle{\sum_{\#I=a}  q^{-|I<I_N| +|I_N<I| }b_I^*\otimes  b_I}
\end{array}
\]
\[
\begin{array}{crcl}
\stackrel{\rightarrow}{\cap}_{a}  :& \Lambda_q^a V_q\otimes( \Lambda_q^a V_q)^*    &\to     & \CC(q) \\
  &  b_I \otimes b_J^{*} &\mapsto & q^{|I<I_N| - |I_N<I|} \delta_{IJ}
\end{array}
\]
We should explain what $|J<I|$ and $|I<J|$ mean here. If $A$ and $B$ are two  subsets of an ordered set $C$, we define
\[
|A<B|:= \#\{(a,b)\in A\times B| a<b\}
\]


Using the Reshetikhin--Turaev functor one can interpret any MOY graph as a morphism in $U_q(\gll_N)$-$\mathsf{mod}_{\Lambda}$. Using identities~(\ref{eq:extcrossplus}) and (\ref{eq:extcrossminus}) we can extend this interpretation to MOY graph with crossings.

\subsection{Symmetric MOY calculus}
\label{sec:symm-moy-calc-1}

In this subsection we work in the full subcategory $U_q(\gll_N)$-$\mathsf{mod_S}$ of finite-dimensional $U_q(\gll_N)$-modules generated (as a monoidal category) by the modules $\mathrm{Sym}^a_qV_q$ for $a$ in $\NN$ and their duals. We define a few morphisms:
\[
\begin{array}{crcl}
\lambda_{a,b}  :& \mathrm{Sym}_q^a V_q\otimes \mathrm{Sym}_q^b V_q  &\to     & \mathrm{Sym}^{a+b}_q V_q \\
  & b'_I \otimes b'_J & \mapsto &
 q^{|J<I|} b'_{I\sqcup J}
\end{array}
\]
\[
\begin{array}{crcl}
Y_{a,b}  :& \mathrm{Sym}^{a+b}_q V_q   &\to     & \mathrm{Sym}_q^a V_q\otimes \mathrm{Sym}\Lambda_q^b V_q  \\
  & b'_K  &\mapsto &  \displaystyle{\sum_{I\sqcup J = K}  [I,J]_q q^{-|J<I|} b'_I\otimes  b'_J}
\end{array}
\]
\[
\begin{array}{crcl}
\stackrel{\leftarrow}{\cup}_{a}  :& \CC(q)   &\to     & \mathrm{Sym}_q^a V_q\otimes (\mathrm{Sym}_q^a V_q)^*  \\
  & 1  &\mapsto &  \displaystyle{\sum_{\#I=a}  q^{|I<J|} b'_I\otimes  (b'_I)^*}
\end{array}
\]
\[
\begin{array}{crcl}
\stackrel{\leftarrow}{\cap}_{a}  :& (\mathrm{Sym}_q^a V_q)^*\otimes \mathrm{Sym}_q^a V_q    &\to     & \CC(q) \\
  &  f\otimes x &\mapsto &  f(x)
\end{array}
\]
\[
\begin{array}{crcl}
\stackrel{\rightarrow}{\cup}_{a}  :& \CC(q)   &\to     & (\mathrm{Sym}_q^a V_q)^*\otimes \mathrm{Sym}_q^a V_q  \\
  & 1  &\mapsto &  \displaystyle{\sum_{\#I=a}  q^{-|I<I_N| +|I_N<I| }(b'_I)^*\otimes  b'_I}
\end{array}
\]
\[
\begin{array}{crcl}
\stackrel{\rightarrow}{\cap}_{a}  :& \mathrm{Sym}_q^a V_q\otimes( \mathrm{Sym}_q^a V_q)^*    &\to     & \CC(q) \\
  &  b'_I \otimes (b'_J)^{*} &\mapsto & q^{+|I<I_N| - |I_N<I|} \delta_{IJ}
\end{array}
\]


We should explain what $|J<I|$, $|I<J|$, and $[I,J]$ mean here. If $A$ and $B$ are two multi-subsets of an ordered set $X$, we define
\[
|A<B|:= \prod_{x<y \in X}A(x)B(y),
\]
\[
[A,B]=\prod_{x\in X} \qbin{A(x)}{B(x)}.
\]

Using the Reshetikhin--Turaev functor one can interpret any MOY graph as a morphism in $U_q(\gll_N)$-$\mathsf{mod_S}$. Using identities~(\ref{eq:symcrossplus}) and (\ref{eq:symcrossminus}) we can extend this interpretation to MOY graph with crossings. Note that this is \emph{not} consistent with the exterior MOY calculus: the braiding has been changed for its inverse.


\section{Koszul resolutions of polynomial algebras}
\label{sec:kosz-reosl-polyn}

In this appendix we recall the definition of the Koszul resolution of polynomial algebras. Then we describe a way to construct other differentials on the Koszul complex which anti-commute with the Koszul differential. 

For an introduction to Koszul resolution see \cite[Section 3.4]{MR1600246} and \cite{ KasselNoteKoszul,berger_lambre_solotar_2017}.


\subsection{Koszul resolutions}
\label{sec:koszul-resolutions}

Let $R$ be an unitary commutative ring and $V$ a free $R$-module of rank $k$.
Let us fix an ordered basis $(x_1,\dots x_k)$. We denote by $A$ the symmetric tensor algebra $SV$ and we will think of $A$ as the polynomial algebra $R[x_1, \dots, x_n]$. The algebra $\Lambda V$ is naturally graded (we speak of \emph{$H$-grading}): the non-zero elements of $V$ seen in $\Lambda V$ have $H$-degree equal to $1$. Let $C(A)$ be the $A$-module-$A$ $A\otimes \Lambda V \otimes A$. It inherits an $H$-grading from $\Lambda V$. We consider the following endomorphism of $A$-module-$A$ on $(C(A)_\bullet)$:
\[
\begin{array}{crcl}
d  \colon\thinspace & C(A) & \to & C(A) \\
  & 1\otimes v_1 \wedge \dots \wedge v_l\otimes 1 &\mapsto & \displaystyle{\sum_{i=1}^l (-1)^{i+1} \left( v_i \otimes  v_1 \wedge \dots \wedge \widehat{v_i}\wedge\dots \wedge v_l\otimes 1 \right.} \\ &&& \qquad \quad \displaystyle{\left.- 1\otimes  v_1 \wedge \dots \wedge \widehat{v_i}\wedge\dots \wedge v_l\otimes v_i\right).}
\end{array}
\]
\begin{lem}
  \label{lem:kozsul_resolution}
  The map $d$ is a differential on $C(A)_\bullet$, and $(C(A)_\bullet, d)$ is a projective resolution of $A$ as an $A$-module-$A$. The complex $C(A)_\bullet$ is called the \emph{Koszul resolution of $A$}.
\end{lem}

\begin{proof}
  If $V$ has dimension 1, then $A = R[x_1]$. It is easy to check, that the short exact sequence
\[
0\longrightarrow A\otimes A \xrightarrow{x_1\otimes1 - 1\otimes x_1} A\otimes A \stackrel{m}{\longrightarrow} A \longrightarrow 0
\]is exact. This proves that
\[ 0\longrightarrow A\otimes A \xrightarrow{x_1\otimes1 - 1\otimes x_1} A\otimes A 
\] is a projective resolution of $A$.
If $V$ is $k$ dimensional, then $C(A)$ is by definition the complex 
\[C(R[x_1]) \otimes C(R[x_2]) \otimes \cdots \otimes C(R[x_k]),\]
and we deduce the result from the one-dimensional case.
\end{proof}

\begin{rmk}
  \label{rmk:gradings}
  The vector space $V$ could be graded, we speak of \emph{$q$-grading}. In this case we suppose that the basis $(x_1, \dots, x_k)$ is homogeneous. The $H$-grading of $C(A)$ is not influenced by the $q$-grading. On the contrary, $q$-grading of $V$ induces a $q$ grading on $A$. Hence if all the $x_i$'s have positive $q$-degree, we have:
\[
\rk^{R}_q A = \prod_{i=1}^k \frac{1}{1- q^{\deg_q{x_i}}}
\]
\end{rmk}

\subsection{An homotopy equivalence}
\label{sec:an-homot-equiv}

We consider the algebra $A:= A_{1^k}=R[x_1, \dots, x_k]$. For consistency with the rest of the paper, we set the indeterminate $x_i$'s to have $q$-degree equal to $2$. If $\listk{k}=(k_1, \dots, k_l)$ is a finite sequence of positive integers of level $k$, we set $A_{\listk{k}{}}:=A^{\mathfrak{S}_{\listk{k}{}}}$, where $\mathfrak{S}_{\listk{k}{}}$ is equal to $\mathfrak{S}_{k_1} \times \dots \times \mathfrak{S}_{k_l}$ and acts naturally on the indeterminates $x_1, \dots, x_k$. Note that the algebras $A_{\listk{k}}$ can be thought of as polynomial algebras as well. One only need to consider some appropriate elementary symmetric polynomials.  

Let us fix two finite sequences $\listk{k^1}= (k^1_1, \dots, k^1_{l_1})$ and $\listk{k}_2 = (k^1_1, \dots, k^1_{l_2})$ of positive integers of levels $k$. We suppose furthermore, that $\listk{k^2}$ is obtained from $\listk{k^1}$ by merging two of consecutive elements of $\listk{k}_1$. For instance:
\[
\listk{k^1} = (2,3,1,1,5,4)\quad \textrm{and} \quad \listk{k^2} = (2,4,1,5,4).
\]
For simplicity of notations we will actually suppose that $\listk{k^1}= (a,b)$ and $\listk{k^2} = (a+b)$. 
In what follows we will be interested in $A^1:= A_{\listk{k^1}}$ and  $A^2:= A_{\listk{k^2}}$.

Since $A^2$ is a sub-algebra of $A^1$, $A^1$ can be consider as a $A^2$-module-$A^1$. As a module-$A^1$ it is free of rank $1$, as a $A^2$-module, it is free of rank ${a+b} \choose a$. Hence it is both a projective module-$A^1$ and a projective $A^2$-module. This implies that 
\[ 
C^1_\bullet := A^1 \otimes_{A^1} C(A^1)_\bullet \simeq C(A^1)_\bullet
\quad \textrm{and} \quad 
C^2_\bullet:=C(A^2)_\bullet \otimes_{A^2} A^1 
\]
are both projective resolutions of $A^1$ as $A^2$-module-$A^1$. Hence we know that these two complexes are homotopic. We denote by $d^1$ (resp. $d^2$) the differential of $C^1_\bullet$ (resp. $C^2_\bullet$).

We want to give an explicit homotopy equivalence of complexes of $A^2$-module-$A^1$ $\varphi: C^2_\bullet \to C^1_\bullet$. 
We consider the vector space $V_1 := <f_1, \dots, f_a, g_1,\dots, g_b>_R$ and $V_2:=<e_1, \dots, e_{a+b}>$. The element $f_i$ (resp. $g_i$, resp. $e_i$) is meant to represent the $i$th elementary symmetric polynomials in the first $a$ variables (resp. the last $b$ variables, resp.  $a+b$ variables).

Thanks to some standard argument of homological algebra (see e.~g.~\cite[Chapter 1, Lemma 7.3]{MR672956} ) we know that if $\varphi$ is a chain map such that
\begin{align}\NB{
\begin{tikzpicture}
    \node (A) at (0,1.5) {$C^2_\bullet$};
    \node (B) at (2,1.5) {$A^1$};
    \node (C) at (0,0) {$C^1_\bullet$};
    \node (D) at (2,0) {$A^1$};
    \draw[-to] (A)--(B) node [midway,above] {$\pipi$};
    \draw[-to] (A)--(C) node [midway,left] {$\varphi$};
    \draw[-to] (B)--(D) node [midway,right] {$\mathrm{id}_{A^1}$};
    \draw[-to] (C)--(D) node [midway,below] {$\pipi$};
\end{tikzpicture} \label{eq:cdvarphi}
}
\end{align}
commutes, then $\varphi$ is an homotopy equivalence.
By definition we have:
\[
C_{\bullet}^1 := A^1\otimes \Lambda V_1 \otimes A^1 \quad \textrm{and} \quad C_{\bullet}^2 := A^2\otimes \Lambda V_2 \otimes A^1.
\]
Since $C_{\bullet}^2$ is a free $A^2$-module-$A^1$ with a basis given by elements of the form 
\[
1_{A^2}\otimes e_{i_1}\wedge \dots \wedge e_{i_l}\otimes 1_{A^1},
\]
with $0\leq l\leq a+b$ and $i_1< i_2 < \dots < i_l$, we only need to define $\varphi$ on these elements. For $l=0$, we define:
\[
\varphi(1_{A^2}\otimes 1_R \otimes 1_{A^1})= 1_{A^1}\otimes 1_R \otimes 1_{A^1}.
\]
For $l=1$, we define:
\[
\varphi(1_{A^2}\otimes e_i \otimes 1_{A^1})= \sum_{j=0}^i f_{j}\otimes g_{i-j} \otimes 1_{A^1} +  1_{A^1}\otimes f_{j} \otimes g_{i-j},
\]
with the convention that the $j$th elementary polynomial in $c$ variables is equal to $0$ whenever $j>c$.
For $l>1$ we set:
\[\varphi (1_{A^2}\otimes e_{i_1}\wedge \dots \wedge e_{i_l}\otimes 1_{A^1}) = \prod_{h=1}^l \varphi(1_{A^2}\otimes e_{i_h}\otimes 1_{A^1}) \]
and extend this map $A^2\otimes (A^1)^{\mathrm{opp}}$-linearly.
In the last formula, the space $A\otimes \Lambda V \otimes A$ is endowed with the algebra structures given by:
\[
(a_1 \otimes v_1 \otimes b_1)\cdot(a_2 \otimes v_2 \otimes b_2) = a_1a_2 \otimes v_1 \wedge v_2 \otimes b_1b_2
\]
Note however that the product is ordered and that $\varphi$ is \emph{not} an algebra morphism.



\begin{prop}
  \label{prop:carphi_works}
  The map $\varphi$ is a morphism of complexes of $A^2$-modules-$A^1$ such that the diagram (\ref{eq:cdvarphi}) commutes.
\end{prop}

\begin{proof}
  The diagram (\ref{eq:cdvarphi}) obviously commutes thanks to the definition of $\varphi(1_{A^2}\otimes 1_R \otimes 1_{A^1})$. It remains to show that $\varphi$ is indeed a chain map. Thanks to the $A^2\otimes (A^1)^{\mathrm{opp}}$-linearity, it is enough to consider elements of the form
\[
1_{A^2}\otimes e_{i_1}\wedge \dots \wedge e_{i_l}\otimes 1_{A^1},
\]
Let us consider the case $l=1$. We have
\[\varphi (d^2(1\otimes e_i \otimes 1)) = e_i\otimes 1_R \otimes 1 - 1\otimes 1_R \otimes e_i\]
and 
\begin{align*}
  d^1(\varphi(1\otimes e_i \otimes 1))   
=&\sum_{j=0}^i (f_j g_{i-j}\otimes 1_R \otimes 1 - f_j \otimes 1_R \otimes g_{i-j}) \\
&\quad  
+\sum_{j=0}^i ( f_j \otimes 1_R \otimes g_{i-j} - 1 \otimes 1_R\otimes f_jg_{i-j}) \\
= &\sum_{j=0}^i (f_j g_{i-j}\otimes 1_R \otimes 1 - 1 \otimes 1_R \otimes f_jg_{i-j})  \\
= & e_i\otimes 1_R \otimes 1 -  1\otimes 1_R\otimes e_i  \\
\end{align*}

If $l>1$, we have:
\begin{align*}
 & d^1(\varphi (1\otimes e_{i_1}\wedge \dots \wedge e_{i_l}\otimes 1))=  
   d^1\left(\prod_{h=1}^l \varphi (1\otimes e_{i_h}\otimes 1)\right) \\ 
& =  \sum_{k=1}^l(-1)^{k+1}\left(\prod_{h=1}^{k-1} \varphi(1\otimes e_{i_h} \otimes 1)\right)\cdot d^1(\varphi((1\otimes e_{i_k} \otimes 1)) \left(\prod_{h=k+1}^{l} \varphi(1\otimes e_{i_h} \otimes 1)\right) \\
& =  \sum_{k=1}^l(-1)^{k+1}\left(\prod_{h=1}^{k-1} \varphi(1\otimes e_{i_h} \otimes 1)\right)\cdot (\varphi(d^2(1\otimes e_{i_k} \otimes 1)) \left(\prod_{h=k+1}^{l} \varphi(1\otimes e_{i_h} \otimes 1)\right) \\
&= \varphi(d^2 (1\otimes e_{i_1}\wedge \dots \wedge e_{i_l}\otimes 1)).\qedhere
\end{align*}

\end{proof}

\subsection{An additional differential}
\label{sec:an-addit-diff}

\begin{notation}
  \label{dfn:extra-diff}
  Suppose $D: A\to A$ is a derivation on $A$ which lets $A_{\listk{k}}$ stable for any finite sequence of positive $\listk{k}$ integers of level $k$. 
We consider the endomorphisms $\delta_{\listk{k}}$ of $C(A_{\listk{k}})$ given by:
\begin{align}
\begin{array}{crcl}
\delta_{\listk{k}}  \colon\thinspace & C(A_{\listk{k}}) & \to & C(A_{\listk{k}}) \\
  & 1\otimes v_1 \wedge \dots \wedge v_l\otimes 1 &\mapsto & \sum_{i=1}^l 
(-1)^{i+1} D(v_i) \otimes  v_1 \wedge \dots \wedge \widehat{v_i}\wedge\dots \wedge v_l\otimes 1  
\end{array} \label{eq:DN}
\end{align}
Both these maps are $H$-homogeneous of degree $-1$. For simplifying notations we denote the maps on $C(A^1)$ and $C(A^2)$ by $\delta^1$ and $\delta^2$.

If we fix a positive integer $N$ and if $R=\SP{N}$, an example of such a derivation is given by
  \[
  \begin{array}{crcl}
D^N   \colon\thinspace & A_{1^k} & \to & A_{1^k} \\
    & P(x_1,\dots,x_k)  &\mapsto & \sum_{i=1}^k \prod_{j=1}^N (x_i - T_j) \partial_{x_i}P(x_1, \dots, x_k).
  \end{array}
  \] 
The maps defined with $D^N$ by the formula~(\ref{eq:DN}) are denoted $d^N_{\listk{k}}$. These are precisely the differentials considered in Section~\ref{sec:an-extra-diff}.
\end{notation}

\begin{lem}
  \label{lem:deltadiff-commute-with-d}
  For $i=1,2$, the map $\delta^i$ anti-commutes with $d^i$ and is a differential on $C(A^i)$.
\end{lem}
\begin{proof}
  The fact that $\delta^1$ and $\delta^2$ are differentials follows as usual from the signs which forces every square to anti-commute. To check that $d^i$ and $\delta^i$ anti-commute is an easy computation. We assume $i=1$. We have:
  \begin{align*}
&    \delta^1(d^1 (1\otimes v_1 \wedge \dots \wedge v_l\otimes 1)) = \\ &\qquad\sum_{1\leq i_1 <i_2 \leq l} (-1)^{i_1+i_2+2} 
( v_{i_1} D(v_{i_2}) \otimes v_1 \wedge \dots \wedge \widehat{v_{i_1}} \wedge \dots \wedge \widehat{v_{i_2}} \wedge \dots \wedge v_l \otimes 1 \\
& \qquad \qquad \qquad \qquad \qquad \qquad \quad -  D(v_{i_2}) \otimes v_1 \wedge \dots \wedge \widehat{v_{i_1}} \wedge \dots \wedge \widehat{v_{i_2}} \wedge \dots \wedge v_l \otimes v_{i_1} \\
&\qquad \qquad \qquad \qquad \qquad \qquad \quad- D(v_{i_1}) v_{i_2} \otimes v_1 \wedge \dots \wedge \widehat{v_{i_1}} \wedge \dots \wedge \widehat{v_{i_2}} \wedge \dots \wedge v_l \otimes 1 \\
&\qquad \qquad  \qquad \qquad \qquad \qquad \quad+ D(v_{i_2}) \otimes v_1 \wedge \dots \wedge \widehat{v_{i_1}} \wedge \dots \wedge \widehat{v_{i_2}} \wedge \dots \wedge v_l \otimes v_{i_1})
  \end{align*}
and
  \begin{align*}
&    d^1(\delta^1( 1\otimes v_1 \wedge \dots \wedge v_l\otimes 1)) = \\ &\qquad\sum_{1\leq i_1 <i_2 \leq l} (-1)^{i_1+i_2+2} 
( D(v_{i_1}) v_{i_2} \otimes v_1 \wedge \dots \wedge \widehat{v_{i_1}} \wedge \dots \wedge \widehat{v_{i_2}} \wedge \dots \wedge v_l \otimes 1 \\
& \qquad \qquad \qquad \qquad \qquad \qquad \quad -  D(v_{i_1}) \otimes v_1 \wedge \dots \wedge \widehat{v_{i_1}} \wedge \dots \wedge \widehat{v_{i_2}} \wedge \dots \wedge v_l \otimes v_{i_2} \\
&\qquad \qquad \qquad \qquad \qquad \qquad \quad- v_{i_1}D(v_{i_2}) \otimes v_1 \wedge \dots \wedge \widehat{v_{i_1}} \wedge \dots \wedge \widehat{v_{i_2}} \wedge \dots \wedge v_l \otimes 1 \\
&\qquad \qquad  \qquad \qquad \qquad \qquad \quad+ D(v_{i_1}) \otimes v_1 \wedge \dots \wedge \widehat{v_{i_1}} \wedge \dots \wedge \widehat{v_{i_2}} \wedge \dots \wedge v_l \otimes v_{i_2}) \\
& \qquad \qquad  \qquad \qquad \qquad \qquad \qquad \qquad\qquad =-\delta^1(d^1(1\otimes v_1 \wedge \dots \wedge v_l\otimes 1))
\qedhere  \end{align*}
\end{proof}

We define $\eta: C^2_\bullet \to C^1_\bullet$ to be the $A^2\otimes (A^1)^{\mathrm{opp}}$-linear map defined by:
\[
  \eta( 1\otimes e_{i_1}\wedge \dots \wedge e_{i_l}\otimes 1) = D
  \otimes \id_{\Lambda V} \otimes \id_{A^1} \left( \varphi( 1\otimes e_{i_1}\wedge
  \dots \wedge e_{i_l}\otimes 1) \right)
\]
Note that we have:
\[
\eta (1\otimes e_{i_1}\wedge \dots \wedge e_{i_l}\otimes 1) = \sum_{k=1}^l\left(\prod_{h=1}^{k-1}\varphi (1\otimes e_{i_h}\otimes 1) \right) \eta (1\otimes e_{i_k}\otimes 1) \left(\prod_{h=k+1}^{l}\varphi (1\otimes e_{i_h}\otimes1 )\right),
\]
This follows from the fact that $D$ is a derivation. We can write a (complicated) explicit formula for $\eta$:

 \begin{align*}
 \eta (1\otimes e_{i_1}\wedge \dots \wedge e_{i_l}\otimes 1) = \sum_{\substack{j_1+j'_1 =i_1\\ \cdots \\j_l+j'_l= i_l}}\sum_{k=0}^l\sum_{\substack{A \sqcup B = \{1, \dots, l\}, |A|=k\\ A=\{a_1< \cdots < a_k\} \\ B=\{b_1< \cdots < b_{k-l} \} }} (-1)^{|B<A|} \\ \qquad \qquad D(f_{j_{a_1}}\cdots  f_{j_{a_k}})\otimes  g_{j'_{a_1}}\wedge \cdots \wedge g_{j'_{a_k}} \wedge f_{j_{b_1}}\wedge \cdots \wedge f_{i_{b_{l-k}}} \otimes  g_{j'_{b_1}} \cdots g_{j'_{a_{l-k}}},
 \end{align*}
where $|B<A| := \#\{(b,a) \in A\times B | b<a\}$.
\begin{lem}
  \label{lem:dN-and-phi-commute}
  We have the following identity:
  \[
\varphi \circ \delta^2 - \delta^1 \circ \varphi = \eta \circ d_2 + d_1 \circ \eta.
\]
\end{lem}

\begin{proof}

Suppose $l=1$. On the one hand, we have
\begin{align*}
  (\varphi \circ \delta^2 - \delta^1 \circ \varphi)(1\otimes e_i \otimes 1) =&
D(e_i)\otimes 1 \otimes 1\\ & - \sum_{j=0}^i \left(f_jD(g_{i-j})\otimes 1 \otimes 1 + D(f_j)\otimes 1 \otimes g_{i-j}\right) \\ 
=&  \sum_{j=0}^i D(f_j)g_{j-i} \otimes 1 \otimes 1 +  f_j D(g_{j-i} )\otimes 1 \otimes 1 \\ &-  \sum_{j=1}^i f_jD(g_{i-j})\otimes 1 \otimes 1 + D_N(f_j)\otimes 1 \otimes g_{i-j}  \\ 
=&  \sum_{j=0}^iD(f_j)g_{j-i} \otimes 1 \otimes 1  -  \sum_{j=0}^i D(f_j)\otimes 1 \otimes g_{i-j}.
\end{align*}
On the other hand, we have:
\begin{align*}
\eta \circ d^2 (1\otimes e_i \otimes 1 ) &= \eta(e_i \otimes 1 \otimes 1) - \eta (1 \otimes 1 \otimes e_i)
\\&= 0
\end{align*}
and 
\begin{align*}
d^1 \circ \eta(1\otimes e_i \otimes 1) &= d^1\left(\sum_{j=1}^i D(f_j)\otimes g_{i_j} \otimes 1 \right)\\ 
&= \sum_{j=1}^i D(f_j)g_{i_j}\otimes 1 \otimes 1- D(f_j)\otimes 1 \otimes g_{i-j} \\
&=(\varphi \circ \delta^2 - \delta^1 \circ \varphi)(1\otimes e_i \otimes 1).
\end{align*}
Suppose now that $l>1$. We have:
\begin{align*}
 & (\varphi \circ \delta^2 - \delta^1 \circ \varphi)(1\otimes e_{i_1}\wedge \dots \wedge e_{i_l}\otimes 1) \\ 
 &\qquad = \sum_{h=1}^l(-1)^{h+1}\prod_{k=1}^{h-1} \varphi(1\otimes e_{i_k}\otimes 1) (\varphi \circ \delta^2 - \delta^1 \circ \varphi)(1\otimes e_{i_h}\otimes 1) \\ &\qquad \qquad \qquad \cdot  \prod_{k=h+1}^{l} \varphi(1\otimes e_{i_k}\otimes 1)\\
& \qquad= \sum_{h=1}^l(-1)^{h+1}\prod_{k=1}^{h-1} \varphi(1\otimes e_{i_k}\otimes 1)  (d^1\circ \eta)(1\otimes e_{i_h}\otimes 1) \\ &\qquad \qquad \qquad \cdot \prod_{k=h+1}^{l} \varphi(1\otimes e_{i_k}\otimes 1)\\
\end{align*}
On the other hand, we have:
\begin{align*}
&  \eta \circ d^2(1\otimes e_{i_1}\wedge \dots \wedge e_{i_l}\otimes 1) \\
\quad &= \eta \left(\sum_{h=1}^l (-1)^{h+1}\prod_{k=1}^{h-1} (1\otimes e_{i_k} \otimes 1) d^2 (1\otimes e_{i_h}\otimes 1) \prod_{k=h+1}^{l} (1\otimes e_{i_k} \otimes 1)\right) \\
\quad &=  \sum_{h=1}^l  \sum_{j=1}^{h-1}(-1)^{h+1} \prod_{k=1}^{j-1} \varphi(1\otimes e_{i_k} \otimes 1) \eta(1\otimes e_{i_j} \otimes 1) 
\\ &\quad \qquad \qquad \qquad \qquad\cdot \prod_{k=j+1}^{h-1} \varphi(1\otimes e_{i_k} \otimes 1)
 \varphi(d^2 (1\otimes e_{i_h}\otimes 1)) \prod_{k=h+1}^{l} (1\otimes e_{i_k} \otimes 1) \\
&+ \sum_{h=1}^l  \sum_{j=h+1}^{l} (-1)^{h+1} \prod_{k=1}^{h-1} \varphi(1\otimes e_{i_k} \otimes 1) 
  \varphi(d^2 (1\otimes e_{i_h}\otimes 1))\\ 
&\quad \qquad \qquad \qquad \qquad\cdot \prod_{k=h+1}^{j-1} (1\otimes e_{i_k} \otimes 1) \eta(1\otimes e_{i_j} \otimes 1) \prod_{k=j+1}^{l} \varphi(1\otimes e_{i_k} \otimes 1) \\
&=  \sum_{h=1}^l \sum_{j=1}^{h-1} (-1)^{h+1} \prod_{k=1}^{j-1} \varphi(1\otimes e_{i_k} \otimes 1) \eta(1\otimes e_{i_j} \otimes 1) 
\\ &\quad \qquad \qquad \qquad \qquad\cdot \prod_{k=j+1}^{h-1} \varphi(1\otimes e_{i_k} \otimes 1) d^1(\varphi (1\otimes e_{i_h}\otimes 1)) \prod_{k=h+1}^{l} (1\otimes e_{i_k} \otimes 1) \\
&+ \sum_{h=1}^l \sum_{j=h+1}^{l}(-1)^{h+1}\prod_{k=1}^{h-1} \varphi(1\otimes e_{i_k} \otimes 1) 
  d^1(\varphi(1\otimes e_{i_h}\otimes 1)) 
\\ &\quad \qquad \qquad \qquad \qquad\cdot\prod_{k=h+1}^{j-1} (1\otimes e_{i_k} \otimes 1) \eta(1\otimes e_{i_j} \otimes 1) \prod_{k=j+1}^{l} \varphi(1\otimes e_{i_k} \otimes 1)
\end{align*}
and
\begin{align*}
& d^1\circ \eta (1\otimes e_{i_1}\wedge \dots \wedge e_{i_l}\otimes 1) \\
&\qquad \qquad = d^1 \left( \sum_{h=1}^l(-1)^{h+1}\prod_{k=1}^{h-1}\varphi (1\otimes e_{i_k}\otimes ) \eta (1\otimes e_{i_h}\otimes ) \prod_{k=h+1}^{l}\varphi (1\otimes e_{i_k}\otimes 1 ) \right) \\
&\qquad \qquad = - \eta \circ d^2((1\otimes e_{i_1}\wedge \dots \wedge e_{i_l}\otimes 1) \\
&\qquad \qquad \quad  +  \sum_{h=1}^l(-1)^{h+1}\prod_{k=1}^{h-1}\varphi (1\otimes e_{i_h}\otimes ) d^1(\eta (1\otimes e_{i_h}\otimes )) \prod_{k=h+1}^{l}\varphi (1\otimes e_{i_k}\otimes 1)  \\
& \qquad \qquad = - \eta \circ d^2(1\otimes e_{i_1}\wedge \dots \wedge e_{i_l}\otimes 1) +  \sum_{h=1}^l(-1)^{h+1}\prod_{k=1}^{h-1}\varphi (1\otimes e_{i_k}\otimes ) \\ 
& \qquad \qquad \quad \cdot(\varphi \circ \delta^2 - \delta^1 \circ \varphi)(1\otimes e_{i_h}\otimes 1)  \prod_{k=h+1}^{l}\varphi (1\otimes e_{i_k}\otimes 1 )  \\
& \qquad \qquad =- \eta \circ d^2(1\otimes e_{i_1}\wedge \dots \wedge e_{i_l}\otimes 1)  
\\ & \qquad \qquad \quad +   (\varphi \circ \delta^2 - \delta^1 \circ \varphi)(
1\otimes e_{i_1}\wedge \dots \wedge e_{i_l}\otimes 1)
\end{align*}
\end{proof}


\section{A pinch of algebraic geometry}
\label{sec:pinch-algebr-geom}

Algebraic geometry has been a very useful guideline for the definition of the (exterior) Khovanov--Rozansky homologies.  The exterior $\sll_N$-invariant of the unknot  labeled by $k$ is equal to $\qbinil{N}{k}$, which is the graded Euler characteristic of the cohomology ring of $\mathrm{Gr}_\CC(k,N)$ the Grassmannian variety of $k$-spaces in $\CC^N$ up to an overall grading shift. Indeed the Frobenius algebra associated with the unknot labeled by $k$ in the Khovanov--Rozansky homology is isomorphic to $H^*(\mathrm{Gr}_\CC(k,N))$. This can be extended to an equivariant\footnote{This motivates the term ``equivariant'' $\sll_N$-homology.} setting see \cite{MR2580427, RW1}.

The symmetric $\sll_N$ invariant of the unknot with a label $k$ is equal to $\qbinil{N+k-1}{k}$. This is (up to an overall grading shift) the graded Euler characteristic of $\mathrm{Gr}_\CC(k,k+N-1)$. However, this does not seem to be the correct point of view when categorifying the symmetric $\sll_N$-invariant\footnote{This remark is due to François Costantino.}. Indeed, in an equivariant setting, one expect to have a natural action of $\QQ[T_1, \dots T_N]$ or $\QQ[T_1, \dots T_N]^{\mathfrak{S}_N}$  on the Frobenius algebra associated with the unknot labeled by $k$. Instead, we believe that it is better to consider the space $S^k\left(\mathrm{Gr}_\CC(1,N)\right)$ of collections of $k$ lines (counted with multiplicity) in $\CC^N$:
\[
S^k\left(\mathrm{Gr}_\CC(1,N)\right) = \left.\left(\mathrm{Gr}_\CC(1,N)\right)^k\right/ \mathfrak{S}_k 
\]
There is a natural of $GL_N$ on this space and its (non-equivariant) cohomology ring is isomorphic\footnote{Note that if $N=2$, these two varieties are actually isomorphic.} to that of $\mathrm{Gr}_\CC(k,k+N-1)$:

\begin{thm}[{\cite[Theorem 2.4]{ELIZONDO200267}}]
  \label{thm:ES}
  There is a birrational map $f: \left(\mathrm{Gr}_\CC(1,N)\right)^k \to \mathrm{Gr}_\CC(k,k+N-1)$ such that the correspondence induced by the graph $\Gamma_f$ of $f$ induces an isomorphism
\[
(\Gamma_f)^\bullet: 
H^{\bullet}(\mathrm{Gr}_\CC(k,k+N-1), \QQ) \to
H^\bullet\left(S^k\left(\mathrm{Gr}_\CC(1,N)\right), \QQ\right) 
\]
of graded $\QQ$-algebras.
\end{thm}

As stated before, we are interested in the equivariant version of the cohomology ring of $S^k(\mathrm{Gr}(1,N))$. 

The group $G:=GL_N$ acts naturally on $\mathrm{Gr}(1,N)$ and diagonally on $S^k(\mathrm{Gr}(1,N))$. Since $H^\bullet(BG, \QQ)\simeq\QQ[T_1,\dots T_N]^{\mathfrak{S}_N}=: \SP{N}$ as a ring, the equivariant cohomologies of both $\mathrm{Gr}(1,N)$ and $S^k(\mathrm{Gr}(1,N))$ have structures of graded $\SP{N}$-algebras (the degrees of the variables $T_\bullet$ are all equal to $2$ ).  
There exists a presentation of $H^\bullet_{G}(\mathrm{Gr}(1,N))$:

\begin{lem}[{\cite[Lecture 3, Example 1.2]{FultonLecture}}]
  \label{lem:equivCPN}
  The $\SP{N}$-algebra $H^\bullet_{G}(\mathrm{Gr}(1,N),\QQ)$ is isomorphic to 
$
 \SP{N}[x]\left/J_{N,1}\right.$, where $x$ has degree $2$ and $J_{N,1}$ is the ideal of $\SP{N}[x]$ generated by $\prod_{i=1}^N(x-T_i)$.
\end{lem}

It follows immediately that $H^\bullet_G(\mathrm{Gr}(1,N),\QQ)$ is a free graded $\SP{N}$-module and that the family $(1, x, \dots, x^{N-1})$ forms an homogeneous $\SP{N}$-basis of this space. The space $M_{N,k}$ considered in Section~\ref{sec:one-quotient-two} is isomorphic to
\[
\mathrm{Sym}^k\left(H^\bullet_G(\mathrm{Gr}(1,N), \QQ)\right).
\]
We believe that this is the same as
\[
H^\bullet_G\left(\mathrm{Sym}^k(\mathrm{Gr}(1,N)), \QQ)\right)
\]
but could not locate such a statement in the literature.   

\bibliographystyle{alphaurl}
\bibliography{biblio}

\end{document}